\documentclass[12pt]{amsart}

\makeatletter

\def\chaptermark#1{}

\def\chapter{%
  \if@openright\cleardoublepage\else\clearpage\fi
  \thispagestyle{plain}\global\@topnum\z@
  \@afterindenttrue \secdef\@chapter\@schapter}

\def\@chapter[#1]#2{\refstepcounter{chapter}%
  \ifnum\c@secnumdepth<\z@ \let\@secnumber\@empty
  \else \let\@secnumber\thechapter \fi
  \typeout{\chaptername\space\@secnumber}%
  \def\@toclevel{0}%
  \ifx\chaptername\appendixname \@tocwriteb\tocappendix{chapter}{#2}%
  \else \@tocwriteb\tocchapter{chapter}{#2}\fi
  \chaptermark{#1}%
  \addtocontents{lof}{\protect\addvspace{10\p@}}%
  \addtocontents{lot}{\protect\addvspace{10\p@}}%
  \@makechapterhead{#2}\@afterheading}
\def\@schapter#1{\typeout{#1}%
  \let\@secnumber\@empty
  \def\@toclevel{0}%
  \ifx\chaptername\appendixname \@tocwriteb\tocappendix{chapter}{#1}%
  \else \@tocwriteb\tocchapter{chapter}{#1}\fi
  \chaptermark{#1}%
  \addtocontents{lof}{\protect\addvspace{10\p@}}%
  \addtocontents{lot}{\protect\addvspace{10\p@}}%
  \@makeschapterhead{#1}\@afterheading}
\newcommand\chaptername{Chapter}

\def\@makechapterhead#1{\global\topskip 7.5pc\relax
  \begingroup
  \fontsize{\@xivpt}{18}\bfseries\centering
    \ifnum\c@secnumdepth>\m@ne
      \leavevmode \hskip-\leftskip
      \rlap{\vbox to\z@{\vss
          \centerline{\normalsize\mdseries
              \uppercase\@xp{\chaptername}\enspace\thechapter}
          \vskip 3pc}}\hskip\leftskip\fi
     #1\par \endgroup
  \skip@34\p@ \advance\skip@-\normalbaselineskip
  \vskip\skip@ }
\def\@makeschapterhead#1{\global\topskip 7.5pc\relax
  \begingroup
  \fontsize{\@xivpt}{18}\bfseries\centering
  #1\par \endgroup
  \skip@34\p@ \advance\skip@-\normalbaselineskip
  \vskip\skip@ }
\def\appendix{\par
  \c@chapter\z@ \c@section\z@
  \let\chaptername\appendixname
  \def\thechapter{\@Alph\c@chapter}}

\newcounter{chapter}

\newif\if@openright

\makeatother

\usepackage{amsmath}
\usepackage{tikz-cd}
\usepackage{mathtools}
\usepackage{amsfonts}
\usepackage{graphicx}
\usepackage{verbatim}
\usepackage{textcomp}\usepackage{amssymb}
\usepackage{cite}
\usepackage{color}
\usepackage[all]{xy}
\usepackage{graphicx}
\usepackage{color}
\usepackage{hyperref}
\usepackage{soul}
\usepackage[normalem]{ulem}
\hypersetup{
    colorlinks,
    citecolor=black,
    filecolor=black,
    linkcolor=black,
    urlcolor=black
}
\usepackage{tikz}
\usepackage[all]{xy}
\usepackage{graphicx}
\usetikzlibrary{backgrounds}
\newtheorem{theorem}{Theorem}[section]
\newtheorem{lemma}[theorem]{Lemma}
\newtheorem{proposition}[theorem]{Proposition}
\newtheorem{corollary}[theorem]{Corollary}

\allowdisplaybreaks
\newtheorem{notation}[theorem]{Notation}

\usepackage{amsthm}

\newtheorem*{hr}{Holomorphic Rigidity of Teichm\"uller space}

\DeclareMathOperator*{\essinf}{ess\,inf}

\theoremstyle{introtheorem}
\newtheorem{introtheorem}{Theorem}

\theoremstyle{introconjecture}

\theoremstyle{definition}
\newtheorem{definition}[theorem]{Definition}
\newtheorem{remark}[theorem]{Remark}

 \newenvironment{siubochner}
{{\it Proof of Siu's Bochner formula, Theorem~\ref{siubochner}.}}
{\hfill $\Box$ \\}

\newenvironment{proof:pluriharmonic}
{{\it Proof of Theorem~\ref{pluriharmonic}.}}
{\hfill $\Box$ \\}

\numberwithin{equation}{section}

\newtheorem*{theorem*}{Theorem}
{{\sc Proof of Lemma~\ref{tri1}.}}%
{{\qed} \\}

\newenvironment{proofTheorem2}%
{{\it Proof of Theorem~\ref{theorem:holomorphic}.}}%
{\hfill $\Box$ \\}

\newenvironment{proofTheorem3}%
{{\it Proof of Theorem~\ref{theorem:sameconclusion}.}}%
{\hfill $\Box$ \\}

\newenvironment{proofTheorem4}%
{{\it Proof of Theorem~\ref{theorem:buildings}.}}%
{\hfill $\Box$ \\}

\newenvironment{proofTheorem5}%
{{\it Proof of Theorem~\ref{theorem:teichmuller}.}}%
{\hfill $\Box$ \\}

{{\it Proof of Corollary~\ref{rigiditytheorem}.}}%
{\hfill $\Box$ \\}

\newcommand{\R}{\mathbb R}
\newcommand{\Z}{\mathbb Z}
\newcommand{\RR}{\mathcal R}
\newcommand{\DD}{\mathcal D}

\newcommand{\Pb}{\mathbb P}
\newcommand{\N}{\mathbb N}
\newcommand{\C}{\mathbb C}
\newcommand{\D}{\mathbb D}
\newcommand{\Sp}{\mathbb S}

\newcommand{\B}{\mathcal B}

\newcommand{\HH}{\mathcal H}
\newcommand{\NN}{\mathcal N}

\newcommand{\PP}{\mathcal P}

\newcommand{\CC}{\mathcal C}

\newcommand{\LL}{\mathcal L}

\newcommand{\domain}{M}

\newcommand{\Hyp}{\mathbb  H}

\newcommand{\Pn}{\mathsf P(n,\mathbb R)}
\newcommand{\Sn}{\mathsf S(n,\R)}
\newcommand{\Mn}{\mathsf M(n,\R)}
\newcommand{\Gl}{\mathsf{GL}(n,\mathbb R)}
\newcommand{\So}{\mathsf{SO}(n)}
\newcommand{\On}{\mathsf{O}(n)}
\newcommand{\An}{\mathsf{A}(n)}
\newcommand{\Nn}{\mathsf{N}(n)}
\newcommand{\Gsf}{{\mathsf G}_{\xi}}
\newcommand{\Nsf}{{\mathsf N}_{\xi}}

\newcommand{\DDD}{\widetilde{\bar \D^*}}
\newcommand{\DDDD}
{\bar \D^* \rightarrow \DDD \times_{\rho} \tilde X}

\newcommand{\Ej}{\mathsf E_j}
\newcommand{\Ei}{\mathsf E_i}

\title{Infinite energy maps and rigidity}

\author[Daskalopoulos]{Georgios Daskalopoulos}
\address{Department of Mathematics \\
                 Brown Univeristy \\
                 Providence, RI}
\email{daskal@math.brown.edu}
\author[Mese]{Chikako Mese}
\address{Johns Hopkins University\\
Department of Mathematics\\
Baltimore, MD}
\email{cmese@math.jhu.edu}

\begin{document}

\thanks{
GD supported in part by NSF DMS-2105226, CM supported in part by NSF DMS-2005406.}
\maketitle

\begin{abstract} 
We extend Siu's and Sampson's celebrated rigidity results to  non-compact domains. More precisely, 
let $M$ be a smooth quasi-projective variety with universal cover $\tilde M$ and let $\tilde X$ be a symmetric space of non-compact type, a locally finite Euclidean building or the Weil-Petersson completion of the Teichm\"uller space of a surface of genus $g$ and $p$ punctures with $3g-3+p>0$.      Under suitable assumptions on a homomorphism $\rho: \pi_1(M) \rightarrow  \mathsf{Isom}(\tilde X)$,  we show that there exists a $\rho$-equivariant pluriharmonic map
$\tilde u: \tilde M \rightarrow \tilde X$ of possibly infinite energy. In the case when the target is  K\"ahler  and  $\mathsf{rank}(d \tilde u) \geq 3$ at some point,  $\tilde u$  is  holomorphic or conjugate holomorphic. This  builds on previous important work by \cite{jost-zuo} and \cite{mochizuki-memoirs}. We also extend these results to the case when the target is a Riemannian manifold with sectional curvature bounded from above by a negative constant. 
\end{abstract}

\vspace*{0.3in}

\begin{center}
{\sc Introduction}
\end{center}

In his influential paper \cite{siu1}, Y.-T.~Siu proved   that harmonic maps between K\"ahler manifolds of complex dimension $\geq 2$ are pluriharmonic, provided that the domain manifold is compact and  the target  non-positively curved in a strong sense.  This result partially settled a conjecture by Yau made a few years earlier.  Siu’s ideas included an extension of the Bochner's formula and were further developed by others.  Notably, Sampson \cite{sampson} showed that a harmonic map from  a compact K\"ahler manifold to a Riemannian manifold of  nonpositive Hermitian sectional curvature is pluriharmonic.  The compactness  of the domain  can be replaced by  completeness and finite volume, provided that there exists a finite energy map (cf. \cite{siu2} and \cite{corlette}).   However, in many situations,  it may be difficult or even impossible to prove that a finite energy  maps exists.  The goal of this manuscript is  to extend Siu's and Sampson's work  by   proving  the  existence of possibly infinite energy pluriharmonic or even holomorphic  maps between K\"ahler manifolds.  

Infinite energy  harmonic maps between  manifolds  previously appeared the work of Lohkamp and Wolf.  Lohcamp \cite{lohkamp}  proved  the existence of a harmonic map in a given homotopy class of maps between two non-compact manifolds, provided that  a certain simplicity condition is satisfied. The most important case  is when the domain  is metrically a product  near infinity.  Wolf \cite{wolf} studied  harmonic maps of infinite energy  when the domain is a nodal Riemann surface and applied this study to describe degenerations of surfaces in the Riemann moduli space.

A few years later, Jost and Zuo (cf. \cite{jost-zuo}, \cite{zuo}) sketched a proof of the existence of  infinite energy pluriharmonic maps from  non-compact K\"ahler manifolds to symmetric spaces and  buildings. 
In a remarkable paper, Mochizuki  \cite{mochizuki-memoirs} rigorously constructed pluriharmonic  maps into the symmetric space    $GL(r, \C)/ U(r)$. More precisely, Mochizuki constructed pluriharmonic metrics on flat vector bundles.  These metrics correspond to harmonic maps  by the Donaldson-Corlette theorem (cf.~\cite{donaldson}, \cite{corlette}).  
Mochizuki  brought into the picture a new Bochner formula (cf.~ \cite[Proposition 21.42]{mochizuki-memoirs}) for harmonic bundles that can be used to bypass the potential obstruction coming from the non-triviality of the normal bundle of the divisor at infinity.

In this paper, we prove the existence of pluriharmonic maps without assuming the existence of a finite energy maps nor the triviality assumption.  One of the contributions of this paper is to broaden Mochizuki's Bochner formula for  harmonic maps.   Applying these Bochner formulas, we prove the following theorems:
\begin{introtheorem} \label{theorem:pluriharmonic} 
Let $M$, $\tilde X$ and $\rho:\pi_1(M) \rightarrow \tilde X$ satisfy the following:
\begin{itemize}
\item $M$ is a smooth quasi-projective variety with universal cover $\tilde M$
\item $\tilde X$ is a symmetric space of non-compact type
 \item  $\rho: \pi_1(M) \rightarrow  \mathsf{Isom}(\tilde X)$  is a proper homomorphism that  does not fix an unbounded closed convex strict subset of $\tilde X$.
 \end{itemize}
Then  there exists a unique  $\rho$-equivariant pluriharmonic map $\tilde u: \tilde M \rightarrow \tilde X$ of sub-logarithmic growth. \end{introtheorem}

Here, {\it sub-logarithmic growth} means that the map grows slower than a logarithmic function as it approaches the divisor, cf.~Definition~\ref{logdec} and Definition~\ref{sub-log}.

\begin{introtheorem} \label{theorem:holomorphic}
In Theorem~\ref{theorem:pluriharmonic}, if we additionally assume that  $\tilde X$ is a  K\"ahler manifold and the $\rho$-equivariant harmonic map $\tilde u:\tilde M \rightarrow \tilde X$ has $\mathsf{rank}_{\R}(d \tilde u) \geq 3$ at some point, then $\tilde u$ is holomorphic or conjugate holomorphic.   \end{introtheorem}

\begin{introtheorem} \label{theorem:sameconclusion}
The conclusions of Theorem~\ref{theorem:pluriharmonic}  and Theorem~\ref{theorem:holomorphic} hold if we replace $\tilde X$ by  a complete simply connected   K\"ahler manifold with strong nonpositive  curvature in the sense of Siu and sectional curvature bounded from above by $\kappa<0$.
Furthermore, 
the conclusion of Theorem~\ref{theorem:pluriharmonic}   holds if we replace $\tilde X$ by a Riemannian manifold $\tilde X$ of  Hermitian-negative curvature and sectional curvature bounded from above by $\kappa<0$.\end{introtheorem}

\begin{introtheorem} \label{theorem:buildings}
The conclusion of Theorem~\ref{theorem:pluriharmonic}  holds if we place $\tilde X$ by a locally finite Euclidean building without a Euclidean factor.  
\end{introtheorem}

\begin{introtheorem} \label{theorem:teichmuller}
The conclusion of Theorem~\ref{theorem:pluriharmonic}  holds if we replace $\tilde X$ by the Weil-Petersson completion of Teichm\"uller space and assume only that $\rho$ is a proper homomorphism (i.e.~we do not need to assume   that it fixes an unbounded closed convex strict subset of $\tilde X$).  
\end{introtheorem}

 Theorems~\ref{theorem:pluriharmonic}, \ref{theorem:holomorphic} and \ref{theorem:buildings}  were first stated in \cite{jost-zuo} and \cite{zuo} where some of the ideas originate. 
 One of the contributions of this paper is to provide complete proofs.  However, our main motivation of this paper is to develop techniques 
 to study rigidity questions associated to the mapping class group of surfaces.   Indeed, in a sequel to this paper, we will use Theorem~\ref{theorem:teichmuller} to provide a proof of the following generalization of the  holomorphic rigidity theorem of \cite{daskal-meseHR}:

 \begin{hr} \label{rigiditytheorem}
Let $ \Gamma$ and $\mathcal T$ denote the mapping class group and the Teichm\"uller space of an oriented  surface  $S$ of genus $g$ and $p$ punctures where $3g-3+p>0$.   Assume that $ \Gamma$ acts as a discrete automorphism group on a contractible K\"{a}hler manifold $\tilde{M}$ and there exists a finite cover  $M'$ of $\tilde M/\Gamma$ that is a smooth quasi-projective variety.
Then $\tilde{M}$ is equivariantly biholomorphic or conjugate biholomorphic to  the Teichm\"{u}ller space $\mathcal T$ of $S$.
\end{hr}

The above statement strengthens  \cite[Corollary 1.5]{daskal-meseHR}. In  \cite{daskal-meseHR}, we made the assumption that $M$ admits a compactification $\overline M$ as an algebraic variety such that   the codimension  of $\overline M \backslash M$ is $\geq 3$. This condition originates in \cite[Lemma 4]{jost-yau}, where they use this to prove the existence of an equivariant homotopy equivalence {\it of finite energy} between $\tilde M$ and $\mathcal T$. This is also the starting point  of \cite[Corollary 1.5]{daskal-meseHR}. With the theory developed in this paper, we will show that we can remove this unnecessary assumption.\\

\thanks{\bf Acknowledgements.} The authors would like to thank T. Mochizuki and Y. Siu for illuminating discussions. \\

\begin{center}
{\sc Outline of the paper}
\end{center}

In Chapter~\ref{chap:preliminaries}, we record the relevant background material needed in this paper.  Of particular importance is the characterization of certain class of  isometries in a space of non-positive curvature $\tilde X$.  To give a rough description of these isometries, first note that a semi-simple (i.e.~elliptic or hyperbolic) isometry $I$ with translation length $\Delta_I$ is characterized by the property that there exists a point $P \in \tilde X$ such that 
\[
d(P, I(P))=\Delta_I.
\]
  We want to allow for isometries that are more general than semi-simple ones.  This motivates the following:  An isometry $I:\tilde X \rightarrow \tilde X$ is said to have an {\it exponential decay to its translation length}  if  there exists a  geodesic ray $c:[0,\infty) \rightarrow \tilde X$ such that 
  \[
 d(c(t),I(c(t))) = \Delta_I(1 +be^{-at})
 \]
 for some constants $a,b>0$ (cf.~\cite[3.2.3]{jost-zuo}).
 We show that all isometries of $\tilde X$ considered  in  Theorem~\ref{theorem:pluriharmonic}, Theorem~\ref{theorem:sameconclusion}, and Theorem~\ref{theorem:buildings} satisfy this property.  We use this to obtain the energy growth estimates for the harmonic map  we construct later.

In Chapter~\ref{chap:bochner}, we present {\it the key tools of this paper:  the variations on the Siu's and Sampson's Bochner formulas}.  If  the domain K\"ahler manifold $M$ is compact, then Siu's or Sampson's Bochner formulas can be used to prove that harmonicity implies pluriharmonicity.    With $M$ as in Theorem~\ref{theorem:pluriharmonic}, neither Siu's nor Sampson's formulas are sufficient to overcome  the complication coming from the non-compactness of the domain.  Indeed,   they need to be rewritten in a different form  to get the  cancellations of the  residues  near the divisor.  
The new formulas, contained in 
Theorem~\ref{mochizukibochnerformula} and Theorem~\ref{mochizukibochnerformula'}, are  inspired by Mochizuki's work on harmonic bundles  \cite[Proposition 21.42]{mochizuki-memoirs}.  

Chapter~\ref{chap:Riemann surfaces}, Chapter~\ref{chap:Kahler surfaces}, and Chapter~\ref{chap:higher dimensions} comprise  the proofs  of Theorem~\ref{theorem:pluriharmonic}, Theorem~\ref{theorem:holomorphic} and Theorem~\ref{theorem:sameconclusion}.  We follow the proof of  \cite{mochizuki-memoirs} with adjustments made  on the account of  the  more general target spaces.  
A brief sketch of the proof is given below.

Let $M$, 
$\tilde X$ and $\rho:\pi_1(M) \rightarrow \tilde X$ be as in Theorem~\ref{theorem:pluriharmonic} with $\dim_\C M=n$.  We construct a pluriharmonic map by an induction on $n$ (cf.~Chapter~\ref{chap:higher dimensions}).
We can assume that  $M=\bar M \backslash \Sigma$ where $\bar M$ is a  smooth projective variety of complex dimension $\geq 2$ and  $\Sigma$ is a divisor with normal crossings. 
The inductive hypothesis implies that if
\begin{itemize}
\item $\bar Y \subset \bar M$ is a sufficiently ample divisor such that $\bar Y \cap \Sigma$ is a normal crossing,
\item $Y=\bar Y \backslash \Sigma$,
\item $\iota: Y \hookrightarrow M$ is the inclusion map, and 
\item $\rho_Y: \pi_1(Y) \rightarrow \tilde X$ is the composition $\rho_Y=\rho \circ \iota_*$,
\end{itemize}
then there exists a unique pluriharmonic $\rho_Y$-equivariant map $\tilde u_Y: \tilde Y \rightarrow \tilde X$.  
Letting $\{\bar Y_s\}$ be a collection of  such subvarieties whose union is $\bar M$,  we show that the map   $\tilde u:\tilde M \rightarrow \tilde X$ given by 
\begin{equation} \label{x983}
\tilde u=\tilde u_{Y_s} \mbox{ on $Y_s:=\bar Y_s \backslash \Sigma$}
\end{equation}  
is a well-defined, thereby proving the existence result for $\dim_\C M=n$.  To get this induction started, we first need  to show existence and uniqueness of   pluriharmonic maps in  dimensions $n=1$ and $n=2$.

First assume  $\dim_\C M=1$.  In other words, we construct harmonic maps from punctured Riemann surfaces of possibly infinite energy (cf.~Chapter~\ref{chap:Riemann surfaces}).  For the purpose of explaining this case, we assume that $\tilde X$ is a universal cover of a compact manifold $X$ and let  $\gamma:\Sp^1 \rightarrow X$ be  an arc length parameterization of a minimizing closed geodesic in $X$.  Denote the energy of $\gamma$ by $E^\gamma$.   A prototype map $v$  corresponding to $\gamma$ is defined on the  infinite cylinder $[0,\infty) \times \Sp^1$ and is   equal to $\gamma$ on slices $\{t\} \times \Sp^1$ for  all sufficiently large $t$.  More precisely, given any Lipschitz map $k: \Sp^1 \rightarrow  X$, let 
\[
v: [1, \infty)  \times \Sp^1 \rightarrow X
\]
be a  Lipschitz continuous map with 
\begin{eqnarray*}
v(1,\theta)=k(1,\theta), & & \forall \theta \in \Sp^1, 
\\
v(t,\theta)=\gamma(\theta), & & \forall t \in [1,\infty), \, \theta \in \Sp^1.
\end{eqnarray*}
The energy of $v$ on the finite cylinder $[0,T] \times \Sp^1$ is bounded from above by $T \cdot E^\gamma +C$ where $C$ is a constant independent of $T$.  Let 
\[
u_T:[0,T] \times \Sp^1 \rightarrow X
\]
be the Dirichlet solution on the finite cylinder $[0,T] \times \Sp^1$  with boundary values equal to $v$.  Since $u_T$  is an energy minimizing map, the energy of $u_T$ is bounded from above by $T \cdot E^\gamma +C$.   Since $\gamma$ is an energy minimizing map, the energy of $u_T$ in any finite cylinder $[0,t] \times \Sp^1$ for $0< t \leq T$ is bounded from below by $t \cdot E^\gamma$.  From this, we conclude that the energy of $u_T$ on any finite cylinder $[t_1,t_2] \times \Sp^1$ has an upper  bound independent of $T$; namely,  $(t_2-t_1) \cdot E^\gamma + C$.
 Thus,  the regularity theory for harmonic maps implies  that the family of maps $\{u_T\}_{T\geq 1}$ have a uniform Lipschitz bound on any compact subset of $[0,\infty) \times \Sp^1$.  By the Arzela-Ascoli theorem, there exists a sequence $T_i \rightarrow \infty$ such that  $u_{T_i}$ converges locally uniformly to a harmonic map $u:[0,\infty) \times \Sp^1 \rightarrow  X$.  We extend this  idea  to prove the existence of an equivariant  harmonic map from any punctured Riemann surface $\Sigma$.  In doing so, we replace the existence of a minimizing geodesic $\gamma$ with the assumption that the the isometry $\rho([\gamma])$ has  exponential decay to its translation length for any $[\gamma] \in \pi_1(M)$  representing a loop around a puncture.  We  note that the idea of the prototype map already appears in \cite{lohkamp}.  The main contribution of Chapter~\ref{chap:Riemann surfaces} to the theory of infinite energy harmonic maps  is the extension of the existence result to general NPC target spaces (i.e.~not necessarily manifolds) and, more importantly,  the  uniqueness result (cf.~Theorem~\ref{existence} and Theorem~\ref{uniqueness} respectively). The uniqueness is a crucial part of the inductive argument.  Indeed, it is used to show that the the pluriharmonic map in dimension $n$ (\ref{x983})    is well-defined.

Next, assume $\dim_\C M=2$.   In other words, we consider harmonic maps from a complex surface (cf.~Chapter~\ref{chap:Kahler surfaces}). The idea, as in the case of $\dim_\C M=1$, is to first construct a prototype map and then use it to construct a harmonic map (cf.~\cite[Sec.~3.2]{jost-zuo}).  By the assumption that the divisor $\Sigma \subset \bar M$ is a normal crossing divisor,  $\Sigma$ consists of a finite set of  irreducible smooth components $\Sigma_j$ intersecting pairwise at most isolated points. We  choose a canonical section $\sigma_j$ of the line bundle $\mathcal O(\Sigma_j)$ vanishing along $\Sigma_j$ and let $h_j$ be a Hermitian metric on $\mathcal O(\Sigma_j)$.  We endow the quasi-projective variety $M=\bar M \backslash \Sigma$ with a  Poincare type metric $g$ (cf.~Definition~\ref{poincarekahler} and Cornabla and Griffiths~\cite{cornalba-griffiths}) making $M$ into a K\"ahler manifold.    
Following \cite{mochizuki-memoirs}, we construct a fiber-wise harmonic map along the divisor at infinity.  In other words, we construct a prototype map $v$ as follows:  
on any punctured disk $\D^*$ where $\D$ is a disk transverse to $\Sigma_j$, let $v$ be  such that the restriction to $\D^*$  is a harmonic map.  Much like in  the construction of the harmonic map from the infinite cylinder $[0,\infty) \times \Sp^1$ (which is conformally equivalent to the punctured disk), we use the prototype map $v$ to construct a sequence of    Dirichlet solutions defined on compact sets that exhaust $M$.  This sequence contains a subsequence that  converges locally uniformly to a $\rho$-equivariant harmonic map $u: M \rightarrow \tilde X$.  By the Bochner method, we then  show that $u$ is in fact pluriharmonic (cf.~Theorem~\ref{theorem:pluriharmonicDim2}).  Technically, the most challenging part is the  analysis near the crossings of the divisor and dealing with the non-triviality of the normal bundle of the divisor.  In fact, even when the divisor is smooth,  an application of the Siu's or Sampson's Bochner formula does not immediately yield pluriharmonicity.   We need the variations of the Siu's and Sampson Bochner formulas provided in Chapter~\ref{chap:bochner}.  Integrating these formulas and applying integration by parts,  a miraculous cancellation (originally discovered by Mochizuki \cite[Section 6.3]{mochizuki-memoirs} in terms of flat vector bundles)  allows us to prove pluriharmonicity and even holomorphicity of $u$.   We can thus start the induction process to prove the existence in higher dimensions.

In Chapter~\ref{chap:Euclidean buildings}, we combine the harmonic map theory to Euclidean buildings developed in \cite{gromov-schoen} and to the Weil-Petersson completion of Teichm\"uller space developed   in \cite{daskal-meseDM}, \cite{daskal-meseER}, \cite{daskal-meseC1} and \cite{daskal-meseHR}  with  ideas in the previous chapters to prove Theorem~\ref{theorem:buildings}. 
The main difference in the proof of these theorems is that we must  deal with the singularities of harmonic maps that apriori may exist because of the singular nature of the target spaces. 
We also use the important fact that every element of the  isometry group for a Euclidean building (resp.~the mapping class group)  acts as an elliptic or hyperbolic element.  For the isometry group of a Euclidean building (resp.~the mapping class group), this follows from \cite[Proposition 2.1]{ramos-cuevas} (resp.~Thurston's classification of mapping classes and \cite{daskal-wentworth}). 
\newpage

\vspace*{0.1in}

\begin{center}
{\sc Table of Contents}
\end{center}

\begin{itemize}

\item[] Chapter~\ref{chap:preliminaries}  Preliminaries
\begin{itemize}
\item[] \S\ref{part1prelim} Definitions  \hfill p.~\pageref{part1prelim}
\item[] \S\ref{sec:Expdecay} Exponential decay of an isometry
\hfill p.~\pageref{sec:Expdecay}
\item[] \S\ref{sec:Expdecay2} Exponential decay of a commuting pair of isometries
\hfill p.~\pageref{sec:Expdecay2}
\end{itemize}
\item[] Chapter~\ref{chap:bochner}  Bochner Formulae
\begin{itemize} 
\item[] \S\ref{sec:KtoK} Between K\"ahler manifolds 
	\hfill p.~\pageref{sec:KtoK}
\item[] \S\ref{sec:KtoR} From K\"ahler manifolds to Riemannian manifolds 
	\hfill p.~\pageref{sec:KtoR}
\item[] \S\ref{boch-appendix} Appendix to Chapter~\ref{chap:bochner}
	\hfill p.~\pageref{boch-appendix}
\end{itemize}
\item[] Chapter~\ref{chap:Riemann surfaces}  Harmonic maps from Riemann surfaces
\begin{itemize}
\item[] \S\ref{disc} Infinite energy harmonic maps from punctured disk 
	\hfill p.~\pageref{disc}
\item[] \S\ref{sec:proofexistence} Proof of existence, Theorem~\ref{existence}	
	\hfill p.~\pageref{sec:proofexistence}
\item[] \S\ref{sec:proofuniqueness} Proof of uniqueness, Theorem~\ref{uniqueness} 
	\hfill p.~\pageref{sec:proofuniqueness}

\end{itemize}
\item[] Chapter~\ref{chap:Kahler surfaces} Harmonic maps from K\"ahler surfaces
\begin{itemize}
\item[] \S\ref{sec:neardivisor} Neighborhoods near the divisor \hfill p.~\pageref{sec:neardivisor}

\item[] \S\ref{sec:metric} Poincar\'e-type metric and its estimates \hfill p.~\pageref{sec:metric}

\item[] \S\ref{prototype} Prototype section \hfill p.~\pageref{prototype}

\item[] \S\ref{sec:ptm} Energy estimates of the prototype section \hfill p.~\pageref{sec:ptm}

\item[] \S\ref{sec:existence} Harmonic maps of possibly infinite energy \hfill p.~\pageref{sec:existence}

\item[] \S\ref{sec:pluriharmonicity} Pluriharmonicity \hfill p.~\pageref{sec:pluriharmonicity}

\item[] \S\ref{appendix:calculus} Appendix to Chapter~\ref{chap:Kahler surfaces} \hfill p.~\pageref{appendix:calculus}

\end{itemize}
\item[] Chapter~\ref{chap:higher dimensions} Harmonic maps in higher dimensions
\begin{itemize}

\item[] \S\ref{premsetup} Preliminary set-up
 \hfill p.~\pageref{premsetup}
\item[] \S\ref{proofpluriharmonic} 
 Proof of Theorem~\ref{theorem:pluriharmonic}, Theorem~\ref{theorem:holomorphic}, and Theorem~\ref{theorem:sameconclusion}
  \hfill p.~\pageref{proofpluriharmonic}
\end{itemize}

\item[] Chapter~\ref{chap:Euclidean buildings} Euclidean buildings and Teichm\"uller spaces
\begin{itemize}
\item[] \S\ref{sec:singularsets} Singular sets of harmonic maps
\hfill p.~\pageref{sec:singularsets}
\item[] \S\ref{proofofthm4} Proof of  Theorem~\ref{theorem:buildings} and Theorem~\ref{theorem:teichmuller} \hfill p.~\pageref{proofofthm4}
\end{itemize}
\item[] References \hfill p.~\pageref{references}
\end{itemize}

\chapter{Preliminaries} \label{chap:preliminaries}

\section{Definitions} \label{part1prelim}

\subsection{NPC spaces}

We refer to \cite{bridson-haefliger} for more details than provided here.

\begin{definition}
A curve $c:[a,b] \rightarrow \tilde X$ into a metric space is called a {\it geodesic} if $\mathsf{length}(l|_{[\alpha,\beta]})=d(c(\alpha),c(\beta))$ for any subinterval $[\alpha,\beta] \subset [a,b]$.  (Note that a identically constant map from an interval is a geodesic.). A metric space $\tilde X$ is a {\it geodesic space} if there exists a geodesic connecting every pair of points in $\tilde X$.
\end{definition}
\begin{definition} \label{def:NPC}
An \emph{NPC space} $\tilde X$ is a complete geodesic space that satisfies the 
 following condition:
  For any three points $P,R,Q \in \tilde X$ and an arclength parameterized geodesic $c:[0,l]
\rightarrow \tilde X$ with $c(0)=Q$ and $c(l)=R$, 
\[
d^2(P,Q_t) \leq (1-t) d^2(P,Q)+td^2(P,R)-t(1-t)d^2(Q,R)
\]
where $Q_{t}=c(tl)$. 
\end{definition}

 \begin{notation} \label{interpolationnotation}
\emph{It follows immediately from Definition~\ref{def:NPC} that, given $P, Q \in \tilde X$ and $t \in [0,1]$, there exists a {\it unique} point with distance from $P$ equal to $td(P,Q)$ and the distance from $Q$ equal to $(1-t)d(P,Q)$.  We denote this point by
\[
(1-t)P+tQ.
\]
}
\end{notation}

\begin{definition}
Let $\tilde X$ be an NPC space.  We say that two geodesics rays $c,c':[0,\infty) \rightarrow X$ are {\it equivalent} if there exists a constant $K$ such that $d(c(t), c'(t)) \leq K$ for all $t \in [0,\infty)$.  Denote the equivalence class of a geodesic ray $c$ by $[c]$.  The set $\partial \tilde X$ of boundary points of $\tilde X$ is the set of equivalence classes of geodesic rays.  
\end{definition}

\begin{definition} 
For any $\kappa>0$, we say an NPC space $\tilde X$ is CAT$(-\kappa)$ if for any  $P,Q,R$ and $Q_t$ as in Definition~\ref{def:NPC}, 
\begin{equation} \label{cat}
\cosh d(P,Q_t) \leq \frac{\sinh (1-t)\kappa d(P,Q)}{\sinh \kappa d(P,Q))} \cosh d(P,Q)+ \frac{\sinh t\kappa d(P,Q)}{\sinh \kappa d(P,Q))} \cosh d(P,R).
\end{equation}
\end{definition}
\begin{definition} \label{neg}
We say that an NPC space $\tilde X$ is {\it negatively curved} if, for any point $P \in \tilde X$, there exists $r>0$ and $\kappa>0$  such that every $P,Q,R \in B_r(P)$ satisfies (\ref{cat}).
\end{definition}

\subsection{Maps into NPC spaces}
\label{subsec:korevaarschoen}

In this paper, we consider harmonic maps into  NPC spaces.  In many cases, the target space $\tilde X$ is a smooth Riemannian manifold of non-positive sectional curvature.  In this case, the energy of a smooth map $u: \Omega \rightarrow \tilde X$ is 
\[
E^u =\int_\Omega |df|^2 d\mbox{vol}_g
\]
where $(\Omega,g)$ is a Riemannian domain and $d\mbox{vol}_g$ is the volume form of $\Omega$.  
We say $u:\tilde M \rightarrow \tilde X$ is {\it harmonic} if it is locally energy minimizing; i.e.~for any $p \in \tilde M$, there exists $r>0$ such that the restriction  $u|_{B_r(p)}$ minimizes energy amongst all maps $v:B_r(p) \rightarrow \tilde X$ with the same boundary values as $u|_{B_r(p)}$.

In other cases, we consider the target $\tilde X$ to be an NPC space, not necessarily smooth.  Below, we will review harmonic maps in this context, and refer to \cite{korevaar-schoen1} for more details than provided here.

Let $(\Omega, g)$ be a bounded Lipschitz Riemannian domain.   Let $\Omega_{\epsilon}$  be the set of points in $\Omega$ at a distance least $\epsilon$ from $\partial \Omega$.  Let $B_{\epsilon}(x)$ be a geodesic ball centered at $x$  and $S_{\epsilon}(x)=\partial B_{\epsilon}(x)$.   We say $f: \Omega \rightarrow X$ is  an $L^2$-map (or that $f \in L^2(\Omega,X)$) if 
\[
\int_{\Omega} d^2(f,P) d\mbox{vol}_g< \infty.
\]
For $f \in L^2(\Omega,X)$, define
\[e_{\epsilon}:\Omega \rightarrow {\bf R}, \ \ \ 
e_{\epsilon}(x) =
\left\{
\begin{array}{ll} 
\displaystyle{\int_{y \in S_{\epsilon}(x)}
\frac{d^2(f(x),f(y))}{\epsilon^2}
\frac{d\sigma_{x,\epsilon}}{\epsilon} }& x \in \Omega_{\epsilon}\\
0 & \mbox{otherwise}
\end{array}
\right.
\]
where $\sigma_{x,\epsilon}$ is the induced measure on  $S_{\epsilon}(x)$. We define a
family of  functionals 
\[
E^f_{\epsilon}:C_c(X) \rightarrow
{\bf R}, \ \ \ \ 
E^f_{\epsilon}(\varphi)=\int_{\Omega} \varphi e_{\epsilon} d\mbox{vol}_g.
\]
We say $f$ has finite energy (or that $f \in W^{1,2}(\Omega,X)$)
if
\[ E^f := \sup_{\varphi \in C_c(\Omega ), 0 \leq \varphi \leq 1}
\limsup_{\epsilon \rightarrow 0} E^f_{\epsilon}(\varphi) < \infty.
\] It is shown in \cite{korevaar-schoen1} that if $f$ has finite energy, the measures $
e_{\epsilon}(x)d\mbox{vol}_g$ converge weakly to a measure which is
absolutely continuous with respect to the Lebesgue measure.
Therefore, there exists a function $e(x)$, which we call the
energy density, so that $e_{\epsilon}(x)d\mbox{vol}_g \rightharpoonup
e(x)d\mbox{vol}_g$. In analogy to the case of smooth targets, we
write $|\nabla f|^2(x)$ in place of $e(x)$. In particular, the (Korevaar-Schoen) energy of $f$ in $\Omega$ is 

\[
E^f[\Omega] =
\int_{\Omega}|\nabla f|^2 d\mbox{vol}_g.
\]

\begin{definition}
We say a continuous map $u: \Omega \rightarrow \tilde X$ from a Lipschitz domain $\Omega$ is {\it harmonic} if it is locally energy minimizing;   more precisely, at each $p \in \Omega$, there exists a neighborhood $\Omega$ of $p$  so that all continuous comparison maps which agree with $u$ outside of this neighborhood have no less energy.
\end{definition}

For  $V \in \Gamma \Omega$ where $\Gamma \Omega$ is the set of
Lipschitz vector fields  on $\Omega$, $|f_*(V)|^2$  is similarly
defined.  The real valued $L^1$ function $|f_*(V)|^2$ generalizes
the norm squared on the directional derivative of $f$.  The
generalization of the pull-back metric is the continuous, symmetric, bilinear, non-negative and tensorial operator
\[
\pi_f(V,W)=\Gamma \Omega \times \Gamma \Omega \rightarrow
L^1(\Omega, {\bf R})
\]
where
\[
\pi_f(V,W)=\frac{1}{2}|f_*(V+W)|^2-\frac{1}{2}|f_*(V-W)|^2.
\]
  We refer to \cite{korevaar-schoen1} for more
details.

Let  $(x^1, \dots, x^n)$ be local coordinates of $(\Omega,g)$ and $g=(g_{ij})$, $g^{-1}=(g^{ij})$ be the local metric expressions.  Then energy density function of
$f$ can be written (cf.~\cite[(2.3vi)]{korevaar-schoen1})
\[
|\nabla f|^2 =  g^{ij} \pi_f(\frac{\partial}{\partial x^i}, \frac{\partial}{\partial x^j})
\]
Next assume  $(\Omega, g)$ is a Hermitian domain and let $(z^1=x^1+ix^2, \dots, z^n=x^{2n-1}+ix^{2n})$ be local complex coordinates.  We extend $\pi_f$ linearly cover $\C$ and  denote
\[
\frac{\partial f}{\partial z^i} \cdot  \frac{\partial f}{\partial \bar z^j}=\pi_f(\frac{\partial}{\partial z^i},  \frac{\partial}{\partial \bar z^j})
\]
and 
\[
\left| \frac{\partial f}{\partial z^i}\right|^2= \pi_f(\frac{\partial}{\partial z^i},  \frac{\partial}{\partial \bar z^i}).
\]
Thus,
\[
\frac{1}{4}|\nabla f|^2= g^{i\bar j} \frac{\partial f}{\partial z^i} \cdot 
\frac{\partial f}{\partial \bar z^j}.
\]

\subsection{Isometries of an NPC space}

Throughout this paper, we denote the group of isometries of an NPC space $\tilde X$ by $\mathsf{Isom}(\tilde X)$.  Isometries of an NPC space are classified as follows.

\begin{definition} \label{Delta}
For $I\in \mathsf{Isom}(\tilde X)$,
let
\[
\Delta_I: =\inf_{P \in \tilde X} d(I(P), P)
\]
denote  its translation length and define
\[
\mathsf{Min}(I):=\{ P \in \tilde X:  d(I(P),P)=\Delta_I\}.
\]
The isometry $I$ is  {\it elliptic}   if $\Delta_I=0$ and $\mathsf{Min}(I) \neq \emptyset$.     It is {\it hyperbolic}  if $\Delta_I>0$ and $\mathsf{Min}(I) \neq \emptyset$.   If $I$ is elliptic or hyperbolic, then we  say $I$ is {\it semisimple}.  Otherwise, $I$ is said to be {\it parabolic}.
\end{definition}

\begin{lemma}
If  $\tilde X$ is locally compact and $I \in \mathsf{Isom}(\tilde X)$ is a parabolic isometry, then there exists a arclength parameterized geodesic ray 
\begin{equation} \label{quasi}
c:[0,\infty) \rightarrow \tilde X \ \mbox{ such that  }\
\lim_{s \rightarrow \infty} d(I(c(s)), c(s))=\Delta_I.
\end{equation}
\end{lemma}

\begin{proof}
Let  $P_i \in \tilde X$ be such that
\[
\lim_{i \rightarrow \infty} d(I(P_i), P_i) = \Delta_I.
\]
Fix $P_0 \in \tilde X$ and let $T_i=d(P_i, P_0)$.  Since  $\mathsf{Min}(I)=\emptyset$, we have that $T_i \rightarrow \infty$.   
Let $c_i:[0,T_i] \rightarrow \tilde X$ be the arclength parameterized geodesic from $P_0$ to $P_i$.  Since $\tilde X$ is assumed to be locally compact, by taking a subsequence if necessary, the sequence  $c_i(s)$ converges uniformly on every compact set to an arclength parameterized geodesic ray $c(s)$.  
\end{proof}

\begin{definition}  \label{proper}
Let $\Gamma$ be a finitely generated group, $\Lambda$ be a finite set of generators of $\Gamma$, $\tilde{X}$ be an NPC space and $\rho:  \Gamma \rightarrow \mathsf{Isom}(\tilde{X})$ be a homomorphism. Define $\delta: \tilde{X} \rightarrow [0,\infty)$ to be the function
\[
\delta(P)=\max \{d(\rho(\lambda)P,P): \lambda \in \Lambda\}.
\]  We say $\rho$ is {\it proper} if the sublevel sets of the function $\delta$ is bounded in $\tilde X$; i.e.~given $c>0$, there exists $P_0 \in X$ and $R_0>0$ such that
\[
\{P \in \tilde X:  \delta(P)\leq c\} \subset B_{R_0}(P_0).
\]
\end{definition}

\subsection{Equivariant maps and section of the flat $\tilde X$-bundle} \label{subsec:donaldsonsection}

Following Donaldson  \cite{donaldson},  we will replace equivariant   maps with  sections of an associated fiber bundle. 
Assume we have the following:
\begin{itemize}
\item
a complete Riemannian manifold $(M,g)$ with universal covering $\Pi:\tilde M \rightarrow M$ 
\item 
NPC space $\tilde X$ 
\item action of $\pi_1(M)$ on  $\tilde \domain$  by deck transformations
\item homomorphism $\rho:\pi_1(M) \rightarrow \mathsf{Isom}(\tilde X)$ 
\end{itemize}

\begin{definition} \label{def:equivariant}A map $\tilde f:\tilde \domain \rightarrow  \tilde X$ is said to be $\rho$-equivariant if 
\[
\tilde f(\gamma p) = \rho(\gamma) \tilde f(p), \ \ \forall \gamma \in \pi_1(M), \ p \in \tilde \domain.
\]
\end{definition}

  \begin{remark}
Assume  $\rho:\pi_1(M) \rightarrow \mathsf{Isom}(\tilde X)$ is a proper homomorphism (cf.~Definition~\ref{proper}). If there exists a finite energy $\rho$-equivariant map $f:\tilde M \rightarrow \tilde X$, then there exists a Lipschitz harmonic map $u:\tilde M \rightarrow \tilde X$ (cf.~\cite[Theorem 2.1.3, Remark 2.1.5]{korevaar-schoen2}).    In this paper, we are trying to establish the existence of a harmonic map {\it without the assuming that there exists a finite energy map to start with.}
\end{remark}

The quotient space under the action of $\pi_1(M)$ on the product $\tilde M \times \tilde X$  is  the {\it twisted product}
\[
\tilde M \times_\rho \tilde X.
\]
In other words, $\tilde M \times_\rho \tilde X$ is the set of orbits  $[(p, x)]$ of a point $(p,x) \in  \tilde M \times  \tilde X$  under the action of $\gamma \in \pi_1(M)$ via   the deck transformation  on the first component and the isometry $\rho(\gamma)$ on the second component.  The fiber bundle
\[
\tilde M \times_\rho \tilde X \rightarrow M
\]
with  fibers over $p \in M$ is isometric to $\tilde X$ is the {\it flat $\tilde X$-bundle} over $M$ defined by $\rho$.

There is a one-to-one correspondence between sections of this fibration   and $\rho$-equivariant maps 
\[
 \tilde f:  \tilde M \rightarrow \tilde X 
 \ \ \ \ \longleftrightarrow \ \ \ \ 
f: M  \rightarrow \tilde M \times_{\rho} \tilde X
 \]
satisfying the relationship  
\[
[(\tilde p, \tilde f(\tilde p))] \leftrightarrow f(p)
\ \mbox{
 where $\Pi(\tilde p) =p$.}
 \]
Since the energy density function  $|\nabla \tilde f|^2$ of $\tilde f$ is a $\rho$-invariant function, we can define
\[
|\nabla f|^2(p):=|\nabla \tilde f|^2(\tilde p). 
\]
(In literature, this quantity is sometimes referred  as the {\it vertical} part the derivative of $f$, e.g.~\cite{donaldson}.) 
We can similarly define the pullback inner product and directional energy density functions of $f$  by using the corresponding notions for $\tilde f$ given in Subsection~\ref{subsec:korevaarschoen}.
For $U \subset M$, 
the (vertical) energy of a section $f$ is 
\begin{equation} \label{vertical}
E^f[U] = \int_U |\nabla f|^2 d\mbox{vol}_g.
\end{equation}
Furthermore,
for   sections $f_1$, $f_2$, we define
\begin{equation} \label{defdis}
d(f_1(p),f_2(p)):=d(\tilde f_1(\tilde p), \tilde f_2(\tilde p))
\end{equation}
where  $\tilde f_1$,  $\tilde f_2$ are  the associated  $\rho$-equivariant maps  to sections $f_1$, $f_2$ respectively.

\section{Exponential decay of an isometry}
\label{sec:Expdecay}

The distinguishing characteristic of a semisimple  isometry $I:\tilde X \rightarrow \tilde X$  is that there is a point $P_* \in \tilde X$ such that  $d(P_*,I(P_*))$ is equal to its translation length $\Delta_I$.  No such point $P_*$ exists if $I$ is a parabolic isometry.  
This motivates the following:  

\begin{definition} \label{jz}
We say that $I \in \mathsf{Isom}(\tilde X)$ has an {\it exponential decay to its translation length} (or simply {\it exponential decay} for short)  if one of the following conditions hold:
\begin{itemize}
\item $I$ is semisimple, or
\item $I$ is parabolic and there exist an arclength parameterized geodesic ray $c:[0,\infty) \rightarrow \tilde X$ and constants $a,b>0$ such that
\[
d^2(I(c(t)), c(t)) \leq  \Delta_I^2(1+be^{-at}).
\]
\end{itemize}
\end{definition}
If $I$ is semisimple, we can simply let $c(t)$ be identically equal to the point $P_*$ as above, and thus $I$ has an exponential decay to its translation length.  On the other hand, for a parabolic isometry  in a general NPC space, we may not be able to find a geodesic ray $c(t)$ such that that    $d(c(t),I(c(t)))$ converges to $\Delta_I$ exponentially fast.

In this section, we  prove that isometries of  important examples of NPC spaces (i.e.~symmetric spaces of non-compact type, CAT($-1$) spaces, and Euclidean buildings) have an exponential decay to their translation length.

\subsection{Symmetric spaces of non-compact type} \label{symmetricspace}
In this subsection, we prove that every isometry of a symmetric space  of non-compact type has exponential decay.  To this end, we introduce the following notations:
\begin{itemize}
\item $\Mn$ is the algebra of $(n \times n)$-matrices with real coefficients.
\item $\Sn$ is the vector subspace of $\Mn$ consisting of symmetric matrices.
\item $\Pn$ is the open cone of positive definite symmetric matrices.
\item $\Gl$ is the general linear group, or the  invertible $(n \times n)$-matrices. 
\item $\On$ is the orthogonal group, i.e.~the elements of $\Gl$ such that its transpose is its inverse.
\item $\An$ is the subgroup of diagonal matices with positive diagonals.
\item $\Nn$ is the subgroup of upper triangular matrices with diagonal entries 1.
\end{itemize}
The group $\Gl$ acts on $\Pn$ by the standard left action; i.e.
\[
G.p:=G p G^T,  \ \ \  \forall G \in \Gl, \, p \in \Pn
\]
where the left hand side consists of matrix multiplication.
Denoting the identity matrix by $e$, we have
\[
T_e\Pn=\Sn.
\]
The manifold $\Pn$ along with the Riemannian metric given by
\[
\langle V,W \rangle_p = \mbox{Trace}(p^{-1} V p^{-1}W), \ \ \ p \in \Pn, \ V,W \in T_p\Pn
\]
defines a  NPC symmetric space with isometry group $\mathsf{Isom}(\Pn)=\Gl$ (cf.~\cite[II.10.33, II.10.34, II.10.39]{bridson-haefliger}).   

Let $e$ be the identity element of $\Pn$.  
A geodesic line 
\[
c:\R \rightarrow \Pn \mbox{ \  with $c(0)=e$ and $c'(0)=V \in T_e\Pn=\Sn$}
\]
 is given by $c(t)=\exp(tV)$.  If $p =\exp(V)$, then $d(e,p)=|V|$ (cf.~\cite[II.10.42]{bridson-haefliger}).

\begin{definition} \label{def:parabolicsubgroup}
For a geodesic ray $c:[0,\infty) \rightarrow \Pn$, let $\xi=[c] \in \partial \Pn$ be the equivalence class of geodesic rays containing $c$.   The {\it parabolic subgroup} associated to $\xi=[c]$ is 
\[
\Gsf:=\{G \in \Pn:  G.\xi=\xi\}.
\]
The {\it horospherical subgroup} associated to $\xi=[c]$ is 
 \[
 \Nsf=\{G \in G_\xi:  \exp(-tV)G\exp(tV) \rightarrow I \mbox{ as } t \rightarrow \infty\}.
 \]
We refer to  \cite[10.62]{bridson-haefliger} for details regarding these subgroups.
\end{definition}
Following an argument of \cite{mochizuki-memoirs},  we will first prove that all isometries of $\Pn$ have exponential decay.

\begin{definition}
For $G \in \Gl$, define
\[
\rho(G):=\left( \sum_{i=1}^n (\log |b_i|^2)^2 \right)^{\frac{1}{2}}
\]
where $b_1, \dots, b_n$ are the eigenvalues of $G$.
\end{definition}

\begin{lemma} \label{tl0}
The translation length of $G \in \Gl$ is at least $\rho(G)$; i.e.~
$
\rho(G) \leq \Delta_G.$
\end{lemma}

\begin{proof}
We need to show 
\begin{equation} \label{tl}
\rho(G) \leq d(p,G.p), \ \ \forall p \in \Pn, \forall G \in \Gl.
\end{equation}
Before we proceed, note the following:
\begin{itemize}
\item {\it It is sufficient to prove the inequality of (\ref{tl}) when $p$ is the identity element $e$ in $\Pn$.}  
We prove this claim in two steps.
\begin{itemize}
\item  Let $q$ be a diagonal matrix; i.e.~$q=\mathsf{diag}(\lambda_1, \dots, \lambda_n)$   and  $q^{\frac{1}{2}}=\mathsf{diag}(\lambda_1^{\frac{1}{2}}, \dots, \lambda_n^{\frac{1}{2}})$.  Thus,  $q=q^{\frac{1}{2}} e q^{\frac{1}{2}}$ and hence
for any $G_0 \in \Gl$ and $G_1:=q^{-\frac{1}{2}} G_0 q^{\frac{1}{2}}$,
\begin{eqnarray*}
d(G_0.q, q) & = & d(G_0 q G_0^t, q) 
 \ = \   
d(G_0 q^{\frac{1}{2}} e q^{\frac{1}{2}} G_0^t, \,q^{\frac{1}{2}}  e q^{\frac{1}{2}})
\\
& = & 
d(q^{-\frac{1}{2}} G_0 q^{\frac{1}{2}} e q^{\frac{1}{2}} G_0^t q^{-\frac{1}{2}}, e)
\ = \ 
d(G_1.e,e).
\end{eqnarray*} 

\item 
Let  $p \in \Pn$.   Choose $O \in \On$ be such that $p=O q O^t=O.q$ where $q$ is a diagonal matrix.   If  $G_0:=OGO^t$ and $G_1:=q^{-\frac{1}{2}} G_0 q^{\frac{1}{2}}$ and, then   
\begin{eqnarray*}
d(G.p,p) & = &  d(GO.q, \,O.q)
\ = \ 
d(O^tGO.q, q)
\\
& = &   
d(G_0.q, q) 
\ = \ 
d(G_1.e,e)
\end{eqnarray*}
 and $\rho(G)=\rho(G_1)$.
\end{itemize}

\item {\it It is sufficient to prove (\ref{tl}) assuming the eigenvalues of $G$ are distinct.}  Indeed, if $G$ does not have distinct eigenvalues, we can take a sequence $G^{(j)}$ of matrices with distinct eigenvalues such that $\rho(G^{(j)}) \rightarrow \rho(G)$ and $d(G^{(j)}.e,e) \rightarrow d(G.e,e)$.
\end{itemize}
  
By virtue of the above two bullet points, we assume that  $G \in \Gl$ has distinct eigenvalues $b_1, b_2, \dots, b_n$ and show (\ref{tl}) with $p=e$.  

Let   $V \in T_e\Pn=\Sn$ be a diagonal matrix with diagonal entries 
\[
\frac{\log |b_1|^2}{\rho(G)}, \ \frac{\log |b_2|^2}{\rho(G)}, \ \dots, \frac{\log |b_n|^2}{\rho(G)}. 
\]
Conjugating by elemen
Let  $c(t)=\exp(tV)$ and $\xi=[c]$. We apply the Iwasawa decomposition to write $G=O A_0 N$  where  $O \in \On$, $A_0 \in \An$ and $N \in \Nn$. 
Since $A'=A_0 N$ is an upper triangular matrix, we can decompose 
$
A'=N' A_0
$
where $N' \in \Nn$.  In summary,
\[
G=O  A_0  N = O  A'=O  N'  A_0.
\] 
The diagonal matrix $A_0$ is the hyperbolic isometry fixing the geodesic $c(t)=\exp (tV)$ and the upper triangular matrix  $N$ with 1's on the diagonals is an  element of the horospherical subgroup $\Nsf$ associated to $\xi$ (cf.~\cite[II.10.66]{bridson-haefliger}).  
By construction, 
\[
c(\rho(G))=\exp(\rho(G) V)=A_0A_0^T=A_0.e.
\]  
Since the geodesic line $c(t)$ intersects the horospheres centered at $\xi=[c]$ orthogonally,   the point $A_0.e$ is the closest point  to $e$ amongst all points on the  horosphere containing $A_0.e$ (cf.~\cite[I.10.5]{bridson-haefliger}).  Furthermore, any orbit of $N'$ is contained in a single horosphere (cf.~\cite[II.10.66]{bridson-haefliger}), and hence 
\[
d(A_0.e,e) \leq d(N'A_0.e, \, e)=d(G.e,e).
\]
This in turn implies
\[
\rho(G)=\rho(A_0) = d(A_0.e,e)  \leq d(G.e,e).
\]
\end{proof}

 \begin{theorem} \label{jzmatrix}
For a parabolic element $G\in \Pn$, let 
\begin{itemize}
\item $\xi \in \partial \Pn$ be such that $G \in \Gsf$ (i.e.~$G$ is an element of the parabolic subgroup associated to $\xi$), 
\item $c:[0,\infty) \rightarrow \Pn$ be the geodesic ray such that $\xi=[c]$ and $c(0)=e$ (cf.~\cite[10.62]{bridson-haefliger}). 
\end{itemize}
Then there exist constants $a,b>0$ such that 
 \[
  \lim_{t \rightarrow \infty} d(c(t),G.c(t)) =  \triangle_G+be^{-at}.
 \]
 \end{theorem}

 \begin{proof}
 We will prove this by an inductive argument on $n$.
 Let 
 \[
 V \in T_e(\Pn)=\Sn \mbox{ such that }c(t)=\exp(tV).
 \]  
 By conjugating by an element of $\So$, we may assume that $V$ is a diagonal matrix 
 \[
 V = \mathsf{diag}(\lambda_1, \dots, \lambda_n)
 \mbox{ 
  with $\lambda_1 \geq \dots \geq \lambda_n$}.
  \]
Let $r_1, \dots, r_k$ be the multiplicities of these eigenvalues.
By \cite[10.64]{bridson-haefliger}, 
\[
G=NA
\]
where 
\begin{itemize}
\item $A$ is a block diagonal matrix,
\begin{equation} \label{eq:A}
A=
\begin{pmatrix}
A_{11} & 0 & \dots & 0\\
0 & A_{22} & \dots & 0\\
\dots & \dots & \dots & \dots
\\
0 & 0 & \dots & A_{kk}
\end{pmatrix}
\end{equation}
\item $N \in \Nsf$, the horospherical subgroup associated to $\xi \in [c_0]$, and $N$ is a block upper triangular matrix
\begin{equation} \label{eqN}
N=
\begin{pmatrix}
I_{r_1} & A_{12} & \dots & A_{1k}\\
0 & I_{r_2} & \dots & A_{2k}\\
\dots & \dots & \dots & \dots
\\
0 & 0 & \dots & I_{r_k}
\end{pmatrix}
\end{equation}
where $A_{ij}$ is an $(r_j \times r_j)$-matrix and $I_{ij}$ is the  $(r_j \times r_j)$-identity matrix.
\end{itemize}
Let $F(c)$ be the union of all geodesic lines parallel to $c$.  We have that (cf.~\cite[10.67]{bridson-haefliger})
\begin{equation} \label{eq:F}
F(c)=
\begin{pmatrix}
\mathsf{P}(r_1,\R) & 0 & \dots & 0\\
0 & \mathsf{P}(r_2,\R) & \dots & 0\\
\dots & \dots & \dots & \dots
\\
0 & 0 & \dots & \mathsf{P}(r_k,\R)
\end{pmatrix}.
\end{equation}

We consider two cases: \\

\vspace*{0.1in}
 {\sc Case 1.} {\it The eigenvalues $\lambda_1, \dots, \lambda_n$ of $V$ are distinct}; i.e.~$r_i=1$ for all $i$.\\
 
The assumption implies that $c(t)=\exp (tV)$ is a regular geodesic (cf.~\cite[10.45]{bridson-haefliger}).  Moreover,  $F(c)$  is the unique maximal flat containing $c(t)$, $A$ is a diagonal matrix $\mathsf{diag}(
\alpha_1, \dots, \alpha_n)$  and acts on  $F(c)$ by translation (cf.~\cite[10.64, 10.67, and 10.68]{bridson-haefliger}).  In particular,   $t \mapsto d(c(t),A.c(t))$ is a constant function in $t$.  Thus,
\begin{equation} \label{anotanot}
d(c(t), A.c(t))  = d(c(0), A.c(0)) = d(e,A.e)=\rho(A)=\rho(G).
\end{equation}
Furthermore, letting $\hat c(t):=A.c(t)$, we have
\begin{eqnarray*}
d(\hat c(t),N.\hat c(t)) & = & d(\hat c(t), \, N  \hat c(t) N^T)  \\
& = & d(\hat c(t)^{\frac{1}{2}} \hat c(t)^{\frac{1}{2}} , \  N  \hat c(t)^{\frac{1}{2}} \hat c(t)^{\frac{1}{2}}  N^T)  \\
& = & d(e, \, (\hat c(t)^{-\frac{1}{2}}N \hat c(t)^{\frac{1}{2}})  e(\hat c(t)^{-\frac{1}{2}}  N \hat c(t)^{\frac{1}{2}})^T)
\\
& = & d(e, \, (\hat c(t)^{-\frac{1}{2}}N \hat c(t)^{\frac{1}{2}}).  e).
\end{eqnarray*}
Since $\hat c(t)$ is a diagonal matrix with diagonal entries $b_1^2 e^{t\lambda_1},  b_2^2 e^{t\lambda_2},  \dots,  b_n^2 e^{t \lambda_n}$ and $N$ is an upper triangular matrix with  diagonal entries equal to 1,  there exists $C_1, C_2>0$ such that 
\[
\| e - \hat c(t)^{-\frac{1}{2}}  N  \hat c(t)^{\frac{1}{2}}\| \leq C_1e^{-C_2t}
\]
Note that the choice of $V$ having distinct diagonal entries is used in the estimate above.  Indeed,  the constant $C_2$ is chosen less than $\min \{\lambda_i-\lambda_{i+1}: i=1, \dots, n-1 \}>0$.
We thus conclude
\begin{equation} \label{eq:N}
d(\hat c(t),N.\hat c(t)) \leq be^{-at}
\end{equation}
for some $a, b>0$.
Combining this with (\ref{anotanot}), the triangle inequality implies that
\begin{equation} \label{c}
d(c(t), G. c(t)) \leq d(c(t), \hat c(t))+ d(\hat c(t), N.\hat c(t)) \leq 
\rho(G)+be^{-at}. 
\end{equation}
\\
\noindent {\sc Case 2:}  {\it The eigenvalues of $V$  are not distinct.}\\

The assumption implies that $c(t)=\exp (tV)$ is a singular geodesic (cf.~\cite[10.45]{bridson-haefliger}).  We write
\[
c(t)=
\begin{pmatrix}
c_1(t) & 0 & \dots & 0\\
0 & c_2(t) & \dots & 0\\
\dots & \dots & \dots & \dots
\\
0 & 0 & \dots & c_k(t)
\end{pmatrix}. 
\]
where $c_i(t)$ is a geodesic ray in  $\mathsf{P}(r_i,\R)$, the block diagonal entry that appear in (\ref{eq:F}).     If $c_i(t)$ is not constant, then the block diagonal $A_{ii}$ of (\ref{eq:A}) is an element of  parabolic subgroup associated to $\xi_i=[c_i] \in \partial \mathsf{P}(r_i,\R)$ by \cite[10.64]{bridson-haefliger}.  By the induction hypothesis, there exists a geodesic ray $c_i:[0,\infty)  \rightarrow  \partial \mathsf{P}(r_i,\R)$ such that 
 \[
 d(c_i(t), A_{ii}.c_i(t)) \leq 
\rho(A_{ii})+b_ie^{-a_it}.
\]
 If $c(t)$ is a constant, then $c(t)=e$ is a fixed point of $A_{ii}$.   
Thus, 
there exist $\alpha, \beta>0$ such that
 \[
  d(c(t), A.c(t)) \leq 
\rho(A)+\beta e^{-\alpha t}. 
\]
By an analogous proof as in (\ref{eq:N}), there exist constants  $\bar \alpha, \bar \beta>0$ such that
 \[
 d(A.c(t), (NA).c(t)) \leq 
\bar \beta e^{-\bar \alpha t}. 
\]
Let $a=\min\{\alpha, \bar \alpha\}$,  $b=\beta+\bar \beta$.  By the triangle inequality,
\[
d(c(t),G.c(t)) \leq d(c(t), A.c(t)) + d(A.c(t), G.c(t)) \leq
\rho(A)+b e^{-a t}. 
\]
 \end{proof}

\begin{theorem} \label{expdecaysymspace}
Let $\tilde X$ be a  irreducible symmetric space of non-compact type.  For $I \in \mathsf{Isom}(\tilde X)$, there exist a geodesic line $c:\R \rightarrow \tilde X$ and constants $a,b>0$ such that
\[
d(c(t),I(c(t))) \leq \triangle_{I} + be^{-at}.
\]
\end{theorem}

\begin{proof}
After scaling, we can assume that  $\tilde X$  is a  totally geodesic submanifold  of $\Pn$ (cf.~\cite[2.6, pp.~14-16]{mostow} or \cite[appendix]{eberlein2})   and $\mathsf{Isom}(\tilde X)=\{I=G|_{\tilde X}:  G \in \Gl, G.\tilde X=\tilde X\}$.  
Let $I=G|_{\tilde X} \in \mathsf{Isom}(\tilde X)$.  Then $I$ semisimple if and only if $G$ is semisimple in $\Gl$ (cf.~\cite[10.61]{bridson-haefliger}).  Thus, the assertion follows from Theorem~\ref{jzmatrix}.
\end{proof}

We have therefore shown the following.
\begin{theorem} \label{thm:jzsymmetricspace}
If  $\tilde X$ is an irreducible symmetric space of non-compact type, then every isometry $I \in \mathsf{Isom}(\tilde X)$ has exponential decay.
\end{theorem}

\subsection{CAT($-1$) spaces}

In the following important cases, every isometry has exponential decay. 

\label{sec:CAT(-1)}
\begin{theorem} \label{thm:jzcat-1}
If  $\tilde X$ is a CAT$(-1)$ space, then every isometry $I:\tilde X \rightarrow \tilde X$ has exponential decay.
\end{theorem}

\begin{proof}
Let  $I$ be a parabolic isometry and $c_0:[0,\infty) \rightarrow \tilde X$ be a geodesic ray such that
\[
\lim_{t \rightarrow \infty} d(c_0(t), I(c_0(t))) = \Delta_{I}=\inf_{t \in [0,\infty)} d(c_0(t), I(c_0(t))).
\]

Let $c_1(t)=I(c_0(t))$.  For each $R \in [0,\infty)$, let $p=c_0(0)$, $q=c_0(R)$, $r=I(c_0(R))$, $s=c_1(0)$ and $\tilde p$, $\tilde q$, $\tilde r$, $\tilde s$ be the comparison quadrilateral in $\Hyp^2$ to the quadrilateral $(p,q,r,s)$ in $\tilde X$ (cf.~\cite{reshetnyak} or \cite[Theorem 2.1.1]{korevaar-schoen1}).  
Let $\tilde c_0^R$ ($\tilde c_1^R$ resp.) be the ray geodesic emanating from $\tilde p$ ($\tilde s$ resp.)and passing through $\tilde q$ ($\tilde r$ resp.). Let $\tilde \gamma_0$ ($\tilde \gamma_1$ resp.) be the geodesic from $\tilde p$ to $\tilde s$ ($\tilde q$ to $\tilde r$ resp.).
By the quadrilateral comparison (cf.~\cite{reshetnyak} or \cite[Theorem 2.1.2]{korevaar-schoen1}),
\[
d(c_0(t), c_1(t)) \leq d(\tilde c_0^R(t), \tilde c_1^R(t)), \ \ \forall t \in [0,R].
\]
Without  loss of generality, we may assume $\tilde c_0^R(0) =\tilde c_0^{R'}(0)$ and $\tilde c_1^R(0) =\tilde c_1^{R'}(0)$ for all $R,R' \in [0,\infty)$.  Arzela-Ascoli then implies that  there exists $R_j \rightarrow \infty$ and geodesic rays $\tilde c_0$, $\tilde c_1=I. \tilde c_0$ such that $\tilde c_0^R \rightarrow \tilde c_0$ and $\tilde c_1^R \rightarrow \tilde c_1$ on every compact subsets of $[0,\infty)$.  Thus, by Theorem~\ref{thm:jzsymmetricspace} for $M=\Hyp^2$,
\[
d(c_0(t), c_1(t)) \leq d(\tilde c_0(t), \tilde c_1(t)) \leq \Delta_{I}+be^{-at}.
\]
\end{proof}

\section{Exponential decay of a commuting pair of isometries} \label{sec:Expdecay2}

In this subsection, we consider the notion of exponential decay for a commuting pair of isometries $I_m: \tilde X \rightarrow \tilde X$, $ m=1,2$.

\begin{definition} \label{jz2}
  We say a commuting pair $I_1, I_2 \in \mathsf{Isom}(\tilde X)$ has {\it exponential decay to their translation length} (or {\it exponential decay} for short) if one of the following conditions hold:
\begin{itemize}
\item $I_1$ and $I_2$ are semisimple, or
\item $I_1$ and $I_2$ are parabolic and there exist an arclength parameterized geodesic ray $c:[0,\infty) \rightarrow \tilde X$ and constants  $a,b>0$ such that
\[
d^2(I_m(c(t)), c(t)) \leq  \Delta_{I_m}^2(1+be^{-at}), \ \ m=1,2.
\]
\end{itemize}
\end{definition}

 \begin{remark} \label{prop:fys}
The following is due to \cite{fsy}.
If  $\tilde X$ is  locally compact, $I_1, I_2 \in \mathsf{Isom}(\tilde X)$ commute and $I_1$ is parabolic, then $I_2$ is also parabolic  and fix the same point at infinity.  
Thus, commuting isometries  are   always both semisimple or both parabolic.   In the latter case, there exists a geodesic ray $c:[0,\infty) \rightarrow \infty$ such that 
\[
\lim_{t \rightarrow \infty} d(I_m(c(t)),c(t)) =\Delta_{I_m}, \ \ m=1,2.
\] 
 \end{remark}

\begin{lemma} \label{flattorus}
If the subgroup $\langle I_1, I_2 \rangle \subset \mathsf{Isom}(\tilde X)$  is generated by commuting semisimple isometries $I_1, I_2$ and 
\[
\rho:2\pi\Z \times 2\pi \Z \rightarrow \Gamma
\]
is a homomorphism defined by $(2\pi,0), (0,2\pi) \in 2\pi\Z \times 2\pi \Z$ mapping to $I_1, I_2$ respectively, then 
there exists a totally geodesic $\rho$-equivariant map
\[
h: \R \times \R \rightarrow \tilde X
\]
such that 
\[
\left| \frac{\partial h}{\partial x}\right|^2=\frac{\Delta_{I_1}^2}{4\pi^2}, \ \ \ \left| \frac{\partial h}{\partial y} \right|=\frac{\Delta_{I_2}^2}{4\pi^2}
\]
where $(x,y)$ are the standard coordinates of $\R \times \R$.
\end{lemma}

\begin{proof}
By \cite[II.7.20]{bridson-haefliger}, we can choose $P_* \in \mathsf{Min}(I_1) \cap \mathsf{Min}(I_2)$. We consider the following cases:
\begin{itemize}
\item
Assume  $I_1$ and $I_2$ are both elliptic.   Then let $h(x,y)$ be identically equal to $P_*$.
\item 
Assume $I_1$ is hyperbolic and $I_2$ is elliptic. Then $P_*$ is a fixed point of $I_2$.  Let $c_1:\R \rightarrow \tilde X$ be an arclength parameterized geodesic such that   $c_1(0)=P_*$ and  $c_1(\R)$ is an axis of $I_1$.  
  By commutativity, 
\[
I_2 \circ I_1^k(P_*) = I_1^k \circ I_2 (P_*)=I_1^k(P_*), \ \ \forall k \in \Z.
\]
Thus,    $c_1(k\Delta_{I_1})=I_1^k(P_*)$ is a fixed point of $I_2$ for any $k \in \Z$ which implies that  every point of  $c_1(\R)$ is a fixed point of $I_2$.  Define $h(x,y)=c_1(\frac{\Delta_I}{2\pi}x)$.
\item Assume $I_1$ and $I_2$ are both hyperbolic.  Then $\langle I_1, I_2 \rangle$ is properly discontinuous, and we can  apply the flat torus theorem  (cf.~\cite[II.7.1]{bridson-haefliger}) to assert the existence of a $\rho$-equivariant totally geodesic map from $\R \times \R$ to $\tilde X$.  After an appropriate reparameterization, we obtain a map $h$ with the desired derivative estimates.
\end{itemize}
\end{proof}

\begin{lemma} \label{almostflattorus}
If the subgroup $\langle I_1, I_2 \rangle \subset \mathsf{Isom}(\tilde X)$  is generated by commuting parabolic isometries $I_1, I_2$ with exponential decay and 
\[
\rho:2\pi\Z \times 2\pi \Z \rightarrow \Gamma
\]
is a homomorphism defined by $(2\pi,0), (0,2\pi) \in 2\pi\Z \times 2\pi \Z$ mapping to $I_1, I_2$ respectively, then we have the following:   there exist  $a,b >0$ and a  
 map
\[
\tilde h: [0,\infty) \times \R \times \R \rightarrow \tilde X
\]
such that for each $t \in [0,\infty)$, $\tilde h_t(x,y):=\tilde h(t,x,y)$ is a $\rho$-equivariant map and 
\[
\left| \frac{\partial \tilde h}{\partial t}\right|^2\leq 1, \ \ \ \left| \frac{\partial \tilde h}{\partial x}\right|^2\leq \frac{\Delta_{I_1}^2}{4\pi^2}(1+be^{-at}), \ \ \ \left| \frac{\partial \tilde h}{\partial y} \right|\leq \frac{\Delta_{I_2}^2}{4\pi^2}(1+be^{-at}).
\]
\end{lemma}

\begin{proof}
Let $c(t)$ be as in Definition~\ref{jz2}.  Let \begin{eqnarray*}
\tilde h(t,0,0)=c(t), & & \tilde h(t,2\pi,0)= I_1(c(t)),\\
\tilde h(t,0,2\pi)=I_2(t) & & \tilde h(t,2\pi, 2\pi)=I_1 \circ I_2(c(t))=I_2 \circ I_1(c(t)).
\end{eqnarray*}
For each $t \in [0,\infty)$, let 
\begin{eqnarray*}
x  \in [0,2\pi]  & \mapsto & \tilde h(t, x,0) \\
   x \in [0,2\pi] & \mapsto & \tilde h(t,x,2\pi) 
   \end{eqnarray*}
   be  geodesics. 
     For each $x \in [0,2\pi]$, let 
     \[
     y \mapsto \tilde h(t,x,y)
     \]
      be a geodesic.  
Finally, we extend $\tilde h_t(x,y)$ equivariantly from $[0,2\pi] \times [0,2\pi]$ to $\R \times \R$.  

The assumption that  the commuting pair $I_1$ and $I_2$ has exponential decay along with  multiple applications of the quadrilateral comparison property of NPC spaces imply the derivative estimates. 
\end{proof}

\subsection{Symmetric spaces of non-compact type} \label{symmetricspace2}
In this section, we prove that every pair of commuting isometries of a symmetric space  of non-compact type has exponential decay.  

 \begin{theorem} \label{jzmatrix}
Let 
\begin{itemize}
\item $G_1, G_2 \in \Gsf$, and
\item   
$c:[0,\infty) \rightarrow \Pn$ be the geodesic ray such that $\xi=[c]$ and $c(0)=e$. 
\end{itemize}
Then there exist constants $a,b>0$ such that 
 \begin{eqnarray*}
  \lim_{t \rightarrow \infty} d(c(t),G_m.c(t)) & = &  \triangle_{G_m}+be^{-at}, \ \ m=1,2.
 \end{eqnarray*}
 \end{theorem}

 \begin{proof}
 We will prove this by an inductive argument on $n$.
 Let 
 \[
 V \in T_e(\Pn)=\Sn \mbox{ such that }c(t)=\exp(tV).
 \]  
 By conjugating by an element of $\So$, we may assume that $V$ is a diagonal matrix 
 \[
 V = \mathsf{diag}(\lambda_1, \dots, \lambda_n)
 \mbox{ 
  with $\lambda_1 \geq \dots \geq \lambda_n$}.
  \]
Let $r_1, \dots, r_k$ be the multiplicities of these eigenvalues.
By \cite[10.64]{bridson-haefliger}, 
\[
G_1=N_1A_1 \ \mbox{ and } \  G_2=N_2A_2
\]
where for $m=1,2$:
\begin{itemize}
\item $A_m,$ is a block diagonal matrix as in (\ref{eq:A}) which we write as
\begin{equation} \label{eq:A2}
A_m=
\begin{pmatrix}
A_{m,11} & 0 & \dots & 0\\
0 & A_{m,22} & \dots & 0\\
\dots & \dots & \dots & \dots
\\
0 & 0 & \dots & A_{m,kk}
\end{pmatrix}
\end{equation}
\item $N \in \Nsf$, the horospherical subgroup associated to $\xi \in [c_0]$, and $N$ is a block upper triangular matrix as in (\ref{eqN}).
\end{itemize}
Let $F(c)$  be the union of all geodesic lines parallel to $c$ as in (\ref{eq:F}).  
We consider two cases: 

\vspace*{0.1in}
 {\sc Case 1.} {\it The eigenvalues $\lambda_1, \dots, \lambda_n$ of $V$ are distinct}; i.e.~$r_i=1$ for all $i$.\\
 
The assumption implies that $c(t)=\exp (tV)$ is a regular geodesic (cf.~\cite[10.45]{bridson-haefliger}).  Moreover,  $F(c)$  is the unique maximal flat containing $c(t)$ and $A_m$ acts on  $F(c)$ by translation (cf.~\cite[10.64, 10.67, and 10.68]{bridson-haefliger}).  In particular,   $t \mapsto d(c(t),A_m.c(t))$ is a constant function in $t$.  Thus,
\begin{equation} \label{anotanot}
d(c(t), A_m.c(t))  = d(c(0), A_m.c(0)) = d(e,A_m.e)=\rho(A_m)=\rho(G_m).
\end{equation}
Furthermore, letting $\hat c(t):=A_m.c(t)$, we have
\begin{eqnarray*}
d(\hat c(t),N_m.\hat c(t)) & = & d(\hat c(t), \, N_m  \hat c(t) N_m^T)  \\
& = & d(\hat c(t)^{\frac{1}{2}} \hat c(t)^{\frac{1}{2}} , \  N_m  \hat c(t)^{\frac{1}{2}} \hat c(t)^{\frac{1}{2}}  N_m^T)  \\
& = & d(e, \, (\hat c(t)^{-\frac{1}{2}}N_m \hat c(t)^{\frac{1}{2}})  e(\hat c(t)^{-\frac{1}{2}}  N_m \hat c(t)^{\frac{1}{2}})^T)
\\
& = & d(e, \, (\hat c(t)^{-\frac{1}{2}}N_m \hat c(t)^{\frac{1}{2}}).  e).
\end{eqnarray*}
Since $\hat c(t)$ is a diagonal matrix with diagonal entries $b_1^2 e^{t\lambda_1},  b_2^2 e^{t\lambda_2},  \dots,  b_n^2 e^{t \lambda_n}$ and $N_m$ is an upper triangular matrix with  diagonal entries equal to 1,  there exists $C_1, C_2>0$ such that 
\[
\| e - \hat c(t)^{-\frac{1}{2}}  N_m  \hat c(t)^{\frac{1}{2}}\| \leq C_1e^{-C_2t}
\]
Note that the choice of $V$ having distinct diagonal entries is used in the estimate above.  Indeed,  the constant $C_2$ is chosen less than $\min \{\lambda_i-\lambda_{i+1}: i=1, \dots, n-1 \}>0$.
We thus conclude
\begin{equation} \label{eq:N2}
d(\hat c(t),N_m.\hat c(t)) \leq b_me^{-a_mt}
\end{equation}
for some $a_m, b_m>0$.
Combining this with (\ref{anotanot}), the triangle inequality implies that
\begin{equation} \label{c}
d(c(t), G_m. c(t)) \leq d(c(t), \hat c(t))+ d(\hat c(t), N_m.\hat c(t)) \leq 
\rho(G_m)+b_me^{-a_mt}. 
\end{equation}
Choose $a=\min\{a_1,a_2\}$ and $b=\max\{b_1,b_2\}$.\\
\\
\noindent {\sc Case 2:}  {\it The eigenvalues of $V$  are not distinct.}\\

The assumption implies that $c(t)=\exp (tV)$ is a singular geodesic (cf.~\cite[10.45]{bridson-haefliger}).  We write
\[
c(t)=
\begin{pmatrix}
c_1(t) & 0 & \dots & 0\\
0 & c_2(t) & \dots & 0\\
\dots & \dots & \dots & \dots
\\
0 & 0 & \dots & c_k(t)
\end{pmatrix}. 
\]
where $c_i(t)$ is a geodesic ray in  $\mathsf{P}(r_i,\R)$, the block diagonal entry that appear in (\ref{eq:F}).     If $c_i(t)$ is not constant, then the block diagonal $A_{m,ii}$ of (\ref{eq:A2}) is an element of  parabolic subgroup associated to $\xi_i=[c_i] \in \partial \mathsf{P}(r_i,\R)$ by \cite[10.64]{bridson-haefliger}.  By the induction hypothesis, there exists a geodesic ray $c_i:[0,\infty)  \rightarrow  \partial \mathsf{P}(r_i,\R)$ such that 
 \[
 d(c_i(t), A_{m,ii}.c_i(t)) \leq 
\rho(A_{m,ii})+b_ie^{-a_it}
\]
 If $c(t)$ is a constant, then $c(t)=e$ is a fixed point of $A_{m,ii}$.   
Thus, 
there exists $\alpha_m, \beta_m>0$ such that
 \[
  d(c(t), A_m.c(t)) \leq 
\rho(A_m)+\beta_m e^{-\alpha_m t}. 
\]
By an analogous proof as in (\ref{eq:N2}), there exist constants  $\bar \alpha_m, \bar \beta_m>0$ such that
 \[
 d(A_m.c(t), (N_mA_m).c(t)) \leq 
\bar \beta_m e^{-\bar \alpha_m t}. 
\]
Let $a_m=\min\{\alpha_m, \bar \alpha_m\}$,  $b_m=\beta+\bar \beta$.  By triangle inequality,
\[
d(c(t),G_m.c(t)) \leq d(c(t), A_m.c(t)) + d(A_m.c(t), G_m.c(t)) \leq
\rho(A_m)+b_m e^{-a_m t}. 
\]
Choose $a=\min\{a_1,a_2\}$ and $b=\max\{b_1,b_2\}$
 \end{proof}

\begin{theorem}
If $\tilde X$ is a  irreducible symmetric space of non-compact type, then every pair of commuting isometries has exponential decay. \end{theorem}

\begin{proof}
The proof is similar to the proof of Theorem~\ref{expdecaysymspace}.
\end{proof}

\subsection{CAT($-1$) spaces}

\begin{theorem}
If $\tilde X$ is a CAT($-1$), then every pair of commuting isometries  has exponential decay.
\end{theorem}

\begin{proof}
Analogously to the proof of Theorem~\ref{thm:jzcat-1}, the assertion follows from comparing $\tilde X$ to $\mathbb{H}^2$.
\end{proof}

\chapter{Bochner Formulae} \label{chap:bochner}

The goal of this chapter is to prove variations of the Siu's and Sampson's Bochner formulas.  These formulas are contained in  Theorem~\ref{mochizukibochnerformula} and Theorem~\ref{mochizukibochnerformula'} and are inspired by Mochizuki's formula in \cite[Proposition 21.42]{mochizuki-memoirs}.

\section{Between K\"ahler manifolds} \label{KtoK}
\label{sec:KtoK}
In \cite{siu1}, Siu introduced the following notion of negative curvature.

\begin{definition}
For local holomorphic coordinates $(w^i)$   of a  K\"ahler manifold $X$, the curvature tensor component is  
 \[
R_{i\bar j l \bar k} =\frac{\partial^2 h_{l \bar k}}{\partial w^i \partial \bar w^j}+h^{p\bar q} \frac{\partial h_{l\bar q}}{\partial w^i} \frac{\partial h_{p \bar k}}{\partial \bar w^j}.
\]  
For $V=(V^i)$ and $W=(W^i)$, let 
 \[
R(V,W)=-\frac{1}{2} R_{i\bar j l \bar k}(V^i \overline{W^j}-W^i \overline{V^j})\overline{(V^l \overline{W^k}-W^l \overline{V^k})}
\] 
We say $X$ has \emph{strongly negative (resp.~strongly non-positive)  curvature in the sense of Siu} \cite{siu1} if
 $R(V,W)<0$ (resp.~$R(V,W) \leq 0$) 
 for any nonzero $n \times n$ complex matrix $(V^i \overline{W^j}-W^i \overline{V^j})$.
 \end{definition}

 \begin{remark}
 It is well known that a K\"ahler  symmetric space of non-compact type has strongly non-positive curvature in the sense of Siu (cf.~\cite{siu1}).
 \end{remark}

\begin{lemma}\label{branch}
Let $X$ be a K\"ahler manifold with K\"ahler form $\omega$ and strongly non-positive curvature in the sense of Siu.  Let $M$ be another K\"ahler manifold and  $u : M \rightarrow X $ be a smooth map.
If the function $Q:M \rightarrow \R$ is defined by setting
\begin{eqnarray*}
Q \frac{ \omega^n}{n!} =-R_{i \bar j l \bar k}  \partial u^l \wedge \bar \partial \bar{u}^k \wedge \bar \partial u^i \wedge \partial \bar{u}^j \wedge \frac{ \omega^{n-2}}{(n-2)!},
\end{eqnarray*}
then $Q \geq 0$.
\end{lemma}

\begin{proof}
Follows from  \cite[Proposition 3, (4.2), and (4.3)]{siu1} .
\end{proof}

Let $(M,g)$ and $(X,h)$ be K\"ahler manifolds. 
For a smooth map $u:M \rightarrow X$, 
let
\begin{equation} \label{E}
E:=u^{-1}(TX \otimes \C).
\end{equation}
  Let
\[
E=E' \oplus E'' \ \mbox{ where } \ E':=u^{-1} (T^{(1,0)}X) , \ \ E'':= u^{-1} (T^{(0,1)}X).
\]
Denote by $\Omega^{p,q}(E)$, $\Omega^{p,q}(E')$ and $\Omega^{p,q}(E'')$  the space of   $E$-, $E'$- and $E''$-valued $(p,q)$-forms respectively.  
If $(w^i)$ are local holomorphic coordinates in $X$, then
$\{\frac{\partial}{\partial u^i}  :=  \frac{\partial}{\partial w^i} \circ u, \
\frac{\partial}{\partial \bar u^i} := \frac{\partial}{\partial \bar w^i} \circ u \}$ 
is a local  frame  of $E$.
Following \cite{siu1}, denote the  natural projections and  inclusions by
\begin{eqnarray*}
\Pi_{(1,0)}: TX \otimes \C \rightarrow  T^{(1,0)}X, & & 
\Pi_{(0,1)}: TX \otimes \C \rightarrow  T^{(0,1)}X,\\
\mbox{I}_{(1,0)}: T^{(1,0)}M \hookrightarrow TM \otimes \C, & & 
\mbox{I}_{(0,1)}: T^{(0,1)}M \hookrightarrow TM \otimes \C.
\end{eqnarray*}
The differential $du:TM \rightarrow TX$ gives rise to a map
\[
du \otimes \C:  TM \otimes \C \rightarrow TX \otimes \C.
\]
Define
\begin{eqnarray*}
\partial u = \Pi_{(1,0)} \circ du \otimes \C  \circ \mbox{I}_{(1,0)} \in \Omega^{(1,0)}(E')& & 
\bar \partial u = \Pi_{(1,0)} \circ du \otimes \C  \circ \mbox{I}_{(0,1)} \in \Omega^{(0,1)}(E')\\ 
\partial \bar u = \Pi_{(0,1)} \circ du \otimes \C  \circ \mbox{I}_{(1,0)} \in \Omega^{(1,0)}(E'')& & 
\bar \partial \bar u = \Pi_{(0,1)} \circ du \otimes \C  \circ \mbox{I}_{(0,1)} \in \Omega^{(0,1)}(E''). 
\end{eqnarray*}
In local holomorphic coordinates of $M$ and $X$,
\begin{eqnarray*}
\partial u=\partial u^i \frac{\partial}{\partial u^i} 
& & 
\bar \partial u=\bar \partial u^i\frac{\partial}{\partial u^i} 
\\
\partial \bar u= \partial\bar u^i \frac{\partial}{\partial \bar u^i} 
& & 
\bar \partial \bar u=\bar \partial \bar u^i \frac{\partial}{\partial \bar u^i}\\
\overline{\partial u} = \bar \partial \bar u \  & &  \overline{\bar \partial u}=\partial \bar u.
\end{eqnarray*}

 Let $\partial_E$, $\bar \partial_E$ denote the exterior derivatives on $E$ induced by the Riemannian  connection  $\nabla$ on $X$.  
Since $\nabla$ preserves  $T^{(1,0)}X$ and   $T^{(0,1)}X$, they induce exterior derivatives on $E'$ and $E''$ denoted by $\partial_{E'}$, $\bar\partial_{E'}$, $\partial_{E''}$, $\bar\partial_{E''}$.

\begin{definition} \label{curvop}
The curvature operators of $E'$ and $E''$ are $R_{E'}=(\partial_{E'} + \bar \partial_{E'} )^2$ and $R_{E''}=(\partial_{E''} + \bar \partial_{E''} )^2$ respectively. 
\end{definition}

\begin{definition} \label{bracket}
Let $\{e^i\}$ be a  local frame of $E$.  For 
\begin{eqnarray*}
\psi  =  \psi_i e^i \in \Omega^{p,q}(E) \ \mbox{and }  \
 \xi  =  \xi_ie^i \in \Omega^{p',q'}(E),
 \end{eqnarray*} 
we set
\[ 
\{\psi, \xi \}= \langle e^i, e^j \rangle
 \psi_i \wedge \bar \xi_j
 \in \Omega^{p+q',q+p'}
\] 
where $ \langle \cdot, \cdot \rangle$ is  the Hermitian metric on $E$ induced by identification $TX  \simeq T^{(1,0)}X$.
\end{definition}

\begin{lemma} \label{form4} 
With $\{ \cdot, \cdot \}$ given as in Definition~\ref{bracket}, 
\[
 \{ \partial u, \partial u \} = - \{ \bar \partial \bar u, \bar \partial \bar u \} , \ \  \{ \partial \bar u, \partial \bar u \} = - \{ \bar \partial u, \bar \partial u \}.
\]
for any smooth map $u : M \rightarrow X $ between  K\"ahler manifolds, 
\end{lemma}
\begin{proof}
In local  holomorphic coordinates $(z^\alpha)$ on $M$ and  normal coordinates $(w^i)$ on $X$ at $u(p)$,
\begin{eqnarray*}
\{ \partial u, \partial u\} 
=
\sum_i
\frac{\partial{u^i}}{ {\partial{ z^\beta}}}  \overline{\frac{\partial{u^i}}{\partial{ z^\alpha}}} dz^\beta \wedge d\bar z^\alpha
=
-\sum_i
  \frac{\partial{\bar u^i}}{\partial{ \bar z^\alpha}} 
 \overline{ \frac{\partial{ \bar u^i}}{ {\partial{ \bar z^\beta}}}
 }
  d\bar z^\alpha \wedge dz^\beta
= -\{\bar \partial \bar u, \bar \partial \bar u\}
\end{eqnarray*}
and
\begin{eqnarray*}
\{ \partial \bar u, \partial \bar u\} 
=
 \sum_i
\frac{\partial \bar u^i}{ {\partial{ z^\beta}}}  \overline{\frac{\partial \bar u^i}{\partial{ z^\alpha}}} dz^\beta \wedge d\bar z^\alpha 
= 
  - \sum_i
 \frac{\partial  u^i}{\partial{ \bar z^\alpha}} 
 \overline{
 \frac{\partial  u^i}{ {\partial{\bar  z^\beta}}} 
 }
 d\bar z^\alpha \wedge dz^\beta =-\{\bar \partial u, \bar \partial u\}.
\end{eqnarray*}
\end{proof} 
\begin{lemma} \label{easycommute}
For any smooth map $u : M \rightarrow X $ between  K\"ahler manifolds, we have
\[
\partial_{E'} \bar \partial u = - \bar \partial_{E'} \partial u, \ \ \partial_{E''} \bar \partial \bar u = - \bar \partial_{E''} \partial \bar u
\]
and
\[
\partial_{E'} \partial u = 0,  \  \bar \partial_{E'} \bar \partial u = 0, \  \partial_{E''} \partial \bar u = 0,  \  \bar \partial_{E''} \bar \partial \bar u = 0.
\]
\end{lemma}

\begin{proof} We will prove the first equality,  all the other being  similar.  In local  holomorphic coordinates $(z^\alpha)$ on $M$ and  normal coordinates $(w^i)$ on $X$ at $u(p)$,

\begin{eqnarray*}
\partial_{E'} \bar \partial u = 
\frac{\partial^2 u^i}{\partial z^{\alpha} \partial \bar z^{\beta}}  dz^\alpha \wedge d\bar z^{\beta} \otimes \frac{\partial}{\partial{ u^i}}  =  
-\frac{\partial^2 u^i}{\partial \bar z^{\beta} \partial z^{\alpha}}  d\bar z^\beta \wedge d z^\alpha \otimes \frac{\partial}{\partial{ u^i}}
= -\bar \partial_{E'}  \partial u.
\end{eqnarray*}
\end{proof}

We now state the well-known Siu-Bochner formula \cite[Proposition 2]{siu1}.  For the sake of completeness, 
 we provide a proof in  the appendix to this chapter (cf.~Section~\ref{boch-appendix}).
\begin{theorem}[Siu-Bochner formula]
\label{siubochner}
For a harmonic map $u: M \rightarrow X$ between K\"{a}hler manifolds, 
\begin{eqnarray*}
 \partial  \bar \partial\{\bar \partial u,\bar \partial u \}\wedge \frac{ \omega^{n-2}}{(n-2)!} \nonumber
& = &    \left(4\left|\partial_{E'}\bar \partial u \right|^2 +Q\right) \frac{ \omega^n}{n!} \ = \    \partial  \bar \partial\{\bar \partial \bar u,\bar \partial \bar u \}\wedge \frac{ \omega^{n-2}}{(n-2)!}. 
\end{eqnarray*}
\end{theorem}

The following is the variation of the Siu's Bochner Formula, inspired by the work of  Mochizuki (cf. \cite[Proposition 21.42]{mochizuki-memoirs}) for harmonic metrics.
\begin{theorem} \label{mochizukibochnerformula}
For a harmonic map $u: M \rightarrow X$ between K\"{a}hler manifolds, 
\begin{eqnarray*}
  d \{\bar \partial_{E'}  \partial u,  \bar \partial u -  \partial u\} \wedge \frac{ \omega^{n-2}}{(n-2)!} 
 &= &  \left(8\left|{\partial}_{E'} \bar \partial u\right|^2+2Q   \right)  \wedge \frac{\omega^n}{n!}.
  \end{eqnarray*}
\end{theorem}

\begin{proof}
They key observation is that, since  $\partial \{ \partial_{E'}  \bar\partial u, \bar \partial u\} \wedge\frac{ \omega^{n-2}}{(n-2)!}
$ is an $(n+1,n-1)$-form and $\bar \partial \{\partial_{E'} \bar \partial u, \partial u\}\wedge \frac{ \omega^{n-2}}{(n-2)!} $ is an $(n-1,n+1)$-form,  these two forms  are both   identically equal to zero.  Thus, by applying Lemma~\ref{form4}  and Lemma~\ref{easycommute}, 
\begin{eqnarray} 
\partial \{  \partial_{E'}  \bar \partial u, \partial u -\bar  \partial u\} \wedge\frac{ \omega^{n-2}}{(n-2)!}
\nonumber 
& = &  \partial \{\partial_{E'} \bar \partial u,  \partial u\} \wedge\frac{ \omega^{n-2}}{(n-2)!}
\nonumber \\
& = &-  \partial \{\bar \partial_{E'} \partial u,  \partial u \} \wedge\frac{ \omega^{n-2}}{(n-2)!} \nonumber \\
  & = &
-  \partial \bar \partial \{ \partial u, \partial u\} \wedge \frac{ \omega^{n-2}}{(n-2)!} \nonumber 
  \nonumber \\
  & = & \partial \bar \partial\{\bar \partial \bar u, \bar \partial \bar u\} \wedge \frac{ \omega^{n-2}}{(n-2)!}
  \label{sb1},
 \\
 \bar\partial \{\partial_{E'}\bar \partial u, \partial u - \bar \partial u\} \wedge \frac{ \omega^{n-2}}{(n-2)!} & = & - \bar\partial \{\partial_{E'}\bar \partial u, \bar \partial u\} \wedge \frac{ \omega^{n-2}}{(n-2)!} 
  \nonumber \\  
  & = &  -  \bar\partial\partial\{\bar \partial u, \bar \partial u\} \wedge \frac{ \omega^{n-2}}{(n-2)!}  \nonumber  \\ 
& = & \partial \bar\partial\{\bar \partial u, \bar \partial u\} \wedge \frac{ \omega^{n-2}}{(n-2)!}.
 \label{sb2}
 \end{eqnarray}
Thus, we obtain
\begin{eqnarray*} \label{mochizukitrick}
 \lefteqn{d \{\bar \partial_{E'}  \partial u,  \bar \partial u -  \partial u\} \wedge \frac{ \omega^{n-2}}{(n-2)!} }\\
  & = &  d \{\partial_{E'} \bar\partial u, \partial u - \bar \partial u\} \wedge \frac{ \omega^{n-2}}{(n-2)!}
 \ \ \mbox{(by Lemma~\ref{easycommute})}
 \\
 & = &  (\partial+\bar \partial) \{\partial_{E'} \bar\partial u, \partial u - \bar \partial u\} \wedge \frac{ \omega^{n-2}}{(n-2)!}
 \\
& = &  \left( \partial \bar\partial\{\bar \partial u, \bar \partial u\}+ \partial \bar \partial\{\bar \partial \bar u, \bar \partial \bar u\}  \right)\wedge \frac{ \omega^{n-2}}{(n-2)!}  \ \ \mbox{(by (\ref{sb1}) and (\ref{sb2}))}.
 \end{eqnarray*} 
 Thus, the asserted identity follows from
Theorem~\ref{siubochner}. 
\end{proof}

\section{From a K\"ahler manifold to a Riemannian manifold}
\label{sec:KtoR}

In \cite{sampson}, Sampson introduced the following notion of negative curvature.
\begin{definition}
Let $X$ be a Riemannian manifold and  $R_{ijkl}$ be its Riemann curvature tensor.
Then $X$ is said to have {\it Hermitian-negative curvature} if  
\begin{equation} \label{Hnc}
R_{ijkl} A^{i\bar l} A^{j\bar k} \leq 0
\end{equation} 
for any Hermitian semi-positive matrix $A = (A^{i\bar l})$. 
\end{definition}
 
 \begin{remark} \label{symHer}
 It is well known that symmetric spaces of non-compact type have Hermitian-negative curvature.  For an exposition, see \cite[Chapter 5]{loustau}.
 \end{remark}
Let $(M,g)$ be K\"ahler manifold, $(X,h)$ be a Riemannian manifold and $u:M \rightarrow X$ be a smooth map.
Again consider the pullback vector bundle 
\begin{equation} \label{ERiemannian}
E:=u^{-1}(TX \otimes \C)
\end{equation}
and denote by $\Omega^{p,q}(E)$  the space of   $E$-valued $(p,q)$-forms respectively.  
If $(x^i)$ are local  coordinates in $X$, then
$\{\frac{\partial}{\partial u^i}  :=  \frac{\partial}{\partial x^i} \circ u \}$ 
is a local  frame  of $E$.
\begin{definition}
Let  
\begin{eqnarray*}
\mbox{I}_{(1,0)}: T^{(1,0)}M \hookrightarrow TM \otimes \C, & & 
\mbox{I}_{(0,1)}: T^{(0,1)}M \hookrightarrow TM \otimes \C.
\end{eqnarray*}
be the inclusions and define
\begin{eqnarray*}
\partial u = du  \circ \mbox{I}_{(1,0)} \in \Omega^{(1,0)}(E),& & 
\bar \partial u = du  \circ \mbox{I}_{(0,1)} \in \Omega^{(0,1)}(E).
\end{eqnarray*}
In local holomorphic coordinates of $M$,
\begin{eqnarray*}
\partial u=\partial u^i   \frac{\partial}{\partial u^i} 
& & 
\bar \partial u=\bar \partial u^i\frac{\partial}{\partial u^i}.
\end{eqnarray*}
\end{definition}
 Let $\partial_E$, $\bar \partial_E$ denote the exterior derivatives on $E$ induced by the Riemannian  connection  $\nabla$ on $X$.

\begin{definition} \label{curvop}
The curvature operator of $E$  is $R_E=(\partial_E + \bar \partial_E )^2$. 
\end{definition}

\begin{definition} \label{bracket}
Let $\{e^i\}$ be a  local frame of $E$.  For 
\begin{eqnarray*}
\psi  =  \psi_i e^i \in \Omega^{p,q}(E) \ \mbox{and }  \
 \xi  =  \xi_ie^i \in \Omega^{p',q'}(E),
 \end{eqnarray*} 
we set
\[ 
\{\psi, \xi \}= \langle e^i, e^j \rangle
 \psi_i \wedge \bar \xi_j
 \in \Omega^{p+q',q+p'}
\] 
where  $ \langle \cdot, \cdot \rangle$ is  the pullback to $E$ of the Hermitian metric defined by extending the Riemannian metric $h$ of  $X$.
\end{definition}

\begin{lemma} \label{form4'} 
With $\{ \cdot, \cdot \}$ given as in Definition~\ref{bracket}, 
\[
 \{ \partial u, \partial u \} = - \{ \bar \partial  u, \bar \partial u \}.
\]
for any smooth map $u : M \rightarrow X $.
\end{lemma}
\begin{proof}
In local  holomorphic coordinates $(z^\alpha)$ on $M$ and  normal coordinates $(x^i)$ on $X$ at $u(p)$,
\begin{eqnarray*}
\{ \partial u, \partial u\} 
=
\sum_i
\frac{\partial{u^i}}{ {\partial{ z^\beta}}}  \overline{\frac{\partial{u^i}}{\partial{ z^\alpha}}} dz^\beta \wedge d\bar z^\alpha
=
-\sum_i
  \frac{\partial{u^i}}{\partial{ \bar z^\alpha}} 
 \overline{ \frac{\partial{u^i}}{ {\partial{ \bar z^\beta}}}
 }
  d\bar z^\alpha \wedge dz^\beta
= -\{\bar \partial u, \bar \partial u\}
\end{eqnarray*}
\end{proof}

\begin{lemma} \label{easycommute'}
For any smooth map $u : M \rightarrow X $, we have
\[
\partial_{E} \bar \partial u = - \bar \partial_{E} \partial u, \ \ \partial_{E} \partial u = 0,  \  \bar \partial_{E} \bar \partial u = 0.
\]
\end{lemma}

\begin{proof} We will prove the first equality,  all the other being  similar.  In local  holomorphic coordinates $(z^\alpha)$ on $M$ and  normal coordinates $(x^i)$ on $X$ at $u(p)$,
\begin{eqnarray*}
\partial_{E'} \bar \partial u = 
\frac{\partial^2 u^i}{\partial z^{\alpha} \partial \bar z^{\beta}}  dz^\alpha \wedge d\bar z^{\beta} \otimes \frac{\partial}{\partial{ u^i}}  =  
-\frac{\partial^2 u^i}{\partial \bar z^{\beta} \partial z^{\alpha}}  d\bar z^\beta \wedge d z^\alpha \otimes \frac{\partial}{\partial{ u^i}}
= -\bar \partial_{E'}  \partial u.
\end{eqnarray*}
\end{proof}

 The following version of Sampson's Bochner formula is contained in \cite[Lemma 6.10]{liu-yang}.
 
 \begin{theorem}[Sampson's Bochner formula]
\label{sampsonbochner}
For a harmonic map $u: (M,g) \rightarrow (X,h)$ from a  K\"{a}hler manifold to a Riemannian manifold, 
\begin{eqnarray}
 \partial  \bar \partial\{\bar \partial u,\bar \partial u \}\wedge \frac{ \omega^{n-2}}{(n-2)!} 
& = &    \left(4\left|\partial_E\bar \partial u \right|^2 +4Q_0\right) \frac{ \omega^n}{n!}  \label{SampsonFormula}
\end{eqnarray}
where 
\[
Q_0=-2g^{\alpha \bar \delta} g^{\gamma \bar \beta} R_{ijkl}
\frac{\partial u^i}{\partial z^\alpha} \frac{\partial u^k}{\partial \bar z^\beta} \frac{\partial u^j}{\partial z^\gamma} \frac{\partial u^l}{\partial \bar z^\delta}\]
in local coordinates $(z^\alpha)$ of $M$ and $(x^i)$ of $X$. 
\end{theorem}

By applying a similar proof as Theorem~\ref{mochizukibochnerformula}, we obtain a generalization the following.
 \begin{theorem} \label{mochizukibochnerformula'}
For a harmonic map $u: M \rightarrow X$ from a K\"{a}hler manifold to a Riemannian manifold, 
\begin{eqnarray*}
  d \{\bar \partial_E  \partial u,  \bar \partial u -  \partial u\} \wedge \frac{ \omega^{n-2}}{(n-2)!} 
 &= &  \left(8\left|{\partial}_E \bar \partial u\right|^2+8Q_0   \right)  \wedge \frac{\omega^n}{n!}.
  \end{eqnarray*}
\end{theorem}
\begin{proof} 
We apply a similar argument as in the proof of Theorem~\ref{mochizukibochnerformula}.  Again, we use the vanishing $\partial \{ \partial_{E'}  \bar\partial u, \bar \partial u\} \wedge\frac{ \omega^{n-2}}{(n-2)!}=0=\bar \partial \{\partial_{E'} \bar \partial u, \partial u\}\wedge \frac{ \omega^{n-2}}{(n-2)!}$ to show \begin{eqnarray} 
\partial \{  \partial_{E'}  \bar \partial u, \partial u -\bar  \partial u\} \wedge\frac{ \omega^{n-2}}{(n-2)!}
& = &  \partial \{\partial_{E'} \bar \partial u,  \partial u\} \wedge\frac{ \omega^{n-2}}{(n-2)!}
\nonumber \\
& = &-  \partial \{\bar \partial_{E'} \partial u,  \partial u \} \wedge\frac{ \omega^{n-2}}{(n-2)!}\ \ \ \mbox{(by Lemma~\ref{easycommute'})}.\nonumber \\
  & = &
-  \partial \bar \partial \{ \partial u, \partial u\} \wedge \frac{ \omega^{n-2}}{(n-2)!}
  \nonumber
  \nonumber \\
  & = & \partial \bar \partial\{\bar \partial u, \bar \partial u\} \wedge \frac{ \omega^{n-2}}{(n-2)!}
 \ \ \ \mbox{(by Lemma~\ref{form4'})}, \label{sb1'}
 \\
 \bar\partial \{\partial_{E'}\bar \partial u, \partial u - \bar \partial u\} \wedge \frac{ \omega^{n-2}}{(n-2)!} & = & - \bar\partial \{\partial_{E'}\bar \partial u, \bar \partial u\} \wedge \frac{ \omega^{n-2}}{(n-2)!} 
  \nonumber \\  & = &  -  \bar\partial\partial\{\bar \partial u, \bar \partial u\} \wedge \frac{ \omega^{n-2}}{(n-2)!} \nonumber \\
& = & \partial \bar\partial\{\bar \partial u, \bar \partial u\} \wedge \frac{ \omega^{n-2}}{(n-2)!}.
 \label{sb2'}
 \end{eqnarray}
Thus,
\begin{eqnarray*} \label{mochizukitrick}
d \{\bar \partial_{E'}  \partial u,  \bar \partial u -  \partial u\} \wedge \frac{ \omega^{n-2}}{(n-2)!}   & = &  d \{\partial_{E'} \bar\partial u, \partial u - \bar \partial u\} \wedge \frac{ \omega^{n-2}}{(n-2)!}
 \ \ \mbox{(by Lemma~\ref{easycommute'})}
 \\
 & = &  (\partial+\bar \partial) \{\partial_{E'} \bar\partial u, \partial u - \bar \partial u\} \wedge \frac{ \omega^{n-2}}{(n-2)!}
 \\
& = & 2\partial \bar\partial\{\bar \partial u, \bar \partial u\} \wedge \frac{ \omega^{n-2}}{(n-2)!}  \ \ \mbox{(by (\ref{sb1'}) and (\ref{sb2'}))}.
 \end{eqnarray*} 
 Thus, the asserted identity follows from
Theorem~\ref{sampsonbochner}. 
\end{proof}

\section{Appendix to Chapter~\ref{chap:bochner}}

\label{boch-appendix}

In this section, we provide a proof of Siu's Bochner formula for the sake of completeness. 
First, by a simple computation (cf.~\cite{liu-yang}), for $\psi   \in \Omega^{p,q}(E)$ and $\xi   \in \Omega^{p',q'}(E)$, 
\begin{eqnarray}\label{bracket}
\partial\{\psi, \xi \}&=&\{\partial_E\psi, \xi \}+ (-1)^{p+q}\{\psi, \bar \partial_E\xi \} \nonumber \\
\bar \partial\{\psi, \xi \}&=&\{\bar\partial_E\psi, \xi \}+(-1)^{p+q}\{\psi,  \partial_E\xi \} \\
\partial \bar\partial\{\psi, \xi \}&=&\{\partial_E\bar\partial_E\psi, \xi \}+ (-1)^{p+q+1}\{\bar\partial_E\psi, \bar \partial_E\xi \}  \nonumber \\
& & \ + (-1)^{p+q}\{\partial_E\psi,  \partial_E\xi \}+\{\psi, \bar \partial_E \partial_E\xi \} \nonumber.
\end{eqnarray}\label{didibarbracket}

\begin{lemma}\label{brac}For any smooth map $u : M \rightarrow X $ between  K\"ahler manifolds, we have

\[
\partial  \bar \partial\{\bar \partial u,\bar \partial u \}\wedge \frac{ \omega^{n-2}}{(n-2)!}=\left(-\{\partial_{E'}\bar \partial u,  \partial_{E'}\bar \partial u \}+ \{\bar \partial u, R_{E'}( \bar \partial u) \}\right)\wedge \frac{ \omega^{n-2}}{(n-2)!}
\]
and
\[
\partial  \bar \partial\{\bar \partial \bar{u},\bar \partial \bar{u} \}\wedge \frac{ \omega^{n-2}}{(n-2)!}
= \left(-\{\partial_{E''}\bar \partial \bar{u},  \partial_{E''}\bar \partial \bar{u} \}+ \{\bar \partial \bar{u}, R_{E''}( \bar \partial \bar{u}) \}\right)\wedge \frac{ \omega^{n-2}}{(n-2)!}.
\] 
\end{lemma}
\begin{proof}
By setting $\psi=\xi= \bar \partial u \in \Omega^{0,1}(E')$ in (\ref{bracket}) and repeatedly using the fact that $\bar \partial_{E'} \bar \partial u=0$ (cf.~Lemma~\ref{easycommute}), 
\begin{eqnarray*}
\partial  \bar \partial\{\bar \partial u,\bar \partial u \}
& =&-\{\partial_{E'}\bar \partial u,  \partial_{E'}\bar \partial u \}+ \{\bar \partial u, \bar \partial_{E'} \partial_{E'}\bar \partial u \}
\\
& =&-\{\partial_{E'}\bar \partial u,  \partial_{E'}\bar \partial u \}+ \{\bar \partial u, (\bar \partial_{E'} \partial_{E'} +   \partial_{E'} \bar \partial_{E'}+ \bar \partial_{E'}^2) \bar \partial u  \}. 
\end{eqnarray*}
Since $\{\bar \partial u, \partial_{E'}^2 \bar \partial u\} \wedge \frac{ \omega^{n-2}}{(n-2)!}$ is an 
$(n-1,n+1)$-form and hence zero  for dimensional reasons, we can complete the square to obtain
\begin{eqnarray}
\partial  \bar \partial\{\bar \partial u,\bar \partial u \}\wedge \frac{ \omega^{n-2}}{(n-2)!}
& =&\left(-\{\partial_{E'}\bar \partial u,  \partial_{E'}\bar \partial u \}+ \{\bar \partial u, (\partial_{E'} + \bar \partial_{E'} )^2 \bar \partial u\}\right) \wedge \frac{ \omega^{n-2}}{(n-2)!} \nonumber\\
& =&\left(-\{\partial_{E'}\bar \partial u,  \partial_{E'}\bar \partial u \}+ \{\bar \partial u, R_{E'}( \bar \partial u) \} \right)\wedge \frac{ \omega^{n-2}}{(n-2)!} \nonumber
\end{eqnarray}
which proves the first equation.
The second equation follows by setting $\psi=\xi= \bar \partial \bar{u} \in \Omega^{0,1}(E'')$ in (\ref{bracket}) and following exactly the same computation.
\end{proof}

\begin{lemma}\label{aux}
For any harmonic map $u : M \rightarrow X $ between  K\"ahler manifolds, we have
\begin{eqnarray*}
-\{\partial_{E'}\bar \partial u,  \partial_{E'}\bar \partial u \} \wedge \frac{ \omega^{n-2}}{(n-2)!} & = &  
4\left|\partial_{E'}\bar \partial u \right|^2 \frac{ \omega^n}{n!} \\
-\{\partial_{E''}\bar \partial \bar{u},  \partial_{E''}\bar \partial \bar{u} \} \wedge \frac{ \omega^{n-2}}{(n-2)!} & = & 
4\left|\partial_{E''}\bar \partial \bar u \right|^2 \frac{ \omega^n}{n!}.
\end{eqnarray*}
\end{lemma}

\begin{proof}
From a straightforward computation (cf.~\cite[Lemma 6.2]{liu-yang}), 
\[
-\{\phi, \phi\}\wedge \frac{ \omega^{n-2}}{(n-2)!}=4(|\phi|^2-|\mbox{Tr}_\omega\phi|^2)\frac{ \omega^n}{n!}, \ \ \ \forall \phi \in \Omega^{1,1}(E).
\]
We apply this identity with $\phi=\partial_{E'}\bar \partial u$ (resp.~$\phi=\partial_{E''}\bar \partial \bar u$).   Since $u$ is harmonic, $\mbox{Tr}_{\omega} \partial_{E'}\bar \partial u=0$ and $\mbox{Tr}_{\omega} \partial_{E''}\bar \partial \bar u=0$ by  \cite[Corollary 3.12]{liu-yang} and the K\"ahler identities.
\end{proof}

\begin{lemma}\label{tree}
For any smooth map $u : M \rightarrow X $ between  K\"ahler manifolds, we have
\[
\{\bar \partial  u, R_{E'}  (\bar \partial  u) \}  = -R_{i \bar j l \bar k}  \partial u^l \wedge \bar \partial \bar{u}^k \wedge \bar \partial u^i \wedge \partial \bar{u}^j=\{\bar \partial \bar u, R_{E''}  (\bar \partial \bar u) \}.
\]
\end{lemma}
\begin{proof}
We compute
 \begin{eqnarray}\label{dminor}
\{\bar \partial  u, R_{E'}  (\bar \partial  u) \}  &=&
\{\bar \partial u^i\frac{\partial}{\partial u^i}, R_{E'}(\bar \partial u^j\frac{\partial}{\partial u^j})\} \nonumber \\
&=&\{\bar \partial u^i\frac{\partial}{\partial u^i}, \bar \partial u^j \wedge R_{E'}(\frac{\partial}{\partial u^j})\} \nonumber \\
&=&\{\bar \partial u^i\frac{\partial}{\partial u^i}, \bar \partial u^j \wedge R^s_{  j k \bar l}\partial u^k \wedge \overline{\partial u^l} \frac{\partial}{\partial u^s}\} \nonumber\\
 &=&\bar \partial u^i \wedge \partial \bar{u}^j \wedge G_{i \bar s} \overline{R^s_{  j k \bar l}\partial u^k \wedge \overline{\partial u^l}} \\
 &=&\bar \partial u^i \wedge \partial \bar{u}^j \wedge \overline{ G_{s \bar i} R^s_{  j k \bar l}\partial u^k \wedge \overline{\partial u^l}} \nonumber\\
 &=&\bar \partial u^i \wedge \partial \bar{u}^j \wedge \overline{  R_{  j \bar i k \bar l}\partial u^k \wedge \overline{\partial u^l}} \ \mbox{(by \cite{koba}, (1.7.10))} \nonumber\\
 &=& -R_{i \bar j l \bar k}  \partial u^l \wedge \bar \partial \bar{u}^k \wedge \bar \partial u^i \wedge \partial \bar{u}^j \ \mbox{(by \cite{koba}, (1.7.21))}
  \nonumber
 \end{eqnarray}
 which proves the first equality.
Similarly as in (\ref{dminor}),
\begin{eqnarray}\label{emajor}
\{\bar \partial  \bar{u}, R_{E''}  (\bar \partial  \bar{u}) \}  &=&
\{\bar \partial \bar{u}^i\frac{\partial}{\partial \bar{u}^i}, R_{E''}(\bar \partial \bar{u}^j\frac{\partial}{\partial \bar{u}^j})\} \nonumber \\ &=&
\{\bar \partial \bar{u}^i\frac{\partial}{\partial \bar{u}^i}, \bar \partial \bar{u}^j \wedge R_{E''}(\frac{\partial}{\partial \bar{u}^j})\} \nonumber \\
&=&\{\bar \partial \bar{u}^i\frac{\partial}{\partial \bar{u}^i}, \bar \partial \bar{u}^j \wedge \overline{R^s_{  j k \bar l}}\partial \bar{u}^k \wedge \bar \partial u^l  \frac{\partial}{\partial \bar{u}^s} \} \ \ \ \ \mbox{(by \cite[(1.5.10)]{koba})}\\
 &=&\bar \partial \bar{u}^i \wedge \partial u^j \wedge G_{s \bar i} R^s_{  j k \bar l}\bar \partial u^k \wedge \partial \bar{u}^l  \nonumber \\ 
 &=&\bar \partial \bar{u}^i \wedge \partial u^j \wedge \  R_{  j \bar i k \bar l}\bar \partial u^k  \wedge \partial \bar{u}^l  \ \mbox{(by \cite[(1.7.10)]{koba})} \nonumber\\
 &=& -R_{  j \bar i k \bar l}  \partial u^j \wedge \bar \partial \bar{u}^i \wedge \bar \partial u^k \wedge \partial \bar{u}^l 
  \nonumber \\
   &=& -R_{ k \bar l j \bar i }  \partial u^j \wedge \bar \partial \bar{u}^i \wedge \bar \partial u^k \wedge \partial \bar{u}^l \ \mbox{(by \cite[(1.7.22)]{koba})}
  \nonumber \\
   &=& -R_{i \bar j l \bar k}  \partial u^l \wedge \bar \partial \bar{u}^k \wedge \bar \partial u^i \wedge \partial \bar{u}^j.
  \nonumber
\end{eqnarray}
 This proves the second equality.
\end{proof}

\begin{lemma} \label{palm}
For any harmonic map $u : M \rightarrow X $ between  K\"ahler manifolds, we have  
\[
\left|\partial_{E''}\bar \partial \bar u \right|^2=\left|\bar  \partial_{E''}\partial \bar u \right|^2=\left|\partial_{E'}\bar \partial u \right|^2.
\]  
\end{lemma}

\begin{proof}
The left equality follows from Lemma~\ref{easycommute} and the right from the fact that conjugation is an isometry.
\end{proof}

\noindent \begin{siubochner}
Combine Lemma~\ref{branch}, Lemma~\ref{brac}, Lemma~\ref{aux}, Lemma~\ref{tree} and Lemma~\ref{palm}.
\end{siubochner}

\chapter{Harmonic maps from Riemann surfaces} \label{chap:Riemann surfaces}

In this chapter,  we present some results about harmonic section with possibly infinite energy from punctured Riemann surfaces.   The target space  $\tilde{X}$ is an NPC space.  For smooth targets, the existence results are due to \cite{lohkamp} and \cite{wolf}. 
   We will also prove new uniqueness theorems for these harmonic maps.   In the case when the harmonic map is of finite energy,  uniqueness of  harmonic maps  follows from previous work (cf.~\cite{mese}, \cite{daskal-meseUnique}) of the authors. Additionally, we give precise estimates  of these harmonic maps near the punctures.  These results are used to find pluriharmonic maps from smooth quasi-projective varieties  of arbitrary dimensions in Chapter~\ref{chap:higher dimensions}.

Throughout this chapter, we use the following notation:
\begin{itemize}
\item
$\bar \RR$ is a compact Riemann surface 
\item $\RR=\bar \RR \backslash \{p^1, \dots, p^n\}$ is a punctured Riemann surface
\item $\Pi:\tilde \RR \rightarrow \RR$ is the universal covering map
\item $\rho:\pi_1(\RR) \rightarrow \mathsf{Isom}(\tilde X)$ is a homomorphism.
\end{itemize}
For each $j=1, \dots, n$, 
\begin{itemize}
\item $\lambda^j \in \pi_1(\RR)$ is the element associated to the loop around the puncture $p^j$
\item $I^j:=\rho(\lambda^j) \in \mathsf{Isom}(\tilde X)$.
\end{itemize}
The main results of this chapter are the following two theorems.

\begin{theorem}[Existence] \label{existence}
Assume that $\tilde X$ is an NPC space.    
If 
\begin{itemize}
 \item[(A)] the action of $\rho(\pi_1(\RR))$ is proper (cf.~Definition~\ref{proper}), and
 \item[(B)] $I^j$ has exponential decay (cf.~Definition~\ref{jz}) for $j=1,\dots, n$,
 \end{itemize}then there exists a harmonic section $u:  \RR \rightarrow \tilde \RR \times_\rho \tilde X$  with logarithmic energy growth (cf.~Definition~\ref{logdec} below).
\end{theorem}

\begin{theorem}[Uniqueness] \label{uniqueness}
The section $u$ of Theorem~\ref{existence} is the unique harmonic section with logarithmic energy growth in the following cases: 
\begin{itemize}
\item[(i)] $\tilde X$ is a negatively curved space as in Definition~\ref{neg}
\item[(ii)] $\tilde X$ is an irreducible symmetric space of non-compact type or
 is an irreducible locally finite Euclidean building such that the action of $\rho(\pi_1(\RR))$ does not fix an unbounded closed convex strict subset  of $\tilde X$.
\end{itemize}
 \end{theorem}

We explain some of the terminology contained Theorem~\ref{existence}.  Loosely speaking,  {\it sub-logarithmic growth} means that, near  the punctures of the Riemann surface,  the section grows at most logarithmically, and  {\it logarithmic energy growth} means that the energy of the section grows like a logarithmic function.

In order to  give a precise definition, consider the universal cover $\R \rightarrow \Sp^1$  with the free group  $2\pi \Z \simeq \pi_1(\Sp^1)$  acting on $\R$ by translation.  

\begin{definition} \label{infE}
For a homomorphism $\rho:  2\pi \Z \simeq \pi_1(\Sp^1) \rightarrow \tilde X$, 
let  $\R \times_\rho \tilde X \rightarrow \Sp^1$ be the flat $\tilde X$-fiber bundle defined by $\rho$ (cf.~Subsection~\ref{subsec:donaldsonsection}).
Define $\mathsf{E}_{\rho}$ to be  the infimum of the (vertical) energies of sections  $\Sp^1 \rightarrow \R \times_{\rho} \tilde X$ (cf.~(\ref{vertical})). 
\end{definition}

\begin{remark}
If  $\Delta_I$ is the translation length of $I=\rho([\Sp^1])$, then
\[
\mathsf{E}_{\rho}= \frac{\Delta^2_{I}}{2\pi}.
\]
\end{remark}

Next, we will
\begin{equation}  \label{fixeddisk}
\mbox{fix a conformal disk $\D^j \subset \bar \RR$ centered at each puncture $p^j$}
\end{equation}
such that $\D^i \cap \D^j=\emptyset$ for $i\neq j$.  Furthermore, we use the following notation:
\begin{itemize}
\item $\D^{j*}=\D^j \backslash \{0\}$
\item $\D^j_r=\{z \in \D^j: |z|<r\}$
\item $\RR_r=\RR \backslash \bigcup_{j=1}^n \D^j_r$
\item $2\pi \Z \simeq \langle \lambda^j  \rangle$ is  the free group generated by $\lambda^j$
\item $\rho^j:\langle \lambda^j \rangle \rightarrow \mathsf{Isom}(\tilde X)$ be the restriction of $\rho: \pi_1(\RR) \rightarrow \mathsf{Isom}(\tilde X)$.
\end{itemize}

Fix $P_0 \in \tilde X$ and a fundamental domain $F$ of $\tilde \RR$.  Let  $f_0$ be the section of the fiber bundle $\tilde \RR \times_\rho \tilde X \rightarrow \RR$  such that, for any $p \in \RR \cap \Pi(F)$,   $f_0(p)=[(\tilde p,P_0)]$ where $\tilde p = \Pi^{-1}(p) \cap F$.  (Note that  $\Pi(F)$ is of full measure in $\RR$.) Let  $p^j \in \RR$ be a puncture and define
\begin{equation} \label{tildeyes}
\delta_j:\D^{*j} \rightarrow [0,\infty), \ \ \ \ \delta_j(z) = \essinf_{\{  z \in \D^{*j}\}} d( f(z),f_0(z)).
\end{equation}

\begin{definition} \label{logdec} 
We say  a section $f:\RR \rightarrow \tilde \RR \times_\rho \tilde X$  has \emph{sub-logarithmic growth} if for any $j=1, \dots, n$
\[
\lim_{|z| \rightarrow 0} \delta_{j}(z) +\epsilon \log |z|=-\infty \mbox{ in }\D^{j*} .
\]
By the triangle inequality, this definition is independent of the choice of  $P_0 \in \tilde X$.
We say $f$ has \emph{logarithmic energy growth} if there exists $\rho_0>0$ such that 
\begin{equation} \label{Cbddef}
\sum_{j=1}^n
\mathsf{E}_{\rho^j} \log \frac{1}{r} \leq E^{f}[\RR_r] \leq \sum_{j=1}^n\mathsf{E}_{\rho^j} \log \frac{1}{r}+C, \ \ \  0<r<r_0<\rho_0
\end{equation}
for some constant $C>0$. 
 \end{definition}

\section{Infinite energy harmonic maps from a punctured disk}\label{disc}

Before we consider the domain to be a general punctured Riemann surface, we first start by studying harmonic maps from a punctured disk $\bar \D^*$ to an NPC space $\tilde X$.  Identify the boundary of the disk with the circle; i.e.
\[
\partial \D=\Sp^1,
\]
This induces a natural identification  
\[
2\pi \Z \simeq \pi_1(\Sp^1) \simeq  \pi_1( \bar \D^*). 
\]
The main result of this section is the following.

\begin{theorem}[Existence and Uniqueness of the Dirichlet solution on $\D^*$] \label{exists}
Assume the following:
\begin{itemize}
\item $\rho:2\pi \Z \simeq \pi_1(\Sp^1) \simeq  \pi_1( \bar \D^*) \rightarrow \mathsf{Isom}(\tilde X)$  
is a homomorphism

 \item $k:\DDDD$ is a locally Lipschitz section
 \item $I:=\rho([\Sp^1])$ has exponential decay to its translation length
 \end{itemize}
Then there  exists a  harmonic section  
\[
u:\DDDD \mbox{  with  }u|_{\Sp^1} =k|_{\Sp^1}.
\]
  Furthermore,  
there exists a constant $C>0$   that depends only on  $\mathsf{E}_\rho$ of Definition~\ref{infE},  $a$, $b$ from Definition~\ref{jz} and the section $k$ satisfying the following properties:\\

\begin{itemize}
\item[(i)]
$\displaystyle{\mathsf{E}_\rho \log \frac{1}{r} \leq E^u[\D_{r,1}] \leq \mathsf{E}_\rho \log \frac{1}{r}+C, \ \ \ 0<r\leq 1}$
\item[(ii)]  $\displaystyle{\left| \frac{\partial u}{\partial r} \right|^2 \leq \frac{C}{r^2(-\log r)} \ \  \mbox{ and } \ \ 
\left( \left| \frac{\partial u}{\partial \theta} \right|^2 - \frac{\mathsf{E}_\rho}{2\pi} \right)   \leq \frac{C}{-\log r}
}$ \ in \  $\D_{\frac{1}{2}}^*$
 \item[(iii)]  $u$ has sub-logarithmic growth. 
\end{itemize}
Moreover, $u$ is the only harmonic section satisfying $u|_{\partial \D} =k|_{\partial \D}$ and property (iii).
\end{theorem}

\begin{proof}
Combine Lemma~\ref{(i)}, Lemma~\ref{(ii)}, Lemma~\ref{(iii)} and Lemma~\ref{uniquenessdisk} below.
\end{proof}
The purpose for the rest of this subsection is to prove the lemmas that comprise the proof of Theorem~\ref{exists}.   
We start with the following preliminary lemma about subharmonic functions.

\begin{lemma} \label{shdisk}
For $r>0$, let $\nu: \D_r^* \rightarrow \R$ be a  subharmonic function that extends as a continuous function to $\bar \D_r^*= \overline{\D_r} \backslash \{0\}$. If we assume that 
\begin{equation} \label{slowblowup}
\lim_{z \rightarrow 0} \  \nu(z) + \epsilon \log |z| =-\infty, \ \ \forall \epsilon>0,
\end{equation}
 then 
 \[
\sup_{z \in \D_r^*} \nu(z)  \leq \max_{\zeta \in \partial \D_r} \nu(\zeta).
\] 
In particular, $\nu$ extends to a subharmonic function on $\D_r$ (cf.~\cite{hayman-kennedy}).
 \end{lemma}

\begin{proof}
Let  $h:\D_r \rightarrow \R$ be  the unique Dirichlet solution  with $h|_{\partial \D_r}=\nu|_{\partial \D_r}$.  By the maximum principle,  $h$ is non-negative and bounded on $\D_r$.  The function 
\[
f_{\epsilon}(z)=\nu(z)-h(z)+\epsilon (\log |z|-\log r)
\]
is subharmonic on $\D_r^*$ with 
$
f_{\epsilon}|_{\partial \D_r} \equiv 0$ on $\partial \D_r$.
Furthermore, since $h$ is a bounded function and $\nu$ satisfies (\ref{slowblowup}), 
\[ 
\lim_{z \rightarrow 0} f_{\epsilon}(z) =-\infty.
\]
By the maximum principle, $f_{\epsilon} \leq 0$.  By letting $\epsilon \rightarrow 0$, we conclude $\nu(z) -h(z) \leq 0$  for all $z \in \D_r^*$.
Thus, 
\[
\sup_{z \in \D_r^*} \nu(z)  \leq \sup_{z \in \D_r}h(z) \leq    \max_{\zeta \in \partial \D_r} h(\zeta)=\max_{\zeta \in \partial \D_r} \nu(\zeta).
\] 
\end{proof}

\begin{lemma} \label{lowerbd}
If  $\rho$ and $I$ are as in Theorem~\ref{exists}, then for any section $f: \D_{r,r_0} \rightarrow \DDD \times_{\rho} \tilde X$,
\[
\mathsf{E}_\rho \log \frac{r_0}{r}  \leq 
  \int_{\D_{r,r_0}}  \frac{1}{r^2} \left| \frac{\partial f}{\partial \theta} \right|^2 rdr \wedge d\theta \leq E^{f}[\D_{r,r_0}], \ \ 0<r<r_0<1. 
\] 
\end{lemma}

\begin{proof}
By the definition of $\mathsf{E}_I$, any section  $c:\Sp^1 \rightarrow \R \times_{\rho} \tilde X$ satisfies
\[
\mathsf{E}_\rho
 \leq   \int_0^{2\pi} \left| \frac{\partial c}{\partial \theta} \right|^2 d\theta.
 \]
After identifying $\Sp^1$ with $\partial \D_r$,  we can view the restriction  $f|_{\partial \D_r}$ as a section  $\Sp^1 \rightarrow \R \times_{\rho} \tilde X$.  
Thus, 
\begin{eqnarray*}\label{ggeo} 
\mathsf{E}_\rho \log \frac{r_0}{r} & = & \mathsf{E}_\rho  \int_r^{r_0} \frac{1}{r^2}  rdr \wedge d\theta \nonumber\\
&  \leq &   \int_0^{2\pi} \int_r^{r_0} \frac{1}{r^2} \left| \frac{\partial f}{\partial \theta} (r,\theta) \right|^2 rdr \wedge d\theta \\
& \leq & E^f[\D_{r,r_0}] \nonumber.
\end{eqnarray*}
\end{proof}

\begin{lemma} \label{defineg}
If $\rho$, $I$ and $k:\DDDD$ are as in  Theorem~\ref{exists}, then there exists a constant $C>0$ and a  Lipschitz section $v:\DDDD$  with $v|_{\partial \D}=k|_{\partial \D}$ such that 
\begin{equation}   \label{glog}
\mathsf{E}_\rho \log \frac{r_0}{r} \leq  E^v[\D_{r,r_0}]  \leq \mathsf{E}_\rho \log \frac{r_0}{r}+C, \ \ \ 0<r<r_0\leq \frac{1}{2}.
\end{equation}
Moreover, the Lipschitz constant of $v$ and $C$ are dependent only on $\mathsf{E}_\rho$, $a$, $b$ from Definition~\ref{jz} and $k$.
\end{lemma}

\begin{proof} 
We will  first construct an one-parameter family of  sections   $\gamma_s:\Sp^1 \mapsto \R \times_{\rho} \tilde X$ as follows:
\begin{itemize}
\item If $I$ is semisimple (i.e.~elliptic or hyperbolic), let $P_{\star} \in \tilde X$ be a point where the infimum $\Delta_I$ is attained.  For any $s \in [0,\infty)$, define 
$\tilde \gamma_s:\R \rightarrow \tilde X$ to be the $\rho$-equivariant geodesic  (a constant map if $I$ is elliptic) such that $\tilde \gamma_s(0)=P_{\star}$ and  $\tilde \gamma_s(2\pi)=I(P_{\star})$ and  let $\gamma_s:\Sp^1 \rightarrow \tilde \Sp^1 \times_{\rho} \tilde X$ be the associated section.
\item
If $I$ is parabolic, then 
let  
 $c:[0,\infty) \rightarrow \tilde X$ be a geodesic ray defined in Definition~\ref{jz}.  Define the $\rho$-equivariant curve 
$\tilde \gamma_s:\R  \rightarrow \tilde X$ such that $\tilde \gamma_s|_{[0,2\pi]}$ is  the geodesic from $c(s)$ and $I(c(s))$ and  let $\gamma_s:\Sp^1 \rightarrow \tilde \Sp^1 \times_{\rho} \tilde X$ be the associated section.
\end{itemize}
We note that by the quadrilateral comparison of NPC spaces, for $\theta
\in [0,2\pi]$,
\[
d(\tilde \gamma_{s_1}(\theta), \tilde \gamma_{s_2}(\theta)) \leq (1-\frac{\theta}{2\pi}) d(c(s_1), c(s_2))+\frac{\theta}{2\pi}d(I \circ c(s_1)), I \circ c(s_2)))=|s_1-s_2|.
\]
 Thus
 \begin{equation} \label{npcquad}
\left| \frac{d}{ds} (\tilde \gamma_s(\theta)) \right| \leq 1, \ \ \ \forall  s \in (0,\infty), \  \theta \in \Sp^1.
\end{equation} 
Since  $I \in \mathsf{Isom}(\tilde X)$ has exponential decay of Definition~\ref{jz}, we have
\begin{equation} \label{gammatbd}
\int_{\Sp^1} \left| \frac{\partial \gamma_s}{\partial \theta} \right|^2 d\tau \leq \mathsf{E}_\rho(1+2\pi be^{-as}).
\end{equation} 

Define a $\rho$-equivariant map $\tilde v:  \DDD \rightarrow \tilde X$ as follows:  First, for $(r,t) \in  \DDD$,  let
\[
\tilde v(r,t):=\tilde \gamma_{(-\log r-\log 2)^{\frac{1}{3}}}(t) \mbox{ for } r \in (0,\frac{1}{2}] 
\  \mbox{ and } \ 
 \tilde v(1,t):=\tilde k(1,t)
\] 
Next,  for each $t \in \R$, let 
\[
r \mapsto  \tilde v(r,t) \mbox{ for } r \in [\frac{1}{2},1]
\]
to be the arclength parameterization of the  geodesic between $\tilde v(\frac{1}{2}, t)$ and $\tilde k(\frac{1}{2},t)$.  Finally, let
$v:\DDDD$ be the section associated to the $\rho$-equivariant map $\tilde v$.

For $0<r<r_0<\frac{1}{2}$, we have
\begin{eqnarray*}
\mathsf{E}_\rho \log \frac{r_0}{r} & \leq  &  \int_0^{2\pi} \int_r^{r_0} \frac{1}{r^2} \left| \frac{\partial v}{\partial \theta} \right|^2 rdr \wedge d\theta \ \ \ \ \  \mbox{(by Lemma~\ref{lowerbd})} \nonumber \\
& \leq  &  \int_0^{2\pi} \int_r^{r_0} \frac{1}{r^2}  \left| \frac{\partial (\gamma_{(-\log r-\log 2)^{\frac{1}{3}}})}{\partial \theta} \right|^2 rdr \wedge d\theta
\\
& =  &  \int_0^{2\pi} \int_{-\log r_0}^{-\log r}  \left| \frac{\partial (\gamma_{(t-\log 2)^{\frac{1}{3}}})}{\partial \theta} \right|^2 dtd\theta \nonumber \\
& \leq  &  \int_{-\log r_0}^{-\log r}( \mathsf{E}_\rho+be^{-a(t-\log 2)^{\frac{1}{3}}})
 dt \ \ \ \ \ \ \ \ \mbox{(by (\ref{gammatbd}))}\nonumber \\
& = & 
\mathsf{E}_\rho \log \frac{r_0}{r}+C.  
\end{eqnarray*}
By (\ref{npcquad}), 
\begin{eqnarray*}
\left| \frac{\partial v}{\partial r} \right|^2(r,\theta)  =  \left| \frac{\partial}{\partial s} \Big|_{s=(-\log r-\log 2)^{\frac{1}{3}} } \gamma_s(\theta) \right|^2  \left| \frac{\partial ( (-\log r-\log 2)^{\frac{1}{3}} )}{\partial r} \right|^2
\leq  \frac{1}{r^2(-\log r-\log 2)^{\frac{4}{3}}},
\end{eqnarray*}
and thus
\begin{eqnarray*}
  \int_0^{2\pi} \int_r^{r_0} \left| \frac{\partial v}{\partial r} \right|^2 rdr \wedge d\theta  & = &  \int_0^{2\pi} \int_r^{r_0}  \frac{1}{r(-\log r-\log 2)^{\frac{4}{3}}} dr \wedge d\theta  \leq C.  
\end{eqnarray*}
\end{proof}

\begin{lemma}[Existence and property (i) of Theorem~\ref{exists}] \label{(i)}
If $\tilde X$, $\rho$,  $\mathsf{E}_\rho$ and $k$ as in  Theorem~\ref{exists}, 
then there  exists a  harmonic section  $u:\D^* \rightarrow \widetilde{\D^*} \times_\rho \tilde X$  with  $u|_{\partial \D} =k|_{\partial \D}$ and a constant $C>0$ that   depends only on $\mathsf{E}_\rho$, $a$, $b$ from Definition~\ref{jz} and the section $k$ such that 
\[
\mathsf{E}_\rho  \log \frac{1}{r} \leq E^u[\D_{r,1}] \leq \mathsf{E}_\rho \log \frac{1}{r}+C, \ \ \ 0<r\leq 1.
\]
\end{lemma}

\begin{proof}
Let  $v$ be as in Lemma~\ref{defineg}. By \cite[Theorem 2.7.2]{korevaar-schoen1}, there exists a  unique harmonic section 
\[
u_r:   \D_{r,1} \rightarrow  \widetilde{ \D_{r,1}}\times_{\rho} \tilde X
\mbox{ with } u_r|_{\partial \D_{r,1}}=v|_{\partial \D_{r,1}}.
 \] 
We have that
\begin{eqnarray}
\mathsf{E}_\rho \log \frac{r_0}{r} +E^{u_r}[\D_{r_0,1}] & \leq & E^{u_r}[\D_{r,r_0}] +E^{u_r}[\D_{r_0,1}]  \ \ \ \ \  \mbox{(by Lemma~\ref{lowerbd})} \nonumber 
\\
& = & E^{u_r}[\D_{r,1}]\nonumber  \\
 & \leq  &  E^{v}[\D_{r,1}] \ \ (\mbox{since $u_r$ is minimizing}) \nonumber  \\
  & = &  E^{v}[\D_{r,r_0}]+ E^{v}[\D_{r_0,1}]  \nonumber  \\
 & \leq &  
\mathsf{E}_\rho \log \frac{r_0}{r} +C +E^{v}[\D_{r_0,1}] 
 \ \ \ \ \  \mbox{(by (\ref{glog}))}.  \nonumber 
 \end{eqnarray}
Therefore,
 \begin{equation} \label{bdvr}
 E^{u_r}[\D_{r_0,1}] \leq C+ E^{v}[\D_{r_0,1}].
 \end{equation}
Note that the right hand side of the above is independent of $r \in (0,\frac{1}{2}]$.  
Thus, $\{u_r|_{\D_{2r_0,1}}\}$ is an equicontinuous family of sections (cf.~\cite[Theorem 2.4.6]{korevaar-schoen1}). Consequently, there exists a sequence  $r_i \rightarrow 0$ and a harmonic section 
\begin{equation} \label{hmpd}
u:  \DDDD \mbox{ with }  u|_{\partial \D} =k|_{\partial \D}.
\end{equation}
 such that the sequence $\{u_{r_i}\}$ converges uniformly on every on compact subset of $\D^*$ to  $u$.

To prove property (i),  we note that for $0<r<r_1<1$,  \begin{eqnarray*}
E^{u_r}[\D_{r,r_1}]  +E^{u}[\D_{r_1,1}] & = & E^{u_r}[\D_{r,1}] \\
& \leq &   E^{v}[\D_{r,1}]
\\
& = & E^{v}[\D_{r,r_1}]+E^{v}[\D_{r_1,1}] \end{eqnarray*}
which, combined with Lemma~\ref{lowerbd} and (\ref{glog}), imply that
\[
E^{u_r}[\D_{r_1,1}]  \leq E^v[\D_{r_1,1}]+C.
\]
Letting $r \rightarrow 0$ and applying the lower semicontinuity of energy \cite[Theorem 1.6.1]{korevaar-schoen1},  we obtain
\[
E^{u}[\D_{r_1,1}]  \leq E^v[\D_{r_1,1}]+C.
\]
Applying Lemma~\ref{lowerbd} and (\ref{glog}), we obtain
\[
\mathsf{E}_\rho \log \frac{1}{r_1}  \leq E^{u}[\D_{r_1,1}]  \leq E^v[\D_{r_1,1}]+C \leq \mathsf{E}_\rho \log \frac{1}{r_1}+C.
\]
\end{proof}

   \begin{lemma}  \label{integrable}
  If    $u:\DDDD$ is a harmonic section  with logarithmic energy growth,
then
 \[
\int_{\D^*}  \left| \frac{\partial u}{\partial r}\right|^2 rdr \wedge d\theta  \leq C \ \mbox{ and } \ \int_{\D^*}  \frac{1}{r^2} \left( \left| \frac{\partial u}{\partial \theta}\right|^2 -  \frac{\mathsf{E}_\rho}{2\pi} \right) rdr \wedge d\theta  \leq C
\]
where $C$ that depends $C_0$ from (\ref{Cbddef}). \end{lemma}

 \begin{proof}
 Follows immediately from the inequality (\ref{Cbddef}) in the Definition~\ref{logdec}.
\end{proof}

   \begin{lemma} [Property (ii) of Theorem~\ref{exists}] \label{(ii)}  If    $u:\DDDD$ is a harmonic section with logarithmic energy growth,
then
\begin{eqnarray} \label{mika1}
\label{mika2}
 \left(\left| \frac{\partial u}{\partial \theta}\right|^2 - \frac{\mathsf{E}_\rho}{2\pi}\right) & \leq &  \frac{C}{(-\log r)}
\ \text{ in } \,
\D_{\frac{1}{4}}^*
\\
\lim_{r \rightarrow 0}  (-\log r)\left(  \left| \frac{\partial u}{\partial \theta} \right|^2(r,\theta) -     \frac{\mathsf{E}_\rho}{2\pi} \right)
 &  = &  0   \label{limitisF}
\\
\left| \frac{\partial u}{\partial r}\right| & \leq &  \frac{C}{r^2 (-\log r)}  
\ \mbox{ in } \,
\D_{\frac{1}{4}}^*
\label{mika1}
\\
 \label{limitisFrad}
\lim_{r \rightarrow 0}  (-\log r) r^2  \left| \frac{\partial u}{\partial r} \right|^2 (r,\theta) &  = &  0.
\end{eqnarray}
where  $C$ that depends $C_0$ from (\ref{Cbddef}). 
\end{lemma}

\begin{proof}
Consider the cylinder
\[
{\mathcal C}
=(0,\infty) \times \Sp^1
\]
and let
\begin{equation} \label{Phi}
\Phi:{\mathcal C} \rightarrow \D^*, \ \ \ \ (t , \psi) = (r=e^{-t}, \theta=\psi). 
\end{equation}

Since $\Phi$ is a conformal map, $u \circ \Phi$ is harmonic.    Thus, the directional energy density functions $\left| \frac{\partial (u \circ \Phi)}{\partial t} \right|^2$ and $\left| \frac{\partial (u \circ \Phi)}{\partial \psi} \right|^2$  are subharmonic by \cite[Remark 2.4.3]{korevaar-schoen1}.
Furthermore,
$\left| \frac{\partial (u \circ \Phi)}{\partial t} \right|^2$ and $\left| \frac{\partial (u \circ \Phi)}{\partial \psi} \right|^2 - \frac{\mathsf{E}_\rho}{2\pi}$ are integrable in ${\mathcal C}=(0,\infty) \times \Sp^1$ by Lemma~\ref{integrable} and the chain rule.   

 The subharmonicity   of the directional energy density functions implies
\[
0 \leq \int_{(t_1,t_2) \times \Sp^1} \triangle \varphi  \left| \frac{\partial (u \circ \Phi)}{\partial \psi} \right|^2 dtd\psi, \ \ \forall \varphi \in C_c^{\infty}((t_1,t_2) \times \Sp^1), \ \varphi \geq 0.
\]
Letting $\varphi$ approximate the characteristic function of $(t_1,t_2) \times \Sp^1$, we obtain
\[
0 \leq  \frac{d}{dt}\Big|_{t=t_2} \left( \int_{\{t\} \times \Sp^1} \left| \frac{\partial (u \circ \Phi)}{\partial \psi} \right|^2 d\psi \right)-\frac{d}{dt}\Big|_{t=t_1} \left( \int_{\{t\} \times \Sp^1} \left| \frac{\partial (u \circ \Phi)}{\partial \psi} \right|^2 d\psi \right).
\]
In other words, 
\[
F(t) := \int_{\{t\} \times \Sp^1}  \left| \frac{\partial (u \circ \Phi)}{\partial \psi} \right|^2 -   \frac{\mathsf{E}_\rho}{2\pi} \ d\psi, \ \ \ t \in (0,\infty)
\]
 is a convex function.    By Lemma~\ref{integrable}, $\int_0^{\infty} F(t) dt< \infty$.  
Since $F(t)$ is convex and  integrable,
$F(t)$ is a  decreasing.  Thus,  we obtain  
\begin{equation} \label{bbbb}
tF(t) \leq 2 \int_{\frac{t}{2}}^{t} F(\tau)\ d\tau \leq 2 \int_{\frac{t}{2}}^{\infty} F(\tau)\ d\tau.
\end{equation}
With $\B_1(t_0,\psi_0)$ denoting the unit disk centered at $(t_0,\psi_0) \in \mathcal C$,  the mean value inequality implies
\begin{eqnarray*}
\lefteqn{
t_0 \left( \left| \frac{\partial (u \circ \Phi)}{\partial \psi}\right|^2
-\frac{\mathsf{E}_\rho}{2\pi} 
\right)(t_0,\psi_0)  
}
\\
& \leq &   
\frac{t_0}{\pi} \int_{\B_1(t_0,\psi_0)}\left( \left| \frac{\partial (u \circ \Phi)}{\partial \psi}\right|^2
-\frac{\mathsf{E}_\rho}{2\pi} 
\right) dtd\psi \\
& \leq & 
\frac{t_0}{\pi} \int_{(t_0-1,t_0+1) \times \Sp^1} \left( \left| \frac{\partial (u \circ \Phi)}{\partial \psi}\right|^2
-\frac{\mathsf{E}_\rho}{2\pi} 
\right) dtd\psi 
\\
&  \leq & \frac{1}{\pi}
\frac{t_0}{t_0-1} \int_{(t_0-1,t_0+1) \times \Sp^1} t   \left( \left| \frac{\partial (u \circ \Phi)}{\partial \psi}\right|^2
-\frac{\mathsf{E}_\rho}{2\pi} 
\right)  dt d\psi  
\\
&  = & 
 \frac{1}{\pi}
\frac{t_0}{t_0-1} \int_{t_0-1}^{t_0+1} t  F(t) dt 
\\
&  \leq & 
 \frac{1}{\pi}
\frac{t_0}{t_0-1}   \int_{t_0-1}^{t_0+1} 2 \int_{\frac{t}{2}}^{\infty} F(\tau)\ d\tau dt \ \ \ \mbox{(by  (\ref{bbbb}))}.
\end{eqnarray*}
Since $F(\tau) \geq 0$, 
\[
\int_{\frac{t}{2}}^{\infty} F(\tau)\ d\tau \leq  \int_{\frac{t_0-1}{2}}^{\infty} F(\tau)\ d\tau, \ \ \ t \in (t_0-1,t_0+1).
\]
Combining the above two inequalities and assuming $t_0 \geq 2$ (and thus $\frac{t_0}{t_0-1}\leq 2$),
\begin{eqnarray}
t_0 \left(  \left| \frac{\partial (u \circ \Phi)}{\partial \psi}\right|^2 (t_0,\psi_0)  -\frac{\mathsf{E}_\rho}{2\pi} \right)
& \leq & \frac{4}{\pi} \int_{\frac{t_0-1}{2}}^{\infty} F(\tau)\ d\tau.
\label{inspector}
\end{eqnarray}
Estimates (\ref{mika2})  and (\ref{limitisF}) both follow from (\ref{inspector}) by the chain rule and the fact that  $F(\tau)$ is integrable on $(0,\infty)$ with  integral bound given by Lemma~\ref{integrable}.

Similar argument for 
\[
G(t):= \int_{\{t\} \times \Sp^1}  \left| \frac{\partial (u \circ \Phi)}{\partial t} \right|^2 d\psi, \ \ \ t \in (0,\infty)
\]
 proves (\ref{mika1}) and (\ref{limitisFrad}).  
 \end{proof}

 \begin{lemma} \label{des}
Let  $u:\DDDD$ be a harmonic section  with logarithmic energy growth.   For any $\epsilon>0$, there exists $\rho_0>0$ such that 
 \[
d^2(u(r_1e^{i\theta}),u(r_0 e^{i\theta})) 
\leq
-\epsilon \log r_1, \ \ \ \forall 0<r_1<r_0\leq \rho_0, \, 0<\theta<2\pi.
\]
\end{lemma}

 \begin{proof}
 Fix $\epsilon>0$.  The  convergence of 
 $ (-\log r) r^2  \left| \frac{\partial u}{\partial r} \right|^2 (r,\theta)$ to 0 as $r \rightarrow 0$ (cf.~(\ref{limitisFrad})) implies that  there exists $\rho_0>0$ such that
 \[
\left| \frac{\partial u}{\partial r} \right| \leq \frac{\sqrt{\epsilon}}{2r(-\log r)^{\frac{1}{2}} }, \ \ r \in (0,\rho_0]. 
\]
Thus, we have for $0<r_1<r_0 \leq \rho_0$, and noting $\rho$-equivariance of $\tilde u$, we have
\begin{eqnarray*}
d^2(u(r_1e^{i\theta}),u(r_0e^{i\theta})) & \leq & 
 \left(  \int_{r_0}^{r_1}
\frac{d}{dr} d(u(r e^{i\theta}),u(r_0 e^{i\theta}))dr \right)^2
\nonumber \\
& \leq  & 
 \left( 
 \int_{r_1}^{r_0} 
\left| \frac{\partial u}{\partial r} \right|(re^{i\theta}) dr 
\right)^2
\nonumber \\
& \leq  & 
\frac{\epsilon}{4}  \left(  \int_{r_1}^{r_0} \frac{dr}{r(-\log r)^{\frac{1}{2}}}   \right)^2 
\nonumber \\
& \leq  & 
\epsilon \left(  (-\log r_1)^{\frac{1}{2}} - (-\log r_0)^{\frac{1}{2}} \right)^2
 \nonumber \\
 & \leq & 
 -\epsilon \log r_1.
\end{eqnarray*}
\end{proof}

\begin{lemma}[Property (iii) of Theorem~\ref{exists}] \label{(iii)}
If    $u:\DDDD$ is a harmonic section with logarithmic energy growth,
 then $u$ has sub-logarithmic growth.
  \end{lemma}
 
 \begin{proof} 
 By Lemma~\ref{des}, for $\epsilon>0$, there exists 
 $r_0>0$ sufficiently small such that
\[
d^2(u(re^{i\theta}), u(r_0e^{i\theta})) \leq  - \frac{\epsilon}{4}\log r, \ \ \ r \in (0,r_0). 
\]
Set $P_0=\tilde u(r_0e^{i\theta_0}) \in \tilde X$ and define $f_0(p)=[(\tilde p, P_0)]$ as in Definition~\ref{logdec}.  For $z =re^{i\theta}$, 
\begin{eqnarray*}
d^2(u(z),f_0(z))+\epsilon \log |z|& = &  
d^2(u(re^{i\theta}),u(r_0e^{i\theta_0}))   +\epsilon \log r \\
& \leq  &  
2d^2(u(re^{\i\theta}),u(r_0e^{i\theta}))+2d^2(u(r_0e^{i\theta}),u(r_0e^{i\theta_0}))   +\epsilon \log r
\\
& \leq &  \frac{\epsilon}{2}\log r +d^2(u(r_0e^{i\theta}),u(r_0e^{i\theta_0})). \end{eqnarray*}
Now let  $r \rightarrow 0$.
 \end{proof}

\begin{lemma}[Uniqueness for sub-logarithmic growth maps] \label{uniquenessdisk}
If   $u, v:\DDDD$ are  harmonic sections  with sub-logarithmic growth  with   $u=v$ on $\partial \D$, then $u=v$ on $\D^*$.
\end{lemma}
\begin{proof}
The function $d^2(u,v)$ (as defined by (\ref{defdis}))
 is a continuous subharmonic function (c.f.~\cite[Remark 2.4.3]{korevaar-schoen1}).  Furthermore,  $d^2(u,v)\equiv 0$ on $\partial \D$.  Since $u$, $v$ both have sub-logarithmic growth,   the triangle inequality implies  
 \[
 \lim_{|z| \rightarrow 0} d^2(u(z),v(z)) +\epsilon \log |z| =-\infty.
 \]
Thus, we can apply Lemma~\ref{shdisk}  to conclude $d^2(u(z),v(z))\equiv 0$ on $ \D^*$ which implies $u \equiv v$.
\end{proof}

\begin{corollary} \label{converse}
Any harmonic section  $v:\DDDD$ 
with sub-logarithmic growth satisfies properties (i), (ii) and (iii)  of Theorem~\ref{exists}.
\end{corollary}

\begin{proof}
By the uniqueness assertion of Lemma~\ref{uniquenessdisk}, $v$ must be the harmonic section $u$ constructed in Theorem~\ref{exists}.
\end{proof}

 \section{Proof of existence, Theorem~\ref{existence}} \label{sec:proofexistence}

In this section, we prove existence of an equivariant harmonic map from a punctured Riemann surface  $\RR$. 
We  use the following notation:
\begin{itemize}
\item $k: \RR \rightarrow  \tilde{\mathcal R} \times_\rho \tilde X$ is a locally Lipschitz section  (cf.~\cite[Proposition 2.6.1]{korevaar-schoen1})
 
\item   $k^j:=k|_{\bar \D^{j*}}$ 

\item  $\langle \lambda^j \rangle$ is the subgroup of $\pi_1(\RR)$ generated by a loop $\lambda^j$ around the puncture $p^j$
\item $\rho^j:  \langle \lambda^j \rangle \rightarrow \mathsf{Isom}(\tilde X)$ is the restriction of $\rho$ to $ \langle \lambda^j \rangle$
\item $I^j=\rho(\lambda^j)$
\item $\Ej:=\mathsf{E}_{\rho^j}$ as in Definition~\ref{infE} where we identify $[\lambda^j]$ with $[\Sp^1]$ 
\end{itemize}
Applying   Lemma~\ref{defineg} with $\rho=\rho^j$, $I=I^j$ and $k=k^j$ yields  a prototype  section 
\[
v^j:  \D^* \rightarrow \widetilde{ \D^*} \times_{\rho^j} \tilde X \subset  \widetilde{{\mathcal R}} \times_{\rho^j} \tilde X.
\]
The composition of $v^j$ and the quotient map 
$\widetilde{{\mathcal R}} \times_{\rho^j} \tilde X \rightarrow  \widetilde{{\mathcal R}} \times_{\rho} \tilde X$ defines 
a section
 of $\tilde {\mathcal R} \times_\rho \tilde X \rightarrow \mathcal R$ over  ${\D}^{*j}$ which we call again $v^j$.
We extend  these local sections $v^j:{\D}^{*j} \rightarrow \tilde {\mathcal R} \times_{\rho} \tilde X$ for $j=1,\dots, n$  to define a smooth section $v:{\mathcal R} \rightarrow   \tilde {\mathcal R}\times_{\rho} \tilde X$.
By the construction and Lemma~\ref{defineg},
\begin{equation} \label{ebv}
\sum_{j=1}^n\Ej \log \frac{1}{r} \leq 
E^v[\RR_r] \leq
 \sum_{j=1}^n \Ej \log \frac{1}{r}+C, \ \ \ 0<r
\leq  1.
\end{equation}

\begin{definition}  
The locally Lipschitz section 
\begin{equation} \label{mochizukimap}
v:  \mathcal R \rightarrow \tilde{\mathcal R} \times_{\rho} \tilde{X}
\end{equation}
constructed above  is called the  {\it prototype section} of the fiber bundle $\tilde{\mathcal R} \times_{\rho} \tilde{X} \rightarrow \mathcal R$. 
The associated $\rho$-equivariant map $\tilde{v}:\tilde{M} \rightarrow \tilde{X}$  is called the {\it prototype map}.   
\end{definition}

Define
\[
u_{\RR_r}:   \RR_r \rightarrow \tilde \RR \times_{\rho} \tilde X
\]
 to be the unique harmonic section with boundary values equal to that of $v|_{\partial \RR_r}$  (cf.~\cite[Theorem 2.7.2]{korevaar-schoen1}).  Fix $r_0 \in (0,\frac{1}{2}]$ and let $r \in (0,r_0)$.  
By Lemma~\ref{lowerbd},
\[
\Ej\log \frac{r_0}{r} 
  \leq    \int_{\D_{r,r_0}} \frac{1}{r^2} \left| \frac{\partial u_{\RR_r}}{\partial \theta} \right|^2 rdr \wedge d\theta\leq  E^{u_{\RR_r}}[\D^j_{r,r_0}].
\]
Thus, by (\ref{ebv}),
\[
E^v[\bigcup_{j=1}^n \D^j_{r,r_0}] \leq E^{u_{\RR_r}}[\bigcup_{j=1}^n\D^j_{r,r_0}]+C
\]
 which implies 
\begin{eqnarray}
E^{u_{\RR_r}}[\bigcup_{j=1}^n \D^j_{r,r_0}] +E^{u_{\RR_r}}[\RR_{r_0}]& = & E^{u_{\RR_r}}[\RR_r]\nonumber  \\
 & \leq  &  E^v[\RR_r] \nonumber  \\
 & = &  
E^v[\bigcup_{j=1}^n \D^j_{r,r_0}] +E^v[\RR_{r_0}]  \nonumber\\
& \leq &  
E^{u_{\RR_r}}[\bigcup_{j=1}^n \D^j_{r,r_0}] +C+E^v[\RR_{r_0}].   \nonumber  
 \end{eqnarray}
In other words,
 \begin{equation} \label{bdvr}
 E^{u_{\RR_r}}[\RR_{r_0}] \leq C+ E^v[\RR_{r_0}], \ \ \ \forall r \in (0,r_0).
 \end{equation}
The right hand side of the inequality (\ref{bdvr})  is independent of the parameter $r$; i.e.~once we fix $r_0$, the quantity $E^{u_{\RR_r}}[\RR_{r_0}]$ is uniformly bounded for all $r\in (0, r_0)$.  This implies a uniform Lipschitz bound, say $L$,  of $u_{\RR_r}$ for $r \in (0,r_0)$ in $\RR_{2r_0}$ (cf.~\cite[Theorem 2.4.6]{korevaar-schoen1}).  

Let $\{\mu_1, \dots, \mu_N\}$ be a set of generators of $\pi_1(\RR)$ and  $\tilde u_{\RR_r}$, $\tilde v$ be the $\rho$-equivariant maps associated to  sections $u_{\RR_r}$, $v$ respectively.
Thus,
  \[
d(\tilde u_{\RR_r}(\mu_i p), \tilde u_{\RR_r}(p)) \leq L d_{\tilde M}(\mu_i p,p),
\ \ 
\forall p \in \tilde \RR_{2r_0}, \ i=1, \dots, N \ \ r \in (0,r_0).
\]
If we let  
\[
c=L \sup \{d_{\tilde{M}}(\mu_i p,p): i=1,\dots, N, \ p \in \tilde \RR_{2r_0} \},
\]
then by  equivariance
\[
d(\rho(\mu_i)\tilde u_{\RR_r}(p),\tilde u_{\RR_r}(p))) \leq c, \ \ \ 
p \in \tilde \RR_{2r_0}, \ i=1, \dots, N, \ \ r \in (0,r_0).
\]
In other words, $\delta(\tilde u_{\RR_r}(p)) \leq c$ for all $p \in \tilde \RR_{2r_0}$ and $r \in (0,r_0)$.
By the properness of $\rho$, there exists $P_0 \in \tilde X$ and $R_0>0$ such that 
\[
\{\tilde u_{\RR_r}(p): \  p \in \tilde \RR_{2r_0}, \ r \in (0,r_0)\} \subset B_{R_0}(P_0).
\]
Thus,  by taking a compact exhaustion, applying Arzela-Ascoli and the usual diagonalization argument, there exists a subsequence of  $\tilde u_{\RR_r}$ converges locally uniformly  to a $\rho$-equivariant harmonic map $\tilde u: \tilde \RR \rightarrow
\tilde X$.   Let $u: \RR \rightarrow \tilde \RR \times_{\rho} \tilde{X}$ be the associated  harmonic section.
The lower semicontinuity of energy (cf.~\cite[Theorem 1.6.1]{korevaar-schoen1}) and (\ref{bdvr}) imply
\[
E^{u}[\RR_{r_0}] \leq  \sum_{j=1}^n\Ej  \log \frac{1}{r_0}+C, 
\ \ \forall r_0 \in (0,\frac{1}{2}]
\]
which, along with Lemma~\ref{lowerbd}, proves that $u$ has  logarithmic energy growth (cf.~Definition~\ref{logdec}).  The fact that $u$ has sub-logarithmic growth follows from Lemma~\ref{(iii)}.  This completes the proof of  Theorem~\ref{existence}.

\section{Proof of uniqueness, Theorem~\ref{uniqueness}} \label{sec:proofuniqueness}

In this section, we prove the uniqueness of  a harmonic map from a punctured Riemann surface  $\RR$ (cf.~Theorem~\ref{uniqueness}). 
First note that by Lemma~\ref{lowerbd},  for $ 0<r<r_0<1$ and any section $f:\DDDD$,
\begin{equation} \label{lwrbd}
\sum_{j=1}^n\Ej \log \frac{r_0}{r}   \leq  E^f[\bigcup_{j=1}^n \D^j_{r,r_0}].
\end{equation}
\begin{definition} \label{minuslogterm} 
Let 
\[
E^f(r)=E^f[\RR_r] - \sum_{j=1}^n \Ej \log \frac{1}{r}.
\]
Note that by (\ref{lwrbd}),   $r \mapsto E^f(r)$ is an increasing function.
Define the {\it modified energy} of $f$ as 
\[
E^f(0) = \lim_{r \rightarrow 0} E^f(r).
\]
Let $\LL_\rho$ be the set of all sections $\RR \rightarrow \tilde \RR \times_\rho \tilde X$ such that  $E^f(0)<\infty$.  Let $\mathcal H_{\rho} \subset \mathcal L_{\rho}$ denote the set of all   harmonic sections $u:  \RR \rightarrow \tilde \RR \times_\rho \tilde X$ of sub-logarithmic growth such that the associated $\rho$-equivariant map $u$ that is not identically constant or does not map into a geodesic. 
\end{definition}

  The goal is to prove that $\mathcal H_{\rho}$ contains only one element.   We first prove the following series of preliminary lemmas.

\begin{lemma} \label{u1}  If $u \in {\mathcal H}_{\rho}$ and $\D^j \subset \bar \RR$ is  the fixed conformal disk at the puncture $p^j$ (cf.~(\ref{fixeddisk})), then the restriction map $u|_{\D^{j*}}$ satisfies the properties (i), (ii) and (iii)  of Theorem~\ref{exists}. 
\end{lemma}

\begin{proof}
This follows immediately from Corollary~\ref{converse}.
\end{proof}

\begin{lemma} \label{sh01}  
If $u_0, u_1 \in {\mathcal H}_{\rho}$,  then
$
d^2(u_0,u_1)=c
$
for some constant $c$.  
\end{lemma}

\begin{proof}
We prove  Lemma~\ref{sh01} by showing that $d^2(u_0,u_1)$ extends as a subharmonic function to $\bar \RR$.  Since $\bar \RR$ is compact,  this implies  $d^2(u_0,u_1)$ is constant.  Indeed,   let   $\D \subset \RR$ be a holomorphic disk.  By  \cite[Lemma 2.4.2 and Remark 2.4.3]{korevaar-schoen1},    the function $d^2(u_0,u_1)$ is subharmonic in $\D$.  Since this statement is true for any holomorphic disk $\D \subset \RR$, we conclude that $d^2(u_0,u_1)$ is subharmonic in $\RR$.  

Next, consider the fixed conformal disk $\D^j \subset \bar{\RR}$  centered at   $p^j \in {\mathcal P}$ (cf.~(\ref{fixeddisk})).    
Since both $u_0$ and $u_1$ have sub-logarithmic growth, 
$d(u_0(z),u_1(z))+\epsilon \log  |z| \rightarrow -\infty$ as $z \rightarrow 0$ in $\D^j$ by the triangle inequality.  Thus,  by Lemma~\ref{shdisk},
$d^2(u_0,u_1)$   is bounded in $\D^{j*}$ 
and  extends as a subharmonic function on $\D^j$.  Hence, we conclude that $d^2(u_0,u_1)$ extends as a subharmonic function on $\bar \RR$, thereby proving    Lemma~\ref{sh01}.
\end{proof}

For $u_0, u_1 \in {\mathcal H}_{\rho}$, let $\tilde u_0$, $\tilde u_1$ be the associated $\rho$-equivariant maps.   For $s \in [0,1]$, define  $u_s:\RR \rightarrow  \tilde \RR \times_{\rho} \tilde X$  to be the associated section of the $\rho$-equivariant map
  \begin{equation} \label{us}
 \tilde u_s:  \tilde \RR \rightarrow \tilde X, \ \ \ \tilde u_s(z)=(1-s)\tilde u_0(z)+s \tilde u_1(z)
  \end{equation}
where the sum on the right hand side above denotes geodesic  interpolation (cf.~Notation~\ref{interpolationnotation}). 
Since $\tilde u_0$ and $\tilde u_1$ are $\rho$-equivariant, $\tilde u_s$ is also $\rho$-equivariant.  
 
 From Lemma~\ref{u1}, Lemma~\ref{sh01}, the convexity of the distance function and the convexity of  energy (cf.~\cite[(2.2vi)]{korevaar-schoen1}), it follows that $u_s|_{\D^{j*}}$ also satisfies the properties $(i)$, $(ii)$ and $(iii)$ of Theorem~\ref{exists} for all $s \in [0,1]$; i.e.
\begin{itemize}
\item
$\displaystyle{
E^{u_s}[\D^j_{r,r_0}] \leq  C+ \Ej \log \frac{r_0}{r}
}
$
\vspace{0.5mm}
\item 
$\displaystyle{
\lim_{r \rightarrow 0} E^{u_s|_{\partial \D^j_r}}[\Sp^1]  =  \Ej,
}
\vspace{0.5mm}
$
\item $u_s$ has sub-logarithmic growth.
 \end{itemize}
 
\begin{lemma}  \label{claim:dist}
Let $u_0, u_1 \in {\mathcal H}_{\rho}$ and $\epsilon>0$.  For any $\rho_0 \in (0,1)$, there exists $r_0 \in (0,\rho_0)$ such that
\[
\frac{1}{-\log r_0}  \sum_{j=1}^n \int_0^{2\pi} d^2(u_0|_{\partial \D^j_{r_0}}, u_{s}|_{\partial \D^j_{r_0^2}})d\theta<\epsilon,
\ \  \ \forall s \in [0,1].
\]
\end{lemma}

\begin{proof}
It suffices to prove Lemma~\ref{claim:dist} for $s=1$. Let $\rho_0>0$ be given. First, Lemma~\ref{des} asserts that there exists $\rho_1 \in (0,\rho_0)$ such that 
\[
\frac{1}{-\log r_0}  \sum_{j=1}^n \int_0^{2\pi} d^2(u_1|_{\partial \D^j_{r_0}}, u_1|_{\partial \D^j_{r_0^2}})d\theta
< \frac{\epsilon}{4}, \ \  \ \ \forall r_0 \in (0,\rho_1). 
\]
For $c>0$ as in Lemma~\ref{sh01}, choose $r_0 \in (0,\rho_1)$ such that 
$
 \frac{2\pi  n c}{-\log r_0}<\frac{\epsilon}{4}.
 $
Then 
\[
\frac{1}{-\log r_0} \sum_{j=1}^n \int_0^{2\pi} d^2(u_0|_{\partial \D^j_{r_0}}, u_1|_{\partial \D^j_{r_0}})d\theta
=\frac{2\pi n c}{-\log r_0} < \frac{\epsilon}{4}.
\]
The inequality of Lemma~\ref{claim:dist} for $s=1$ follows from the above two inequalities and the triangle inequality.
\end{proof}

\begin{lemma} \label{claim:lim}
Let $u_0, u_1 \in {\mathcal H}_{\rho}$ and $s \in [0,1]$.  For $\epsilon>0$, there exists  $\rho_0>0$ sufficiently small such that 
\[
-\log r_0  \sum_{j=1}^n \left(E^{u_0|_{\partial \D^j_{r_0}}}[\Sp^1] +E^{u_s |_{\partial \D^j_{r_0^2}}}[\Sp^1] \right) <-2 \log r_0 \sum_{j=1}^n \Ej +\epsilon, \ \ \ 0<r_0<\rho_0.
\]
\end{lemma}

\begin{proof}
For $s=1$,  Lemma~\ref{claim:lim}  follows from Lemma~\ref{(ii)}.  The general case of $s \in [0,1]$ follows immediately by convexity of energy (cf.~\cite[(2.2vi)]{korevaar-schoen1}). 
\end{proof} 

\begin{lemma} \label{sae}   For $u_0, u_1 \in {\mathcal H}_{\rho}$ and $s \in [0,1]$,  we have
$
E^{u_0}(0) = E^{u_s}(0).
$
\end{lemma}

\begin{proof}
It suffices to prove Lemma~\ref{sae} for $s=1$.   Assume on the contrary  that $E^{u_0}(0) \neq  E^{u_1}(0)$.  By changing the role of $u_0$ and $u_1$ if necessary, we can assume $E^{u_0}(0) < E^{u_1}(0)$.  Let  $\rho_0>0$ smaller of  the $\rho_0$ in Lemma~\ref{claim:dist} and Lemma~\ref{claim:lim}.  Choose $\epsilon>0$  and $r_0 \in (0,\rho_0]$ such that
\[
E^{u_0}(r_0) <  E^{u_1}(r_0) -2\epsilon
\]
which implies (cf.~Definition~\ref{minuslogterm})\begin{equation} \label{ococ}
E^{u_0}[\RR_{r_0}] <  E^{u_1}[\RR_{r_0}] -2\epsilon.
\end{equation}
Next fix $r_1 \in (0,r_0)$.  Let  $\tilde u_0$, $\tilde u_1:  \tilde \RR \rightarrow \tilde X$ be the $\rho$-equivariant maps associated to sections $u_0$, $u_1$. 
 For the fixed conformal disk $\D^j \subset \bar \RR$ centered at the puncture $p^j$ (cf.~(\ref{fixeddisk})), let  $ \tilde \D^j_{r_1,r_0} \subset \tilde \RR$ the lift of   $\D^j_{r_1,r_0} \subset \RR$.
We define a ``bridge" between  map $\tilde u_0$ and $\tilde u_1$  by setting 
\[
\tilde b: \bigcup_{j=1}^n \tilde \D^j_{r_1,r_0}  \rightarrow \tilde X
\] 
 to be the geodesic  interpolation $(1-t) U_0(\theta) +tU_1(\theta)$ where
  \[
 t=\frac{\log |z|-\log r_0}{\log r_1 - \log r_0}, \ \   U_0(\theta) = \tilde u_0(r_0,\theta) \ \text{ and } \  U_1(\theta) = \tilde u_1(r_1,\theta).
 \]
In other words,
\[
\tilde b(r,\theta)=\frac{\log r_1 - \log r}{\log r_1 - \log r_0} \tilde u_0(r_0, \theta)+\frac{\log r - \log r_0}{\log r_1 - \log r_0} \tilde u_1(r_1, \theta)
\]
 for $(r,\theta)  \in \tilde \D^j_{r_1,r_0}$,  $j=1,\dots, n$.  
Let
 \[
b:\bigcup_{j=1}^n \D^j_{r_1,r_0} \rightarrow 
\tilde \RR \times_{\rho}  \tilde X
\]
be the local section associated with $\tilde b$.

The NPC condition implies (by an argument analogous to the proof the  bridge lemma  \cite[Lemma 3.12]{korevaar-schoen2})
\begin{eqnarray*}
E^b[\bigcup_{j=1}^n \D^j_{r_1,r_0}] 
& \leq & 
\frac{1}{2} \log \frac{r_0}{r_1} \sum_{j=1}^n \left(E^{u_0|_{\partial \D^j_{r_0}}}[\Sp^1] +E^{u_1|_{\partial \D^j_{r_1}}}[\Sp^1] \right) \\
& & \  + \frac{1}{\log \frac{r_0}{r_1}} \sum_{j=1}^n \int_0^{2\pi} d^2(u_0|_{\partial \D^j_{r_0}}, u_1|_{\partial \D^j_{r_1}})d\theta.
\end{eqnarray*} 

Choose $r_1=r_0^2$ (cf.~ Lemma~\ref{claim:dist} and Lemma~\ref{claim:lim}) to obtain 
 \begin{equation} \label{acac}
E^b[\bigcup_{j=1}^n \D^j_{r_0^2,r_0}] < -\log r_0
\sum_{j=1}^n \Ej  +2\epsilon.
\end{equation}
Since the section 
\[
h:\RR_{r_0^2} \rightarrow \tilde \RR \times_\rho \tilde X
\]
 defined by setting 
 \[
 h=
 \left\{ 
 \begin{array}{ll}
 u_0 & \mbox{ in }\RR_{r_0}\\
b & \mbox{ in } \displaystyle{\bigcup_{j=1}^n \D^j_{r_0^2,r_0}}
\end{array}
\right.
\] is a competitor for $u_1$, we have
\begin{eqnarray*}
E^{u_1}[\RR_{r_0^2}]  & \leq &  E^h[\RR_{r_0^2}]
\\
& = & E^{u_0}[\RR_{r_0}]+ E^b[\bigcup_{j=1}^n \D^j_{r_0^2,r_0}]\\
&  <&  E^{u_1}[\RR_{r_0}]-2\epsilon + E^b[\bigcup_{j=1}^n \D^j_{r_0^2,r_0}] \ \ \ \ \ \ \mbox{(by (\ref{ococ}))}\\
& < & E^{u_1}[\RR_{r_0}] -\log r_0 \sum_{j=1}^n\frac{\Delta_{I^j}^2}{2\pi}  \ \ \ \ \ \ \mbox{(by (\ref{acac}))}.
\end{eqnarray*}
Thus,
\[
E^{u_1}[\bigcup_{j=1}^n \D^j_{r_0^2,r_0}] < -\log r_0 \sum_{j=1}^n \Ej.
\]
This contradicts   (\ref{lwrbd}) and proves Lemma~\ref{sae}.
\end{proof}

\begin{lemma} \label{twothings}
For $u_0, u_1 \in \HH_\rho$, there exists a constant $c$ such that
\begin{eqnarray} 
d(u_s(p), u_1(p)) & \equiv  & cs, \forall p \in  \RR
\label{d=1}
 \\
 |(u_s)_*(V)|^2(p) & = & |(u_0)_*(V)|^2(p), \ \ \mbox{for a.e.~}s \in [0,1], \ p \in \RR, V \in T_p\tilde M. \label{pullback}
\end{eqnarray}
\end{lemma}

\begin{proof}
By the  convexity of energy (cf.~\cite[(2.2vi)]{korevaar-schoen1}), 
\[
E^{u_s}[\RR_r] \leq (1-s)E^{u_0}[\RR_r] + s E^{u_1}[\RR_r] -s(1-s) \int_{\RR_r} |\nabla d(u_0, u_1)|^2 d\mbox{vol}_{\domain}
\]
for any $r>0$.
\[
E^{u_s}(r) \leq (1-s)E^{u_0}(r) + s E^{u_1}(r) -s(1-s) \int_{\RR_r} |\nabla d(u_0, u_1)|^2 d\mbox{vol}_{\domain}.
\]
Letting $r \rightarrow 0$ and applying Lemma~\ref{sae},  we conclude 
\begin{eqnarray}
0 &= & \int_{\RR} |\nabla d(u_0, u_1)|^2 d\mbox{vol}_{\domain}
 \label{graddist}
 \\
E^{u_s}& = & E^{u_0}, \ \ \forall s \in [0,1] \label{energyconstant}
 \end{eqnarray}

First, (\ref{graddist}) implies that $\nabla d(u_0, u_1)=0$~a.e.~in $\domain$ which in turn implies $d(u_0,u_1) \equiv c$ for some $c$.  By the definition of the map $u_s$, (\ref{d=1}) follows immediately.

Next, for $\{P,Q,R,S\} \subset \tilde X$, the quadrilateral comparison for NPC spaces   implies
\[
d^2(P_t,Q_t) \leq (1-t) d^2(P,Q)+td^2(R,S)
\]
where $P_t=(1-t)P+tS$ and $Q_t=(1-t)Q+tR$.
Applying the above inequality with  $P=u_0(p)$, $Q=u_1(p)$, $S=u_1(\exp_p(tV))$ and $Q=u_0(\exp_t(tV))$ where $t >0$ and $V \in T_p\tilde M$, dividing by $y$ and letting $t \rightarrow 0$, we obtain (cf.~\cite[Theorem 1.9.6]{korevaar-schoen1})
\[
|(u_s)_*(V)|^2(p) \leq (1-s) |(u_0)_*(V)|^2(p) + s|(u_1)_*(V)|^2(p),
\ 
\mbox{a.e.~$p \in \tilde M, V \in T_p\tilde M$.}
\]   Integrating the above over all unit vectors $V \in T_p\tilde M$  and then over $p \in F$, we obtain
\[
E^{u_s} \leq (1-s) E^{u_0}+ s E^{u_1}.
\]
Combining this with  (\ref{energyconstant}) implies (\ref{pullback}).
\end{proof}

 We assume there exists $u_0, u_1 \in \HH_\rho$ such that $u_0 \not \equiv u_1$ and treat the two  cases (i) and (ii) of Theorem~\ref{uniqueness} separately. 
 
 \begin{itemize}
 \item {\it {\bf $\tilde X$ is a negatively curved space as in Definition~\ref{neg}}}:  
For  $p \in \tilde \RR$,  choose $\kappa>0$ and $R>0$ such that $B_{R}(u_0(p))$ is a CAT(-$\kappa$) space.  Let   $\mathcal U$ be a sufficiently small neighborhood of  $p$  and $s_0 \in [0,1]$ sufficiently small such that $\tilde u_s(\mathcal U) \subset B_{R}(u(p))$ for $s \in [0,s_0]$.  Thus, if $\tilde u_{s_0}(p)\neq \tilde u_0(p)$, then applying  \cite[Section 5]{mese} implies that the image under $u_0$ of a  sufficiently small  neighborhood of $p$ is contained in an image $\sigma(\R)$ of a geodesic line.   (If $\tilde X$ is a smooth manifold, this follows by Hartman \cite{hartman}).  If $ \tilde u_0 \not \equiv \tilde u_1$, then $\tilde u_s(p) \neq \tilde u_0(p)$ for all $p \in \tilde \RR$ and $s \in (0,s_0]$ by (\ref{d=1}).  Thus, we conclude that $\tilde u_0(\tilde \RR)$ is contained in $\sigma(\R)$.  Consequently, $\rho(\pi_1(\tilde \RR))$ fixes $\sigma(\R)$ which contradicts  the fact that $\rho(\pi_1(\RR))$ is satisfies assumption (B). 

\item {\bf $\tilde X$ is an irreducible symmetric space of non-compact type or locally finite Euclidean buildings}:
For these two target spaces, Lemma~\ref{twothings} combined with \cite{daskal-meseUnique}  implies  the uniqueness result.
\end{itemize}

\chapter{Harmonic maps from K\"ahler surfaces} \label{chap:Kahler surfaces}

In this chapter, we assume that  \emph{the complex dimension of $M$ is 2.}
In dimension 2, the intersection of two irreducible components of the normal crossings divisor consists of  isolated points.  Thus, the analysis around the juncture of a normal crossing divisor  is much easier to understand than in higher dimensions.

Throughout this section, we assume:
\begin{itemize}
\item $M$ is a smooth quasi-projective variety such that $M=\bar M \backslash \Sigma$ where $\bar M$ is a  smooth projective variety of complex dimension $2$,  $\Sigma$ is a divisor with normal crossings and $M=\bar M \backslash \Sigma$ 
\item $\tilde X$ is an NPC space
 \item  $\rho: \pi_1(M) \rightarrow  \mathsf{Isom}(\tilde X)$  is a proper homomorphism that  does not fix an unbounded closed convex strict subset of $\tilde X$.
\item every element (resp.~commuting pair) of $\mathsf{Isom}(\tilde X)$ has exponential decay (cf.~Definition~\ref{jz} and Definition~\ref{jz2})
 \end{itemize}
 In other words, $M$ and $\rho$ are as in 
Theorem~\ref{theorem:pluriharmonic} with the additional assumption that 
\[
\dim_\C M=2,
\]
 but the target space $\tilde X$ is more general  than in Theorem~\ref{theorem:pluriharmonic}.  By relaxing the restriction on the target space, we  include for example the case when $\tilde X$ is a Euclidean building   (cf.~Theorem~\ref{theorem:buildings}) and the Weil-Petersson completion of Teichm\"uller space (cf.~Theorem~\ref{theorem:teichmuller}).

The main theorems  of Chapter~\ref{chap:Kahler surfaces} are:

\begin{theorem}  \label{theorem:pluriharmonicDim2}
With the Poincar\'e-type K\"ahler metric $g$  on $M$ (cf.~Section~\ref{poincarekahler}), there exists    a   harmonic section 
$
u: M \rightarrow \tilde M \times_\rho \tilde{X}
$
 of logarithmic energy growth.
\end{theorem}

\begin{theorem}
\label{pluriharmonicity}
If in addition $\tilde X$ is as in Theorem~\ref{theorem:pluriharmonic}, the harmonic section $
u: M \rightarrow \tilde M \times_\rho \tilde{X}
$ of Theorem~\ref{theorem:pluriharmonicDim2}
is pluriharmonic.
\end{theorem}

The main idea of the proof  of Theorem~\ref{theorem:pluriharmonicDim2} is to   construct a \emph{prototype section} which almost minimizes energy near infinity.  This idea goes back to \cite{lohkamp} and  \cite{jost-zuo}.  On the other hand,   since we want include the case when  the normal bundle of $\Sigma$ is non-trivial and also the case when the divisor consists of more than one irreducible components,  our argument closely follows the construction in Mochizuki \cite{mochizuki-memoirs}.  We use the prototype section to construct a harmonic section.  The idea is to consider the Dirichlet problem on a compactly supported domain with boundary values equal to $v$ and then to take a compact exhaustion.

Throughout this Chapter, we let $\bar M$ be a   smooth projective variety of complex dimension $\geq 2$ and  $\Sigma$ be a divisor with normal crossings such that
\[
M=\bar M \backslash \Sigma.
\]  Furthermore, let 
\[
\Sigma= \bigcup_{j=1}^L \Sigma_j
\]
where $\{\Sigma_j\}$ is the set of irreducible components of $\Sigma$. \\

\noindent Chapter~\ref{chap:Kahler surfaces} consists of the following: 
\vspace*{0.05in}

In Section~\ref{sec:neardivisor}, we study neighborhoods of the divisor of $\bar M$.  To do so, we consider the following:
\begin{itemize}
\item A neighborhood $\Omega \times \D_{\frac{1}{4}}$  of a point in $\Sigma_j$ away from the crossings
\item A neighborhood
  $\D_{\frac{1}{4}}  \times  \D_{\frac{1}{4}}$ near  the crossing $\Sigma_j \cap \Sigma_i$.
  \end{itemize}
  
  In Section~\ref{sec:metric},  we discuss a  Poincar\'e-type metric $g$ (cf.~Definition~\ref{poincarekahler}) due to Cornalba and Griffiths \cite{cornalba-griffiths}.  This is a complete metric which puts the divisor $\Sigma$ at infinity.

    In Section~\ref{prototype}, we construct  a $\rho$-equivariant {\it prototype section} $v:M \rightarrow \tilde M \times_\rho \tilde X$   with controlled growth near infinity. The crucial tool in this contruction is Theorem~\ref{exists} (i.e.~the Dirichlet solution on the punctured disk).  This enables us to construct a fiber-wise harmonic map on the normal bundle of the divisor.  This map defined near the divisor is then extended to all of $M$.  
    
    In Section~\ref{sec:ptm},  we give  precise estimates for energy growth of the prototype section near the divisor at infinity.  These are important because they imply the estimates for the harmonic section.
    
In Section~\ref{sec:existence}, we use the prototype section $v:M \rightarrow \tilde M \times_\rho \tilde X$  in order to construct a harmonic section $u:M \rightarrow \tilde M \times_\rho \tilde X$ (cf.~Theorem~\ref{theorem:pluriharmonicDim2}). We end with some energy estimates of the harmonic section constructed. 

In Section~\ref{sec:pluriharmonicity}, we show that the harmonic section is in fact pluriharmonic (cf.~Theorem~\ref{pluriharmonicity}). This is where we take advantage of the variation of the Bochner formulas proved in Chapter~\ref{chap:bochner}.

 \section{Neighborhoods near the divisor} \label{sec:neardivisor}
We continue to denote by $\D$ the unit disk in the complex plane. For clarity, we will also denote the unit disk by $\D_{z}$ to indicate that $\D$ is being parameterized by the complex variable $z=re^{i\theta}$.    We use analogous notation $\D_{z,r}$, $\D_{z, r_1,r_2}$, etc.~that corresponding to the sets defined in Section~\ref{disc}. We also use the notation $\Sp^1_{\theta}$ to denote the circle  $\Sp^1$  parameterized by the real variable $\theta$ and identify $\Sp^1_{\theta}$ as the boundary of ${\D}_{z}$ via the map $\theta \mapsto e^{i\theta}$.

Recall that $\bar M$ is a compact K\"ahler surface with a divisor $\Sigma$ containing only  simple normal crossings,
$
M=\bar M \backslash \Sigma,
$
 and $\{\Sigma_j : j=1,\cdots, L\}$ are the irreducible components of  $\Sigma$.  Let $\sigma_j$ be the canonical section of ${\mathcal O}(\Sigma_j)$ with zero set $\Sigma_j$. 

To study a neighborhood of the juncture, let $P \in \Sigma_i \cap \Sigma_j$ for some $i, j \in \{1, \dots, L\}$ with $i \neq j$, and let $V_P$ be a neighborhood of $P$ containing no other crossings.
Choose holomorphic  trivializations $e_i$ (resp. $e_j$)  of $\mathcal O(\Sigma_i)$ (resp. $\mathcal O(\Sigma_j)$)  on 
$V_P$ and define $z^1$ (resp. $z^2$) by setting
\begin{equation} \label{z12}
\sigma_i=z^1e_i, \mbox{ (resp. }\sigma_j=z^2e_j).
\end{equation}
For each $j=1, \dots, L$, let $h_j$ be a Hermitian metric  on $\mathcal\mathcal {\mathcal O}(\Sigma_j)$ such that $|e_j|_{h_j}=1$ in $V_P$ for any crossing $P$.
Let $h$ be a Hermitian metric on $\bar{M}$,
not necessarily K\"ahler, such that the following holds:
\begin{itemize}
\item[(i)] The metric $h$ is the Euclidean metric in a neighborhood $V_P$ of every crossing $P$, i.e.~
\begin{equation} \label{hiseuc}
h|_{V_P} =dz^1d\bar z^1+ dz^2d\bar z^2. 
\end{equation}
By rescaling $\sigma_1$ and $\sigma_2$ if necessary, we can assume without the loss of generality that 
  \begin{equation} \label{D12}
  \bar{\D}_{z^1} \times \bar{\D}_{z^2} \subset V_P.
\end{equation}  
\item[(ii)] The metric $h$ induces the orthogonal decomposition $T \bar M|_{\Sigma_j} = T \Sigma_j \oplus N \Sigma_j$ and  under the natural isomorphism 
\begin{equation} \label{natiso}
N \Sigma_j\simeq \mathcal\mathcal O(\Sigma_j)|_{\Sigma_j},
\end{equation}
 the restriction of $h$ to $N \Sigma_j$ is same as $ h_j$.
\end{itemize}

For   $r \in (0,1]$, we  set
  \begin{eqnarray*} \label{nbhds}
{\mathcal D}_{j,r}  & = &   \{\nu \in  N\Sigma_j: |\nu|_{h_j} < r\},   \nonumber \\
{\mathcal D}_{j,r}^*  & = &   \{\nu \in   N\Sigma_j: 0<|\nu|_{h_j} < r\},    \nonumber \\
\bar{\mathcal D}_{j,r}  & = &   \{\nu  \in N\Sigma_j: |\nu|_{h_j} \leq r\} \nonumber\\
\bar{\mathcal D}_{j,r}^*  & = &    \{\nu \in \subset N\Sigma_j: 0<|\nu|_{h_j} \leq r\},  
\end{eqnarray*}

\begin{eqnarray*}
{\mathcal D}_r = \bigcup_{j=1}^L {\mathcal D}_{j,r}, \ \  {\mathcal D}_r^* = \bigcup_{j=1}^L {\mathcal D}_{j,r}^*, \ \ 
\bar{\mathcal D}_r = \bigcup_{j=1}^L\bar{\mathcal D}_{j,r}, \ \   \bar{\mathcal D}_r^* = \bigcup_{j=1}^L\mathcal D_{j,r}^*
\end{eqnarray*}

\begin{eqnarray} \label{nbhd12'} 
{\mathcal D}_{r_1,r_2} = {\mathcal D}_{r_2} \backslash \bar{\mathcal D}_{r_1}
\mbox{ for
$
0< r_1< r_2 \leq 1.
$}
\end{eqnarray} 
There exists  $r>0$ such that the restriction of the exponential map  
\[
\exp : N \Sigma_j \subset T \bar M|_{\Sigma_j} \rightarrow \bar M
\] 
defines diffeomorphism of 
${\mathcal D}_{j,r}$ to a neighborhood of $\Sigma_j$ in $\bar M$. By rescaling $\sigma_j$ if necessary, we may assume $r>1$.  
In particular, we identify (\ref{nbhd12'}) for sufficiently small $r>0$ as an open subset of $M$ via the exponential map; i.e.
\begin{equation} \label{nbhd12}
{\mathcal D}_{r_1,r_2} := \bigcup_j \{\exp \nu: \nu \in N\Sigma_j, \ r_1<|\nu|_{h_j} <r_2\} \subset M.
\end{equation}
Denote
\[
\bar{\mathcal D}_j:=\bar{\mathcal D}_{j,1} \subset N\Sigma_j.
\]
The restriction of the normal bundle $N\Sigma_j \rightarrow \Sigma_j$ to $\bar{\mathcal D}_j$ defines a disk bundle
\begin{equation} \label{diskbundle}
\pi_j : \bar{\mathcal D}_j \rightarrow \Sigma_j.
\end{equation}
We also identity $\bar{\mathcal D}_j$  as a subset of $\bar{M}$; i.e.
\begin{equation}  \label{assubsets}
\bar{\mathcal D}_j \simeq \exp(\bar{\mathcal D}_j) \subset \bar{M}.
\end{equation}
We denote by $J_{\bar M}$  the  holomorphic structure on $\bar{\mathcal D}_j$ defined by pulling back the complex  structure on $\bar M$ via the exponential map.

 We now consider a finite collection of sets near the divisor 
of the following two types:
\begin{itemize}
 
\item A  set   of type (A) admits a local unitary trivialization
\begin{equation} \label{ref:typeA}
\pi_j^{-1}(\Omega) \simeq  \Omega \times \bar{\D}_{z^2}, \ \ 
\end{equation}
of $\pi_j:  \bar {\mathcal D}_j \rightarrow \Sigma_j$ where $\Omega \subset \Sigma_j$ is a contractible open subset of  $\Sigma_j$ containing no crossings.  We will use coordinates $(z^1,z^2)$ where  $z^1$  is a holomorphic local coordinate in  $\Omega$ and $z^2$ is the standard coordinate of $\bar \D$. 
Although the coordinates $(z^1,z^2)$  are holomorphic with respect to the  product complex structure $J_{prod}$ of $\Omega \times \bar \D_{z^2} $, they  are {\it not holomorphic} with respect to the complex structure $J_{\bar M}$.  However, by construction, we have that  
\[
J_{\bar M}=J_{prod} \ \ \mbox{on} \ \  T_{(z^1,0)} (\Omega \times \bar \D_{z^2}).
\]

\item  
A set of type (B) is as in (\ref{D12}); i.e.
\begin{equation} \label{ref:typeB}
 \bar{\D}_{z^1} \times \bar{\D}_{z^2} \subset V_P
 \end{equation}
where $V_P$ be an  open set containing a single crossing $P \in \Sigma_i \cap \Sigma_j$ ($i\neq j$).
 By the property (i) of the hermitian metric $h$,  
$(z^1,z^2)$ are holomorphic coordinates with respect to $J_{\bar M}$. 
  Furthermore, with the identification
$\bar\D_{z^1} \simeq \bar\D_{z^1} \times \{0\} \subset \Sigma_1$ (resp. $\bar\D_{z^2} \simeq  \{0\} \times \bar\D_{z^2}\subset \Sigma_2$),  
 \begin{equation} \label{unitarytrivialization}
 \pi_j^{-1}(\bar\D_{z^1}) \simeq \bar\D_{z^1} \times \bar\D_{z^2} \ \ \ (\mbox{resp. }\pi_i^{-1}(\bar\D_{z^2}) \simeq \bar\D_{z^1} \times \bar\D_{z^2})
 \end{equation}
 is a local unitary trivialization of $\pi_j:  \bar{\mathcal D}_j \rightarrow \Sigma_j$ (resp. $\pi_i: \bar {\mathcal D}_i \rightarrow \Sigma_i$). 
\end{itemize}
\begin{definition} \label{stdprod}
We will refer to the coordinates $(z^1,z^2)$ discussed in (A) above as the \emph{standard product coordinates} on a set  $\Omega \times \bar\D_{z^2} $ of type (A).  
\end{definition}
\begin{definition} \label{product=holomorphic}
We will refer to the coordinates $(z^1,z^2)$ discussed in (B) above as the {\it standard product coordinates} and the {\it holomorphic coordinates} on a set $\bar{\D}_{z^1}^* \times \bar{\D}_{z^2}^*$ of type (B). 
\end{definition}
\begin{remark}
Holomorphic coordinates in a set of type (A) is defined later (cf.~Definition~\ref{holcoord}).
\end{remark}

\section{The Poincar\'e-type metric and its estimates} \label{sec:metric}

\subsection{Poincar\'e-type metric}

Recall the Poincar\'e metric
\begin{equation} \label{poincaremetric}
g_{poin} = \mbox{Re} \left( \frac{dz \oplus d
\bar z}{|z|^2(\log|z|^2)^2}\right) \mbox{ on } \D^*.
\end{equation}
Using the canonical section $\sigma_j \in {\mathcal O}(\Sigma_j)$ and the Hermitian metric $h_j$ on ${\mathcal O}(\Sigma_j)$ given in Section~\ref{sec:neardivisor}, we define the Poincare-type metric on $M$.
\begin{definition} \label{poincarekahler} 
Let $\bar{\omega}$ be the K\"ahler form on $\bar{M}$.  Scale   the metric $h_j$ such that  $|\sigma_j|_{h_j}<1$ and define
\begin{equation} \label{donaldsonmetric}
\omega=  \bar{\omega}-\frac{\sqrt{-1}}{2}  \sum_{l=1}^L  \partial \bar{\partial} \log (\log |\sigma_j|_{h_j}^{-2}).
\end{equation}
By scaling $\bar \omega$ if necessary, we can assume that $\omega$ defines a positive form. We denote by $g$ the  K\"ahler metric on $M$ induced by the K\"ahler form $\omega$. 
\end{definition}

\begin{definition} \label{(j)} 
Fix $j$ and define on $\bar M \backslash \bigcup_{i \neq j} \Sigma_i$ the K\"ahler form
\begin{equation} \label{rest}
\omega + \frac{\sqrt{-1}}{2} \partial \bar{\partial }\log \log |\sigma_j|_{h_j}^{-2} = \bar{\omega}- \frac{\sqrt{-1}}{2}  \sum_{i \neq j} \partial \bar{\partial} \log \log |\sigma_i|_{h_i}^{-2} .
\end{equation}
Define $g_{\Sigma_j}$ to be the restriction  to $\Sigma_j \backslash\bigcup_{i \neq j} \Sigma_i$ of the K\"ahler metric associated to this K\"ahler form.  Thus,  away from $\bigcup_{i \neq j} \Sigma_i$, the metric $g_{\Sigma_j}$ is the smooth part of the metric $g|_{\Sigma_j}$.
\end{definition}

Below we derive some estimates for the metric $g$ in a set of type (A) and of type (B) (See~(\ref{ref:typeA}) and (\ref{ref:typeB}) for definitions of a set of type (A) and (B)).  These are an expanded form of the estimates derived by Mochizuki \cite{mochizuki-memoirs}.  

\subsection{Metric estimates in set of Type (A)}

Let $\pi_j^{-1}(\Omega) \simeq \Omega \times \bar\D$ be a set of type (A) with the standard product coordinates $(z^1, z^2)$.   (Recall that  $\Omega \subset \Sigma_j$ does not intersect $\Sigma_i$ for $i \neq j$.) 
We will write
\[
z^1=x+iy \ \mbox{ and } \ z^2=re^{i\theta}.
\]
Fix  a local trivialization $e$ of ${\mathcal O}(\Sigma_j)$, holomorphic with respect to the complex structure $J_{\bar{M}}$.    
Define
\begin{equation} \label{bU}
b:\Omega \times \bar \D  \rightarrow [0,\infty), \  \ b= |e|^{-2}_{h_j}.
\end{equation}
With $\sigma_j$ the canonical section of ${\mathcal O}(\Sigma_j)$ as before, define a  function $\zeta$ on $\Omega \times \bar\D$ by
\begin{equation} \label{bU'}
\sigma_j = \zeta e.
\end{equation}
Thus, $\zeta$ is holomorphic with respect to $J_{\bar{M}}$. 
\begin{definition}  \label{holcoord}
We refer to  
\[
(z^1,\zeta)
\]
as the  {\it holomorphic coordinates (with respect to $J_{\bar{M}}$}) on  a set $\Omega \times \bar\D$ of type (A).
\end{definition}

Since $z^2=0=\zeta$ on $\Omega \times \{0\}$ and $z^2 \neq 0$, $\zeta \neq 0$ on $\Omega \times \bar \D^*$, 
\[
d\log z^2 - d\log \zeta = \frac{dz^2}{z^2} - \frac{d\zeta}{\zeta} = O(1)(dz^1+d\bar z^1 + dz^2+d\bar z^2).
\]
Taking real and imaginary parts, 
\begin{eqnarray} \label{sreta}
 \frac{dr}{r} - \frac{ds}{s} &  = &   O(1)(dz^1+d\bar z^1 + dz^2+d\bar z^2)
 \nonumber 
 \\
d\theta-d\eta & = &  O(1)(dz^1+d\bar z^1 + dz^2+d\bar z^2).
\end{eqnarray}

Let
\[
a:\Omega \times \bar\D  \rightarrow \C^*
\]
be a smooth function  bounded above and bounded away from 0 satisfying 
\begin{equation} \label{az}
a d\zeta\big|_{\Omega \times \{0\}} = dz^2\big|_{(p,0)}, \ \ \forall z^1 \in \Omega.
\end{equation}
Thus,
\[
ad\zeta=dz^2(1+O(r))  =\frac{dz^2}{z^2} z^2(1+O(r))= \left(\frac{d\zeta}{\zeta} +O(1)\right)z^2(1+O(r)).
\]
Plugging in $\frac{\partial}{\partial\zeta}$ in the above equation, we obtain
\begin{equation} \label{laer}
a=\frac{z^2}{\zeta}(1+ O(r)).
\end{equation}
From this, we immediately obtain
\begin{eqnarray} \label{az'}
|a \zeta| & = &  r(1+O(r))
\nonumber  \\
\log |a\zeta|^2 & = & \log r^2+ \log (1+O(r))= \log r^2+ O(r).\ \nonumber 
\end{eqnarray}
This implies 
\[
\log |\sigma_j|_{h_j}^{-2}=\log b|\zeta|^{-2}
=\log b -\log r^2 +\log |a|^2+O(r)=-\log r^2+A +O(r)
\]
where
\[
A(z^1)=\log b(z^1,0)+\log |a(z^1,0)|^2.
\]

The  function $a$ depends on the choice of $e$ and $\sigma$ whereas the function $b$ depends on the choice of $e$ and $h_j$.  Thus, by scaling $\sigma$  if necessary, we can assure that $b$ satisfies the following two conditions:
\begin{eqnarray} 
-\log |\zeta|^2+\log b & > & 0 \mbox{ on } \Omega \times \bar\D \label{mochizukirequirement1}\\
\log b +\log |a|^2& > & 0 \mbox{ on } \Omega \times \{0\}.  \label{mochizukirequirement2}
\end{eqnarray}  
We compute
 \begin{eqnarray}
-\frac{\sqrt{-1}}{2}\partial \bar{\partial }\log \log |\sigma_j|_{h_j}^{-2} 
& = & 
\frac{\sqrt{-1}}{2} \left( \frac{\partial \log |\sigma_j|_{h_j}^{2} \wedge 
\bar{\partial} \log |\sigma_j|_{h_j}^{2}}{ (\log |\sigma_j|_{h_j}^2)^2 } -\frac{\partial \bar{\partial} \log |\sigma_j|_{h_j}^{2}}{ \log |\sigma_j|_{h_j}^{2}}\right) \nonumber\\
& = & 
\frac{\sqrt{-1}}{2} \left( 
 \frac{ (\partial \log \zeta +\partial \log b) \wedge
  (\bar{\partial} \log \bar{\zeta}+\bar{\partial} \log b)}
{ (\log |\sigma_j|_{h_j}^{2})^2 }  -\frac{\partial \bar{\partial} \log b}{ \log |\sigma_j|_{h_j}^{2}}\right).
\nonumber\\
& = & 
\frac{\sqrt{-1}}{2} \left( 
 \frac{ \partial \log \zeta \wedge
  \bar{\partial} \log \bar{\zeta}}
{ (\log |\sigma_j|_{h_j}^{2})^2 } +  \frac{ \partial \log \zeta \wedge
  \bar{\partial} \log b}
{ (\log |\sigma_j|_{h_j}^{2})^2 }  \right.
\nonumber \\
& &  \ \ \ \ \ \ \ \ \ \ \  +  \frac{ \partial \log b \wedge
  \bar{\partial} \log \bar{\zeta}}
{ (\log |\sigma_j|_{h_j}^{2})^2 }+ \left. \frac{ \partial \log b \wedge
  \bar{\partial} \log b}
{ (\log |\sigma_j|_{h_j}^{2})^2 } -\frac{\partial \bar{\partial} \log b}{ \log |\sigma_j|_{h_j}^{2}}\right).
\label{expandedemetric}
 \end{eqnarray} 
Note that because of (\ref{az}) and since $b$ is a smooth function bounded away from 0, we have
\begin{eqnarray} \label{mainterm}
\frac{\partial \log \zeta \wedge 
\bar{\partial} \log \zeta}{ (\log |\sigma_j|_{h_j}^2)^2 }
& = &  \frac{dz^2 \wedge 
 d \bar{z}^2}{r^2(-\log r^2+A)^2}+ Error_1+ Error_2
 \nonumber 
 \\
\frac{ \partial \log \zeta \wedge
  \bar{\partial} \log b} 
{ (\log |\sigma_j|_{h_j}^{2})^2 }, \ \frac{ \partial \log b \wedge
  \bar{\partial} \log b}
{ (\log |\sigma_j|_{h_j}^{2})^2 }
& = &  Error_1+ Error_2
\nonumber \\
\frac{ \partial \log b \wedge
  \bar{\partial} \log b}
{ (\log |\sigma_j|_{h_j}^{2})^2 }, \ \frac{ \partial \bar{\partial} \log b}
{ \log |\sigma_j|_{h_j}^{2} } & = & Error_2.
\end{eqnarray}
where $Error_1$ is a form of the type
\[
O\left(\frac{1}{ r(-\log r^2+A)^2 }\right)dz^1\wedge d \bar{z}^2 \
\mbox{or} \
O\left(\frac{1}{ r(-\log r^2+A)^2 }\right)dz^2\wedge d \bar{z}^1
\]
and $Error_2$ is a form of the type
\[
O\left(\frac{1}{ (-\log r^2+A)^2 }\right)dz^1\wedge d \bar{z}^1\
\mbox{or} \
O\left(\frac{1}{ (-\log r^2+A)^2 }\right)dz^2\wedge d \bar{z}^2.
\]

In coordinate $z^1$ of $\Omega \subset \Sigma_j$, let the local expression of the metric  $g_{\Sigma_j}$ given by  Definition~\ref{(j)} be  $\lambda \, dz^1d\bar z^1$.  Then it follows from the above estimates that in the coordinates $(z^1, z^2)$ of $\Omega \times \bar \D$ and with $r=|z^2|$,   the metric expression of $g$ is 
\begin{equation} \label{ginA}
\left( 
\begin{array}{cc}
g_{1\bar 1} & g_{1 \bar 2}
\\
g_{2 \bar 1} & g_{2 \bar 2}
\end{array}
\right)
=
\left(
\begin{array}{ll}
\lambda + O\left(\frac{1}{ (-\log r^2+A)^2 }\right)& O\left(\frac{1}{ r(-\log r^2+A)^2 } \right) 
\\
O\left(\frac{1}{ r(-\log r^2+A)^2 } \right) 
& \frac{1}{r^2(-\log r^2+A)^2} + O\left(\frac{1}{ (-\log r^2+A)^2 }\right)
\end{array}
 \right).
 \end{equation}
Furthermore, the local expression for the inverse $g^{-1}$ is
 \begin{equation} \label{ginverse}
\left( 
\begin{array}{cc}
g^{1\bar 1} & g^{1 \bar 2}
\\
g^{2 \bar 1} & g^{2 \bar 2}
\end{array}
\right)
=
\left(
\begin{array}{ll}
\frac{1}{\lambda} +O\left( r^2\right)& O\left(r \right) 
\\
O\left(r \right) 
&r^2(-\log r^2+A)^2+O(r^2)
\end{array}
 \right).
 \end{equation}
The  {it product metric} $P$ on $\Omega \times \bar\D$  is defined by taking the dominant terms of $g$.  More precisely, let 
\begin{equation} \label{def:P}
\left( 
\begin{array}{cc}
P_{1\bar 1} & P_{1 \bar 2}
\\
P_{2 \bar 1} & P_{2 \bar 2}
\end{array}
\right)
=
\left(
\begin{array}{lc}
\lambda & 0\\
0 & \frac{1}{r^2(-\log r^2+A)^2} 
\end{array}
 \right).
 \end{equation}
The inverse $P^{-1}$
\[
\left( 
\begin{array}{cc}
P^{1\bar 1} & P^{1 \bar 2}
\\
P^{2 \bar 1} & P^{2 \bar 2}
\end{array}
\right)
=
\left(
\begin{array}{lc}
\frac{1}{\lambda}& 0\\
0
&r^2(-\log r^2+A)^2
\end{array}
 \right).
 \]
Thus,
\begin{equation} \label{closegP}
g^{-1} -P^{-1}= 
\left(
\begin{array}{ll}
 O\left(r^2\right)& O\left(r \right) 
\\
O\left(r \right) 
& O\left(r^2\right)
\end{array}
 \right).
 \end{equation}
Comparing the local expression of $g$ and $P$, we obtain
  \begin{equation} \label{volcomp}
 d\mbox{vol}_g =d\mbox{vol}_P \left(1+O\left(\frac{1}{(-\log r^2+A)^2} \right)\right).
  \end{equation}
A straightforward computation gives
  \begin{eqnarray} \label{volumeformofP}
  d\mbox{vol}_P & = & d\mbox{vol}_{g_{\Sigma_j}}  \wedge \frac{dz^2\wedge d\bar{z}^2}{-2i r^2(-\log r^2+A)^2}  
  \nonumber  \\
    P^{1\bar 1} d\mbox{vol}_P & = &g_{\Sigma_j}^{1\bar 1} d\mbox{vol}_{g_{\Sigma_j}}  \wedge \frac{dz^2\wedge d\bar{z}^2}{-2i r^2(-\log r^2+A)^2}\\
    P^{2\bar 2} d\mbox{vol}_P & = &  d\mbox{vol}_{g_{\Sigma_j}} \wedge \frac{dz^2 \wedge d\bar{z}^2}{-2i}. \label{P11}\nonumber
  \end{eqnarray} 
  The metric $P$ (and hence the metric $g$) is of finite volume  over $\Omega \times \D^*$  since
$g_{\Sigma_j}$ is a smooth metric on $\Omega \subset \Sigma_j$ and  
 \begin{eqnarray*} 
\mbox{Vol}_P(\Omega \times \D^*) & = & \mbox{Area}(\Omega) \cdot 2\pi \lim_{\epsilon \rightarrow 0} \int_{\epsilon}^1 \frac{rdr}{r^2(-\log r^2+A)^2} 
\nonumber 
\\
& = &  -\mbox{Area}(\Omega) \cdot 2\pi \lim_{\epsilon \rightarrow 0} \int_{\epsilon}^1 \frac{d(-\log r^2+A)}{2(-\log r^2+A)^2}
 \nonumber 
 \\
 & = & - \mbox{Area}(\Omega) \cdot \pi \lim_{\epsilon \rightarrow 0} \left. \frac{1}{-\log r^2+A}\right|_{\epsilon}^{1}<\infty.
  \end{eqnarray*}
 Moreover, since
 \begin{eqnarray} \label{fingers}
2\pi \mbox{Area}_{g_{\Sigma_j}}(\Omega)  \log \frac{r_2}{r_1} 
\nonumber 
& = & 
  \int_{\Omega \times \D_{r_1,r_2}}  d\mbox{vol}_{g_{\Sigma_j}}  \wedge \frac{dr \wedge d\theta}{r} =  \int_{\Omega \times  \D_{r_1,r_2}} P^{\theta \theta} d\mbox{vol}_P,
\end{eqnarray}
there exists a constant $C>0$ such that 
 \begin{eqnarray} \label{areacompA}
\left|\frac{1}{2\pi}  \int_{\Omega \times \D_{r_1,r_2}}g^{\theta \theta} d\mbox{vol}_g -   \mbox{Area}(\Omega) \log \frac{r_2}{r_1} \right|
\leq 
 C, \ \ \ \ 0<r_1<r_2<1.
\end{eqnarray}

   \begin{lemma} \label{vol} 
The Poincar\'e type metric $g$ defined by Definition~\ref{poincarekahler} satisfies the following:  There exists $c>0$ such that, on the set $\Omega \times \D^*_{\frac{1}{4}}$  away from the crossings with holomorphic coordinates $(z^1,\zeta=re^{i\theta})$, 
 \begin{equation} 
 \frac{1}{c} \, d\mbox{vol}_g \ \leq \  \rho d\rho \wedge d\phi  \wedge \frac{r dr \wedge d\theta}{r^2(-\log r^2)^2} \  \leq \ c \, d\mbox{vol}_g \nonumber
 \end{equation}
 \end{lemma}

\begin{proof}
This  is  immediate from the fact that the metric  $g_{\Sigma_j}$ given by  Definition~\ref{(j)}   is smooth combined with (\ref{volcomp}).
\end{proof}

 \begin{remark} \label{alike}
  The key feature  of the metric $P$ is the following:   Define $Q$ to be the product metric 
  \[
  Q=g_{\Sigma_i} \oplus \mbox{Re}\left(  \frac{dz^2 d\bar{z}^2}{-2i} \right)
  =
  g_{\Sigma_i} \oplus (dr^2+r^2 d\theta^2)
 \  \mbox{
 on $\Omega \times \D^*$.}
 \]
   Then
  \[
  P^{2\bar{2}} d\mbox{vol}_P = Q^{2\bar{2}} d\mbox{vol}_{Q}.
  \] 
This is important in   Section~\ref{sec:ptm} below where we estimate the energy of $v$.  In particular, we have
\[
\int_{\Omega \times \D}  P^{2\bar 2} \left| \frac{\partial v}{\partial z^2} \right|^2 d\mbox{vol}_P= \int_{\Omega \times \D} Q^{2\bar 2} \left| \frac{\partial v}{\partial z^2} \right|^2 d\mbox{vol}_Q= \int_{\Omega} \left( \int_{\{z^2\} \times \D} \left| \frac{\partial v}{\partial z^2} \right|^2 \frac{dz^2 \wedge d\bar z^2}{-2i} \right) d\mbox{vol}_{g_{\Sigma_j}}.
\] 
Note that  the inside integral on the right hand side above is exactly the energy of the harmonic map $v|_{\{z^1\} \times \D}$ from the disk.  Indeed, we have already recorded an estimate for this integral in Lemma~\ref{Aest}.  We will take advantage of this   in the proof of  Lemma~\ref{lemmadirvA} below.
    \end{remark}

\subsection{Metric estimates in a set of Type $\mbox{(B)}$}  First recall that a type (B) set of Section~\ref{sec:neardivisor} is a set  $\bar{\D}_{z^1} \times \bar{\D}_{z^2}$  with  
\[
\bar\D_{z^1} \subset \Sigma_j, \ \bar\D_{z^2} \subset \Sigma_i \ \mbox{ and }(0,0) \in \Sigma_j \cap \Sigma_i
\] 
such  that the standard product coordinates $(z^1,z^2)$ are also holomorphic coordinates with respect to complex structure $J_{\bar{M}}$. 
Since $|\sigma_i|_{h_i}=|z^1|$ and $|\sigma_j|_{h_j}=|z^2|$,
\begin{eqnarray*}
-\frac{\sqrt{-1}}{2} \left(\partial \bar{\partial }\log \log |\sigma_i|_{h_i}^{-2} \partial \bar{\partial }\log \log |\sigma_j|_{h_j}^{-2} \right)
 & = &  \frac{\partial \log z^1 \wedge 
\bar{\partial} \log \bar{z}^1}{ (-\log |z^1|^2)^2 }+  \frac{\partial \log z^2 \wedge 
\bar{\partial} \log \bar{z}^2}{ (-\log |z^2|^2)^2 }.
\end{eqnarray*}
In the  coordinates $(z^1,z^2)$ and with $\rho=|z^1|$ and $r=|z^2|$, the local expression of the   metric $P$ associated to the above K\"ahler form is 
\begin{equation} \label{productmetricB}
\left( 
\begin{array}{cc}
P_{1\bar 1} & P_{1 \bar 2}
\\
P_{2 \bar 1} & P_{2 \bar 2}
\end{array}
\right)
=
\left(
\begin{array}{ll}
\frac{1}{\rho^2 (\log \rho^2)^2} & 0
\\
0 
& \frac{1}{r^2 (\log r^2)^2} 
\end{array}
 \right)
 \end{equation}
 and the  inverse $P^{-1}$ is given by 
 \[
\left( 
\begin{array}{cc}
P^{1\bar 1} & P^{1 \bar 2}
\\
P^{2 \bar 1} & P^{2 \bar 2}
\end{array}
\right)
=
\left(
\begin{array}{ll}
\rho^2 (\log \rho^2)^2 & 0
\\
0 
& r^2 (\log r^2)^2 
\end{array}
 \right).
 \]
With $\diamond=O( \rho^2 (\log \rho^2)^2)$ and $\Box=O(r^2 (\log r^2)^2)$,  the local expression of the metric $g$ and its inverse $g^{-1}$ is
 \begin{eqnarray} \label{ginB}
\left( 
\begin{array}{cc}
g_{1\bar 1} & g_{1 \bar 2}
\\
g_{2 \bar 1} & g_{2 \bar 2}
\end{array}
\right)
& = & 
\left(
\begin{array}{ll}
\frac{1}{\rho^2 (\log \rho^2)^2} +O(1)& O(1)
\\
O(1)
& \frac{1}{r^2 (\log r^2)^2} +O(1)
\end{array}
 \right)
\\
\left( 
\begin{array}{cc}
g^{1\bar 1} & g^{1 \bar 2}
\\
g^{2 \bar 1} & g^{2 \bar 2}
\end{array}
\right)
& = & 
\left(
\begin{array}{ll}
\rho^2 (\log \rho^2)^2 (1+\diamond+ \Box) & \diamond \Box
\\
\diamond \Box
& r^2 (\log r^2)^2(1 +\diamond+ \Box)
\end{array}
 \right). \label{ginverseinB}
 \end{eqnarray}
Thus,
 \begin{equation} 
\label{closegPB}
 g^{-1} - P^{-1} = \left(
\begin{array}{ll}
\rho^2 (\log \rho^2)^2(\diamond + \Box) & \diamond \Box
\\
\diamond \Box
& r^2(\log r^2)^2 (\diamond +\Box)
\end{array}
 \right)
 \end{equation}

  \begin{equation} \label{volcompB}
 d\mbox{vol}_g =d\mbox{vol}_P \left(1+O\left(\frac{1}{(-\log r^2+A)^2} \right) + O\left(\frac{1}{(-\log r^2+A)^2} \right)\right).
  \end{equation}
A straightforward computation gives
  \begin{eqnarray} \label{volumeprodB}
  d\mbox{vol}_P & = & \frac{dz^1\wedge d\bar{z}^1}{-2i \rho^2(\log \rho^2)^2}  \wedge \frac{dz^2\wedge d\bar{z}^2}{-2ir^2(\log r^2)^2} \nonumber  \\
    P^{1\bar 1} d\mbox{vol}_P & = &\frac{dz^1\wedge d\bar{z}^1}{-2i }  \wedge \frac{dz^2\wedge d\bar{z}^2}{-2i r^2(\log r^2)^2}\\
    P^{2\bar 2} d\mbox{vol}_P & = &\frac{dz^1\wedge d\bar{z}^1}{-2i \rho^2(\log \rho^2)^2}  \wedge \frac{dz^2\wedge d\bar{z}^2}{-2i }.
    \label{P11'} \nonumber
  \end{eqnarray} 
Similarly to (\ref{areacompA}), we also obtain for subsets $\D_{r_1,r_2} \times \Omega$, $ \Omega \times \D_{r_1,r_2}$ of  $\bar \D_{z^1} \times \bar \D_{z_2}$
  \begin{eqnarray} \label{areacompB}
\left| \frac{1}{2\pi} \int_{\D_{r_1,r_2}  \times  \Omega}  g^{\phi \phi} d\mbox{vol}_g -  \mbox{Area}_{g_{\Sigma_i}}(\Omega) \log \frac{r_2}{r_1}   \right|  &  \leq &     C,
\nonumber \\
\left| \frac{1}{2\pi} \int_{\Omega \times   \D_{r_1,r_2}}  g^{\theta \theta} d\mbox{vol}_g -  \mbox{Area}_{g_{\Sigma_j}}(\Omega) \log \frac{r_2}{r_1}   \right|  &  \leq &     C.
\end{eqnarray}

   \begin{lemma} \label{volB} 
The Poincar\'e type metric $g$ defined by Definition~\ref{poincarekahler} satisfies the following:  There exists $c>0$ such that in  neighborhood 
  $\bar \D_{\frac{1}{4}}^*  \times  \bar \D_{\frac{1}{4}}^*$ near  a crossing with holomorphic coordinates $(z^1=\rho e^{i\phi}, z^2=re^{i\theta})$,
\begin{equation} 
  \frac{1}{c} \, d\mbox{vol}_g \ \leq \ \frac{\rho  d\rho \wedge d\phi}{\rho^2(-\log \rho^2)^2}  \wedge \frac{r dr \wedge d\theta}{r^2(-\log r^2)^2} \  \leq \ c \, d\mbox{vol}_g. \nonumber
 \end{equation}
 \end{lemma}

\begin{proof}
This  is  immediate from (\ref{volumeprodB}).
\end{proof}

\section{The prototype section}\label{prototype}
The goal of this section is to construct a   {\it prototype section}  
with logarithmic energy growth near the divisor.  The key is the fiber-wise harmonic sections on the normal bundle of the divisor $\Sigma$, the existence of which follows from the Dirichlet problem on the punctured disk (cf.~Theorem~\ref{exists}).

The first step in the construction  is  to define local sections near the divisor. Recall the sets of type (A) and of type (B) described in (\ref{ref:typeA}) and (\ref{ref:typeB}) respectively. In Subsection~\ref{NJ},  we construct a local prototype  section in a set of type (A).  In Subsection~\ref{NJ}, we construct a local prototype section  in a set of type (B).  In Subsection~\ref{gluing}, we glue these sections together to define a prototype section near the divisor and extend it to all of $M$.   
In summary, we construct  a locally Lipschitz global section 
\[
v:  M \rightarrow \tilde M \times_\rho \tilde X
\]
of logarithmic energy growth near the divisor.  

\subsection{In a neighborhood away from the junctures} \label{NY}

The goal of this subsection is to construct a local prototype section in a set of type (A) and derive some energy estimates.  We start with the following:
\begin{itemize}
\item $\Omega  \times  \bar {\D}$ is a set of type (A)  with $\Omega \subset \Sigma_j$
\item $(z^1,z^2)$ are the standard product coordinates of $\Omega  \times  \bar {\D}$ (cf~(\ref{stdprod}))
\item
$r$, $\theta$ are parameters defined by
$z^2=re^{i\theta}
$  

\item $\Sp^1_{\theta}
$ is  the boundary $\partial \bar{\D}_{z^2}$ of $ \bar{\D}_{z^2}$

\item $\R_\theta \rightarrow \Sp^1$ is the universal cover

\item $\widetilde{\bar{\D}^*} \rightarrow \bar \D^*$ is the universal cover

\item 
$[\Sp^1_{\theta}]$ is  the element of $\pi_1(\Sp^1_\theta) \simeq \pi_1(\Omega \times \bar \D^*)$ associated to the loop $\Sp^1_\theta \simeq \{z^1\} \times \Sp^1_{\theta}$

\item $\rho':\pi_1(\Omega \times \bar \D^*) \rightarrow \mathsf{Isom}(\tilde X)$ is a homomorphism

\item $ (\Omega \times \widetilde{\bar{\D}^*})  \times_{\rho'} \tilde X \rightarrow \Omega \times \bar{\D}^*$ is a   fiber bundle

 \item $k:\Omega \times \bar \D^* \rightarrow  (\Omega \times \widetilde{\bar{\D}^*})  \times_{\rho'} \tilde X$ is a locally Lipschitz section

\item  $\R_\theta \times_{\rho'} \tilde X \rightarrow \Sp^1_\theta$ is a  fiber bundle

\item  $\Ej$ is  the infimum of the energies of sections $\Sp^1_\theta  \rightarrow \R_\theta \times_{\rho'} \tilde X$.

 \end{itemize}

\subsubsection{Construction of a prototype section in a set  of type (A)}  \label{AJ} 

We define
\begin{equation} \label{mochizukimap}
v: \Omega \times \bar \D^* \rightarrow (\Omega \times \widetilde{\bar{\D}^*})  \times_{\rho'} \tilde X
\end{equation}
by setting $v$ to be 
the  fiber-wise harmonic section  with boundary values given by $k\big|_{ \Omega \times \Sp^1 }$.  More precisely, we apply Theorem~\ref{exists} as follows:  For each $z^1 \in \Omega$, the restriction 
\[
v_{z^1}:=
v\big|_{ \{z^1\}\times \bar\D^*}:\bar{\D}^* \approx \{z^1\} \times \bar{\D}^* \rightarrow \widetilde{\bar \D^*} \times_{\rho'} \tilde X
\]
is the unique  harmonic section with logarithmic energy growth and boundary values
\[
v_{z^1}\big|_{ \Sp^1 \approx \{z^1 \} \times \Sp^1}=k\big|_{ \Sp^1 \approx \{z^1 \} \times \Sp^1}.
\]
\subsubsection{Derivative estimates   in a set of type (A)}

\begin{lemma}[Derivative estimates in set of type (A)]  \label{Aest}
There exists a constant $C$  such that 
\begin{eqnarray*}
\left| \frac{\partial v}{\partial z^1}(z^1_0,z^2_0)  \right| & \leq &  C,  \ \ \forall (z^1_0,z^2_0) \in \Omega \times \bar \D^*  \label{tangdir}
\\
\int_{ \{z^1_0\}  \times \D_{r,r_0}} \left| \frac{\partial v}{\partial z^2} \right|^2   \frac{dz^2 \wedge d\bar{z}^2}{-2i} & \leq &  C+\Ej\log \frac{r_0}{r} , \ \ \  \forall z^1_0 \in \Omega \label{dirz2}
\end{eqnarray*}
where $0<r<r_0\leq \frac{1}{4}$ and $z^2=re^{i\theta}$.
\end{lemma}
\begin{proof}
Denote the Lipschitz constant of $k$ on $\Omega \times \Sp^1$ by $L$.
Let  $z^1_0$,  $z^1 \in \Omega$.  Since $v_{z^1_0}$ and $v_{z^1}$ are harmonic sections, the function $z \mapsto d^2(v_{z^1_0}(z),v_{z^1}(z))$ is subharmonic in $\D^*$ (cf.~\cite[Remark 2.4.3]{korevaar-schoen1}).

 \begin{figure}[h!]
  \includegraphics[width=2.5in]{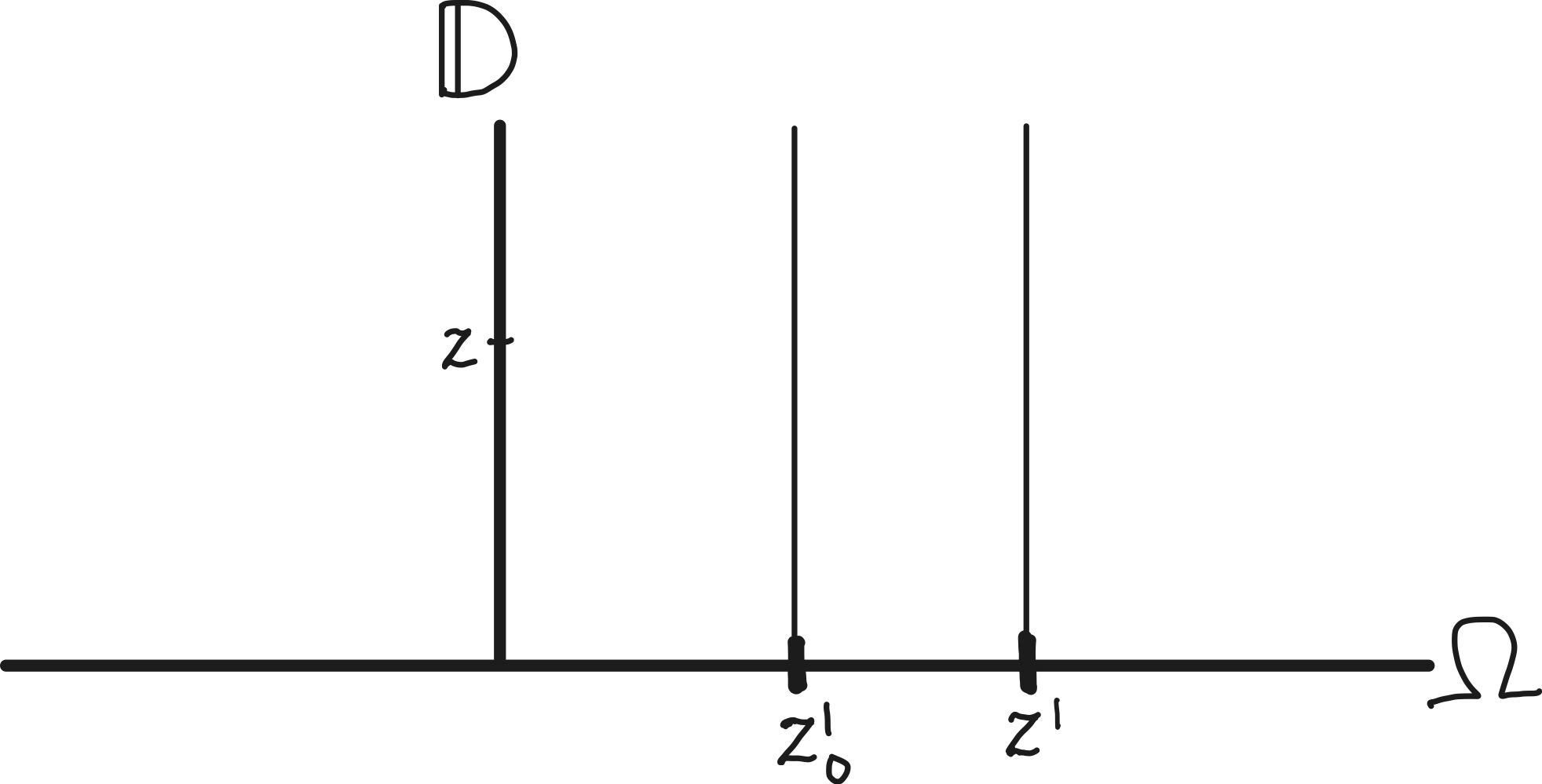}
  \caption{The map $z \mapsto d^2(v_{z^1_0}(z),v_{z^1}(z))$ is defined on $\D$.}
  \label{fig:slices}
\end{figure}

By Theorem~\ref{exists} (iii) and the triangle inequality,  
 \[
 \lim_{z \rightarrow 0} \  d^2(v_{z^1_0}(z),v_{z^1}(z)) + \epsilon \log |z| =-\infty, \ \ \forall \epsilon>0.
 \]  Thus, $d^2(v_{z^1_0},v_{z^1})$
 extends to subharmonic function on $\D$ (cf.~Lemma~\ref{shdisk}).  Thus, the maximum principle implies that for any $z^2_0=r_0e^{i\theta_0}$, we have that,
 \begin{eqnarray} \label{mp}
  d^2(v_{z^1_0}(r_0 e^{i\theta_0}), v_{z^1} (r_0 e^{i\theta_0})) &\leq& \sup_{\theta \in \Sp^1} d^2(k(z^1_0, e^{i\theta}), k(z^1, e^{i\theta}))\\
&\leq& L^2 |z_0^1-z^1|^2 \nonumber.
 \end{eqnarray}
In other words, for every fixed $z^2_0=r_0e^{i\theta_0}$, the map  $z_1 \mapsto v_{z^1}(z^2_0)$ is Lipschitz  which immediately implies the first estimate.    The second estimate follows from
Theorem~\ref{exists}.
  \end{proof}

\subsection{In a neighborhood of the juncture} 
\label{NJ}

The goal of this subsection is to construct a local prototype section in a set of type (B) and derive some derivative estimates.  We start with the following:

\begin{itemize}
\item $\bar{\D}_{z^1} \times \bar{\D}_{z^2}$  is a set of type (B) with $\D_{z^1} \subset \Sigma_j$ and $\D_{z^2} \subset \Sigma_i$
\item  $(z^1,z^2)$ are the standard product (and holomorphic) coordinates 
\item
 $\rho$, $\phi$, $r$, $\theta$ are the parameters defined by
$
z^1=\rho e^{i\phi} \ \mbox{  and } z^2=re^{i\theta}
$  

\item $
\Sp^1_{\phi}$ is the boundary of $\bar \D_{z^1}$ and   $\Sp^1_{\theta}
$ is  the boundary of $\bar{\D}_{z^2}$

\item $\R_\phi \rightarrow \Sp^1_\phi$ and $\R_\theta \rightarrow \Sp^1_\theta$ are the universal covers

\item $\widetilde{\bar{\D}^*_{z^1}} \rightarrow \bar \D^*_{z^1}$ and $\widetilde{\bar{\D}^*_{z^2}} \rightarrow \bar \D^*_{z^2}$ are  the universal cover

\item $
\Sp^1_{\phi} \times \Sp^1_{\theta}
$ is  the boundary of $\bar{\D}_{z^1} \times \bar{\D}_{z^2}$

\item $\R_{\phi} \times \R_{\theta} \rightarrow \Sp^1_{\phi} \times \Sp^1_{\theta}
$ is the universal covering map

\item $[\Sp^1_{\phi}]$  is the element of $\pi_1(\bar \D^*_{z^1}) \simeq \pi_1(\Sp^1_\phi)$  generated by $\Sp^1_{\phi}$  and $[\Sp^1_{\theta}]$  is the element of $\pi_1(\bar \D^*_{z^2}) \simeq \pi_1(\Sp^1_\theta)$  generated by $\Sp^1_{\theta}$

\item 
$[\Sp^1_{\phi}]$ and $[\Sp^1_{\theta}]$ also are the elements of $\pi_1(\bar{\D}^*_{z^1} \times \bar{\D}^*_{z^2})$ generated by $\Sp^1_{\phi}  \simeq \Sp^1_{\phi} \times \{z^2\}$ and $\Sp^1_{\theta} \simeq \{z^1\} \times \Sp^1_{\theta}$ respectively

\item $\pi_1(\bar \D_{z^1}^*) \simeq \pi_1(\Sp^1_\phi)$  and $\pi_1(\bar \D_{z^2}^*) \simeq \pi_1(\Sp^1_\theta)$ are  identified as a subgroup of  $\pi_1(\bar{\D}^*_{z^1} \times \bar{\D}^*_{z^2})$ by the above identification 
\item $\rho':\pi_1(\bar \D^*_{z^1} \times \bar \D^*_{z^2}) \rightarrow \mathsf{Isom}(\tilde X)$ is a homomorphism and $\rho'_k=\rho'|_{ \pi_1(\bar \D_{z^k}^*)}$ for $k=1,2$

\item $(\widetilde{\bar{\D}_{z^1}^*} \times \widetilde{\bar{\D}^*_{z^2}})  \times_{\rho'} \tilde X \rightarrow \bar{\D}_{z^1}^* \times \bar{\D}^*_{z^2}$ and $\widetilde{\bar{\D}_{z^k}^*}  \times_{\rho'_k} \tilde X \rightarrow \bar{\D}_{z^k}^*$ for $k=1,2$ are  fiber bundles

\item $\R_\phi \times_{\rho'_1} \tilde X \rightarrow \Sp^1_\phi$ and $\R_\theta \times_{\rho'_2} \tilde X \rightarrow \Sp^1_\theta$ are fiber bundles

\item    
 $\Ei$ is  the infimum of the energies of sections $\Sp^1_\phi \rightarrow \R_\phi \times_{\rho'_1} \tilde X$
 \item  $\Ej$ is  the infimum of the energies of sections $\Sp^1_\theta \rightarrow \R_\theta  \times_{\rho'_2} \tilde X$
 
\item  If $\rho'([\Sp^1_{\theta^1}])$ and $\rho'([\Sp^2_{\theta^2}]) \in \mathsf{Isom}(\tilde X)$ are semisimple, then call this set-up the {\it  semisimple case}.  If they are parabolic we call it the {\it  parabolic case}.  Note that by Proposition~\ref{prop:fys}, $\rho([\Sp^1])$ and $\rho([\Sp^2])$ are both semisimple or both parabolic.

\item $k: \bar{\D}_{z^1}^* \times \bar{\D}^*_{z^2} \rightarrow  (\widetilde{\bar{\D}_{z^1}^*} \times \widetilde{\bar{\D}^*_{z^2}})  \times_{\rho'} \tilde X$ is a locally Lipschitz section 
\item $\tilde k:  \widetilde{\bar{\D}_{w^1}^*} \times \widetilde{\bar{\D}^*_{w^2}}  \rightarrow \tilde X$ is the corresponding $\rho'$-equivariant map

\item 
$
\kappa$
is the restriction of  $\tilde k$ to $\R_\phi \times \R_\theta$.
\end{itemize}

\subsubsection{Construction of a prototype section  in a set  of type (B)}  \label{prototype(B)}
We construct the local section 
\[
v:  \bar{\D}^*_{z^1} \times \bar{\D}^*_{z^2} \rightarrow (\widetilde{\bar{\D}^*_{z^1}} \times \widetilde{\bar{\D}^*_{z^2} }) \times_{\rho'} \tilde X
\]
 as follows:

\begin{itemize}
\item 
Apply Lemma~\ref{flattorus} in the semisimple case  and Lemma~\ref{almostflattorus} in the parabolic case.  In either case, 
 for each $t \in[0,\infty)$, 
there exist constants $a,b>0$ and a $\rho'$-equivariant  map 
\[
\tilde h_t: \R_\phi \times \R_\theta \rightarrow \tilde X
\]
and an associated section
\[
h_t: \Sp^1_{\phi} \times \Sp^1_{\theta} \rightarrow (\R_\phi \times \R_\theta) \times_{\rho'} \tilde X
\]
such that 
 \begin{equation} \label{hder}
\left| \frac{\partial  h_t}{\partial t} \right|^2 \leq 1, \ \ \ \left| \frac{\partial  h_t}{\partial \phi} \right|^2 \leq \frac{1}{2\pi}\Ei(1+be^{-at}), \ \ \ 
\left| \frac{\partial  h_t}{\partial \theta}
\right|^2 \leq \frac{1}{2\pi}\Ej(1+be^{-at}).
 \end{equation}
 \item 
Define the diagonal set 
 \[
 D=\{(\rho e^{i\phi}, \rho e^{i\theta}) \in \bar{\D}_{z^1}^* \times \bar{\D}^*_{z^2}:  \rho \in (0,1], \ \ \phi, \theta \in \Sp^1\}.
 \]
\item Define $v_D:D \rightarrow X$  as follows: Fix $(\theta,\phi) \in \Sp^1 \times \Sp^1$.   
\begin{itemize}
\item[$\circ$] For $\rho \in (0,\frac{1}{2}]$, 
let
  \[
v_D(\rho e^{i\phi}, \rho e^{i\theta})=h_{3(-\log \rho)^{\frac{1}{3}}}(\phi,\theta).
 \]
\item[$\circ$]  For $\rho \in [\frac{1}{2},1]$, 
let
  the curve
 \[
 \rho \mapsto \gamma(\rho) \mbox{ for } \rho \in [\frac{1}{2},1]
 \]
 be the  the geodesic  between $\tilde h_{3(\log 2)^{\frac{1}{3}}}(\phi, \theta)$ and $\kappa(\phi, \theta)$.
 Define
 \[
  v_D(\rho e^{i\phi}, \rho e^{i\theta})=\gamma(\rho).
   \]
\end{itemize}

 \item  Let  
 \[
 Z_1:=\{ (z^1,z^2) \in  \bar{\D}_{z^1}^* \times \bar{\D}^*_{z^2}:  |z^1|\geq |z^2|\}
 \]
 and  
 \[
 \varphi_1: \bar{\D}_{w^1}^* \times \bar{\D}^*_{w^2} \rightarrow  Z_1 \subset  \bar{\D}_{z^1}^* \times \bar{\D}^*_{z^2}
 \]
be a homeomorphism defined by (see Figure~\ref{fig:phi1})
 \[
(z^1,z^2)=\varphi_1(w^1,w^2), \ \ \  z^1=w^1, \ z^2=|w^1|w^2.
 \]
 \begin{figure}[h!]
  \includegraphics[width=2.5in]{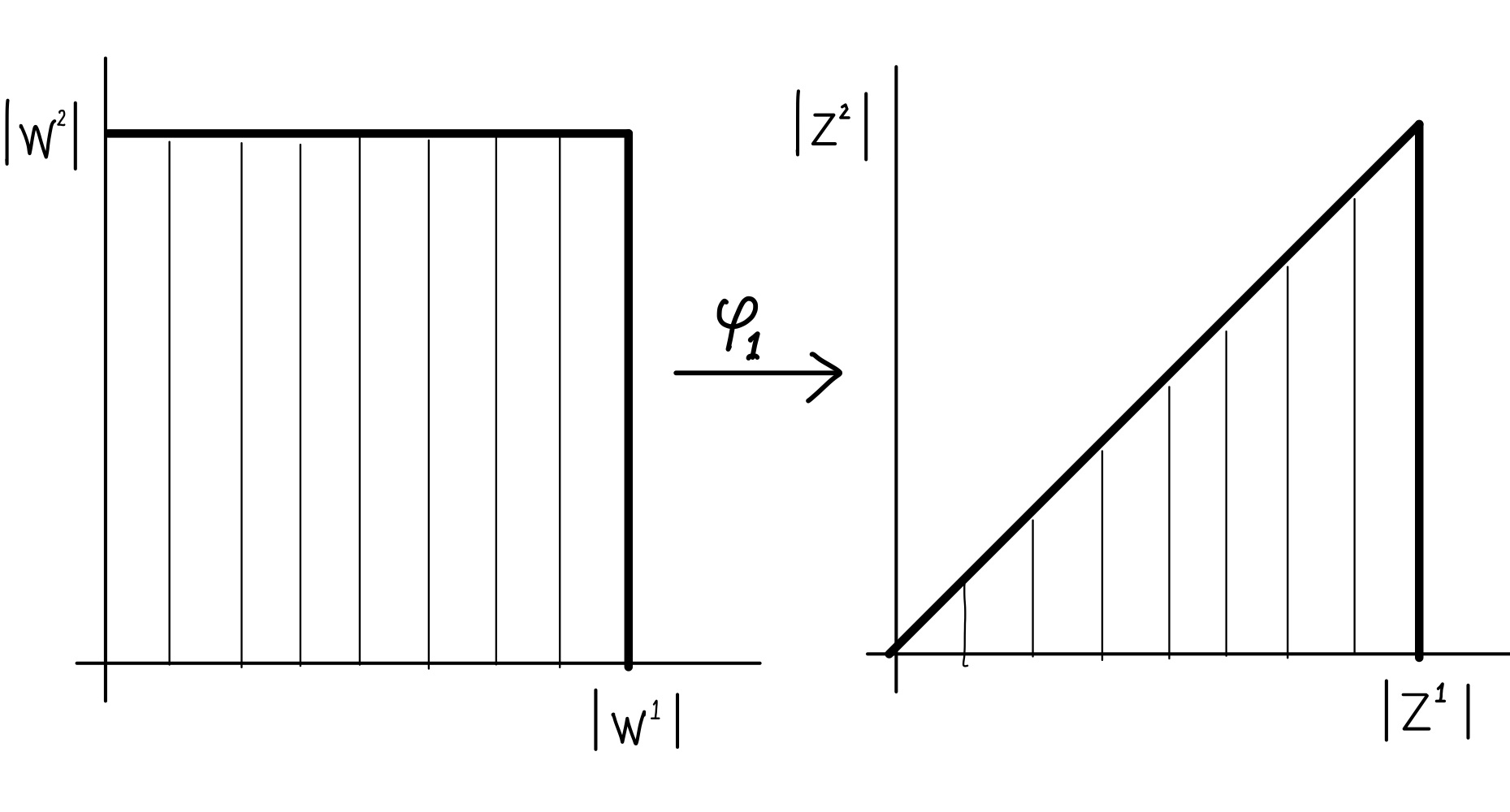}
  \caption{The map $\varphi_1$.}
  \label{fig:phi1}
\end{figure}
\item Define
 \[
v_1:  \bar{\D}_{w^1}^* \times \bar{\D}^*_{w^2} \rightarrow (\widetilde{\bar{\D}_{w^1}^*} \times \widetilde{\bar{\D}^*_{w^2}})  \times_{\rho'} \tilde X
 \]
by setting $v_1$ to be 
the  fiber-wise harmonic section with boundary values given by $v_D \circ \varphi_1 \big|_{ \bar{\D}^*_{w^1} \times \Sp^1_{\theta} }$.  More precisely, we apply Theorem~\ref{exists} as follows:  For each $w^1 \in \bar{\D}^*_{w^1} $, the restriction  
\[
v_{1,w^1}:=v_1\big|_{\{w^1\} \times \bar{\D}^*_{w_2}}
\]
is the unique  harmonic section  with boundary values 
\[
v_{1,w^1}\big|_{\Sp^1 \approx \{w^1\} \times \Sp^1_{\theta}}
=
v_D \circ \varphi_1 \big|_{\Sp^1 \approx \{w^1\} \times \Sp^1_{\theta}}
\]   as in Theorem~\ref{exists}.

\item 
Similarly, let 
\[
 Z_2:=\{ (z^1,z^2) \in  \bar{\D}_{z^1}^* \times \bar{\D}^*_{z^2}:  |z^1|\leq |z^2|\}
 \]
and
\[
 \varphi_2: \bar{\D}_{w^1}^* \times \bar{\D}^*_{w^2} \rightarrow  Z_2 \subset  \bar{\D}_{z^1}^* \times \bar{\D}^*_{z^2}
 \]
 be the homeomorphism defined by
 \[
(z^1,z^2)=\varphi_2(w^1,w^2), \ \ \  z^1=|w^2|w^1, \ z^2=w^2.
 \]
We define
 \[
 {v_2}:  \bar{\D}_{w^1}^* \times \bar{\D}^*_{w^2} \rightarrow (\widetilde{\bar{\D}_{w^1}^*} \times \widetilde{\bar{\D}^*_{w^2}}) \times_{\rho'} \tilde X
 \]
by setting $v_2$ to be 
the  fiber-wise harmonic section  with boundary values given by $v_D \circ \varphi_2 \big|_{\Sp^1_{\phi} \times  \bar{\D}^*_{w^2}  }$.  
\item
Let 
\[
v: \bar{\D}^*_{z^1} \times \bar{\D}^*_{z^2} \rightarrow (\widetilde{\bar{\D}_{z^1}^*} \times \widetilde{\bar{\D}^*_{z^2}}) \times_{\rho'} \tilde X
\]
be the section defined by 
\begin{eqnarray*}
v =
\left\{
\begin{array}{ll}
 v_1 \circ \varphi_1^{-1} & \mbox{ on}  \ \ Z_1,
\\
{v_2}  \circ \varphi_2^{-1}  & \mbox{ on}  \ \ Z_2.
\end{array}
\right.
\end{eqnarray*}
Note that $v$ is well defined since $Z_1 \cap Z_2 =D$ and 
\[
{v_1} \circ \varphi_1^{-1} \big|_D = v_D ={v_2} \circ \varphi_2^{-1} \big|_D.
\]
\end{itemize}

\subsubsection{Derivative estimates in a set  of type (B)}

Since $z^1=w^1$ and $z^2=|w^1|w^2$, we have
\begin{equation} \label{coins}
w^1=z^1=\rho e^{i\phi}, \ \ z^2=\rho e^{i\theta}, \ \ w^2=\frac{z^2}{|z^2|}=\frac{r}{\rho} e^{i\theta}.
\end{equation}
By (\ref{hder}), 
\[
\left| \frac{\partial h_{3(-\log \rho)^{\frac{1}{3}}}}{\partial \rho}(\theta,\phi)\right|^2=  \left| \frac{\partial h_t}{\partial t}(\theta,\phi)\right|^2 \left| \frac{\partial t}{\partial \rho}(\theta,\phi)\right|^2 \leq  \frac{1}{\rho^2(-\log \rho)^{\frac{4}{3}} }.
\] 
Thus, noting that $w^1=z^1=\rho e^{i\phi}$,
\[
\left| \frac{\partial (v_D \circ \varphi_1)}{\partial \rho}\right|^2 \leq \frac{1}{\rho^2(-\log \rho)^{\frac{4}{3}} } \mbox{ on  $\D^*_{w^1,\frac{1}{2}} \times \{ |w^2|=1\}$}.
\]
Furthermore,  (\ref{hder}) implies
\[
\left| \frac{\partial (v_D \circ \varphi_1)}{\partial \phi }\right|^2  \leq \frac{\Ei}{2\pi}(1+be^{-a(-\log \rho)})^2=\frac{\Ei}{2\pi}(1+b\rho^a)^2 \ \mbox{ on $\D^*_{w^1,\frac{1}{2}} \times \{ |w^2|=1\}$.}
\]
Since $v_1$ is a fiber-wise harmonic section,  an argument analogous to the proof  of first inequality of Lemma~\ref{Aest} (i.e.~apply the maximum principle for subharmonic functions $d(u(\rho_1e^{i\phi}), u(\rho_2e^{i\phi}))$ and $d(u(\rho e^{i\phi_1}), u(\rho e^{i\phi_2}))$) implies
\[
\left| \frac{\partial  {v_1}}{\partial \rho}\right|^2 \leq  \frac{1}{\rho^2(-\log \rho)^{\frac{4}{3}} } \ \mbox{ and } \left| \frac{\partial  {v_1}}{\partial \phi }\right|^2  \leq \frac{\Ei}{2\pi}(1+b\rho^a)^2 \ \ \mbox{ in } \D^*_{w^1,\frac{1}{2}} \times \D^*_{w^2}.
\]
Thus,
\begin{equation} \label{penny}
\left| \frac{\partial v_1}{\partial w^1 }\right|^2  \leq  \frac{1}{\rho^2(-\log \rho)^{\frac{4}{3}} } +\frac{\Ei}{2\pi \rho^2}(1+b\rho^a)^2 
\ \ \mbox{ in } \D^*_{w^1,\frac{1}{2}} \times \D^*_{w^1}.
\end{equation}
Since $\kappa$ is a Lipschitz section, $v_D \circ \varphi_1$ is a Lipschitz section for $\rho\geq \frac{1}{2}$.  Thus,
\begin{equation} \label{penny1}
\left| \frac{\partial v_1}{\partial w^1 }\right|^2 \leq C \ \ \mbox{ in } \D^*_{w^1,\frac{1}{2},1} \times \D^*_{w^2}.
\end{equation}
Furthermore, using the harmonicity of  $v_1$  restricted to the slice $\{w^1_0\} \times \bar \D^*_{w^2}$,   we have by Theorem~\ref{exists} that
\begin{eqnarray} \label{nickel}
\Ej\log \frac{r_0}{r} \leq \int_{ \{w^1_0\}  \times \D_{r,r_0}} \left| \frac{\partial   {v_1}}{\partial w^2} \right|^2   \frac{dw^2 \wedge d\bar{w}^2}{-2i} & \leq &  C+\Ej\log \frac{r_0}{r}
\end{eqnarray}
for $0<r<r_0\leq \frac{1}{4}$.

\begin{lemma}[Derivative estimates a set of type (B) away from the juncture] \label{Bestaway}

For  $\Omega :=\D_{\frac{1}{4},1}$,  there exists a constant $C$  such that the following estimates hold:  
\begin{eqnarray*}
\left| \frac{\partial v}{\partial z^1}(z^1_0,z^2_0)  \right| & \leq &  C,  \ \ \forall (z^1_0,z^2_0) \in \Omega \times \D^*\\
\Ej\log \frac{r_0}{r} \leq \int_{ \{z^1_0\}  \times \D_{r,r_0}} \left| \frac{\partial v}{\partial z^2} \right|^2   \frac{dz^2 \wedge d\bar{z}^2}{-2i} & \leq &  C+\Ej\log \frac{r_0}{r}
\end{eqnarray*}
where $0<r<r_0\leq \frac{1}{4}$, $z_0^1 \in \Omega$ and $z^2=re^{i\theta}$.
\end{lemma}

\begin{proof}
Since $\D_{\frac{1}{4},1} \times \D_{\frac{1}{4}}^* \subset Z_1$, the estimates follow by applying the change of variables (\ref{coins}) to estimates (\ref{penny}), (\ref{penny1}), (\ref{nickel}) and noting that $\rho>\frac{1}{4}$ in $\Omega$.
\end{proof}

\begin{lemma}[Derivative estimates  in a set of type (B) near the juncture] \label{Bestnear}

For $v$ restricted to  $\D^*_{\frac{1}{4}} \times \D^*_{\frac{1}{4}}$,  there exists a constant $C$  such that the following estimates hold:
\begin{eqnarray*}
\Ej\log \frac{r_0}{r} \leq \int_{ \{z^1_0\}  \times \D_{r,r_0}} \left| \frac{\partial v}{\partial z^2} \right|^2   \frac{dz^2 \wedge d\bar{z}^2}{-2i} & \leq &  C+ \Ej\log \frac{r_0}{r},
\\
\Ei\log \frac{r_0}{r} \leq \int_{ \D_{r,r_0} \times \{z^2_0\} } \left| \frac{\partial v}{\partial z^1} \right|^2   \frac{dz^1 \wedge d\bar{z}^1}{-2i} & \leq &  C+ \Ei\log \frac{r_0}{r}
\end{eqnarray*}
for $z^1_0, z^2_0 \in \D^*_{\frac{1}{4}}$ and $0<r<r_0 \leq \frac{1}{4}$.
\end{lemma}
\begin{proof}
We only prove the first estimate, the second being similar.  First, note that the lower bound follows from the definition of $\Ej$.  Next, we estimate the upper bound.  For $z^1_0 \in \D^*_{\frac{1}{4}}$, we have the inclusion $\{z^1_0\} \times \D_{r,r_0} \subset Z_1$ whenever $r_0 \leq |z^1_0|$.  Thus, the change of variables $w^2 \mapsto z^2=|w^1|w^2$ in   (\ref{nickel}) yields
\[
\int_{ \{z^1_0\}  \times \D_{r,r_0}} \left| \frac{\partial {v}}{\partial z^2} \right|^2   \frac{dz^2 \wedge d\bar{z}^2}{-2i}  \leq   C+\Ej\log \frac{r_0}{r}
\]
for  $z^1_0 \in \D^*_{\frac{1}{4}}$ and  $0<r < r_0\leq |z_0^1|$.

If $r<|z_0^1|<r_0$, 
we break up the integral  into two integrals since the estimate for $\left| \frac{\partial v}{\partial z^2} \right|^2$ is different in $Z_2$ than in $Z_1$.  Indeed, by  an analogous argument to the proof of  (\ref{penny}),  
we have 
\[
\left| \frac{\partial v_2}{\partial w^2 }\right|^2  \leq  \frac{1}{r^2(-\log r)^{\frac{4}{3}} } +\frac{\Ej}{2\pi r^2}(1+r^a)^2 
\]
After a change of variables $z^1=|w^2|w^1, \ z^2=w^2$, 
\begin{equation} \label{Z2estimate}
\left| \frac{\partial v}{\partial z^2 }\right|^2  \leq \frac{1}{r^2(-\log r)^{\frac{4}{3}}} +\frac{\Ej}{2\pi r^2}(1+r^a)^2 
\ \ \mbox{ in } Z_2.
\end{equation}
Thus, we have
\begin{eqnarray*}
\lefteqn{\int_{ \{z^1_0\}  \times \D_{r,r_0}} \left| \frac{\partial v}{\partial z^2} \right|^2   \frac{dz^2 \wedge d\bar{z}^2}{-2i} }
\\
& \leq &  \int_0^{2\pi} \left( \int_r^{|z_0^1|}  \left| \frac{\partial v}{\partial z^2} \right|^2  \, rdr +\int_{|z_0^1|}^{r_0} \left| \frac{\partial v}{\partial z^2} \right|^2  \, rdr \right) d\theta 
\\
&  \leq & \left( C+\Ej\log \frac{|z_0^1|}{r} \right) + \int_0^{2\pi} \int_{|z_0^1|}^{r_0} \left( \frac{1}{r^2(-\log r)^{\frac{4}{3}} } + \frac{\Ej}{2\pi r^2}(1+r^a)^2 \right) \ rdr  d\theta
\\
& \leq & C +  \Ej\log \frac{r_0}{r}.
\end{eqnarray*}
Finally, if $|z_0^1|<r<r_0$, we only use the estimate (\ref{Z2estimate}).  We omit the details. 
\end{proof}

\subsection{Gluing the maps} \label{gluing}

Given the homomorphism $\rho:\pi_1(M) \rightarrow \mathsf{Isom}(\tilde X)$ of Theorem~\ref{theorem:pluriharmonic}, we will construct a prototype section $v:M \rightarrow \tilde M \times_\rho \tilde X$.  

Let $P \in \Sigma_j \cap \Sigma_i$ and $U = \bar\D_{z^1} \times \bar\D_{z^2}$ be a set of  of type (B) with 
\[
\bar\D_{z^1} \simeq \bar\D_{z^1} \times \{0\} \subset \Sigma_j \ \mbox{ and } \   \bar\D_{z^2} \simeq \{0\} \times \bar\D_{z^2} \subset \Sigma_i
\]
and 
\[
P=\{0\} \times \{0\}.
\]
The identification of the product  space
$
\bar\D_{z^1} \times \bar\D_{z^2}
$
as a subset of $M$ is simultaneously induced by  the local trivializations 
of the disk bundles $\pi_j:\bar{\mathcal D}_j \rightarrow \Sigma_j$ and  $\pi_i: \bar{\mathcal D}_i \rightarrow \Sigma_i$
via
\[
 \pi_j^{-1}(\bar\D_{z^1}) \simeq \bar\D_{z^1} \times \bar\D_{z^2}
  \ \mbox{ and } \   \pi_i^{-1}(\bar\D_{z^2}) \simeq \bar\D_{z^1} \times \bar\D_{z^2}  \]  
(cf.~(\ref{unitarytrivialization})).

Recall the following items associated with the set $U:=\D_{z^1} \times \D_{z^2}$.
\begin{itemize}

\item $[\Sp^1_{\theta^k}]$  is the element of $\pi_1(\bar \D^*_{z^k})$  generated by $\Sp^1_{\theta^k}$ for $k=1,2$

\item 
$[\Sp^1_{\theta^1}]$ and $[\Sp^1_{\theta^2}]$ also are the elements of $\pi_1(\bar{\D}^*_{z^1} \times \bar{\D}^*_{z^2})$ generated by $\Sp^1_{\theta^1}  \simeq \Sp^1_{\theta^1} \times \{z^2\}$ and $\Sp^1_{\theta^1} \simeq \{z^1\} \times \Sp^1_{\theta^2}$ respectively

\item $ \pi_1(\bar \D_{z^k}^*)$  is identified as a subgroup of  $\pi_1(\bar{\D}^*_{z^1} \times \bar{\D}^*_{z^2})$ for $k=1,2$ 
\item $\rho':\pi_1(\bar \D^*_{z^1} \times \bar \D^*_{z^1}) \rightarrow \mathsf{Isom}(\tilde X)$ is a homomorphism and $\rho'_k=\rho'|_{ \pi_1(\bar \D_{z^k}^*)}$ for $k=1,2$

\item $(\widetilde{\bar{\D}_{z^1}^*} \times \widetilde{\bar{\D}^*_{z^2}})  \times_{\rho'} \tilde X \rightarrow \bar{\D}_{z^1}^* \times \bar{\D}^*_{z^2}$ and $\widetilde{\bar{\D}_{z^k}^*}  \times_{\rho'_k} \tilde X \rightarrow \bar{\D}_{z^k}^*$ for $k=1,2$ are  fiber bundles.
\end{itemize}

Next, let 
$V  = \Omega  \times \bar\D_{ z_2}$ be a set of type (A) with 
\[
\Omega \simeq \Omega \times \{0\} \subset \Sigma_j.
\]
The identification $V \simeq \Omega \times \bar \D^*_{ z^2}$ is induced by the local trivialization 
\[
\pi_j^{-1}(\Omega) \simeq \Omega \times \bar \D_{z^2}
\]
of the bundle $\pi_j:\bar{\mathcal D}_j \rightarrow \Sigma_j$ (cf.~(\ref{ref:typeA})).
If $U \cap V \neq \emptyset$ (and hence
 $\Omega \cap \bar\D_{z^1} \neq \emptyset$), 
 then the transition function of the disk bundle $\pi_j: \bar{\mathcal D}_j \rightarrow \Sigma_j$ defines a smooth map
\[
\tau: \Omega \cap \bar{\D}_{z^1} \rightarrow U(1). 
\]

By \cite[Proposition 2.6.1]{korevaar-schoen1}, there exists a locally Lipschitz section $k: M \rightarrow \tilde M \times_\rho \tilde X$ of the fiber bundle $\tilde M \times_\rho \tilde X \rightarrow X$. 
Let $k_U$ be the lift to $(\widetilde{\bar{\D}_{z^1}^*} \times \widetilde{\bar{\D}^*_{z^2}})  \times_{\rho'} \tilde X$ of the restriction  
of $k$ to $U^* := \bar{\D}^*_{z^1} \times \bar{\D}^*_{z^2}$ and let 
\[
v_U: U^* := \bar{\D}^*_{z^1} \times \bar{\D}^*_{z^2} \rightarrow  (\widetilde{\bar{\D}^*_{z^1}} \times \widetilde{\bar{\D}^*_{z^2} } )\times_{\rho'} \tilde X
\]
be the local prototype section defined  in Subsection~\ref{NJ}.
The composition of $v_U$ and the quotient map 
$\tilde M \times_{\rho'} \tilde X \rightarrow  \tilde M \times_{\rho} \tilde X$ defines 
a section
 of $\tilde M \times_\rho \tilde X \rightarrow U$  which we call again $v_U$.

 Also
let $k_V$ be the lift to $(\Omega \times \widetilde{\bar \D^*_{z^2}}) \times_{\rho'_2} \tilde X$ of  the restriction  of $k$ to $V^*=\Omega \times \bar \D^*_{z^2}$ and let  
\[
v_V: V^* := \Omega \times \bar \D^*_{z^2} \rightarrow (\Omega \times \widetilde{\D^*_{z^2}}) \times_{\rho'_2} \tilde X
\]
be the local prototype section defined  in Subsection~\ref{AJ}.  The composition of $v_V$ and the quotient map 
$\tilde M \times_{\rho'_2} \tilde X \rightarrow  \tilde M \times_{\rho} \tilde X$ defines 
a section
 of $\tilde M \times_\rho \tilde X \rightarrow V$  which we call again $v_V$.
 
We claim that we can glue these local sections together to define $v$ in $U \cup V$. To do so, we have to show the following.
\begin{lemma} \label{vpdefined}
If $U$ and $V$ are sets of type (B) and (A) respectively, then the  sections $v_U$ and $v_V$  agree on $U^* \cap V^*$.  
\end{lemma}
\begin{proof}
For $p \in \bar \D_{z^1} \cap \Omega$, let $v_{U,p}$, $k_{U,p}$ be the restrictions of $v_U$, $k_U$ respectively to 
$\{p\} \times  \bar\D_{z^2}^* $ and $v_{V,p}$, $k_{V,p}$  be the restrictions of $v_V$, $k_V$ respectively to $\{p\} \times  \bar\D_{z^2}^* $.  We claim that the harmonic sections $v_{U,p}$ and $ v_{V,p}$ are related by the transition relation
\[
v_{U,p} = 
v_{V,p} \circ \tau(p).
\]
Indeed, since multiplication by $\tau(p) \in U(1)$ is a conformal map, 
$v_{V,p} \circ \tau(p)$ is  harmonic on $\bar \D^*$ with boundary values $k_{U,p}=k_{V,p} \circ \tau(p)\big|_{\partial \D \simeq \{p\} \times \Sp^1_{\theta} }$.  The assertion follows from the uniqueness of harmonic maps  (cf.~Theorem~\ref{exists}).
\end{proof}

 Similar construction holds for two sets of type (A).

\begin{lemma}
If  $U$ and $V$ are both of  type (A) and $U \cap V \neq \emptyset$, then $v_U$ and $v_V$  agree on $U^*\cap V^*$.
\end{lemma}

\begin{proof}
Apply the same argument as Lemma~\ref{vpdefined}.
\end{proof}

Let ${\mathcal U}$ be a finite open cover of $\bar{\mathcal D}$ by sets of type (A) and of type (B).  Let ${\mathcal U}_A \subset {\mathcal U}$ be the collection of sets of type (A) and ${\mathcal U}_B \subset {\mathcal U}$ be the collections of sets of type  (B).   Without the loss of generality, we can assume ${\mathcal U}_B$ is a collection of disjoint sets.  Set
\begin{equation}\label{join}
v = \left\{ 
\begin{array}{ll}
v_U & \mbox{in } U \cap {\mathcal D}^*_{\frac{1}{4}} \mbox{ where }U \in {\mathcal U}_B  \\
v_V & \mbox{in } (V \cap {\mathcal D}^*_{\frac{1}{4}}) \backslash \bigcup_{U \in {\mathcal U}_B} U \mbox{ where }V \in {\mathcal U}_A
\end{array}
\right.
\end{equation}
and extend to the rest of $M$ as a well-defined, locally Lipschitz global section of $\tilde M \times_\rho \tilde X \rightarrow M$.

\begin{definition}  \label{prototypemap}
The  map 
\begin{equation} \label{mochizukimap}
v:  M \rightarrow \tilde M \times \tilde X
\end{equation}
constructed above is called the \emph{prototype section}.
The corresponding $\rho$-equivariant map $\tilde{v}:\tilde{M} \rightarrow \tilde{X}$  is called the {\it prototype map}. 
\end{definition}

\section{Energy estimates of the prototype section} \label{sec:ptm}

The goal of this section is to obtain  the energy estimates  of the prototype section  \[
v:  M \rightarrow \tilde{M} \times_{\rho} \tilde{X}
\] 
of Definition~\ref{prototypemap} with respect to the Poincar\'e-type metric $g$ given by Definition~\ref{poincarekahler}.
Throughout this section, we use $C$ to denote constants that are independent of the distance to the divisor. (Note that $C$ may change from line to line.)  

We consider the following three types of sets intersecting the divisor $\Sigma \subset \bar M$: 
\begin{itemize}
\item[(i)] $\Omega \times \bar{\D}_{\frac{1}{4}} \subset \Omega \times \bar \D$, a subset in a set  of type (A)
\item[(ii)] $\Omega \times  \bar \D_{\frac{1}{4}} :=\D_{\frac{1}{4},1}  \times \bar \D_{\frac{1}{4}}  \subset \bar{\D} \times\bar{\D}$, a subset in a set of type (B) away from the crossing (cf.~Figure~\ref{fig:subsetsB})
\item[(iii)] $\bar{\D}_{\frac{1}{4}} \times\bar{\D}_{\frac{1}{4}} \subset \bar{\D} \times\bar{\D}$, a subset in a set of type (B) at the crossing (cf.~Figure~\ref{fig:subsetsB})
 \end{itemize}
 \begin{figure}[h!]
  \includegraphics[width=2in]{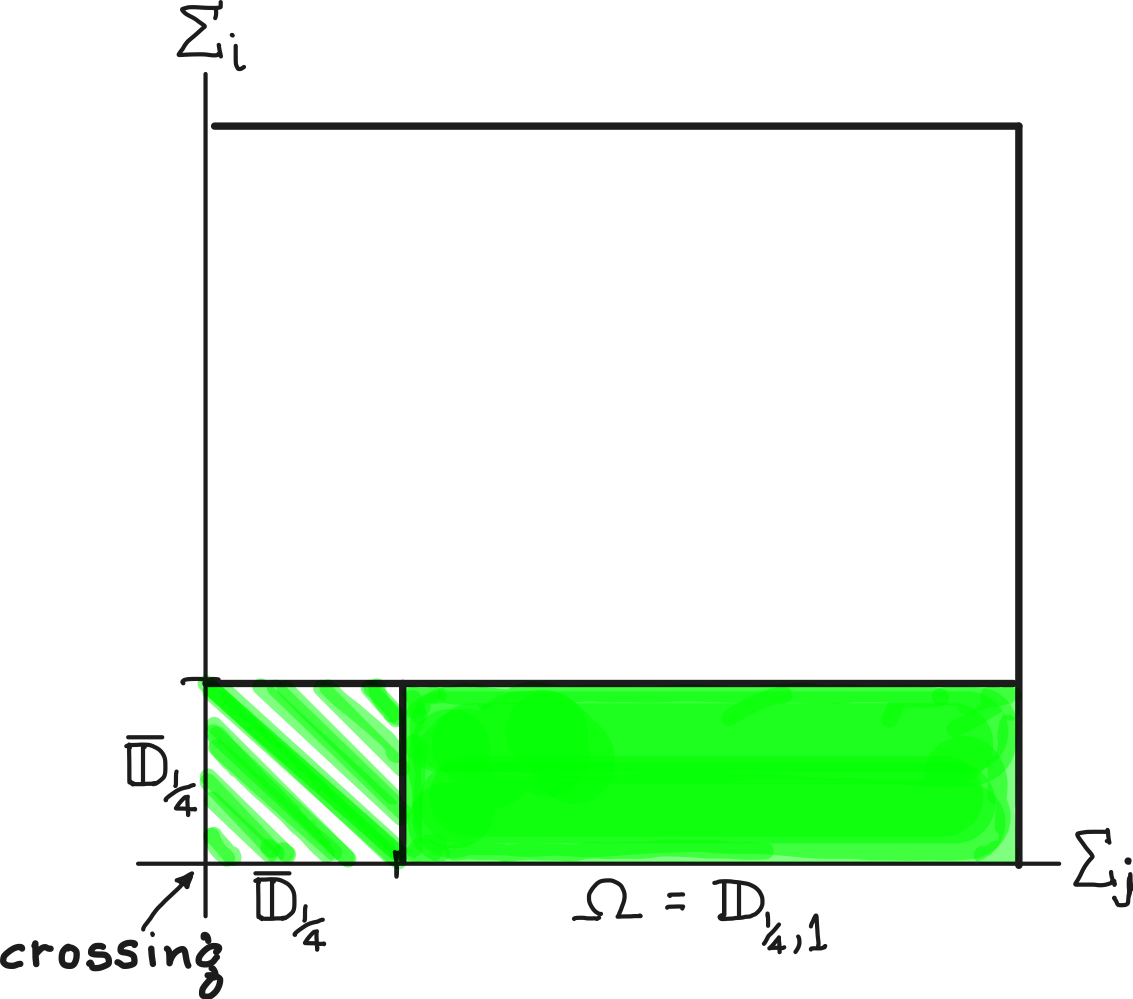}
  \caption{Subsets of a set of type (B). }
  \label{fig:subsetsB}
\end{figure}
A neighborhood of $\Sigma$ can be covered by a {\it{finite collection of sets}} of the above type.  In order to estimate the energy of $v$, we will compute its energy in each such set.

\subsection{Energy in a set of type (A)} \label{subsec:EA}
In this subsection we will use the following  notation in addition to the one used in Subsection~\ref{NY}.

\begin{itemize}

\item $\Omega \times  \D_{r_1,r_2}$ is the subset of  $\Omega \times \bar \D$ with $0<r_1<|z^2| <r_2 < \frac{1}{4}$.

\item
 $g_{\Sigma_j}$ is the smooth metric on $\Omega$ as defined in Definition~\ref{(j)}

 \item $\mbox{Area}_{g_{\Sigma_j}}$ is  the area  with respect to  $g_{\Sigma_j}$
 
  \item $P$ is the product metric on $\Omega \times \bar \D^*$ defined by (\ref{def:P}).
 \end{itemize}
  Note that
  \[
  {\mathcal D}_{r_1,r_2} \cap (\Omega \times \bar \D) = \Omega \times \D_{r_1,r_2} \ \  \mbox{(cf.~(\ref{nbhd12}))}.
  \]

The strategy for estimating the energy of the prototype section $v$ in $\Omega \times \bar \D$ will be to first compute the energy of $v$ with respect to the product metric $P$ (cf.~(\ref{def:P})).  Since the metrics $P$ and $g$ are close (cf.~(\ref{closegP})), this will give us the estimate of the energy of $v$ with respect to $g$.

  \begin{lemma}\label{lemmadirvA}
For  a subset $\Omega \times \bar{\D}$ of a set of  type (A),  there exists a constant $C>0$ such that  the energy with respect to the metric $P$ of the prototype section $v$ satisfies
 \begin{eqnarray*}
0 \leq {^PE}^v[\Omega \times \D_{r_1,r_2}] -  \Ej\mbox{Area}_{g_{\Sigma_j}}(\Omega)  \log \frac{r_2}{r_1} \leq  C,  \ \ \ \ 0<r_1<r_2\leq \frac{1}{4}.
 \end{eqnarray*}
  \end{lemma}
  
   \begin{proof}
By  (\ref{volumeformofP}),
 \begin{eqnarray*}
 ^PE^v[\Omega \times \D_{r_1,r_2}] & = & 
 \int_{\Omega \times  \D_{r_1,r_2}} P^{1\bar1}\left| \frac{\partial v}{\partial z^1} \right|^2  +P^{2\bar2} \left| \frac{\partial v}{\partial z^2} \right|^2    d\mbox{vol}_P
 \\
 &  = &   \int_{\Omega \times  \D_{r_1,r_2}} P^{1\bar1}\left| \frac{\partial v}{\partial z^1} \right|^2   d\mbox{vol}_P
 + 
   \int_{\Omega}
   \left(
   \int_{\D_{r_1,r_2}}
  \left| \frac{\partial v}{\partial z^2} \right|^2 \frac{dz^2 \wedge d\bar z^2}{-2i} 
  \right)  d\mbox{vol}_{g_{\Sigma_j}},
\end{eqnarray*}
hence the inequality on the right follows  from Lemma~\ref{Aest} (cf~Remark~\ref{alike}).

By the definition of $\Ej$,
\begin{eqnarray*}
\Ej\mbox{Area}_{g_{\Sigma_j}}(\Omega)  \log \frac{r_2}{r_1} 
 & = &  \int_{\Omega} \int_{r_1}^{r_2} \Ej \frac{dr}{r}
d\mbox{vol}_{g_{\Sigma_j}}
\\
& \leq &  \int_{\Omega} \left(  \int_{ \D_{r_1,r_2}} \left| \frac{\partial v}{\partial \theta} \right|^2    \frac{dr \wedge d\theta}{r} \right)
d\mbox{vol}_{g_{\Sigma_j}} \\
& \leq & 
\int_{\Omega  \times \D_{r_1,r_2}} 
P^{1\bar 1} \left| \frac{\partial v}{\partial z^1} \right|^2 +P^{2\bar 2} \left| \frac{\partial v}{\partial z^2} \right|^2
d\mbox{vol}_P
\\
& = & ^PE^v[\Omega \times \D_{r_1,r_2}],
\end{eqnarray*}
which proves the inequality on the left.
  \end{proof}

\begin{lemma} \label{egrowthA}
For  a subset $\Omega \times \bar{\D}$ of a set  type (A),  there exists a constant $C>0$ such that  the prototype section $v$ satisfies

\begin{eqnarray*}
\left|
 ^gE^v[\Omega \times \D_{r_1,r_2}] -  \Ej\mbox{Area}_{g_{\Sigma_j}}(\Omega)  \log \frac{r_2}{r_1} \right| \leq  C, \  \ 0<r_1<r_2 \leq \frac{1}{4}.
\end{eqnarray*}
\end{lemma}

\begin{proof}
By Lemma~\ref{lemmadirvA}, it suffices to show that the  difference of the energy of $v$ with respect to $g$ and with respect to  $P$  is bounded; i.e.~
\[
\left| ^gE^v[\Omega \times \D_{r_1,r_2}]-{^PE}^v[\Omega \times \D_{r_1,r_2}] \right| \leq C.
\]
  We obtain the above estimate with the help of  Lemma~\ref{calculus} found in the Appendix to this chapter. Therefore we need to first show  that the assumption (\ref{alb}) of Lemma~\ref{calculus} is satisfied; in other words, we need an estimate  of the intergral of  $\frac{1}{r^2}\left| \frac{\partial v}{\partial \theta} \right|^2$.  Below, we will derive the estimate (\ref{alb}) by bounding the $z^1$-energy and $r$-energy of $v$ and then subtracting those from the full energy estimate of Lemma~\ref{lemmadirvA}. 

First, to bound the $z^1$-energy of $v$, we use  estimate  $ \left| \frac{\partial v}{\partial z^1} \right|^2 \leq C$ of (\ref{Aest}) to see that 
\[
\int_{\Omega  \times \D_{r_1,r_2}} 
  \left| \frac{\partial v}{\partial z^1} \right|^2  d\mbox{vol}_{g_{\Sigma_j}}  \wedge \frac{dz^2\wedge d\bar{z}^2}{-2i r^2(\log r^2+A)^2}  \leq C.
\]
Second, to bound the $r$-energy of $v$,  
use estimate $\left| \frac{\partial v}{\partial r} \right|^2 \leq \frac{C}{r^2 (-\log r)}$ from Theorem~\ref{exists} to see that 
\[
 \int_{\Omega  \times \D_{r_1,r_2}}   \left| \frac{\partial v}{\partial r} \right|^2 
d\mbox{vol}_{g_{\Sigma_j}} \wedge \frac{dz^2\wedge d\bar{z}^2}{-2i } \leq C.
    \]
    Thus, Lemma~\ref{lemmadirvA} and the identities for $P^{1\bar 1} d\mbox{vol}_P$ and $    P^{2\bar 2} d\mbox{vol}_P$ given by (\ref{volumeformofP}) imply the following integral estimate on $\frac{1}{r^2} \left| \frac{\partial v}{\partial \theta} \right|^2$:
   \[
   \int_{\Omega \times \D_{r_1,r_2}}
 \frac{1}{r^2} \left| \frac{\partial v}{\partial \theta} \right|^2 
d\mbox{vol}_{g_{\Sigma_j}}  \wedge \frac{dz^2\wedge d\bar{z}^2}{-2i }  -\Ej\mbox{Area}_{g_{\Sigma_j}}(\Omega)  \log \frac{r_2}{r_1} \leq C. 
   \]
We set  $r_2=\frac{1}{4}$ and let $r_1 \rightarrow 0$ above to obtain
\[
 \int_{\Omega \times \D_{\frac{1}{4}}}
\left(   \left| \frac{\partial v}{\partial \theta} \right|^2 -\Ej
\right) 
d\mbox{vol}_{g_{\Sigma_j}} \wedge \frac{dr \wedge d\theta}{r} 
\leq  C.
\]
Noting that $g_{\Sigma_j}$ is a smooth metric, the  assumption (\ref{alb}) of Lemma~\ref{calculus} is satsified.  Consequently (noting that $A$ appears in the metric expression of $P$  is a bounded function), 
\begin{equation} \label{conlem}
 \int_{\Omega \times \D_{\frac{1}{4}}}
  \left| \frac{\partial v}{\partial \theta} \right|^2 
d\mbox{vol}_{g_{\Sigma_j}} \wedge \frac{dr \wedge d\theta}{r(\log r^2+A)^2} 
\leq C'
\end{equation}
where the constant $C'$ depends only on $C$.

We use the above estimate to  compute the difference between $^gE^v[\Omega \times \D_{r_1,r_2}]$ and $^PE^v[\Omega \times \D_{r_1,r_2}]$. The trickiest to bound include the following  two terms for which we use the  estimate (\ref{conlem}):  \\
\\
$\bullet$ $\displaystyle{\left| \int_{\Omega \times \D_{\frac{1}{4}}} P^{\theta \theta}  
  \left| \frac{\partial v}{\partial \theta} \right|^2  (d\mbox{vol}_g  - d\mbox{vol}_P) \right|
}$\\
\\
To bound this term, note that  since by (\ref{volcomp})
\[
d\mbox{vol}_P - d\mbox{vol}_g=O\left(\frac{1}{(-\log r^2 +A)^2} \right) d\mbox{vol}_P=O(1) d\mbox{vol}_{g_{\Sigma_j}}  \wedge \frac{dz^2\wedge d\bar{z}^2}{-2i r^2(-\log r^2+A)^4}\]
 and 
 \[
 P^{2\bar 2} = r^2( \log r^2+A)^2.
 \]
 and thus
 \[
 P^{2\bar 2} (d\mbox{vol}_P - d\mbox{vol}_g)=O(1) d\mbox{vol}_{g_{\Sigma_j}}  \wedge \frac{dz^2\wedge d\bar{z}^2}{-2i (-\log r^2+A)^2}.
 \]
Thus, 
\[
P^{\theta \theta}   \left( d\mbox{vol}_P - d\mbox{vol}_g \right)=d\mbox{vol}_{g_{\Sigma_j}} \wedge   \frac{dz^2 \wedge d\bar z^2}{-2ir^2(-\log r^2 +A)^2}= d\mbox{vol}_{g_{\Sigma_j}} \wedge \frac{ dr \wedge d\theta}{r(-\log r^2 +A)^2}
\]
which in turn implies
\begin{eqnarray*}
\left| \int_{\Omega \times \D_{\frac{1}{4}}} P^{\theta \theta}  
  \left| \frac{\partial v}{\partial \theta} \right|^2  (d\mbox{vol}_g  - d\mbox{vol}_P) \right|
  & \leq & 
C\int_{\Omega \times \D_{\frac{1}{4}}}   \left| \frac{\partial v}{\partial \theta} \right|^2 d\mbox{vol}_{g_{\Sigma_j}} \wedge \frac{ dr \wedge d\theta}{r(-\log r^2 +A)^2} \\
  & \leq & CC'.
\end{eqnarray*}

\noindent $\bullet$ $\displaystyle{ \left| \int_{\Omega \times \D_{\frac{1}{4}}} \left( g^{\theta \theta} - P^{\theta \theta}  \right)
  \left| \frac{\partial v}{\partial \theta} \right|^2  d\mbox{vol}_P   \right|}$\\
  \\
To bound this term, note that since $g^{2\bar2}-P^{2\bar 2}=O(r^2)$, we have
\[
g^{\theta \theta} - P^{\theta \theta}=O(1) \mbox{ and }
  d\mbox{vol}_P = d\mbox{vol}_{g_{\Sigma_j}}   \wedge \frac{dz^2\wedge d\bar{z}^2}{-2i r^2(\log r^2+A)^2}
  \] by (\ref{volumeformofP}).  Combining the above,  we obtain
\[
\left( g^{\theta \theta} - P^{\theta \theta} \right) d\mbox{vol}_P =O(1) d\mbox{vol}_{g_{\Sigma_j}} \wedge \frac{dr\wedge d\theta}{r(-\log r^2+A)^2}.
\]  Thus
\begin{eqnarray*}
\left| \int_{\Omega \times \D_{\frac{1}{4}}} \left( g^{\theta \theta} - P^{\theta \theta}  \right)
  \left| \frac{\partial v}{\partial \theta} \right|^2  d\mbox{vol}_P   \right|
  &  \leq & C \int_{\Omega \times \D_{\frac{1}{4}}}
  \left| \frac{\partial v}{\partial \theta} \right|^2 
d\mbox{vol}_{g_{\Sigma_j}}\wedge \frac{dr \wedge d\theta}{r(-\log r^2+A)^2} 
\leq CC'.
\end{eqnarray*}
The other terms of $\left| ^gE^v[\Omega \times \D_{r_1,r_2}]-{^PE}^v[\Omega \times \D_{r_1,r_2}] \right|$ are also bounded by similar computations.  We omit the details.
  \end{proof}

\begin{lemma} \label{comparisonA}
For  a subset $\Omega \times \bar{\D}$ of  a set type (A),  there exists a constant $C>0$ such that  the prototype section $v$ satisfies
\[
^gE^v[\Omega \times \D_{r_1,r_2}]
\leq {^gE^f}[\Omega \times \D_{r_1,r_2}]+C, \ \ 0<r_1<r_2 \leq \frac{1}{4}
\]
for  any locally Lipschitz section $f:\Omega  \times \D_{r_1,r_2} \rightarrow \tilde M \times_{\rho} \tilde X $.
\end{lemma}

\begin{proof}
Since $\gamma_j$ is minimizing,
\begin{eqnarray}
 \Ej\mbox{Area}_{g_{\Sigma_j}}(\Omega)  \log \frac{r_2}{r_1} 
\nonumber 
 & \leq & \int_0^{2\pi} \left| \frac{\partial f}{\partial \theta} \right|^2 d\theta \int_\Omega d\mbox{vol}_{g_{\Sigma_j}} \int_{r_1}^{r_2} \frac{dr}{r}
 \\
 & = & 
 \int_{\Omega  \times \D_{r_1,r_2}} \left| \frac{\partial f}{\partial \theta} \right|^2  d\mbox{vol}_{g_{\Sigma_j}}  \wedge \frac{dr \wedge d\theta}{r}
\nonumber \\
& \leq & 
\int_{\Omega  \times \D_{r_1,r_2}} 
P^{1\bar 1} \left| \frac{\partial f}{\partial z^1} \right|^2 +P^{2\bar 2} \left| \frac{\partial f}{\partial z^2} \right|^2 \
d\mbox{vol}_P
\nonumber \\
& = & ^PE^f[\Omega \times \D_{r_1,r_2}].
\label{piggy}
\end{eqnarray}
Thus, the desired estimate with $g$ replaced by $P$ follows from combining the above estimate with Lemma~\ref{lemmadirvA}.
Thus, we are left to show that 
\[
\left| {^gE^f}[\Omega \times \D_{r_1,r_2}] -{^PE^f}[\Omega \times \D_{r_1,r_2}]\right|
\leq C.
\]
To do so, note that if the inequality
\[
 \int_{\Omega \times \bar \D_{\frac{1}{4}} }
\left(   \left| \frac{\partial f}{\partial \theta} \right|^2 -\frac{\Ej}{2\pi}
\right) 
d\mbox{vol}_{g_{\Sigma_j}}\wedge \frac{dr \wedge d\theta}{r} 
<\infty
\]
does not hold, then we are done since the desired estimate  holds by Lemma~\ref{egrowthA}.  Hence, we can  assume the above inequality and apply Lemma~\ref{calculus}  to conclude 
\[
 \int_{\Omega \times \D_{\frac{1}{4}}}
  \left| \frac{\partial f}{\partial \theta} \right|^2 
d\mbox{vol}_{g_{\Sigma_j}} \wedge \frac{dr \wedge d\theta}{r(\log r^2+A)^2} 
\leq C'
\]
where the constant $C'$ depends only on $C$. 
The rest of the proof is exactly as in the proof of Lemma~\ref{lemmadirvA}. 
\end{proof}

\subsection{Energy in a set of Type (B) away from the crossing}
In this subsection we will use the following   notation in addition to the one used in Subsection~\ref{NJ}.

\begin{itemize}
\item $\Omega:=\D_{\frac{1}{4},1}$

\item 
$\Omega \times  \D_{r_1,r_2}$ is the subset of  $\Omega \times \bar \D=\D_{\frac{1}{4},1} \times \bar \D$ with $0<r_1<|z^2| <r_2 < \frac{1}{4}$.

\item $g_{\Sigma_j}$ is the smooth metric on $\Omega$  as in Definition~\ref{(j)}

\item $\mbox{Area}_{g_{\Sigma_j}}$ is the area  with respect to  $g_{\Sigma_j}$

\item $P$ is the product metric on $\bar \D \times \bar \D$ defined by (\ref{productmetricB}).
\end{itemize}

Since $\Omega:=\D_{\frac{1}{4},1} \subset \D$, the points of $\Omega \times \D_{r_1,r_2}$ are uniformly away from the crossing.  In particular, since 
\[
\frac{1}{4}<\rho 
\mbox{ in $\Omega \times  \D_{r_1,r_2}$},
\]
 the metric expressions of $g$ in a set of type (A) and of type (B) (cf.~(\ref{ginA}) and (\ref{ginB}) respectively) show  that $g$ restricted to $\Omega \times  \D_{r_1,r_2} $ in set type (B)  has   the same asymptotic behavior as $r \rightarrow 0$ as $g$ in a set of type (A).   Thus, the procedure for obtaining  energy estimates of $v$ will be analogous to that in the previous subsection.
   Note that 
   \[
   {\mathcal D}_{r_1,r_2} \cap (\Omega \times \bar {\D})) = \Omega  \times \D_{r_1,r_2} \mbox{ for $0<r_1<r_2<\frac{1}{4}$ (cf.~(\ref{nbhd12}))}.
   \]

  \begin{lemma}  \label{lemmadirvB1} 
For  a subset $\Omega \times  \D_{r_1,r_2}$ of a set  of  type (B) (with $\Omega=\D_{\frac{1}{4},1}$),  there exists a constant $C>0$ such that  the prototype section $v$ satisfies
 \begin{eqnarray*}
^PE^v[\Omega  \times \D_{r_1,r_2}]\leq \Ej\mbox{Area}_{g_{\Sigma_j}}(\Omega)  \log \frac{r_2}{r_1} + C,  \ \ \ \ 0<r_1<r_2\leq \frac{1}{4}.
 \end{eqnarray*}
  \end{lemma}
  
  \begin{proof}
   Follow the proof of Lemma~\ref{lemmadirvA} but by replacing (\ref{volumeformofP}) by (\ref{volumeprodB}) and Lemma~\ref{Aest} by Lemma~\ref{Bestaway}.  (Note that $\frac{1}{4}<\rho$, i.e.~$\rho$ is bounded away from 0, so the expressions in  (\ref{volumeformofP}) and (\ref{volumeprodB}) are comparable.)
  \end{proof}

\begin{lemma} \label{egrowthB1}
For  a subset $\Omega \times  \D_{r_1,r_2}$ of a set  of  type (B) (with $\Omega=\D_{\frac{1}{4},1}$),  there exists a constant $C>0$ such that  the prototype section $v$ satisfies
\begin{eqnarray*}
\left|
 ^gE^v[\Omega  \times \D_{r_1,r_2}]-  \Ej\mbox{Area}_{g_{\Sigma_j}}(\Omega)  \log \frac{r_2}{r_1} \right| \leq  C, \  \ 0<r_1<r_2 \leq \frac{1}{4}.
\end{eqnarray*}
\end{lemma}

\begin{proof}
Follow the proof of Lemma~\ref{egrowthA} using Lemma~\ref{lemmadirvB1} instead of Lemma~\ref{lemmadirvA}.
\end{proof}

    \begin{lemma} \label{comparisonB1}
For  a subset $\Omega \times  \D_{r_1,r_2}$ of a set  of  type (B) (with $\Omega=\D_{\frac{1}{4},1}$),  there exists a constant $C>0$ such that  the prototype section $v$ satisfies
\[
^gE^v[\Omega  \times \D_{r_1,r_2}]
\leq {^gE^f}[\Omega  \times \D_{r_1,r_2}]+C, \ \ 0<r_1<r_2 \leq \frac{1}{4}
\]
for  any locally Lipschitz section $f:\Omega  \times \D_{r_1,r_2} \rightarrow \tilde M \times_{\rho} \tilde X $.
\end{lemma}

\begin{proof}
Follow the proof of Lemma~\ref{comparisonA} using Lemma~\ref{egrowthB1} instead of Lemma~\ref{egrowthA}
\end{proof}

\subsection{Energy in a set of Type $\mbox{(B)}$ at the  crossing}  
In this subsection we will use the following notation in  addition the one used in Subsection~\ref{NJ}.

\begin{itemize}

\item $\D_{r_1,r_2} \times \D_{r_1,r_2}$ is the subset of $\bar \D \times \bar \D$ with $0<r_1<|z^k| <r_2 <\frac{1}{4}$ for $k=1,2$.

\item
$g_{\Sigma_j}$, $g_{\Sigma_i}$ are as in Definition~\ref{(j)}

\item  $\mbox{Area}_{g_{\Sigma_j}}$ and $\mbox{Area}_{g_{\Sigma_i}}$ are the areas with respect to   $g_{\Sigma_i}$

\item $P$ is the product metric on $\bar \D \times \bar \D$ defined by (\ref{productmetricB}).

\end{itemize}

The goal is to  estimate the energy of $v$  in the set 
\[
U_{r_1,r_2} = (\D_{r_1,\frac{1}{4}}  \times \D_{r_1,\frac{1}{4}}) \backslash (\D_{r_2,\frac{1}{4}}  \times \D_{r_2, \frac{1}{4}}),
\]
 pictured in Figure~\ref{fig:Ur1r2}.
\begin{figure}[h!]
  \includegraphics[width=2in]{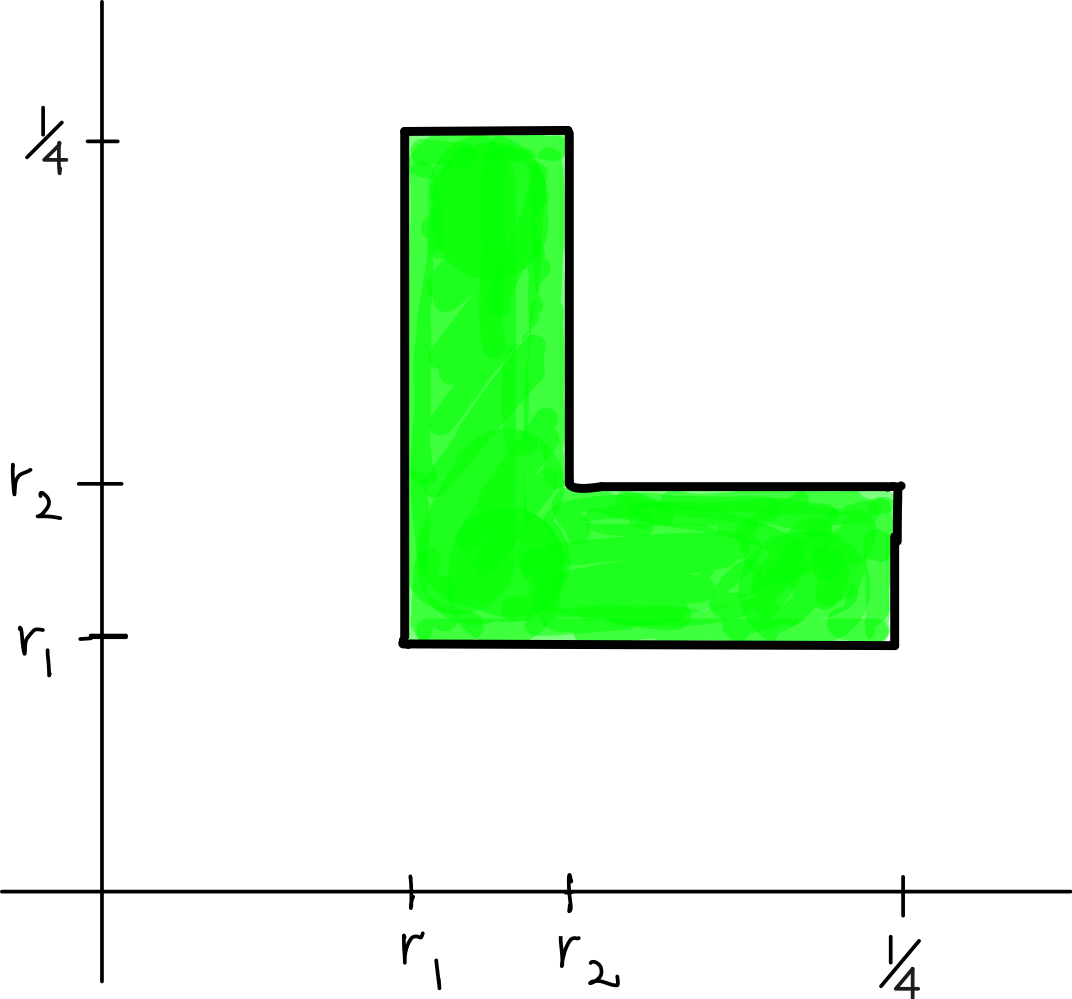}
  \caption{The region  $U_{r_1,r_2}$.}
  \label{fig:Ur1r2}
\end{figure}
The procedure for doing so involves an extra step compared to the procedure in  the previous two subsections.  Namely, we will first derive an expression for the energy with respect to the product metric $P$ in the box  $\D_{r,\frac{1}{4}}  \times \D_{r,\frac{1}{4}}$ pictured in Figure~\ref{fig:Boxr1r2}  (cf. Lemma~\ref{lemmadirvB2} below).
\begin{figure}[h!]
  \includegraphics[width=2in]{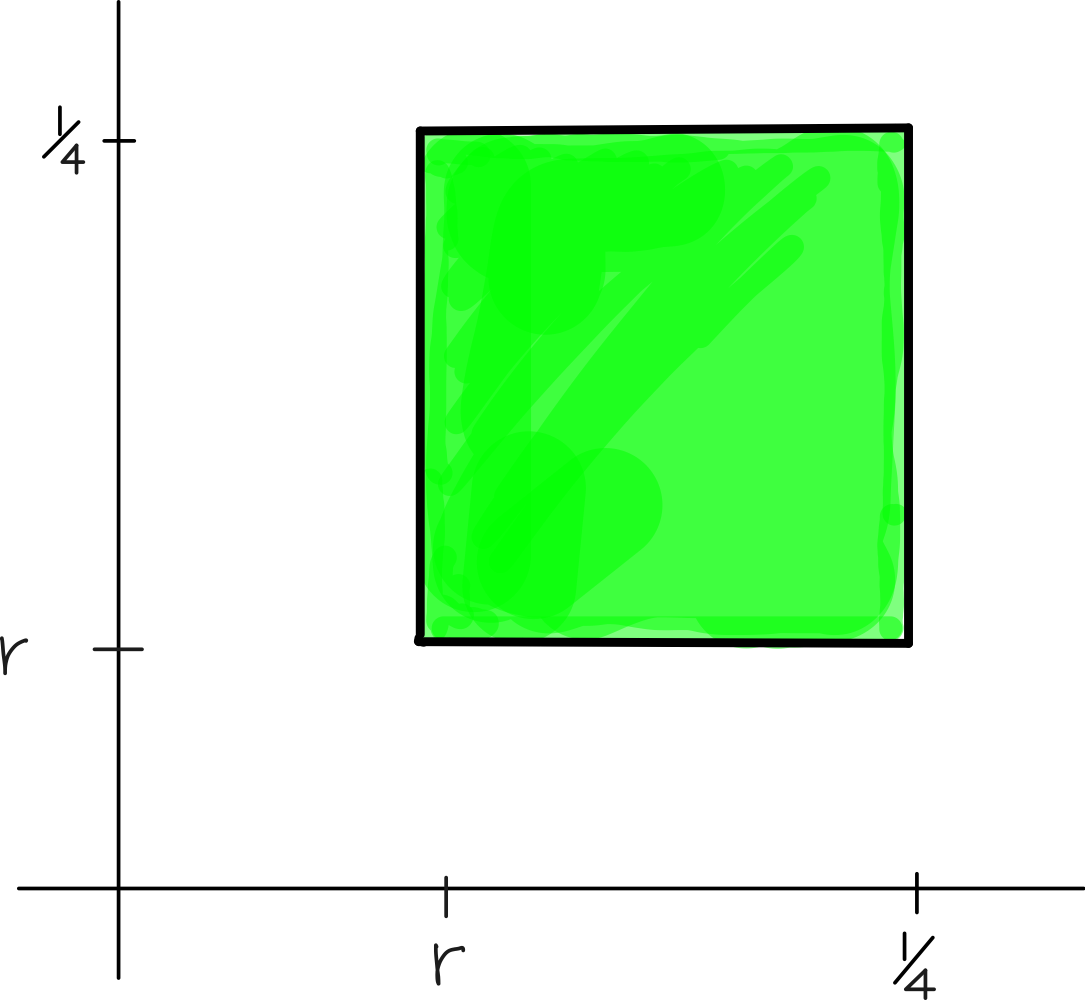}
  \caption{The region  $\D_{r,\frac{1}{4}}  \times \D_{r,\frac{1}{4}}$.}
  \label{fig:Boxr1r2}
\end{figure}
Then we take the difference of the energy with respect to $P$ contained in $\D_{r_1,\frac{1}{4}}  \times \D_{r_1,\frac{1}{4}}$ and in  $\D_{r_2,\frac{1}{4}}  \times \D_{r_2, \frac{1}{4}}$ to   bound the energy   in $U_{r_1,r_2}$  (cf.~Lemma~\ref{arealemma} below).  
Finally, since the difference between the metric $g$ and $P$ is small,  we obtain a  bound for the energy with respect to $g$ in $U_{r_1,r_2}$ (cf.~Lemma~\ref{egrowthB2} below).

\begin{definition}
We define $P_{\Sigma_i}$, $P_{\Sigma_j}$ to be the restriction of $P$ (cf.~(\ref{productmetricB})) to $\Sigma_i$, $\Sigma_j$ respectively.  
\end{definition}

\begin{remark} \label{areasymmetry}
Note that $\mbox{Area}_{P_{\Sigma_i}}(\D_{r,r'})=\mbox{Area}_{P_{\Sigma_j}}(\D_{r,r'})$ for $0<r<r' \leq \frac{1}{4}$ by the symmetry of $P$.
\end{remark}

\begin{lemma} \label{lemmadirvB2}
In a subset $\bar\D_{\frac{1}{4}} \times \bar\D_{\frac{1}{4}}$ of a set  $\D^* \times \D^*$ of type (B),   the prototype section $v$  satisfies
  \[
 \left|  ^PE^v[\D_{r,\frac{1}{4}}  \times \D_{r,\frac{1}{4}}] - \left(  \Ej\mbox{Area}_{P_{\Sigma_j}}(\D_{r,\frac{1}{4}})  \log  {\frac{1}{4r}} +\Ei\mbox{Area}_{P_{\Sigma_i}}(\D_{r,\frac{1}{4}})  \log  {\frac{1}{4r}} \right) \right| \leq C
  \]
for $0<r<\frac{1}{4}$.
\end{lemma}

\begin{proof}
 By Lemma~\ref{Bestnear},
\begin{eqnarray*}
\Ei\log \frac{r_0}{r} \leq \int_{ \D_{r,r_0} \times \{z^2_0\} } \left| \frac{\partial v}{\partial z^1} \right|^2   \frac{dz^1 \wedge d\bar{z}^1}{-2i} & \leq &  C+\Ei\log \frac{r_0}{r}
\end{eqnarray*}
for $z^2_0 \in \D^*_{\frac{1}{4}}$ and $0<r<r_0$.   By (\ref{volumeprodB}),
\begin{eqnarray*}
 \int_{ \D_{r,\frac{1}{4}} \times \D_{r,\frac{1}{4}}}  
P^{1\bar 1} \left|\frac{\partial v}{\partial z^1} \right|^2  d\mbox{vol}_P 
& = &   \int_{ \D_{r,\frac{1}{4}}}   \left(  \int_{ \D_{r,\frac{1}{4}}}  
 \left|\frac{\partial v}{\partial z^1} \right|^2 \frac{dz^1\wedge d\bar{z}^1}{-2i } 
 \right) \frac{dz^2\wedge d\bar{z}^2}{-2i r^2(\log r^2)^2}.
 \end{eqnarray*}
Combining the above, we obtain
 \begin{eqnarray*}
\Ei \mbox{Area}_{P_{\Sigma_i}}(\D_{r,\frac{1}{4}})\log  {\frac{1}{4r}} \leq  \int_{ \D_{r,\frac{1}{4}} \times \D_{r,\frac{1}{4}}}  
P^{1\bar 1} \left|\frac{\partial v}{\partial z^1} \right|^2  d\mbox{vol}_P    \leq   \Ei \mbox{Area}_{P_{\Sigma_i}}(\D_{r,\frac{1}{4}})\log  {\frac{1}{4r}} +C
\end{eqnarray*}
and similarly
\[
\Ej \mbox{Area}_{P_{\Sigma_j}}(\D_{r,\frac{1}{4}})\log  {\frac{1}{4r}} \leq  \int_{ \D_{r,\frac{1}{4}} \times \D_{r,\frac{1}{4}}}  
P^{2\bar 2} \left| \frac{\partial v}{\partial z^2} \right|^2 d\mbox{vol}_P 
\leq \Ej\mbox{Area}_{P_{\Sigma_i}}(\D_{r,\frac{1}{4}}) \log  {\frac{1}{4r}}+C.
\]
\end{proof}

\begin{lemma} \label{arealemma}
In a subset $\bar\D_{\frac{1}{4}} \times \bar\D_{\frac{1}{4}}$ of a set  of type (B), there exists a constant $C>0$ such that  in any 
 subset (cf.~(\ref{nbhd12}))
\[
U_{r_1,r_2} ={\mathcal D}_{r_1,r_2} \cap (\bar\D_{\frac{1}{4}} \times \bar\D_{\frac{1}{4}}),   \ \ \ 0<r_1<r_2\leq\frac{1}{4},
\]
the prototype section $v$ satisfies
\begin{eqnarray*}
0 \leq {^PE^v[}U_{r_1,r_2}]
-\Ej  \mbox{Area}_{P_{\Sigma_j}} (\D^*_{\frac{1}{4}}) \log \frac{r_2}{r_1} - \Ei  \mbox{Area}_{P_{\Sigma_i}}(\D^*_{\frac{1}{4}}) \log \frac{r_2}{r_1}  \leq C.
\end{eqnarray*}
\end{lemma}

\begin{proof}
By a straightforward computation,
\begin{eqnarray} \label{itoarea1}
\lefteqn{ 
 \mbox{Area}_{P_{\Sigma_j}} (\D_{r_1,\frac{1}{4}}) \log \frac{1}{4r_1}
-
  \mbox{Area}_{P_{\Sigma_j}} (\D_{r_2,\frac{1}{4}}) \log \frac{1}{4r_2}} 
  \nonumber \\
& = &   \mbox{Area}_{P_{\Sigma_j}} (\D_{r_1,\frac{1}{4}}) \log \frac{1}{4r_1}
-
  \mbox{Area}_{P_{\Sigma_j}} (\D_{r_1,\frac{1}{4}}) \log \frac{1}{4r_2}
    \nonumber \\
 & & +  \mbox{Area}_{P_{\Sigma_j}} (\D_{r_1,\frac{1}{4}}) \log \frac{1}{4r_2}
-  \mbox{Area}_{P_{\Sigma_j}} (\D_{r_2,\frac{1}{4}}) \log \frac{1}{4r_2}
 \nonumber  \\
& = &  \mbox{Area}_{P_{\Sigma_j}} (\D_{r_1,\frac{1}{4}})\log \frac{r_2}{r_1} 
 +  \left( \mbox{Area}_{P_{\Sigma_j}} (\D_{r_1,r_2})
 \right)
 \log \frac{1}{4r_2}.
   \nonumber
\end{eqnarray}
The second term is bounded by
\begin{eqnarray*}
0 \leq 
 \left( \mbox{Area}_{P_{\Sigma_j}} (\D_{r_1,r_2})
 \right)
 \log \frac{1}{4r_2}
 & = & \left( \int_{r_1}^{r_2} \frac{dr}{r(\log r^2)^2}\right) \log \frac{1}{4r_2} 
\\
 &
= &  \left( \frac{1}{\log r_2} - \frac{1}{\log r_1} \right) \log 4r_2 \leq 1. 
\end{eqnarray*}
Thus, combining the above equality and the inequality and then multiplying by $\Ej $, we obtain 
\begin{eqnarray*}
\lefteqn{
\Ej \mbox{Area}_{P_{\Sigma_j}} (\D_{r_1,\frac{1}{4}}) \log \frac{1}{4r_1}
-
\Ej  \mbox{Area}_{P_{\Sigma_j}} (\D_{r_2,\frac{1}{4}}) \log \frac{1}{4r_2} 
}
\\
& \leq  & \Ej  \left( \mbox{Area}_{P_{\Sigma_j}} (\D^*_{\frac{1}{4}}) \log \frac{r_2}{r_1}+1 \right).  \hspace{1.5in}
\end{eqnarray*}
Similarly,
\begin{eqnarray*}
\lefteqn{
\Ei \mbox{Area}_{P_{\Sigma_i}} (\D_{r_1,\frac{1}{4}}) \log \frac{1}{4r_1}
-
\Ei  \mbox{Area}_{P_{\Sigma_i}} (\D_{r_2,\frac{1}{4}}) \log \frac{1}{4r_2} 
}
\\
& \leq  & \Ei  \left( \mbox{Area}_{P_{\Sigma_i}} (\D^*_{\frac{1}{4}}) \log \frac{r_2}{r_1}  +1 \right). \hspace{1.5in}
\end{eqnarray*}
Furthermore,  Lemma~\ref{lemmadirvB2} implies 
\begin{eqnarray*}
^PE^v[U_{r_1,r_2}] & = & {^PE}^v[\D_{r_1,\frac{1}{4}} \times \D_{r_1,\frac{1}{4}}]- {^PE}^v[\D_{r_2,\frac{1}{4}} \times \D_{r_2,\frac{1}{4}}]
\\
& \leq  & \Ej\mbox{Area}_{P_{\Sigma_j}}(\D_{r_1,\frac{1}{4}})  \log \frac{1}{4r_1} +\Ei\mbox{Area}_{P_{\Sigma_j}}(\D_{r_1,\frac{1}{4}})  \log \frac{1}{4r_1}\\
& & \ - \Ej\mbox{Area}_{P_{\Sigma_i}}(\D_{r_2,\frac{1}{4}})  \log \frac{1}{4r_2} -\Ei\mbox{Area}_{P_{\Sigma_i}}(\D_{r_2,\frac{1}{4}})  \log \frac{1}{4r_2}+C.
\end{eqnarray*}
Thus, the desired estimate follows from the the fact that $\mbox{Area}_{P_{\Sigma_i}}(\D_{r,\frac{1}{4}})=\mbox{Area}_{P_{\Sigma_j}}(\D_{r,\frac{1}{4}})$ (cf.~Remark~\ref{areasymmetry}).  
 \end{proof}

\begin{lemma} \label{egrowthB2}
In a subset $\bar\D_{\frac{1}{4}} \times \bar\D_{\frac{1}{4}}$ of a set  $\bar \D^* \times \bar \D^*$ of type (B), there exists a constant $C>0$ such that  in any 
 subset (cf.~(\ref{nbhd12}))
\[
U_{r_1,r_2} ={\mathcal D}_{r_1,r_2} \cap (\bar\D_{\frac{1}{4}} \times \bar\D_{\frac{1}{4}}),   \ \ \ 0<r_1<r_2\leq\frac{1}{4},  \
\]
the prototype section $v$ satisfies\begin{eqnarray*}
\left| ^gE^v[U_{r_1,r_2}]
-\Ej  \mbox{Area}_{g_{\Sigma_j}} (\D^*_{\frac{1}{4}}) \log \frac{r_2}{r_1} - \Ei  \mbox{Area}_{g_{\Sigma_i}} (\D^*_{\frac{1}{4}}) \log \frac{r_2}{r_1} \right| \leq C.
\end{eqnarray*}
\end{lemma}

\begin{proof}
Follows from the metric estimate (\ref{closegPB})
 and Lemma~\ref{arealemma}.
\end{proof}

\begin{lemma} \label{comparisonB2}
In a subset $\bar\D_{\frac{1}{4}} \times \bar\D_{\frac{1}{4}}$ of a set $\bar \D^* \times \bar \D^*$ of type (B), there exists a constant $C>0$ such that  in any 
 subset (cf.~(\ref{nbhd12}))
\[
U_{r_1,r_2} ={\mathcal D}_{r_1,r_2} \cap (\bar\D_{\frac{1}{4}} \times \bar\D_{\frac{1}{4}}),   \ \ \ 0<r_1<r_2\leq\frac{1}{4},
\]
the prototype section $v$ satisfies 
\[
^gE^v[U_{r_1,r_2}]
\leq {^gE^f}[U_{r_1,r_2}]+C
\]
for  any locally Lipschitz section $f:U_{r_1,r_2}  \rightarrow \tilde M \times_{\rho} \tilde X $.
\end{lemma}
\begin{proof}
Follow the proof of Lemma~\ref{comparisonA} using Lemma~\ref{egrowthB2} instead of Lemma~\ref{egrowthA}
\end{proof}

\subsection{Energy estimates for the prototype section near the divisor}

Combining the results of the previous three subsections, we obtain the following estimate  in an open set  ${\mathcal D}_{r_1,r_2}$ (cf.~(\ref{nbhd12})) near the divisor.
\begin{proposition} \label{l}
There exists a constant $C>0$ such that the prototype section $v$ of Definition~\ref{prototypemap} satisfies
\[
\left| {^gE^v}[{\mathcal D}_{r_1,r_2}] - \sum_{j =1}^L  \Ej \mbox{Area}_{g_{\Sigma_j}}(\Sigma_j) \log \frac{r_2}{r_1}  \right|< C, \ \ \ 0<r_1<r_2 \leq \frac{1}{4}.
\]
\end{proposition}

\begin{proof}
Since we can cover a neighborhood of $\Sigma$  by a finite collection of sets of type (A) and type (B), the estimate   follows from Lemma~\ref{egrowthA}, Lemma~\ref{egrowthB1} and Lemma~\ref{egrowthB2}.
\end{proof}

\begin{proposition} \label{comparison}
The section $v$ is almost minimizing in $M$ in the following sense:  There exists a constant $C>0$ such that
\[
^gE^v[{\mathcal D}_{r_1,r_2}]
\leq {^gE^f}[{\mathcal D}_{r_1,r_2}]+C, \ \ \ 0 <r_1<r_2\leq \frac{1}{4}
\]
for any section $f: M \rightarrow \tilde{M} \times_{\rho} \tilde{X}$. 
\end{proposition}

\begin{proof}
Since we can cover a neighborhood of $\Sigma$  by a finite collection of sets of type (A) and type (B), the estimate, the estimate follows from Lemma~\ref{comparisonA}, Lemma~\ref{comparisonB1} and Lemma~\ref{comparisonB2}.
\end{proof}

\section{Harmonic maps of possibly infinite energy}  \label{sec:existence}

The goal of this section is to prove Theorem~\ref{theorem:pluriharmonicDim2}, the existence of a harmonic map of logarithmic energy growth.  
In Subsection~\ref{subsec:thmdim2}, we show the existence of a harmonic map with the help of the prototype map.  In Subsection~\ref{subsec:thmdim2energy}, we record the energy growth estimates for this map.

Throughout this section, we use $C$ to denote constants that are independent of the parameter $r$.  Note that $C$ may change from line to line.

\subsection{Proof of existence, Theorem~\ref{theorem:pluriharmonicDim2}} 
\label{subsec:thmdim2}

\begin{proof}
For $r  \in (0,\frac{1}{4}]$, let  $M_r =M \backslash {\mathcal D}_r$ (see Figure~\ref{Fig:Mr}).
  \begin{figure}[h!]
  \includegraphics[width=2in]{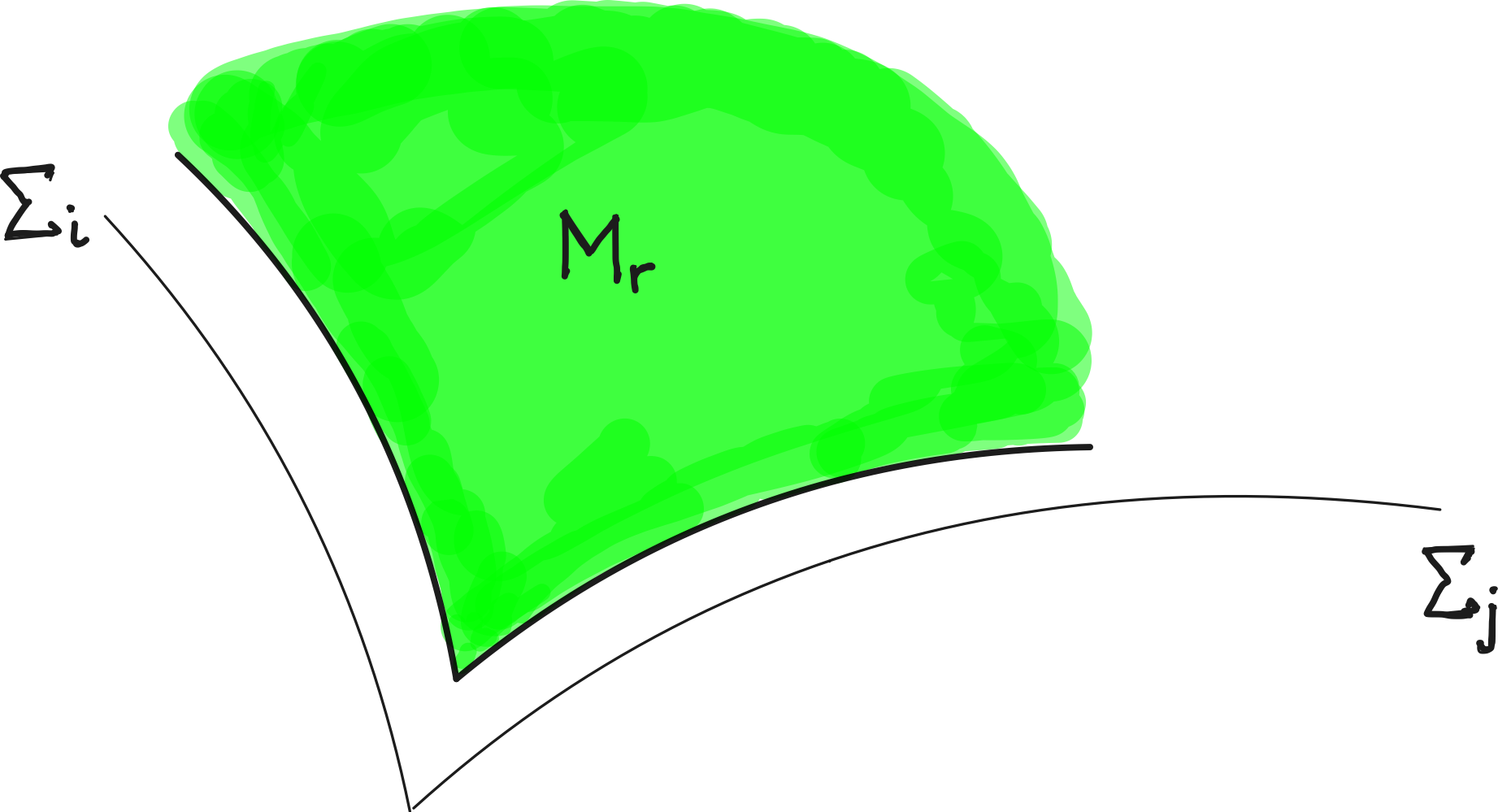}
  \caption{The region  $M_r \subset M$.}
  \label{Fig:Mr}
\end{figure}

Next, let $v:M \rightarrow  \tilde{M} \times_{\rho} \tilde{X}$ be the prototype section  of Definition~\ref{prototypemap} and 
let
\[
u_r:M_r  \rightarrow \tilde{M} \times_{\rho} \tilde{X}
\]
 be the energy minimizer among all  sections that  agree with $v$ on
$\partial M_r$ for each $r \in (0, r_1]$.  
The existence of such a section $u_r$ 
follows from  the  proof of ~\cite{korevaar-schoen1} Theorem 2.7.2.  

 Since
 \begin{eqnarray}
^gE^{u_r}[{\mathcal D}_{r,r_1}]+  {^gE^{u_r}}[M_{r_1}] & = & {^gE^{u_r}}[M_r] \nonumber  \\
 & \leq  &  {^gE^v}[M_r] \ \ \ \mbox{ (since $u_r$ is minimizing in $M_r$)} \nonumber  \\
 & = &  
{^gE^v}[{\mathcal D}_{r,r_1}] + {^gE^v}[M_{r_1}]\nonumber  \\
 & \leq &  
{^gE^{u_r}}[{\mathcal D}_{r,r_1}] +C+ {^gE^v}[M_{r_1}] \  \ \ \ \mbox{(by Proposition~\ref{comparison})},  \nonumber  
 \end{eqnarray}
we have that
\begin{equation} \label{mpe}
{^gE^{u_r}}[M_{r_1}] \leq {^gE^v}[M_{r_1}] + C.
\end{equation}
The right hand side of the inequality (\ref{mpe})  is independent of the parameter $r$; i.e.~once we fix $r_1 \in (0,\frac{1}{4}]$, the quantity $^gE^{u_r}[M_{r_1}]$ is uniformly bounded for all $r\in (0, r_1]$.  This implies a uniform Lipschitz bound, say $L$,  of $u_r$ for $r \in (0,r_1]$ in $M_{2r_1}$ (cf.~\cite[Theorem 2.4.6]{korevaar-schoen1}).  

Let $\tilde{u}_r$ and $\tilde v$ be the $\rho$-equivariant maps corresponding to  sections $u_r$ and $v$.
Thus,
  \[
d(\tilde u_r(\lambda(p)), \tilde u_r(p)) \leq L d_{\tilde M}(\lambda(p),p),
\ \ 
p \in M_{2r_1}, \ \lambda \in \Lambda, \ \ r \in (0,r_1].
\]
If we let  
\[
c=L \max \{d_{\tilde{M}}(\lambda(p),p): \lambda \in \Lambda, \ p \in \overline{M_{2r_1}} \},
\]
then by  equivariance
\[
d(\rho(\lambda)\tilde{u}_r(p),\tilde{u}_r(p))) \leq c, \ \ \ 
p \in M_{2r_1}, \ \lambda \in \Lambda, \ \ r \in (0,r_1].
\]
In other words, $\delta(\tilde{u}_r(p)) \leq L$ for all $p \in M_{2r_1}$ and $r \in (0,r_1]$.
By the properness of $\rho$, there exists $P_0 \in \tilde X$ and $R_0>0$ such that 
\[
\{\tilde{u}_r(p): \  p \in M_{2r_1}, \ r \in (0,r_1]\} \subset B_{R_0}(P_0).
\]
Thus,  by taking a compact exhaustion, applying Arzela-Ascoli and the usual diagonalization argument, there exists a subsequence of  $\tilde{u}_r$ that converges locally uniformly  to a $\rho$-equivariant harmonic map $\tilde u: \tilde M \rightarrow
\tilde X$.   Let $u: M \rightarrow \tilde{M} \times_{\rho} \tilde{X}$ be the corresponding harmonic section.
 \end{proof}
 
  \subsection{Energy estimates for the harmonic section} \label{subsec:thmdim2energy}

\begin{lemma} \label{lemma:mpe}
For the harmonic section $u:M \rightarrow \tilde M \times_\rho \tilde X$ of Theorem~\ref{theorem:pluriharmonicDim2} and the prototype section  of Definition~\ref{prototypemap}, we have
\[
{^gE^u}[M_{r_1}] \leq {^gE^v}[M_{r_1}] + C, \ \ \ \forall r_1 \in (0,\frac{1}{4}].
\]
\end{lemma}

\begin{proof}
Follows from (\ref{mpe}) and the lower semicontinuity of energy (cf.~\cite[Lemma~1.6.1]{korevaar-schoen1}).
\end{proof}
 \begin{lemma} \label{comparinguv}
If  $v:M \rightarrow \tilde M \times_{\rho} \tilde X$ is the prototype section of Definition~\ref{prototypemap} and $u:M \rightarrow \tilde M \times_{\rho} \tilde X$ is the  harmonic section of Theorem~\ref{theorem:pluriharmonicDim2},  there exists a constant $C>0$ such that 
\[
\left| {^gE^{u}}[\mathcal D_{r_1,r_2}] - {^gE^v}[\mathcal D_{r_1,r_2}] \right|\leq  C,  \  \ 0<r_1<r_2 \leq \frac{1}{4}.
\]
 \end{lemma}
 
 \begin{proof}
From the fact that $\mathcal D_{r_1,\frac{1}{4}} \subset \mathcal D_{r_1,\frac{1}{4}} \cup M_{\frac{1}{4}}= M_{r_1}$, Lemma~\ref{mpe}, and the lower semicontinuity of energy (cf.~\cite[Theorem 1.6.1]{korevaar-schoen1}), we obtain
\[
{^gE^{u}}[\mathcal D_{r_1,\frac{1}{4}}] \leq {^gE^u}[M_{r_1}] \leq {^gE^v}[M_{r_1}]+C
={^gE^v}[\mathcal D_{r_1,\frac{1}{4}}] + {^gE^v}[M_{\frac{1}{4}}] + C.
\]
 Proposition~\ref{comparison} implies
\[
{^gE^v}[\mathcal D_{r_1,\frac{1}{4}}] \leq {^gE^u}[\mathcal D_{r_1,\frac{1}{4}}] + C.
\]
Combining the above two inequalities we obtain
\[
\left| {^gE^{u}}[\mathcal D_{r_1,\frac{1}{4}}] - {^gE^v}[\mathcal D_{r_1,\frac{1}{4}}] \right|\leq  {^gE^v}[M_{\frac{1}{4}}]+C
\]
and similarly
\[
\left| {^gE^{u}}[\mathcal D_{r_2,\frac{1}{4}}] - {^gE^v}[\mathcal D_{r_2,\frac{1}{4}}] \right|\leq  {^gE^v}[M_{\frac{1}{4}}]+ C.
\]
The desired estimate follows from the above two inequalities.
 \end{proof}
 
 \begin{lemma} \label{prelude}
If $u:M \rightarrow \tilde M \times_{\rho} \tilde X$ is the  harmonic section  of Theorem~\ref{theorem:pluriharmonicDim2}, then there exists $C>0$ such that 
\begin{eqnarray*}
\left| ^gE^u[{\mathcal D}_{r_1,r_2}] - \sum_{j =1}^L \Ej \mbox{Area}_{g_{\Sigma_j}}(\Sigma_j) \log \frac{r_2}{r_1}  \right| & \leq &  C, \ \ \
0<r_1<r_2 \leq \frac{1}{4}.
\end{eqnarray*}
\end{lemma}
 
 \begin{proof}
The estimate follows from Proposition~\ref{l} and Lemma~\ref{comparinguv}. 
 \end{proof}

  \begin{lemma} \label{lulocal1}
If $u:M \rightarrow \tilde M \times_{\rho} \tilde X$ is the  harmonic section  of Theorem~\ref{theorem:pluriharmonicDim2}, then we have the following estimates in the subset $\Omega \times \D^*_{\frac{1}{4}}$ of a  set $ \Omega \times \bar\D $  of type (A) or  the subset $\Omega \times \bar \D^*: = \D_{\frac{1}{4},1} \times \bar\D^*$ of  a set $\bar \D \times \bar \D$ of type (B):
\begin{eqnarray*}
\int_{\Omega \times \bar \D^*_{\frac{1}{4}} }  \left| \frac{\partial u}{\partial z^1} \right|^2 dz^1 \wedge d\bar z^1 \wedge \frac{d\zeta \wedge d\bar z^2}{r^2(-\log r^2)^2} & < & \infty
 \\
  \int_{\Omega \times \bar \D^*_{\frac{1}{4}} }  \left( \left| \frac{\partial u}{\partial z^2} \right|^2 - \frac{\Ej}{8\pi r^2} \right) dz^1 \wedge d\bar z^1 \wedge dz^2 \wedge d\bar z^2    & < & \infty
\\
    \int_{\Omega \times \bar \D^*_{\frac{1}{4}} }   \left| \frac{\partial u}{\partial r} \right|^2  dz^1 \wedge d\bar z^1 \wedge dz^2 \wedge d\bar z^2    & < & \infty
  \\
 \int_{\Omega \times \bar \D^*_{\frac{1}{4}} }
  \left(   \left| \frac{\partial u}{\partial \theta} \right|^2 -\frac{\Ej}{2\pi} 
\right) 
dz^1 \wedge d\bar z^1 \wedge \frac{dz^2 \wedge d\bar z^2}{r^2}  & < & \infty
   \\
 \int_{\Omega \times \bar\D^*_{\frac{1}{4}}}
  \left| \frac{\partial u}{\partial \theta} \right|^2 
dz^1 \wedge d\bar z^1 \wedge \frac{dz^2 \wedge d\bar z^2}{r^2(-\log r^2)^2} 
& < & \infty
\\
   \int_{\Omega \times \bar \D^*_{\frac{1}{4}} }  \left| \frac{\partial u}{\partial z^2} \right|^2 dz^1 \wedge d\bar z^1 \wedge \frac{dz^2 \wedge d\bar z^2}{(-\log r^2)^2}    & < & \infty
\end{eqnarray*}
where $(z^1, z^2=re^{i\theta})$ are the standard product coordinates on $\Omega \times \bar \D$.
\end{lemma}

\begin{proof}
All the estimates except for the last two follow immediately from Lemma~\ref{prelude}.  The last two follow from the other estimates and  Lemma~\ref{calculus}.\end{proof}

Recall that the standard product coordinates $(z^1,z^2)$ on a set $\Omega \times \bar \D$ of type (A) are not necessarily the holomorphic coordinates $(z^1,\zeta)$ of  Definition~\ref{holcoord}.   We will now reframe the statements of Lemma~\ref{lulocal1}  in terms of  the holomorphic coordinates on the set of type (A).  We first need some estimates that compares  $\zeta$ to $z^2$.

 \begin{lemma}\label{zzetaest1}
If $(z^1,z^2=re^{i\theta})$ and $(z^1, \zeta=se^{i\eta})$ are the standard product coordinates and holomorphic coordinates respectively on a set $\Omega \times \bar \D$ of type (A), then 
 \begin{eqnarray*}
 r=as+O(r^2)\\
 \frac{\partial z^2}{\partial \zeta} = a(1+O(r)) & & \frac{\partial \bar z^2}{\partial \zeta} = O(r)
 \\
 \frac{\partial r}{\partial s}=|\alpha|(1+O(r)) & & \frac{\partial \theta}{\partial s}=O(1) \\
\frac{\partial r}{\partial \eta} = O(r^2) & & \frac{\partial \theta}{\partial \eta}=O(1)
\end{eqnarray*}
where $a$ is a smooth function both bounded above and bounded away from 0 (cf.~(\ref{az}) and (\ref{laer})).
\end{lemma}
\begin{proof}
Since 
\[
z^2 =a\zeta + O(r^2)
\]
by  (\ref{laer}), the first estimate follows immediately.
Furthermore, differentiating the above with respect to $\zeta$, we obtain the next two estimates. 
The last four estimates are obtained  
by evaluating the differential forms of (\ref{sreta}) on the vector fields 
$\frac{\partial }{\partial s}$,
$\frac{\partial }{\partial \eta}$ 
and using the fact that 
$|\frac{\partial }{\partial s}|=O(1)$, 
$|\frac{\partial }{\partial \eta}|=O(r^2)$.
\end{proof}

  \begin{theorem} \label{weneedA}
If $u:M \rightarrow \tilde M \times_{\rho} \tilde X$ is the  harmonic section  of Theorem~\ref{theorem:pluriharmonicDim2}, then we have the following estimates in the set $\Omega \times \D^*_{\frac{1}{4}}$ away from a crossing (i.e.~a subset of a  set $ \Omega \times \bar\D $  of type (A) or a subset $\Omega \times \bar \D^*: = \D_{\frac{1}{4},1} \times \bar\D^*$ of  a set $\bar \D \times \bar \D$ of type (B)),
\begin{eqnarray*}
\int_{\Omega \times \bar \D^*_{\frac{1}{4}} }  \left| \frac{\partial u}{\partial z^1} \right|^2 dz^1 \wedge d\bar z^1 \wedge \frac{d\zeta \wedge d\bar \zeta}{s^2(-\log s^2)^2} & < & \infty
 \\
  \int_{\Omega \times \bar \D^*_{\frac{1}{4}} }  \left( \left| \frac{\partial u}{\partial \zeta} \right|^2 - \frac{\Ej}{8\pi s^2} \right) dz^1 \wedge d\bar z^1 \wedge d\zeta \wedge d\bar \zeta    & < & \infty
  \\
   \int_{\Omega \times \bar \D^*_{\frac{1}{4}} }  \left| \frac{\partial u}{\partial \zeta} \right|^2 dz^1 \wedge d\bar z^1 \wedge \frac{d\zeta \wedge d\bar \zeta}{(-\log s^2)^2}    & < & \infty
\\
    \int_{\Omega \times \bar \D^*_{\frac{1}{4}} }   \left| \frac{\partial u}{\partial s} \right|^2  dz^1 \wedge d\bar z^1 \wedge d\zeta \wedge d\bar \zeta    & < & \infty
  \\
 \int_{\Omega \times \bar \D^*_{\frac{1}{4}} }
  \left(   \left| \frac{\partial u}{\partial \eta} \right|^2 -\frac{\Ej}{2\pi} 
\right) 
dz^1 \wedge d\bar z^1 \wedge \frac{d\zeta \wedge d\bar \zeta}{s^2}  & < & \infty \\
 \int_{\Omega \times \bar\D^*_{\frac{1}{4}}}
  \left| \frac{\partial u}{\partial \eta} \right|^2 
dz^1 \wedge d\bar z^1 \wedge \frac{d\zeta \wedge d\bar \zeta}{s^2(-\log s^2)^2} 
& < & \infty
\end{eqnarray*}

where $(z^1, \zeta=s e^{i\eta})$ are the holomorphic coordinates on $\Omega \times \bar \D$ (cf.~Definition~\ref{holcoord}).
\end{theorem}

\begin{proof} 
By Lemma~\ref{zzetaest1}, we have
 \begin{eqnarray*}  
s^2 (\log s^2)^2 & = & r^2(\log r^2)O(1),
\\
\frac{1}{s^2} & = & \frac{|a|}{r^2}(1+O(r)),
\\
\frac{\partial u}{\partial \zeta}\
& = &\frac{\partial u}{\partial z^2} \frac{\partial z^2}{\partial \zeta} +  \frac{\partial u}{\partial \bar z^2} \frac{\partial \bar z^2}{\partial \zeta}=\frac{\partial u}{\partial z^2}a(1+O(r)) +  \frac{\partial u}{\partial \bar z^2} O(r),
\\
\frac{\partial u}{\partial \eta}  & = &  \frac{\partial u}{\partial r}\frac{\partial r}{\partial \eta}+\frac{\partial u}{\partial \theta}\frac{\partial \theta}{\partial \eta} 
= \frac{\partial u}{\partial r}O(r^2)  +  \frac{\partial u}{\partial \theta}O(1).
\end{eqnarray*}
Thus,
\begin{eqnarray*}
 \left|\frac{\partial u}{\partial  \zeta}\right|^2 - \frac{\Ej}{8\pi s^2}
&=& 
|a|^2\left(  \left|\frac{\partial u}{\partial z^2}\right|^2-\frac{\Ej}{8\pi r^2}\right) (1+O(r))
\\
\left| \frac{\partial u}{\partial \eta} \right|^2 
& = &   \left| \frac{\partial u}{\partial \theta}\right|^2 O(1) +\left|\frac{\partial u}{\partial r}\right|^2 O(r^2). 
\end{eqnarray*}
Thus, the second, third and fourth estimates now follow from Lemma~\ref{lulocal1}.   The first estimate is a restatement of the first estimate of  Lemma~\ref{lulocal1}.   
\end{proof}

\begin{theorem} \label{weneedB}
If $u:M \rightarrow \tilde M \times_{\rho} \tilde X$ is the  harmonic section  of Theorem~\ref{theorem:pluriharmonicDim2}, then we have the following estimates in the set  $\bar \D^*_{\frac{1}{4}}  \times  \bar \D^*_{\frac{1}{4}}$ at a crossing (i.e.~a subset of a set $\bar \D \times \bar\D$ of type (B)),
\begin{eqnarray*}
\int_{\bar \D^*_{\frac{1}{4}}  \times  \bar \D^*_{\frac{1}{4}}}
\left( 
   \left| \frac{\partial u}{\partial z^1} \right|^2 -  \frac{\Ei}{8\pi \rho^2}
   \right)  dz^1 \wedge d\bar z^1\wedge \frac{dz^2 \wedge d\bar z^2}{r^2 (-\log r^2)^2} & < & \infty 
   \\
   \int_{\bar \D^*_{\frac{1}{4}}  \times  \bar \D^*_{\frac{1}{4}}}
  \left(  \left| \frac{\partial u}{\partial z^2} \right|^2 -  \frac{\Ej}{8\pi r^2} \right)  \frac{dz^1 \wedge d\bar z^1}{\rho^2 (-\log \rho^2)^2} \wedge dz^2 \wedge d\bar z^2 & < & \infty
\\
   \int_{\bar \D^*_{\frac{1}{4}}  \times  \bar \D^*_{\frac{1}{4}}}
  \left| \frac{\partial u}{\partial \rho} \right|^2  dz^1 \wedge d\bar z^1 \wedge \frac{dz^2 \wedge d\bar z^2}{r^2(-\log r^2)^2} & < & \infty
  \\
   \int_{\bar \D^*_{\frac{1}{4}}  \times  \bar \D^*_{\frac{1}{4}}}
  \left| \frac{\partial u}{\partial r} \right|^2  \frac{dz^1 \wedge d\bar z^1}{\rho^2 (-\log \rho^2)^2}  \wedge dz^2 \wedge d\bar z^2 & < & \infty
 \\
\int_{\bar \D^*_{\frac{1}{4}}  \times  \bar \D^*_{\frac{1}{4}}}
 \left(  \left| \frac{\partial u}{\partial \phi} \right|^2 - \frac{\Ej}{2\pi}
\right)  \frac{dz^1 \wedge d\bar z^1}{\rho^2 } \wedge \frac{dz^2 \wedge d\bar z^2}{r^2 (-\log r^2)^2} & < & \infty
\\
\int_{\bar \D^*_{\frac{1}{4}}  \times  \bar \D^*_{\frac{1}{4}}}
 \left( \left| \frac{\partial u}{\partial \theta} \right|^2 -  \frac{\Ei}{2\pi}
\right) 
\frac{dz^1 \wedge d\bar z^1}{\rho^2 (-\log \rho^2)^2} \wedge \frac{dz^2 \wedge d\bar z^2}{r^2} & < & \infty
\end{eqnarray*}
where $(z^1=\rho e^{i\phi},z^2=re^{i\theta})$  are the holomorphic coordinates on $\bar \D \times \bar\D$ (cf.~Definition~\ref{product=holomorphic}).
\end{theorem}

\begin{proof}
The standard product coordinates of a set of type (B) are also the holomorphic coordinates.
Thus, these estimates follow immediately from 
Lemma~\ref{lemmadirvB2} and  
(\ref{volcompB}).
\end{proof}

\section{Pluriharmonicity}
\label{sec:pluriharmonicity}
 
 The goal of this section is to prove Theorem~\ref{pluriharmonicity}.  In other words, we show  the pluriharmonicity of the  harmonic section $u:M \rightarrow \tilde M \times_\rho \tilde X$ of Theorem~\ref{theorem:pluriharmonicDim2} in the case when the target space $\tilde X$ is as in Theorem~\ref{theorem:pluriharmonic}.  
 \begin{quote}
{\it  For simplicity of exposition, we will assume that $\tilde X$ is a complete, simply connected K\"ahler manifold of strong nonpositive curvature in the sense of Siu.}
 \end{quote}
  Thus, we can apply the variation of Siu's Bochner formula, Theorem~\ref{mochizukibochnerformula}.  
  More generally, $\tilde X$ is a complete, simply connected Riemannian manifold of Hermitian negative curvature (cf.~Remark~\ref{symHer}).  Thus, the other cases follow analogously by using the variation of Sampson's Bochner formula, Theorem~\ref{mochizukibochnerformula'}, and thus we omit the details.

 The proof involves two steps:  first is to show the integrability of all the terms in the Siu's Bochner formula  of Theorem~\ref{siubochner} (cf.~Subsection~\ref{sec:integrability})  and second is to  apply the modified Siu's Bochner formula of Theorem~\ref{mochizukibochnerformula} to prove each of these terms are actually equal to 0 (cf.~Subsection~\ref{plur}).
Our proof closely follows Mochizuki's  in \cite{mochizuki-memoirs}, but we apply the modified Bochner formula of Theorem~\ref{mochizukibochnerformula} instead of the one used in his paper.

In Subsection~\ref{sec:integrability}, we prove the integrability result.   In the case when $M$ is a closed K\"ahler manifold, Stoke's theorem implies that the integration over the  $\partial \bar \partial$-term of Siu's formula (i.e.~the left hand side of the formula from Theorem~\ref{siubochner}) is equal to 0.  On the other hand,  we are considering  non-compact domains. Thus, we must first multiply Siu's formula by an appropriate cut-off function to apply Stoke's Theorem.  Unlike the compact case, this process yields terms   that apriori seem to be an obstruction to pluriharmonicity. Nevertheless,  this computation shows  that the curvature function $Q$ defined in Lemma~\ref{branch} and the norm squared of the Hessian associated to $u$ are integrable.   

In Subsection~\ref{plur}, we prove the pluriharmonicity result, Theorem~\ref{pluriharmonicity}.  
This is accomplished by combining the integrability results from Subsection~\ref{sec:integrability} and  the modified Bochner formula of Theorem~\ref{mochizukibochnerformula}.
 
\subsection{Integrability of the curvature and Hessian terms from Siu's formula} \label{sec:integrability}

Let  $u:M \rightarrow \tilde M \times_\rho \tilde X$ be the harmonic section of Theorem~\ref{theorem:pluriharmonicDim2}). 
Denote by $\mathcal V \rightarrow \tilde M \times_\rho \tilde X \rightarrow M$  the vertical tangent bundle of $\tilde M \times_\rho \tilde X \rightarrow M$.  Let $\mathcal V^\C$ denote the vertical complexified  tangent bundle and 
\[
E= u^*(\mathcal V^\C).
\]
The Levi-Civita connection on $\tilde X$ induces a connection on $E$, and  the
Bochner formulas of  Chapter~\ref{chap:bochner}  can be interpreted in terms of sections of the appropriate bundles induced by $E$.  In particular, the section $u$ satisfies Siu's Bochner formula (cf.~Theorem~\ref{siubochner} with $n=2$)
\begin{eqnarray*}
 \partial  \bar \partial\{\bar \partial u,\bar \partial u \} \nonumber
& = &    \left(4\left|\partial_{E'}\bar \partial u \right|^2 +Q\right) \frac{ \omega^2}{2}\end{eqnarray*}
 where $Q$ is the curvature term defined in Lemma~\ref{branch} and $\omega$ is the K\"ahler form of $M$.  The goal is to show that the right hand side above is integrable in $M$.  We first prove the following preliminary lemma.

\begin{lemma} \label{lemma:prelim}
Let $V=\bar \D_{\frac{1}{4}}^*  \times  \bar \D_{\frac{1}{4}}^*$ be a set  at  a crossing (cf.~Theorem~\ref{weneedB})
and  $(z^1=\rho e^{i\phi},z^2=r e^{i\theta})$ be holomorphic coordinates in $V$.  
  If   $F_N$ is a sequence of functions defined on $V$ satisfying the following:
   \begin{itemize}
   \item[(a)] $\displaystyle{|F_N(z^1,z^2)| \leq \frac{c}{(-\log r^2)^2}}$ for some constant $c>0$
   \item[(b)] $\displaystyle{ c_0:=\int_V   F_N(z^1,z^2)
\frac{dz^1 \wedge d\bar z^1}{\rho^2} \wedge \frac{dz^2 \wedge d\bar z^2}{r^2}}$ is  independent of $N$.

\item[(c)] For any  $z^2 \in \D_{\frac{1}{4}}^*$ with $|z^2|=r$, 
$F_N(z^1,z^2)=0$  for $N$ sufficiently large,
   \end{itemize}
    then
  \[
  \lim_{N \rightarrow \infty}  \int_V F_N(z^1,z^2) \left| \frac{\partial u}{\partial z^1}\right|^2 dz^1 \wedge d\bar z^1 \wedge \frac{dz^2 \wedge d\bar z^2}{r^2} =\frac{c_0 \Ei}{8\pi}.
    \]
  \end{lemma}
 
 \begin{proof}
 We first rewrite
\begin{eqnarray} 
\lefteqn{
 \int_V F_N(z^1,z^2) \left| \frac{\partial u}{\partial z^1}\right|^2 dz^1 \wedge d\bar z^1 \wedge \frac{dz^2 \wedge d\bar z^2}{r^2}
 } \nonumber 
\\
 & = &  
  \frac{\Ei}{8\pi} \int_V   F_N(z^1,z^2)
\frac{dz^1 \wedge d\bar z^1}{\rho^2} \wedge \frac{dz^2 \wedge d\bar z^2}{r^2}
 \nonumber \\
 &  & +\int_V F_N(z^1,z^2) \left( \left| \frac{\partial u}{\partial z^1} \right|^2 -
\frac{\Ei}{8\pi \rho^2} \right) dz^1 \wedge d\bar z^1 \wedge \frac{dz^2 \wedge d\bar z^2}{r^2}. 
\end{eqnarray}
The first term is equal to $\frac{c_0 \Ei}{8\pi}$  by assumption (b).
For the second term of (\ref{int-(i)}), we first rewrite the integral as
\[
\int_0^{\frac{1}{4}} \left( \int_{\D^*_{\frac{1}{4}} }\int_0^{2\pi}F_N(z^1,z^2) \left(   \left| \frac{\partial u}{\partial z^1} \right|^2 -\frac{\Ei}{8\pi \rho^2} 
\right) \rho d\phi \wedge \frac{dz^2 \wedge d\bar z^2}{-2ir^2}  \right)d\rho.
\]
By assumption (a), the integral inside the bracket, i.e.~the function
\begin{equation}  \label{fncFN'}
r \mapsto \int_{\D^*_{\frac{1}{4}} }\int_0^{2\pi}F_N(z^1,z^2) \left(   \left| \frac{\partial u}{\partial z^1} \right|^2 -\frac{\Ei}{8\pi \rho^2} 
\right) \rho d\phi \wedge \frac{dz^2 \wedge d\bar z^2}{-2ir^2},
\end{equation}
is bounded from above (independently of $N$) by a non-negative function
\[
r \mapsto c \int_{\D^*_{\frac{1}{4}} }\int_0^{2\pi} \left(   \left| \frac{\partial u}{\partial z^1} \right|^2 -\frac{\Ei}{8\pi \rho ^2}
\right)  \rho d\phi \wedge \frac{dz^2 \wedge d\bar z^2}{-2ir^2(-\log r^2)^2}.
\]
The above is non-negative by the definition of $\Ei$ and integrable over the interval $[0,\frac{1}{4}]$ by Lemma~\ref{weneedB}.
Furthermore, 
the function (\ref{fncFN'}) converges to 0 for each  $r \in (0,\frac{1}{4})$ by assumption (c).  Thus,   Lebesgue's dominated  convergence theorem implies   the result.
\end{proof}

\begin{proposition} \label{integrable}  
Let $u:M \rightarrow \tilde M \times_\rho \tilde X$ be the harmonic section of  Theorem~\ref{theorem:pluriharmonicDim2}.    
If
 $\tilde X$ is a K\"ahler manifold with strong nonpositive curvature in the sense of Siu,    then
\[
\int_M Q \, \omega^2 < \infty \mbox{ and } \int_M |\partial_{E'} \bar\partial  u|^2= \int_M |\partial_{E''} \bar\partial  \bar u|^2  \, \omega^2 < \infty.
\] 
\end{proposition}

\begin{proof}
Let $\eta:[0,\infty) \rightarrow [0,1]$ be a non-increasing, non-negative smooth function
satisfying
\[
\eta(x)=1 \mbox{ for } 0 \leq x \leq \frac{1}{2},  \ \ \eta(x) = 0 \mbox{ for } \frac{2}{3} \leq x <\infty.
\]
For $N \in {\mathbb N}$, define a cut-off function 
\begin{equation} \label{defchi}
\chi_N: M \rightarrow [0,1], \ \ \ \chi_N =\left\{
\begin{array}{ll}
\displaystyle{\prod_{j=1}^L  \eta \left(N^{-1}\log |\sigma_j|_{h_j}^{-2} \right) } & \mbox{in } \displaystyle{\bigcup_{j=1}^L {\mathcal D}_j^*}
\\
1 & \mbox{otherwise}
\end{array}
\right.
\end{equation}
where $\sigma_j$ is the canonical section and $h_j$ is the Hermitian metric on ${\mathcal O}(\Sigma_j)$ (cf.~Section~\ref{sec:neardivisor}). Note that we use (\ref{assubsets}) to identify the neighborhood ${\mathcal D}_j$ of the zero section of ${\mathcal O}(\Sigma_j)$ with the neighborhood of $\Sigma_j$ in $\bar M$.   
Multiplying the Siu-Bochner formula of Theorem~\ref{siubochner} by $\chi_N$, integrating over $M$ and applying integration by parts, we obtain
\begin{equation} \label{multiply}
 \int_M\partial \bar{\partial}  \chi_N \wedge \{ \bar \partial u, \bar\partial u \} = \frac{1}{2} \int_M \chi_N \left(   4\left|\partial_{E'} \bar \partial u\right|^2 +Q  \right) \omega^2.
\end{equation}
By the assumption that $\tilde{X}$ has strong non-positive curvature in the sense of Siu (cf.~Lemma~\ref{branch}),  the right hand side of the above equality is non-negative.

Let $V$  be either a set $\Omega \times \D^*_{\frac{1}{4}}$ away from the crossings (cf.~Theorem~\ref{weneedA}) or a set 
  $\bar \D_{\frac{1}{4}}^*  \times  \bar \D_{\frac{1}{4}}^*$  at  a crossing (cf.~Theorem~\ref{weneedB}).
 Since   $\partial \bar \partial \chi_N$  is supported in the finite union of such sets for sufficiently large $N$,  it suffices to prove 
 \begin{equation} \label{ibp}
\lim_{N\rightarrow \infty} \int_{V} \partial \bar{\partial} \chi_N \wedge \{ \bar\partial u, \bar\partial u \}<\infty
\end{equation}
for either    $V=\Omega \times \D^*_{\frac{1}{4}}$ or $V=\bar \D_{\frac{1}{4}}^*  \times  \bar \D_{\frac{1}{4}}^*$.
Throughout this proof of (\ref{ibp}), we will use $c$ to denote a generic positive constant that may change from line to line but is independent of $N \in {\mathbb N}$.

First, consider   the subset
  $V=\bar \D_{\frac{1}{4}}^*  \times  \bar \D_{\frac{1}{4}}^*$ near  a crossing with local holomorphic coordinates $(z^1=\rho e^{i\phi},z^2=r e^{i\theta})$.  In $V$ and  for $N$ sufficiently large,   
  \begin{eqnarray*}
\chi_N(z^1,z^2)  
& = & \eta \left(-N^{-1} \log \rho^2 \right) \eta \left(-N^{-1} \log r^2 \right). 
\end{eqnarray*}
The support of $\eta'\left(-N^{-1} \log \rho^2 \right) $ and $\eta''\left(-N^{-1} \log \rho^2 \right) $ is contained in 
\begin{equation} \label{sptphi1}
W_N:=\left\{\frac{1}{2} \leq -N^{-1} \log \rho^2 \leq \frac{2}{3}\right\} 
\end{equation}
and the support of $\eta'\left(-N^{-1} \log r^2 \right) $ and $\eta''\left(-N^{-1} \log r^2 \right)$ is contained in 
\begin{equation} \label{sptphi2}
V_N:=\left\{\frac{1}{2} \leq -N^{-1} \log r^2 \leq \frac{2}{3}\right\}.
\end{equation}
Therefore,
\begin{eqnarray} 
(-\log \rho^2) \left| \frac{\eta'(-N^{-1} \log \rho^2)}{N} 
\right|
\leq  c, & &  (-\log r^2)\left| \frac{\eta'(-N^{-1} \log r^2)}{N} 
\right|
\leq c \nonumber  \\
(-\log \rho^2)^2 \left| \frac{\eta'' (-N^{-1} \log \rho^2)}{N^2} \right| \leq  c,
& & 
(-\log r^2)^2 \left| \frac{\eta'' (-N^{-1} \log r^2)}{N^2} \right| \leq  c.
\label{cleverbd!} 
\end{eqnarray}
We have
\begin{eqnarray}
\partial \bar{\partial} \chi_N & = &  \eta(-N^{-1}\log \rho^2)\eta''(-N^{-1}\log r^2) \frac{dz^2 \wedge d\bar{z}^2}{N^2r^2}
\nonumber 
\\
& & 
\
+
\eta''(-N^{-1}\log \rho^2) \eta(-N^{-1}\log r^2)\frac{dz^1 \wedge d\bar{z}^1}{N^2 \rho^2}
\nonumber
\\
& & \ \ +\eta'(-N^{-1}\log \rho^2) \eta'(-N^{-1}\log r^2) \frac{dz^1 \wedge d\bar{z}^2}{N^2 z^1\bar{z}^2}
\nonumber 
\\
& & \ \ \ +\eta'(-N^{-1}\log r^2) \eta'(-N^{-1}\log \rho^2) \frac{dz^2 \wedge d\bar{z}^1}{N^2 \bar{z}^1z^2}.
\label{sum'}
\end{eqnarray}
Using  (\ref{sum'}),  we write the integral of  (\ref{ibp}) as the sum $(i)+(ii)+(iii)+(iv)$
where
\begin{eqnarray*}
(i) &= & 
  \int_V  \eta(-N^{-1}\log \rho^2)\eta''(-N^{-1}\log r^2) \frac{dz^2 \wedge d\bar{z}^2}{N^2r^2}   \wedge \{ \bar\partial u, \bar\partial u \}\\
 (ii) &= & 
 \int_V  \eta''(-N^{-1}\log \rho^2)\eta(-N^{-1}\log r^2) \frac{dz^1 \wedge d\bar{z}^1}{N^2\rho^2}   \wedge \{ \bar\partial u, \bar\partial u \}\\
 (iii) &= & 
 \int_V  \eta'(-N^{-1}\log \rho^2)\eta'(-N^{-1}\log r^2) \frac{dz^1 \wedge d\bar{z}^2}{N^2z^1\bar{z}^2}   \wedge \{ \bar\partial u, \bar\partial u \}
\\
 (iv) &= & 
 \int_V  \eta'(-N^{-1}\log \rho^2)\eta'(-N^{-1}\log r^2) \frac{dz^2 \wedge d\bar{z}^1}{N^2\bar{z}^1z^2}   \wedge \{ \bar\partial u, \bar\partial u \}.
\end{eqnarray*}

First, consider the integral $(i)$.  Using the identity
\[ 
\langle  \frac{\partial  u}{\partial{ \bar z^\alpha}},  \frac{\partial  u}{ {\partial{\bar  z^\beta}}} \rangle  d\bar z^\alpha \wedge dz^\beta = h_{i\bar j}  \frac{\partial  u^i}{\partial{ \bar z^\alpha}} 
 \overline{
 \frac{\partial  u^j}{ {\partial{\bar  z^\beta}}} 
 }
 d\bar z^\alpha \wedge dz^\beta =\{\bar \partial u, \bar \partial u\},
 \]
 we have
\begin{equation}   \label{int-(i)}
(i) =  \int_V  
\eta(-N^{-1}\log \rho^2) \frac{\eta''(-N^{-1}\log r^2)}{N^2} 
\left| \frac{\partial u}{\partial z^1} \right|^2 dz^1\wedge d\bar z^1 \wedge \frac{dz^2 \wedge d\bar z^2}{r^2}.
\end{equation}
We now check that  
\[
F_N(z^1,z^2) = \eta(-N^{-1}\log \rho^2) \frac{\eta''(-N^{-1}\log r^2)}{N^2}
\]
satisfies the assumptions (a), (b) and (c) of  Lemma~\ref{lemma:prelim}.
First, $F(z^1,z^2)$ satisfies assumption (a) of Lemma~\ref{lemma:prelim} by (\ref{cleverbd!}).
Next, we will check that the function $F_N(z^1,z^2)$ also satisfies assumption (b) of Lemma~\ref{lemma:prelim}.
Indeed, after  a change of variables, 
\begin{equation} \label{stvar}
t=-N^{-1}\log \rho \ \mbox{ and } \ s=-N^{-1}\log r,
\end{equation}
we obtain
\[
\frac{dz^1 \wedge d\bar{z}^1}{N\rho^2} =-2i \frac{d\rho \wedge d\phi}{N\rho} = 2i dt \wedge d\phi \ \mbox{ and } \ \frac{dz^2 \wedge d\bar{z}^2}{Nr^2} = -2i \frac{dr \wedge d\theta}{Nr} =  2i ds \wedge d\theta.
\]
Thus, 
\begin{eqnarray*}  \label{int-(i)1}
\lefteqn{  c_0:=
\int_V 
\eta(-N^{-1}\log \rho^2)  \frac{\eta''(-N^{-1}\log r^2)}{N^2}  \frac{dz^1 \wedge d\bar{z}^1}{\rho^2} \wedge \frac{d\bar{z}^2 \wedge dz^2}{r^2}
 } 
\nonumber  \\
& = & c \int_0^{\frac{1}{3}} \eta(2t) dt \int_{\frac{1}{4}}^{\frac{1}{3}}  \eta''(2s) ds= c \int_0^{\frac{1}{3}} \eta(2t) dt \cdot \left(  \eta'(\frac{2}{3}) - \eta'(\frac{1}{2}) \right) =0.
 \end{eqnarray*}
Finally,  $F_N(z^1,z^2)$ satisfies assumption (c) of Lemma~\ref{lemma:prelim} by (\ref{sptphi2}).
By applying Lemma~\ref{lemma:prelim}, we conclude
\begin{equation} \label{ito0}
\lim_{N \rightarrow \infty}|(i)| =0.
\end{equation}
The same argument also implies
\[
\lim_{N \rightarrow \infty}|(ii)| =0.
\]
We will now bound  the term $(iii)$.    Indeed, we can rewrite
\begin{eqnarray}
|(iii)|&= &
\left| \int_V \frac{ \eta'(-N^{-1}\log \rho^2)}{N}\frac{\eta'(-N^{-1}\log r^2)}{N} \frac{dz^1 \wedge d\bar{z}^2}{z^1 \bar{z}^2}   \wedge \{ \bar\partial u, \bar\partial u \} \right|
 \nonumber  \\
  &\leq   & 
 \int_V
\left|  \langle \frac{\partial u}{\partial \bar{z}^1}, \frac{\partial  u}{\partial \bar{z}^2}\rangle  \right| 
\frac{ \eta'(-N^{-1}\log \rho^2)}{N}\frac{\eta'(-N^{-1}\log r^2)}{N}  
\frac{dz^1 \wedge d\bar{z}^1}{\rho} \wedge \frac{dz^2 \wedge d \bar z^2}{r} 
 \nonumber  \\
& \leq &   \int_{V_N}
\left|  \frac{\partial u}{\partial \bar{z}^1}\right|^2
\left( \frac{\eta'(-N^{-1}\log \rho^2)}{N}  \right)^2
dz^1 \wedge d\bar{z}^1 \wedge \frac{dz^2 \wedge d \bar z^2}{r^2}
 \nonumber  \\
& & \ \ +  \int_{W_N}
\left| \frac{\partial  u}{\partial \bar{z}^2} \right|^2 
\left( \frac{ \eta'(-N^{-1}\log r^2)}{N}\right)^2  
\frac{dz^1 \wedge d\bar{z}^1}{\rho^2} \wedge dz^2 \wedge d \bar z^2.
 \label{iiiint}
 \end{eqnarray}
 
 For the first integral on the right hand side of (\ref{iiiint}), we let
 \[
 F_N(z^1,z^2) =\chi_{V_N}
\left( \frac{ \eta'(-N^{-1}\log r^2)}{N}\right)^2 
 \]
 where $\chi_{V_N}$ is the characteristic function of $V_N$.
First,  $F_N(z^1,z^2)$ satisfies assumption (a) of Lemma~\ref{lemma:prelim} by (\ref{cleverbd!}),.
Next, we check that it satisfies assumption (b) of Lemma~\ref{lemma:prelim}.  Indeed, using the substitution (\ref{stvar}), 
\begin{eqnarray*}
c_0 & := &  \int_V \chi_{V_N} \left( \frac{\eta'(-N^{-1}\log r^2)}{N}  \right)^2 \frac{dz^1 \wedge d\bar{z}^1}{\rho^2} \wedge \frac{dz^2 \wedge d \bar z^2}{r^2} 
\\
& = &  c   \int_{\frac{1}{4}}^{\frac{1}{3}} dt \  \int_{\frac{1}{4}}^{\frac{1}{3}}   (\eta'(2s))^2  ds.
\end{eqnarray*}
Finally,  $F_N(z^1,z^2)$ satisfies assumption (c) of Lemma~\ref{lemma:prelim} by (\ref{sptphi2}).
  Thus, the second integral on the right hand side of (\ref{iiiint}) limits to $\frac{c_0\Ei}{8\pi}$ as $N \rightarrow \infty$ by Lemma~\ref{lemma:prelim}.  Analogously, the second integral on the right hand side of (\ref{iiiint}) limits to $\frac{c_0\Ej}{8\pi}$ as $N \rightarrow \infty$. 
Thus, we have shown 
 \begin{equation} \label{iii0}
\lim_{N \rightarrow \infty}|(iii)| =\frac{c_0(\Ei +\Ej)}{8\pi}.
\end{equation}
Same argument shows 
\[
\lim_{N \rightarrow \infty}|(iv)| =\frac{c_0(\Ei +\Ej)}{8\pi}.
\]
Summing the limits of $(i)$, $(ii)$, $(iii)$ and $(iv)$, we conclude that (\ref{ibp}) is satisfied  in the case   $V=\bar \D_{\frac{1}{4}}^*  \times  \bar \D_{\frac{1}{4}}^*$.

Next, consider  set $V=\Omega \times \D^*_{\frac{1}{4}}$ away from the crossings with holomorphic coordinates  $(z^1,\zeta=r e^{i\theta})$.   
In $V$ and for sufficiently large $N$, 
\begin{equation}\label{setb0} 
\chi_N(z^1,\zeta) =\eta \left(N^{-1} \log b|\zeta|^{-2} \right). 
\end{equation}
We compute
\begin{eqnarray}
\partial \bar{\partial} \chi_N & = & \frac{\eta''(N^{-1} \log b|\zeta|^{-2})}{N^2} \left(\frac{d\zeta \wedge d\bar{\zeta}}{|\zeta|^2}
+  \frac{\partial b \wedge \bar{\partial} b}{b^2}- \frac{\partial b \wedge d\bar{\zeta}}{b \bar \zeta}- \frac{d\zeta \wedge \bar{\partial} b}{\zeta b}\right)\nonumber \\
&  &  \ +\frac{\eta'(N^{-1} \log b|\zeta|^{-2})}{N}  \partial \bar{\partial} \log b. \label{sum}
\end{eqnarray}
The support of $\eta'$ and $\eta''$ is contained in 
\begin{eqnarray*} 
\left\{\frac{1}{2} \leq N^{-1} \log b|\zeta|^{-2} \leq \frac{2}{3}\right\},
\end{eqnarray*}
which  is contained in the set
 \begin{equation}\label{sprt}
V_N=\Omega \times \D_{z^2,c_1e^{-\frac{N}{3}}, c_2e^{-\frac{N}{4}}}
\end{equation}
for appropriate constants $c_1$ and $c_2$ depending only on $b$. 
Therefore,
\begin{eqnarray} 
\left| \frac{\eta'(N^{-1} \log b|\zeta|^{-2})}{N} 
\right|
& \leq &  \frac{c}{\log br^{-2}}
\leq  \frac{c}{-\log r^2}, \label{cleverbd1}\\
\left| \frac{\eta'' (N^{-1} \log b|\zeta|^{-2})}{N^2} \right|& \leq & \frac{c}{(\log br^{-2})^2}
\leq  \frac{c}{(-\log r^2)^2}.
\label{cleverbd2} 
\end{eqnarray}
Using  (\ref{sum}),  we write the integral of  (\ref{ibp}) as the sum $(I)+(II)+(III)+(IV)+(V)$.  For the integral $(I)$, we write
\begin{eqnarray*}
|(I)| 
&  = &   
\left|
\int_V \frac{\eta'' (N^{-1} \log b|\zeta|^{-2})}{N^2}  \frac{d\zeta \wedge d\bar{\zeta}}{|\zeta|^2} \wedge \{ \bar\partial u, \bar\partial u \}   
\right|
 \\
 & \leq &
c\int_{V_N}  \left| \frac{\partial u}{\partial \bar z^1} \right|^2  dz^1 \wedge d\bar z^1 \wedge \frac{d\zeta \wedge d\bar \zeta}{r^2 (-\log r^2)^2}\ \ \ (\mbox{by (\ref{cleverbd2}))}.
\end{eqnarray*}
By Theorem~\ref{weneedA}, 
\[
\int_V     \left| \frac{\partial     u}{\partial \bar z^1} \right|^2  dz^1 \wedge d\bar z^1 \wedge \frac{d\zeta \wedge d\bar \zeta}{r^2 (-\log r^2)^2}<\infty.
\]
Thus, Lebesgue's dominated  convergence Theorem implies
\begin{equation} \label{I0}
\lim_{N \rightarrow \infty} (I) = 0.
\end{equation}
For the integral $(II)$, we write
\begin{eqnarray*}
 |(II)| 
 & = &
 \left|   \int_V \frac{\eta'' (N^{-1} \log b|\zeta|^{-2})}{N^2} \frac{\partial b \wedge \bar{\partial} b}{b^2} \wedge \{ \bar\partial u, \bar\partial u \}  \right| 
 \\
 & \leq & c   \int_{V_N}  \left|\frac{\partial u}{\partial \bar{z}^1}\right |^2  dz^1 \wedge d\bar z^1 \wedge \frac{d\zeta \wedge d\bar \zeta}{(-\log r^2)^2}+ \int_{V_N} \left |\frac{\partial u}{\partial \bar{\zeta}}\right |^2  dz^1 \wedge d\bar z^1 \wedge \frac{d\zeta \wedge d\bar \zeta}{(-\log r^2)^2}  \\
  & &  \ \ \  (\mbox{by}  \ \frac{\partial b}{b}=O(1)\ \mbox{and }  (\ref{cleverbd2})).
 \end{eqnarray*}
By Theorem~\ref{weneedA}, we can apply an analogous argument to (\ref{I0}) to conclude
 \[
 \lim_{N \rightarrow \infty} (II)=0.
 \] 
 In order to estimate $(III)$, notice that
  \begin{eqnarray*} \label{cancel1}
  d\bar{\zeta} \wedge \{\bar \partial u, \bar\partial u \}=d\bar{\zeta} \wedge\left(\left|\frac{\partial u}{\partial \bar{z}^1} \right|^2d\bar{z}^1\wedge d{z}^1
 + \langle \frac{\partial u}{\partial \bar{z}^1}, \frac{\partial u}{\partial \bar{\zeta}}\rangle d\bar{z}^1\wedge d\zeta \right).
 \end{eqnarray*}
Thus,
  \begin{eqnarray*}
|(III)| & = &  \left|     \int_V \frac{\eta'' (N^{-1} \log b|\zeta|^{-2})}{N^2} \frac{\partial b \wedge d\bar{\zeta}}{b\bar{\zeta}} \wedge \{\bar \partial u, \bar\partial u \}     \right|    
\\
& \leq & 
 c     \int_{V_N} \left(\left|\frac{\partial u}{\partial \bar{z}^1} \right|^2
 +\left| \langle \frac{\partial u}{\partial \bar{z}^1}, \frac{\partial u}{\partial \bar\zeta}\rangle \right| \right) dz^1 \wedge d\bar z^1 \wedge \frac{d\zeta \wedge d\bar \zeta}{r(-\log r^2)^2} 
 \\
 & & 
\ \ \  ({\mbox{by}  \ \frac{\partial b}{b}=O(1) \mbox{ and } (\ref{cleverbd2}))}\\
\\
& = & 
 c     \int_{V_N}   r \left(\left|\frac{\partial u}{\partial \bar{z}^1} \right|^2
 +\left| \langle \frac{\partial u}{\partial \bar{z}^1}, \frac{\partial u}{\partial \bar\zeta} \rangle\right| \right)\ dz^1 \wedge d\bar z^1 \wedge \frac{d\zeta \wedge d\bar \zeta}{r^2(-\log r^2)^2} 
\\
& \leq & 
 c     \int_{V_N}  \left(  \left|\frac{\partial u}{\partial \bar{z}^1} \right|^2 
 +  r^2\left| \frac{\partial u}{\partial \bar\zeta} \right|^2 \right)dz^1 \wedge d\bar z^1 \wedge \frac{d\zeta \wedge d\bar \zeta}{r^2(-\log r^2)^2} \ \   \mbox{(by Cauchy-Schwartz)}.
\end{eqnarray*}
By Theorem~\ref{weneedA}, we can apply an analogous argument to (\ref{I0}) to conclude
 \[
 \lim_{N \rightarrow \infty} (III)=0.
 \] 
Similarly,  
 \[
 \lim_{N \rightarrow \infty} (IV)=0.
 \] 
We thus conclude
\[
\lim_{N \rightarrow \infty} |(I)|+|(II)|+|(III)|+|(IV)|=0.
\]
Next,
 \begin{eqnarray*}
 (V) & = &
   \int_V \frac{\eta'(N^{-1} \log b|\zeta|^{-2})}{N}   \partial  \bar{\partial} \log b \wedge \{ \bar\partial u, \bar\partial u \}
   \\   
   & = &
   \int_V \frac{\eta'(N^{-1} \log b|\zeta|^{-2})}{N}   \frac{\partial^2 \log b} { \partial  z^1  \partial \bar{z}^1} dz^1 \wedge d{\bar z}^1 \wedge \left|\frac{\partial u}{\partial \bar \zeta} \right|^2  d \bar\zeta \wedge d{\zeta}\\
   & & \ +
   \int_V \frac{\eta'(N^{-1} \log b|\zeta|^{-2})}{N}   \frac{\partial^2 \log b} { \partial \zeta  \partial \bar\zeta} d\zeta \wedge d{\bar \zeta} \wedge \left|\frac{\partial u}{\partial \bar z^1} \right|^2  d \bar{z}^1 \wedge dz^1\\
   & & \ \ +
   \int_V \frac{\eta'(N^{-1} \log b|\zeta|^{-2})}{N}   \frac{\partial^2 \log b} { \partial z^1  \partial \bar\zeta} dz^1 \wedge d{\bar \zeta} \wedge <\frac{\partial u}{\partial \bar z^1},  \frac{\partial u}{\partial \bar \zeta}>  d \bar{z}^1 \wedge d\zeta\\
   & & \ \ \ +
   \int_V \frac{\eta'(N^{-1} \log b|\zeta|^{-2})}{N}   \frac{\partial^2 \log b} { \partial \zeta  \partial  \bar{z}^1} d \zeta \wedge d \bar{z}^1 \wedge <\frac{\partial u}{\partial \bar \zeta}, \frac{\partial u}{\partial \bar{z}^1}>  d \bar\zeta \wedge d{z}^1\\
   & =: & (V)_1+(V)_2+(V)_3+(V)_4.
    \end{eqnarray*}
We estimate
 \begin{eqnarray*}
|(V)_2| 
& \leq  &  c
 \int_V
\left|  \frac{\eta'(N^{-1} \log b|\zeta|^{-2})}{N}  \right| \left| \frac{\partial^2 \log b} { \partial  z^1  \partial \bar{z}^1} \right| \left|\frac{\partial u}{\partial \bar z^1} \right|^2 dz^1 \wedge   d \bar{z}^1 \wedge d\zeta \wedge d\bar \zeta\nonumber 
\\
& \leq  &  c
 \int_{V_N}
\left|\frac{\partial u}{\partial \bar z^1} \right|^2 dz^1 \wedge   d \bar{z}^1 \wedge \frac{d \zeta \wedge d\bar \zeta}{(-\log r^2)}\nonumber \\
&& \left(\mbox{by} \  \frac{\partial^2 \log b} { \partial \zeta  \partial \bar \zeta}=O(1) \ \mbox{ and }  (\ref{cleverbd1})\right)
\\
|(V)_3| & \leq & 
  \int_V \left| \frac{\eta'(N^{-1} \log b|\zeta|^{-2})}{N}\right|  \left| \frac{\partial^2 \log b} { \partial z^1  \partial \bar\zeta} \right| \left| \frac{\partial u}{\partial \bar z^1}\right|\left| \frac{\partial u}{\partial \bar \zeta}\right| dz^1 \wedge   d \bar{z}^1 \wedge d{\bar \zeta} \wedge d\zeta \\
  & \leq & c
\int_{V_N} 
\frac{1}{(-\log r^2)}\left|\frac{\partial u}{\partial \bar z^1} \right|\left|\frac{\partial u}{\partial \bar \zeta} \right| dz^1 \wedge   d \bar{z}^1 \wedge d{\bar \zeta} \wedge d\zeta \nonumber 
\\
&& \left(\mbox{by}  \ \frac{\partial^2 \log b} { \partial \zeta  \partial \bar \zeta}=O(1) \ \mbox{and } \  (\ref{cleverbd1})\right)\\
& \leq   &  c
\int_{V_N}
 \left( \left|\frac{\partial u}{\partial \bar z^1} \right| ^2+\frac{1}{(-\log r^2)^2} \left|\frac{\partial u}{\partial \bar \zeta} \right|^2 \right) dz^1 \wedge   d \bar{z}^1 \wedge d\zeta \wedge d\bar \zeta\nonumber 
  \end{eqnarray*}
and similarly 
 \begin{eqnarray*}
|(V)_4|  & \leq   &  c
\int_{V_N}
 \left( \left|\frac{\partial u}{\partial z^1} \right| ^2+\frac{1}{(-\log r^2)^2} \left|\frac{\partial u}{\partial \zeta} \right|^2 \right) dz^1 \wedge   d \bar{z}^1 \wedge d\zeta \wedge d\bar \zeta.   \end{eqnarray*}
With these estimates, we can argue as in  the proof of (\ref{I0}) to  conclude
\[
\lim_{N \rightarrow \infty} (V)_2+(V)_3+(V)_4=0.
\]
 We are left  to compute
 \[
  (V)_1 =
   \int_V \frac{\eta'(N^{-1} \log b|\zeta|^{-2})}{N}   \frac{\partial^2 \log b} { \partial  z^1  \partial \bar{z}^1} dz^1 \wedge d{\bar z}^1 \wedge \left|\frac{\partial u}{\partial \bar \zeta} \right|^2  d \bar\zeta \wedge d{\zeta}.
   \]   
   First,  use the identity
  \[
  \frac{\partial^2 \log b} {  \partial z^1  \partial \bar{z}^1}(z^1,\zeta)=\frac{\partial^2 \log b} { \partial z^1  \partial \bar{z}^1}(z^1,0)+O(r)
  \]
  to write 
  \[
  (V)_1=(V)_{1a}+(V)_{1b}.
  \]
We estimate 
   \begin{eqnarray*}
|(V)_{1b}|& = &   \int_V \left| \frac{\eta'(N^{-1} \log b|\zeta|^{-2})}{N}  \right|  \left|\frac{\partial u}{\partial \bar \zeta} \right|^2O(\zeta) dz^1 \wedge d{\bar z}^1 \wedge   d \zeta \wedge d \bar \zeta\\   
   & \leq  &  c
\int_{V_N}  
\frac{r}{(-\log r^2)}\left|\frac{\partial u}{\partial \bar \zeta} \right|^2dz^1 \wedge d{\bar z}^1 \wedge   d \zeta \wedge d \bar \zeta \  \mbox{ (by (\ref{cleverbd1})}).
   \end{eqnarray*}
Thus, we can argue as in  the proof of (\ref{I0}) to  conclude 
   \[
   \lim_{N \rightarrow \infty} (V)_{1b} =0.
   \]
Furthermore,
  \begin{eqnarray*}
(V)_{1b} & = & 
   \int_V \frac{\eta'(N^{-1} \log b|\zeta|^{-2})}{N}   \frac{\partial^2 \log b} { \partial z^1  \partial \bar{z}^1}(z^1,0) dz^1 \wedge d{\bar z}^1 \wedge \left|\frac{\partial u}{\partial \bar \zeta} \right|^2  d \bar\zeta \wedge d{\zeta}
   \\
    & = &
  \int_V \frac{\eta'(N^{-1} \log b|\zeta|^{-2})}{N}   \frac{\partial^2 \log b} { \partial z^1  \partial \bar{z}^1}(z^1,0) dz^1 \wedge d{\bar z}^1 \wedge \left(\left|\frac{\partial u}{\partial \bar \zeta} \right|^2-\frac{\Ej}{8\pi r^2}\right) d \bar\zeta \wedge d{\zeta}
  \\
   &  &
\ \ + \  \frac{\Ej}{2\pi}  \int_V \frac{\eta'(N^{-1} \log b|\zeta|^{-2})}{Nr^2}    \frac{\partial^2 \log b} { \partial z^1  \partial \bar{z}^1}(z^1,0) dz^1 \wedge d{\bar z}^1 \wedge d \bar\zeta \wedge d{\zeta}.
   \end{eqnarray*}
The first  term on the right hand side above can be estimated by 
  \begin{eqnarray*}
\lefteqn{
\left| \int_V \frac{\eta'(N^{-1} \log b|\zeta|^{-2})}{N}    \frac{\partial^2 \log b} { \partial z^1  \partial \bar{z}^1}(z^1,0) \left(\left|\frac{\partial u}{\partial \bar \zeta} \right|^2-\frac{\Ej}{2\pi r^2}\right) dz^1 \wedge d{\bar z}^1 \wedge  d \bar\zeta \wedge d{\zeta}  \right|
}
\\
   & \leq  &  c
\int_{V_N}  
\left(\left|\frac{\partial u}{\partial \bar \zeta} \right|^2-\frac{\Ej}{2\pi r^2}\right) dz^1 \wedge d{\bar z}^1 \wedge   \frac{d \bar\zeta \wedge d{\zeta}}{(-\log r^2)} \  \left(\mbox{by}  \ \frac{\partial^2 \log b} { \partial z^1  \partial \bar{z}^1} =O(1) \ \mbox{and} \  (\ref{cleverbd1})\right).
  \end{eqnarray*} 
With these estimates, we can argue as in  the proof of (\ref{ito0}) to  conclude \begin{eqnarray*}
 \lim_{N \rightarrow \infty} (V)_{1} & = &  \lim_{N \rightarrow \infty} (V)_{1a} + (V)_{1b} \\
 & = & \frac{\Ej}{2\pi}  \lim_{N \rightarrow \infty}\int_V \frac{\eta'(N^{-1} \log b|\zeta|^{-2})}{Nr^2}   \frac{\partial^2 \log b} { \partial z^1  \partial \bar{z}^1}(z^1,0)dz^1 \wedge d{\bar z}^1\wedge d \bar\zeta \wedge d{\zeta}\\
 &=&\frac{\Ej}{2\pi}  \int_{\Omega} \frac{\partial^2 \log b} { \partial z^1  \partial \bar{z}^1}(z^1,0) d{z}^1\wedge d{\bar z}^1   \cdot \lim_{N \rightarrow \infty} \int_{\bar{\D}^*_{\frac{1}{4}}}
 \frac{\eta'(N^{-1} \log b|\zeta|^{-2})}{Nr^2}  d \bar\zeta \wedge d{\zeta}\\
 &=&\frac{\Ej}{2\pi i} \int_{\Omega} \Theta({\mathcal O}(\Sigma_j))
\cdot  \lim_{N \rightarrow \infty}  \int_0^{\frac{1}{4}}
 \frac{\eta'(-N^{-1} \log br^2)}{Nr}dr\\
 &=&\frac{\Ej}{2\pi i} \int_{\Omega} \Theta({\mathcal O}(\Sigma_j)).
\end{eqnarray*}
In the above $ \Theta({\mathcal O}(\Sigma_j))$  denotes the curvature of the hermitian metric $h_j$ on the line bundle ${\mathcal O}(\Sigma_j)$. The estimates for $(I)$, $(II)$, $(III)$, $(IV)$ and $(V)$ imply that (\ref{ibp}) also holds for $V=\Omega \times \D^*_{\frac{1}{4}}$ away from the crossings.
\end{proof}

\subsection{Proof of Theorem~\ref{pluriharmonicity}}\label{plur}  

Let $\chi_N$ be as in the proof of Proposition~\ref{integrable} (cf.~(\ref{defchi})).  
We multiply the modified Siu-Bochner formula of Theorem~\ref{mochizukibochnerformula} by $\chi_N$, integrate over $M$ and apply integration by parts to obtain
\begin{equation} \label{multiply'}
 \int_M d \chi_N \wedge  \{ \bar\partial_E \partial u, \partial u -\bar{\partial} u\}  = \frac{1}{2} \int_M \chi_N  \left( 4\left| {\partial}_{E'} \bar \partial u\right|^2+Q \right) \omega^2.
\end{equation}
By the monotone convergence theorem and Proposition~\ref{integrable},
\begin{eqnarray}
\lim_{N \rightarrow \infty} \int_M \chi_N \left(  \left|{\partial}_{E'} \bar \partial u\right|^2 +Q \right) \omega^2 
\nonumber 
& = & \int_M  \left(  \left|{\partial}_{E'} \bar \partial u\right|^2 +Q \right) \omega^2
\nonumber \\
& = & \int_M \left(  \left|{\partial}_{E''} \bar \partial \bar u\right|^2 +Q \right) \omega^2. \label{lint}
\end{eqnarray}
By the assumption that $\tilde{X}$ has strong non-positive curvature in the sense of Siu,  the $\left|{\partial}_{E'} \bar \partial u\right|^2 +Q=\left|{\partial}_{E''} \bar \partial \bar u\right|^2 +Q$  is non-negative.
Let $V$  be either a set $\Omega \times \D^*_{\frac{1}{4}}$  away from the crossings or the set
  $\bar \D_{\frac{1}{4}}^*  \times  \bar \D_{\frac{1}{4}}^*$ near  a crossing (cf.~Section~\ref{sec:neardivisor}).
 We claim that 
\begin{equation} \label{limitis0}
\lim_{N \rightarrow \infty}  \int_V d \chi_N \wedge  \{\bar\partial_{E'}\partial u, \partial u -\bar{\partial} u\}   = 
0
\end{equation}
for any such set $V$.   This implies that the integral (\ref{lint}) is equal to 0 from which we conclude $\partial_{E'} \bar \partial u=0=\partial_{E'} \bar \partial \bar u$.

The rest of the proof is devoted to proving (\ref{limitis0}).  For the sequel, the constant $c>0$ is an arbitrary constant independent of the parameter $N$.  
First, consider the set
  $V=\bar \D_{\frac{1}{4}}^*  \times  \bar \D_{\frac{1}{4}}^*$ at  a crossing with local holomorphic coordinates $(z^1=\rho e^{i\phi},z^2=r e^{i\theta})$ (c.f.~Theorem~\ref{weneedB}).
   We have
\begin{eqnarray*}
\partial u - \bar{\partial} u  
& = &
\left( \frac{\partial u}{\partial z^1 }dz^1- \frac{\partial u}{\partial \bar{z}^1 }d\bar{z}^1 \right) + \left( \frac{\partial u}{\partial z^2 }dz^2- \frac{\partial u}{\partial \bar{z}^2 }d\bar{z}^2 \right)
\\
& = &    i \left(   \frac{\partial u}{\partial \rho} \rho d\phi - \frac{\partial u}{\partial \phi} \frac{d\rho}{\rho} \right)+ i \left(   \frac{\partial u}{\partial r} r d\theta - \frac{\partial u}{\partial \theta} \frac{dr}{r} \right)
\end{eqnarray*}
and
\begin{eqnarray*}
d\chi_N & = &  -\eta(-N\log \rho^2) \frac{\eta'(-N^{-1}\log r^2)}{N}  \frac{2dr}{r}  -\frac{\eta'(-N^{-1}\log \rho^2)}{N}  \eta(-N\log r^2)  \frac{2d\rho}{\rho}.
\end{eqnarray*}
Thus,
\begin{eqnarray} \label{goto00}
\lefteqn{
  \int_V d \chi_N \wedge  \{\bar\partial_{E'}\partial u, \partial u -\bar{\partial} u\} }
 \nonumber \\
 & = &   
 - \frac{2}{N} \int_V   \eta(-N^{-1}\log \rho^2) \eta'(-N^{-1}\log r^2)  \frac{dr}{r} \wedge \{ \partial_{E'} \bar \partial u, \frac{\partial u}{\partial z^1 }dz^1- \frac{\partial u}{\partial \bar{z}^1 }d\bar{z}^1  \} 
 \nonumber \\
 & & 
 \      - \frac{2i}{N}  \int_V  \eta(-N^{-1}\log \rho^2) \eta'(-N^{-1}\log r^2) \frac{dr}{r}  \wedge \{ \partial_{E'} \bar \partial u,   \frac{\partial u}{\partial r} rd\theta \}  \nonumber \\
 & & 
 \ \ \  -\frac{2}{N}  \int_V  \eta'(-N^{-1}\log \rho^2) \eta(-N^{-1}\log r^2)   \frac{d\rho}{\rho} \wedge \{ \partial_{E'} \bar \partial u,  \frac{\partial u}{\partial z^2 }dz^2- \frac{\partial u}{\partial \bar{z}^2 }d\bar{z}^2 \} 
\nonumber  \\
 & & \ \ \ \  - \frac{2i}{N}    \int_V  \eta'(-N^{-1}\log \rho^2) \eta(-N^{-1}\log r^2) \frac{d\rho}{\rho}  \wedge \{ \partial_{E'} \bar \partial u,   \frac{\partial u}{\partial \rho} \rho d\phi \}  \\
  & = &(i)+(ii)+(i')+(ii'). \nonumber 
 \end{eqnarray}
 
 We will show that all the terms $(i)$, $(ii)$, $(i')$ and $(ii')$ go to 0 as $N \rightarrow \infty$. We start with $(i)$.
Note that  $|\eta(-N^{-1}\log \rho^2)|$ has  support in $\rho \geq e^{-\frac{N}{3}}$ and $|\eta'(-N^{-1}\log r^2)|$ has  support in $ e^{-\frac{N}{3}} \leq r \leq e^{-\frac{N}{4}}$  (cf.~(\ref{sptphi1})). Thus, the integrand of $(i)$ has support in 
\[
D_N:=  \D_{z^2,e^{-\frac{N}{3}}, \frac{1}{4}} \times \D_{z^1,e^{-\frac{N}{3}}, e^{-\frac{N}{4}}}.
\]
We estimate
\begin{eqnarray}
\lefteqn{|(i)|
\leq 
c  \int_{D_N} \left|  \frac{ \eta'(-N^{-1}\log r^2)}{N} \right|   |\partial_{E'} \bar \partial u|  \left| \frac{\partial u}{\partial z^1}\right|  \rho d\rho \wedge d\phi \wedge \frac{r dr \wedge d\theta}{r}
 }
 \nonumber   \\
& \leq &
 c \left(\int_{D_N}  |\partial_{E'} \bar \partial u|^2  dz^1 \wedge d\bar z^1 \wedge 
 dz^2 \wedge d\bar z^2
\right)^{\frac{1}{2}} \nonumber
 \\
& & \ \times 
 \left(  \int_{D_N}  \left( \frac{ \eta'(-N^{-1}\log r^2)}{N} \right)^2 \left| \frac{\partial u}{\partial z^1}\right|^2 dz^1 \wedge d\bar z^1\wedge \frac{dz^2 \wedge d\bar z^2}{r^2}
\right)^{\frac{1}{2}} \nonumber\\
 && \ \ \ \ \ \ \mbox{ (by Cauchy-Schwartz and (\ref{cleverbd!}))}. 
  \label{sades}
\end{eqnarray}
The first integral above limits to 0 as $N \rightarrow \infty$ by Proposition~\ref{integrable}, volume estimate (\ref{volcompB}) and  Lebesgue's dominated convergence theorem.  The limit as $N \rightarrow \infty$ of the second integral exists by Lemma~\ref{lemma:prelim} by  following the proof  of (\ref{iii0}). Thus
$
\lim_{N \rightarrow \infty} (i) =0.
$
An analogous argument shows $
\lim_{N \rightarrow \infty} (i') =0.
$

Next,  \begin{eqnarray*}
|(ii)| 
 & \leq &    
c \int_{V}   \left| \frac{ \eta'(-N^{-1}\log r^2)}{N} \right| | \bar\partial_{E'}\partial u|
 \left| \frac{\partial u}{\partial r} \right|
dz^1  \wedge d\bar z^1 \wedge
\frac{dz^2 \wedge d\bar z^2}{r} 
\\
  & \leq &    
c \left(  \int_{D_N} 
  | \bar\partial_{E'}\partial u|^2 
dz^1  \wedge d\bar z^1 \wedge
\frac{dz^2 \wedge d\bar z^2}{r^2 (-\log r^2)^2}
 \right)^{\frac{1}{2}}  
 \\
 &  & \ \times
\left(  \int_{D_N}  
 \left|\frac{\partial u}{\partial r} \right|^2 
  dz^1  \wedge d\bar z^1 \wedge
dz^2 \wedge d\bar z^2 
\right)^{\frac{1}{2}}\\
&  & (\mbox{by Cauchy-Schwartz and (\ref{cleverbd!})}).
 \end{eqnarray*}
 The first integral limits to 0 as $N \rightarrow \infty$ by  Proposition~\ref{integrable}, volume estimate (\ref{volcompB}) and  Lebesgue's dominated   convergence Theorem.  The second integral also limits to 0 by Theorem~\ref{weneedB} and Lebesgue's dominated  convergence theorem. Thus, $
\lim_{N \rightarrow \infty} (ii) =0$, and an analogous argument shows 
$\lim_{N \rightarrow \infty} (ii') =0$.

Next, consider a set $V=\Omega \times \D^*_{\frac{1}{4}}$ away from the crossings   with holomorphic coordinates  $(z^1,\zeta=r e^{i\theta})$.  
Since
\[
\partial u - \bar{\partial} u =  \left( \frac{\partial u}{\partial z^1 }dz^1- \frac{\partial u}{\partial \bar{z}^1 }d\bar{z}^1 \right)+ i \left(   \frac{\partial u}{\partial r} r d\theta - \frac{\partial u}{\partial \theta} \frac{dr}{r} \right),
\]
we have
\begin{eqnarray} \label{goto0}
\lefteqn{
  \int_M d \chi_N \wedge  \{\bar\partial_{E'}\partial u, \partial u -\bar{\partial} u\} }
 \\
 & = &   
 i \int_Md\chi_N \wedge \{ \bar\partial_{E'}\partial u, \frac{\partial u}{\partial r} r d\theta \} 
 +
 i \int_Md\chi_N \wedge \{ \bar\partial_{E'}\partial u,  \frac{\partial u}{\partial \theta}  \frac{dr}{r} \} 
  \nonumber  \\
 & & \ +   \int_Md\chi_N \wedge \{ \bar\partial_{E'}\partial u, \frac{\partial u}{\partial z^1 }dz^1- \frac{\partial u}{\partial \bar{z}^1 }d\bar{z}^1 \}  \nonumber \\
  & = & (I)+(II)+(III) \nonumber
 \end{eqnarray}
where  the integrals $(I)$, $(II)$, and $(III)$ are estimated below. Let
\[
G_N:=\Omega \times\D_{z^1,c_1e^{-\frac{N}{3}}, c_2e^{-\frac{N}{4}}}.
\]
Since, 
\begin{eqnarray*}
d\chi_N 
 & = & -\frac{\eta'(N^{-1}\log br^{-2})}{N} \left( \frac{2dr}{r} - \frac{db}{b} \right), \end{eqnarray*}
integral $(I)$ is bounded by
 \begin{eqnarray*}
\left| (I) \right| 
& = & 
 \left|  \int_V  \frac{\eta'(N^{-1}\log br^{-2})}{N} \left( \frac{2dr}{r} - \frac{db}{b}    \right) \wedge \{ \bar\partial_{E'}\partial u, \frac{\partial u}{\partial r} r d\theta \}  \right|  
 \\
 & \leq &c \int_V   |\bar{\partial}_{E'} \partial u| \left| \frac{\partial u}{\partial r} \right| \rho d\rho \wedge d\phi \wedge \frac{r dr \wedge d\theta}{r (-\log r^2)}  \ \  \left(\mbox{since} \ \frac{db}{b}=O(1) \right)
 \\
& \leq  & c \left(  \int_{G_N} |\bar{\partial}_{E'} \partial u|^2  d z^1 \wedge d\bar z^1 \wedge  \frac{d \zeta \wedge d\bar \zeta}{r^2 (-\log r^2)^2} \right)^{\frac{1}{2}} \\
& & \ \ \times  \left( \int_{G_N}   \left| \frac{\partial u}{\partial r} \right|^2 
d z^1 \wedge d\bar z^1 \wedge d\zeta \wedge d\bar \zeta  \right)^{\frac{1}{2}} \ \  \mbox{(by Cauchy-Schwartz)}.
 \end{eqnarray*}
The first integral limits to 0 by  Proposition~\ref{integrable}, volume estimate (\ref{volcomp}) and Lebesgue's dominated  convergence theorem.  The second integral also limits to 0 by Theorem~\ref{weneedA} (with $s=r$) and  Lebesgue's dominated  convergence theorem. Thus, $
\lim_{N \rightarrow \infty} (I) =0$.

   Next, we estimate $(II)$ and note that the modified Siu's  Bochner formula is crucial.  Indeed, we hightlight the cancellation $\frac{dr}{r} \wedge \frac{dr}{r}=0$ below:
\begin{eqnarray*}
|(II)| 
& =& 
 \left|  \int_V \frac{\eta'(N^{-1}\log br^{-2})}{N} \left( \frac{2dr}{r} - \frac{db}{b}  \right) \wedge \{ \bar\partial_{E'}\partial u, \frac{\partial u}{\partial \theta} \frac{dr}{r} \}  \right|
 \\
 & \leq &c \int_V  \left| \frac{\eta'(N^{-1}\log br^{-2})}{N} \right| |\bar{\partial}_{E'} \partial u| \left| \frac{\partial u}{\partial \theta} \right|  \rho d\rho \wedge d\phi \wedge \frac{r dr \wedge d\theta}{r} \\
 & & \ \ \  \ \left(\mbox{since} \ \frac{db}{b}=O(1)  \mbox{ and } \frac{dr}{r} \wedge \frac{dr}{r}=0\right)
\\   & \leq & 
c \left( \int_{G_N}   | \bar\partial_{E'}\partial u|^2  dz^1 \wedge d\bar z^1 \wedge d\zeta \wedge d\bar \zeta \right)^{\frac{1}{2}}
\\
& & \ \ \times \left(\int_{G_N}    \left| \frac{\partial u}{\partial \theta} \right|^2 dz^1 \wedge d\bar z^1 \wedge \frac{d\zeta \wedge d\bar \zeta}{r^2 (-\log r^2)^2} \right)^{\frac{1}{2}}  \   \mbox{ (by Cauchy-Schwartz)}.
\end{eqnarray*}
The first integral limits to 0 by  Proposition~\ref{integrable}, volume estimate (\ref{volcomp}) and  Lebesgue's dominated  convergence theorem.  The second integral also limits to 0 by Theorem~\ref{weneedA} (with $r=s$ and $\theta=\eta$) and Lebesgue's dominated  convergence theorem. Thus, $
\lim_{N \rightarrow \infty} (II) =0$.

Finally, 
\begin{eqnarray*}
|(III)| 
 & = & 
 \left|  \int_V  \frac{\eta'(N^{-1}\log br^{-2})}{N} \left( \frac{2dr}{r} - \frac{db}{b} \right) \wedge
 \{ \bar\partial_{E'}\partial u, \frac{\partial u}{\partial z^1 }dz^{1}- \frac{\partial u}{\partial \bar{z}^1 }d\bar{z}^1 \}  \right|
 \\
 & \leq & 
 c\int_{G_N}  | \bar\partial_{E'}\partial u| \left|  \frac{\partial u}{\partial z^1}\right| dz^1 \wedge d\bar z^1 \wedge \frac{d\zeta \wedge d\bar \zeta}{r (-\log r^2)} \\
 &&  
\ \ \  \left(\mbox{since} \ \frac{db}{b}=O(1) \ \mbox{and by} \ (\ref{cleverbd1}) \right)
 \\  & \leq & 
c \left( \int_{G_N}   | \bar\partial_{E'}\partial u|^2 dz^1 \wedge d\bar z^1 \wedge d\zeta \wedge d\bar \zeta \right)^{\frac{1}{2}} 
\\
& & \ \times \left( \int_{G_N} \left|  \frac{\partial u}{\partial z^1}\right|^2   dz^1 \wedge d\bar z^1 \wedge \frac{d\zeta \wedge d\bar \zeta}{r^2 (-\log r^2)^2} \right)^{\frac{1}{2}} \ \ \ \  \mbox{ (by Cauchy-Schwartz)}.
\end{eqnarray*}
The first integral limits to 0 by  Proposition~\ref{integrable}, volume estimate (\ref{volcomp}) and  Lebesgue's dominated  convergence theorem.  The second integral also limits to 0 by Theorem~\ref{weneedA} (with $r=s$) and  Lebesgue's dominated  convergence theorem. Thus, $
\lim_{N \rightarrow \infty} (III) =0$.

We now conclude that  (\ref{goto0}) $\rightarrow 0$ as $N \rightarrow \infty$, which combined with the fact that (\ref{goto00}) $\rightarrow 0$ as $N \rightarrow \infty$
implies (\ref{limitis0}).  This concludes the proof of Theorem~\ref{pluriharmonicity}.

\section{Appendix to Chapter~\ref{chap:Kahler surfaces}}
\label{appendix:calculus}
We conclude this chapter with the following calculus result which was used  in the derivation of the energy estimates.

\begin{lemma} \label{calculus}
Let  $\Omega \times \D^*_{\frac{1}{4}}$ be a  subset of  $ \Omega \times \bar \D $  of type (A) or  $\Omega \times \bar \D^*: = \D_{\frac{1}{4},1} \times \bar\D$ of  type (B)$_1$ with standard product coordinates  $(z^1, z^2=re^{i\theta})$.
If a locally Lipschitz map $f$ defined on $\Omega \times \D^*_{\frac{1}{4}}$ 
satisfies
\begin{equation} \label{alb}
 \int_{\Omega \times \bar \D^*_{\frac{1}{4}} }
\left(   \left| \frac{\partial f}{\partial \theta} \right|^2 -c
\right) 
dz^1 \wedge d\bar z^1 \wedge \frac{dz^2 \wedge d\bar z^2}{r^2}
\leq C
\end{equation}
where
\[
\int_{\{z^1\} \times \partial \D_r} \left| \frac{\partial f}{\partial \theta} \right|^2 d\theta \geq c \geq 0,
\]
then
\[
 \int_{\Omega \times \bar\D^*_{\frac{1}{4}}}
  \left| \frac{\partial f}{\partial \theta} \right|^2 
dz^1 \wedge d\bar z^1 \wedge \frac{dz^2 \wedge d\bar z^2}{r^2(-\log r^2)^2}   
\leq C'
\]
where $C'$ is a constant depending only on $C$ and $c$.  \end{lemma}

\begin{proof}
We start with the following claim:  For any function $\psi:[0,\frac{1}{4}] \rightarrow \R$  satisfying $\psi(r) \geq c$, 
\begin{equation} \label{calclemma}
\int_0^\frac{1}{4} \psi(r)  \frac{dr}{r(\log r^2)^2}  \leq c  \log 2 +\int_0^\frac{1}{4} \left( \psi(r) -c \right) \frac{dr}{r}.
\end{equation}
To prove (\ref{calclemma}),  first note that
$
2^{-i-1}\leq
 r \leq 2^{-i}
 $ 
  implies 
\[
\frac{1}{(\log r^2)^2}  \leq \frac{1}{(\log 2^{-2i})^2}=\frac{1}{4(\log 2)^2}\frac{1}{i^2}. 
 \]
Furthermore,
 by the assumption that $\psi(r) \geq c \geq 0$, 
\[
 \int_{2^{-i-1}}^{2^{-i}}  \psi(r) \frac{dr}{r}  =  c \int_{2^{-i-1}}^{2^{-i}}  \frac{dr}{r}+ \int_{2^{-i-1}}^{2^{-i}} \left( \psi(r) -c \right) \frac{dr}{r}     \leq  c \log2 + \int_0^\frac{1}{4} \left( \psi(r) -c \right) \frac{dr}{r}.
\]
The above two inequalities imply  
\begin{eqnarray*}
\int_0^\frac{1}{4} \frac{\psi(r)}{r(\log r^2)^2} dr 
& = & 
\sum_{i=2}^{\infty}
\int_{2^{-i-1}}^{2^{-i}} \frac{\psi(r)}{r(\log r^2)^2} dr\\
& \leq & 
\frac{1}{4(\log 2)^2} \sum_{i=2}^{\infty}
\frac{1}{i^2} \int_{2^{-i-1}}^{2^{-i}} \psi(r) \frac{dr}{r}  
\\
& = & 
\frac{1}{4(\log 2)^2} \sum_{i=2}^{\infty}
\frac{1}{i^2} \cdot\left( c \log2 + \int_0^\frac{1}{4} \left( \psi(r) -c \right) \frac{dr}{r} \right)\\
& \leq &  c\log 2 + \int_0^\frac{1}{4} \left( \psi(r) -c \right) \frac{dr}{r}
\end{eqnarray*}
which proves (\ref{calclemma}).

Let  
\[
\psi(r):=\int_{\{z^1\} \times \partial \D_r} \left| \frac{\partial f}{\partial \theta} \right|^2 d\theta. 
\]
Since $\psi(r) \geq c$, we have by (\ref{calclemma}) that 
\[
\int_0^{\frac{1}{4}}
\left(
 \int_{\{z^1\} \times \partial \D_r} \left| \frac{\partial f}{\partial \theta} \right|^2 d\theta
 \right)
  \frac{dr}{r(\log r^2)^2} 
\leq 
c \log 2 + \int_0^{\frac{1}{4}} \left( \int_{\{z^1\} \times \partial \D_r} \left| \frac{\partial f}{\partial \theta} \right|^2 d\theta - c \right) \frac{dr}{r}
\]
for a.e.~$z^1 \in \Omega$.
Thus,
\begin{eqnarray*}
\lefteqn{
\int_{\Omega \times \D_{z^2,\frac{1}{4}} }
  \left| \frac{\partial f}{\partial \theta} \right|^2 
\frac{dz^1 \wedge d\bar z^1}{-2i} \wedge \frac{r dr \wedge d\theta}{r^2(\log r^2)^2}
}
\\
  & = & 
\int_{\Omega} \int_0^{\frac{1}{4}} \left( \int_{\{z^1\} \times \partial \D_r} 
  \left| \frac{\partial f}{\partial \theta} \right|^2 d\theta  \right) \frac{dr}{r(\log r^2)^2}  
\wedge \frac{dz^1 \wedge d\bar z^1}{-2i}
\\
  & \leq & 
 \int_{\Omega} 
\left(  c \log 2+ \int_0^{\frac{1}{4}}  
 \left( 
\int_{\{z^1\}  \times \D_{z^2,\frac{1}{4}} }    \left| \frac{\partial f}{\partial \theta} \right|^2 d\theta - c   
\right) 
 \frac{dr}{r}  
 \right)
\frac{dz^1 \wedge d\bar z^1}{-2i}
\\
& = &
  c \log 2 \cdot \int_\Omega \frac{dz^1 \wedge d\bar z^1}{-2i}
\\
& & 
\ +   \int_{\Omega \times \bar \D_{\frac{1}{4}} }
\left(   \left| \frac{\partial f}{\partial \theta} \right|^2 -c
\right) 
\frac{dz^1 \wedge d\bar z^1}{-2i} \wedge \frac{dr \wedge d\theta}{r} .
\end{eqnarray*}
\end{proof}

\chapter{Harmonic maps in higher dimensions} \label{chap:higher dimensions}

The goal of this section is to prove Theorem~\ref{theorem:pluriharmonic}.  
The proof closely follows the argument in \cite{mochizuki-memoirs} with some modifications due to the fact that we are dealing with a more general situation.    The preparation for the proof is contained in Section~\ref{premsetup}.   The proof,  contained in Section~\ref{proofpluriharmonic},   is by induction on the dimension of $\bar M$. Namely, the unique pluriharmonic map of logarithmic growth from a domain of dimension $n$ is constructed by piecing together pluriharmonic maps from lower dimensional subvarieties.

\section{Preliminary set-up}
\label{premsetup}
Throughout this section,   $M$, $\tilde X$ and $\rho:  \pi_1(M) \rightarrow \mathsf{Isom}(\tilde X)$ are as in Theorem~\ref{theorem:pluriharmonic}.  
Furthermore, we let $\bar M$ be a compact projective variety and $\Sigma$ be a normal crossing divisor such that $M =\bar M \backslash \Sigma$.
We start by recording  the following  version of the Lefschetz hyperplane theorem (cf.~\cite[Lemma 21.8]{mochizuki-memoirs}). 
\begin{lemma}\label{lefschetz}
If $\bar Y \subset \bar M$ is a sufficiently ample smooth divisor such that $\bar Y \cap \Sigma$  is a normal crossing divisor, then 
the map $\iota_*: \pi_1(Y ) \rightarrow \pi_1(M)$ is onto where $Y = \bar Y \backslash \Sigma$ and 
$
\iota: Y  \rightarrow M
$
is the inclusion map. 
\end{lemma}

We will use the following notation throughout this section. For 
\[
\bar Y, \  Y=\bar Y \backslash \Sigma, \  \mbox{ and } \ \iota:Y \rightarrow M
\] 
as in Lemma~\ref{lefschetz}:
\begin{itemize}
\item
$
\rho_Y: \pi_1(Y) \rightarrow \mathsf{Isom}(\tilde X)$ is the homomorphism  $\rho_{Y}:=\rho \circ \iota_*
$
\item $K = \mbox{kernel}(\iota_*:  \pi_1(Y) \rightarrow \pi_1(M))$ 
\item 
$\tilde Y$ is the universal cover of $Y$
\item $\tilde Y/K$ is the set of orbits $[p]$  of $p \in \tilde Y$ by the action of $K$
\item  
$
\hat \rho_Y:  \pi_1(Y)/K \simeq \pi_1(M)
\rightarrow \mathsf{Isom}(\tilde X)$ is the homomorphism
$\hat \rho_Y([\gamma]) \mapsto \rho_Y(\gamma)$
\end{itemize}

\begin{lemma}
The homomorphism $\rho_Y:\pi_1(Y) \rightarrow \mathsf{Isom}(\tilde X)$ is proper.
\end{lemma}

\begin{proof}
Lemma~\ref{lefschetz} implies the image of $\rho_Y$ is equal to the image of $\rho$. Thus, $\rho_Y$ is proper  since $\rho$ is proper.\end{proof}

Let $\iota^*(\tilde M) \rightarrow Y$ denote the pullback bundle of the universal cover $\tilde M \rightarrow M$.   Since $\pi_1(Y)/K \simeq \pi_1(M)$, we have the identification $\iota^*(\tilde M) \simeq \tilde Y/K$.
The commutative diagram 
\begin{equation*}
\begin{tikzcd}
  \tilde Y/K  \simeq \iota^*(\tilde M) \arrow[r, "{\iota^*}"] \arrow[d]
    &  \tilde M \arrow[d ] \\
  Y  \arrow[r, "\iota" ]
& M \end{tikzcd}
\end{equation*}
induces  an injective morphism of fiber bundles 
\begin{equation*}
\begin{tikzcd}
\tilde Y/K \times_{\hat \rho_Y} \tilde X \simeq {\iota^*}(\tilde M \times_\rho \tilde X)  \ \ \ \ \ \  \hookrightarrow \arrow[d,]
    &  \tilde M \times_\rho \tilde X \arrow[d ] \\
  Y  \arrow[r, "\iota" ]
& M
 \end{tikzcd}
\end{equation*}
Finally, by the identification 
\[
\tilde Y \times_{\rho_Y} \tilde X \rightarrow \tilde Y/K \times_{\hat \rho_Y} \tilde X, \ \ \ [(p,x)]_{\rho_Y} \mapsto [([p],x)]_{\hat \rho_Y}
\]
we obtain
\begin{equation*}
\begin{tikzcd}
 \tilde Y \times_{\rho_Y} \tilde X  \ \ \ \ \ \  \hookrightarrow \arrow[d,]
    &  \tilde M \times_\rho \tilde X \arrow[d ] \\
  Y  \arrow[r, "\iota" ]
& M
 \end{tikzcd}
\end{equation*}
Thus, we conclude that any section
 \begin{equation} \label{ustilde}
s: Y \rightarrow  \tilde Y \times_{\rho_Y} \tilde X
\end{equation}
of the fiber bundle $\tilde Y \times_{\rho_Y} \tilde X \rightarrow Y$ can be viewed as a map $Y \rightarrow  \tilde M \times_\rho \tilde X$.  Conversely, 
given a section $u:M \rightarrow \tilde M \times_\rho \tilde X$ and $Y \subset M$, its restriction    
\begin{equation} \label{restrict}
u|_Y: Y \rightarrow  \tilde M \times_\rho \tilde X
\end{equation}
defines a section $Y \rightarrow \tilde Y \times_{\rho_Y} \tilde X$ which we denote by 
\begin{equation} \label{restrictionsection}
u_Y:Y  \rightarrow 
\tilde Y \times_{\rho_Y} \tilde X.
\end{equation}

In Definition~\ref{logdec}, we defined the notion of sub-logarithmic growth for harmonic maps from a domain of complex dimension 1.
We use the identification of the restriction (\ref{restrict}) of a section from $M$  to a section (\ref{restrictionsection}) from a subvariety $Y \subset M$ to define the notion of \emph{sub-logarithmic growth} for a pluriharmonic map from a domain of arbitrary dimensions.  
To this end,
let
 $L$ be an ample line bundle over $\bar{M}$ and  $H^0(\bar M,L)$ be the space of global holomorphic sections.   Define the projectified space
\begin{equation} \label{Pb}
\Pb = \Pb(H^0(\bar M, L)^*).
\end{equation}
Let 
$\CC$ be the set of sufficiently ample smooth  curves $\bar \RR \subset \bar M$ with the following properties:  
\begin{itemize}
\item
$\RR  := \bar \RR \backslash \Sigma= \{s_1=0\} \cap \dots \cap \{s_{n-1}=0\}, \ s_1, \dots, s_{n-1} \in \Pb
$ with $\dim_\C \bar M=n$.
\item  $\bar \RR \cap \Sigma$ is a  divisor in $\bar M$ with  normal crossings. 
\end{itemize}

\begin{definition} \label{sub-log}
 We say a pluriharmonic section $u: M \rightarrow \tilde M \times_\rho  \tilde X$ has \emph{sub-logarithmic growth} if, for $\bar \RR \in \CC$, the section $u_\RR:\RR \rightarrow \tilde \RR \times_{\rho_\RR} \tilde X$ (as defined by (\ref{restrictionsection})) has sub-logarithmic growth (cf.~Definition~\ref{logdec}).  
\end{definition}

\begin{remark}
It follows from the definition that if $u:M \rightarrow \tilde M \times_\rho \tilde X$ is a pluriharmonic section of sub-logarithmic growth and  $\bar Y$ is as in Lemma~\ref{lefschetz}, then its restriction $u_Y: Y \rightarrow \tilde Y \times_{\rho_Y} \tilde X$  (as defined by (\ref{restrictionsection})) is a pluriharmonic section of sub-logarithmic growth.  
\end{remark}

\section{Proof of the Theorem~\ref{theorem:pluriharmonic}, Theorem~\ref{theorem:holomorphic}, and Theorem~\ref{theorem:sameconclusion}}
\label{proofpluriharmonic}

We will prove Theorem~\ref{theorem:pluriharmonic}   by an induction on the (complex) dimension of $M$.    
Consider the following:

\begin{quote}
{\sc statement} $(n)$:  Theorem~\ref{theorem:pluriharmonic} is true whenever $\dim_{\C} M=n$.  
\end{quote}
\subsection{Intial step of the induction} \label{dim2case}
In this subsection, we  prove {\sc statment} ($2$).  Thus, assume $\dim_\C \bar M=2$.  We have already shown that there exists a pluriharmonic section $u:M \rightarrow \tilde M\times_\rho \tilde X$ in Theorem~\ref{pluriharmonicity}.  What remains to be shown is that $u$ is the unique pluriharmonic section of sub-logarithmic growth.

\begin{lemma}
\label{uniqueH}
If $u: M \rightarrow \tilde M \times_\rho \tilde{X}$  is the  pluriharmonic section  of Theorem~\ref{pluriharmonicity} (i.e.~$\dim_C \bar M=2$), then $u$ has sub-logarithmic growth.  \end{lemma}

\begin{remark} \label{uniqueH'}
Intuitively, the proof is clear.  Near the divisor, $M$ has local holomorphic coordinates $(z^1,\zeta)$.  From the energy estimate  Theorem~\ref{weneedA}, the section $u$ is Lipschitz in the direction parallel to the divisor (in the $z^1$-direction) and has logarithmic  energy growth when restricted to the  transverse disks to the divisor (i.e.~on the holomorphic  disks given by parameter $\zeta$). Since  harmonic maps from the punctured disk have sub-logarithmic growth (cf.~Theorem~\ref{exists}),  $u$ restricted to the transverse holomorphic disks has sub-logarithmic growth.  It then follows that  the restriction of  $u$ to $\RR$ for   any   $\bar \RR \in C$ also has sub-logarithmic growth. Details are below.
\end{remark}

\begin{proof}
Let $\RR \in \CC$.  Since $u: M \rightarrow \tilde M \times_\rho \tilde X$ is pluriharmonic,  $u_\RR: \RR \rightarrow \tilde \RR \times_{\rho_\RR} \tilde X$ (as defined by (\ref{restrictionsection})) is harmonic. 
 Let $\bar \RR \cap \Sigma=\{p^1,\dots,p^n\}$ be the punctures of $\RR$.  Fix a puncture $p^k$ 
and let $V$ be a neighborhood of $p^k$.
The assumption that  $\bar \RR \cup \Sigma$ is a divisor with  normal crossings means that $p_k$ is not a juncture of $\Sigma$.  Thus,  we may assume that  $V= \Omega \times \D_{\frac{1}{4}}$   is a subset of a set $\Omega \times \D$ away from the crossing (cf.~Theorem~\ref{weneedA}) and  such that $p^k \in \Omega \times \{0\}$.

Let $(z^1,\zeta)$ be the holomorphic coordinates  of $V$ (cf.~(\ref{holcoord})).  Choose  $V_0$ to be a compactly contained subset of $V$.  More precisely,  we choose an open subset $\Omega_0$  of $\Omega$ and $\rho_0>0$ such that $p^k \in \Omega_0 \times \{0\}$ and set
\[
V_0:= \{(z^1,\zeta): \  z^1 \in \Omega_0, \zeta \in \D_{\rho_0}^*\} \subset V.
\]   
For each $z^1 \in \Omega_0$, the disk $\DD_{\rho_0}^{z^1}:=\{(z^1,\zeta) \in V_0: \ \zeta \in \D_{\rho_0} \}$ is holomorphic in $\bar M$.  Define the subsets
\[
\DD_{\rho_0}^{z^1*}:=\{(z^1,\zeta):  \ \zeta \in \D^*_{\rho_0} \}, \ \ \ \DD_{r,r_0}^{z^1*}:=\{(z^1,\zeta):  \ \zeta \in \D^*_{r,r_0} \}.
\]
Let $\D$ be a holomorphic disk of $\RR$  centered at $p:=p^k$ and $z$ be its holomorphic coordinate.   Assume  that $\D \subset \bar \RR \cap V_0$ and let 
$(z^1(z), \zeta(z))$ be the coordinates of $z \in \D^*$ as a point in $V_0$.  
Since $\bar \RR$ is transverse to $\Sigma_j$ and $\zeta(z)$ is an analytic function in $z$, there exists $m \in \N \backslash \{0\}$ and a constant $c>0$ such that
\begin{equation} \label{orderofintersection}
\frac{1}{c} \leq \left| \frac{\zeta(z)}{z^m} \right| \leq c.
\end{equation}

Define the restriction map
\[
u_{z^1}=u\big|_{\DD^{z^1*}}:\DD_{\rho_0}^{z^1*} \rightarrow \tilde \RR \times_\rho \tilde X.
\]
The pluriharmonicity of $u$ implies  the harmonicity of $u_{z^1}$ and $u_{z^1_0}$ for any $z^1, z^1_0 \in \Omega_0$.  Thus,   
the function 
\[
\zeta \mapsto d^2 (u_{z^1}(\zeta),u_{z^1_0}(\zeta)), \ \zeta \in \D^*_{\rho_0} \simeq  \DD_{\rho_0}^{z^1*} \simeq \DD_{\rho_0}^{z_0^1*}
\] 
 is a subharmonic function.  (Note that this is analogous to the function $z \mapsto d^2(v_{z^1_0}(z),v_{z^1}(z))$ defined in proof of Lemma~\ref{Aest} and Figure~\ref{fig:slices}.)

Theorem~\ref{weneedA} and Fubini's theorem, for a.e.~$z^1 \in \Omega_0$, 
that there exists a constant $C=C(z^1)>0$ such that
\[
\int_{\DD_{r,r_0}^{z^1*}} \left| \frac{\partial u}{\partial r}\right|^2 + \frac{1}{r^2}  \left| \frac{\partial u}{\partial \theta}\right|^2 \ \frac{dz^1 \wedge d\zeta}{-2i} \leq \frac{\Delta_I^2}{2\pi}\log \frac{r_0}{r}+C(z^1), \ \ 0<r<r_0<\rho_0.
\]
Thus, noting the lower bound of Lemma~\ref{lowerbd}, we conclude that 
 for a.e.~$z^1 \in \Omega_0$, 
\begin{equation} \label{aez}
\frac{\Delta_I^2}{2\pi} \log \frac{r_0}{r} \leq  E^{u_{z^1}}[\DD_{r,r_0}^{z^1*}]  \leq \frac{\Delta_I^2}{2\pi}\log \frac{r_0}{r}+C(z^1), \ \ \ 0<r<r_0\leq \rho_0.
\end{equation}
Let $\Omega_0'\subset \Omega_0$ be the set of  all $z^1 \in \Omega_0$ such that (\ref{aez}) holds.
By Theorem~\ref{exists}, $u_{z^1}$ has sub-logarithmic growth for $z^1 \in \Omega_0'$. 
In other words, $u_{z^1}$ has sub-logarithmic growth for a.e.~$z_1$.  We will show that this implies that $u_{z^1}$ has sub-logarithmic growth {\it for every }$z_1$ since $z^1 \mapsto u(z^1,\zeta)$ is Lipschitz continuous.

Indeed, given $z^1, z^1_0 \in \Omega_0'$, the triangle inequality implies that for all $\epsilon>0$,
\[
\lim_{\zeta \rightarrow 0} d^2 (u(z^1,\zeta),u(z^1_0,\zeta)) +\epsilon \log |\zeta| =-\infty.
\]
Hence, $\zeta \mapsto d^2 (u(z^1,\zeta),u(z^1_0,\zeta))$ extends as a subharmonic function on $\D_{\rho_0}$ by Lemma~\ref{shdisk}. The maximum principle therefore implies that
for $z^1, z^1_0 \in \Omega_0'$  and for $\zeta \in \D_{\frac{\rho_0}{2}}$, 
 \begin{eqnarray} \label{mp}
  d^2(u(z^1,\zeta), u (z^1_0,\zeta)) &\leq& \sup_{ \eta \in \partial \D_{\frac{\rho_0}{2}} } d^2(u(z^1,\eta), u(z^1_0, \eta))\\
&\leq& L^2 |z^1-z^1_0|^2   \nonumber
 \end{eqnarray}
where $L$ is the Lipschitz bound of $u$ in a compact set of $M$ containing $\{(z^1,\zeta):  z^1 \in \Omega_0, \ |\zeta|=\frac{\rho_0}{2} \}$.  
The continuity of $u$ implies that (\ref{mp}) is valid for any $z^1, z^1_0 \in \Omega_0$.  In other words, we have shown the Lipschitz continuity of $z^1 \mapsto u(z^1, \zeta)$.

We now finish the proof.
Fix $\epsilon>0$ and let $z^1_0 \in \Omega_0'$.  Since $u_{z^1_0}$ has sub-logarithmic growth, there exists   $\delta \in (0,\rho_0)$ be such that, with $f_0$ as in Definition~\ref{logdec}, 
\begin{equation} \label{ppl}
d^2(u(\zeta),f_0(\zeta)) \leq -\frac{\epsilon}{2cm} \log |\zeta|, \ \ \ \forall p'  \in V_0 \text{ and } |\zeta|<\delta
\end{equation}
where $c$,  $m$ as in (\ref{orderofintersection}).
For a point $p \in V_0$, write 
$
(z^1(p), \zeta(p))
$
in terms of the holomorphic coordinates $(z^1,\zeta)$ of $V$.
Furthermore, let $p' \in V_0$, write  $(z_0^1, \zeta(p'))$ with $|\zeta(p')|< \delta$.  Then
\begin{eqnarray*} 
d^2(u_\RR(p), f_0(p)) & = & d^2(u(p), f_0(p))
\nonumber \\
& \leq
&  2d^2(u(p), u(p'))+ 2d^2(u(p'), f_0(p))\ \  \ \ \mbox{(by the triangle inequality)}
\nonumber \\ 
& = & 2d^2(u(z^1(p), \zeta(p)), u(z^1_0, \zeta(p')))-\frac{\epsilon}{2cm} \log |\zeta(p)|  \ \ \ \ \ \ \mbox{(by (\ref{ppl}))}\nonumber \\
&
\leq
&  2L^2|z^1(p)-z^1_0|^2 -\frac{\epsilon}{2cm} \log |\zeta(p)| \ \ \ \ \mbox{(by (\ref{mp})}).
\end{eqnarray*}
By (\ref{orderofintersection}), 
\[
\lim_{|\zeta(p)| \rightarrow 0} d^2(u_\RR(p), P_0)+\epsilon \log |\zeta(p)| =-\infty.
\]
Since $p$ is an arbitrary puncture of $\RR$, we have shown that  $u_\RR$ has sub-logarithmic growth.
\end{proof}

For any element $s \in \Pb = \Pb(H^0(\bar M, L)^*)$, we set 
\[
{\bar Y}_s := s^{-1}(0) \ \mbox{and} \  Y_s={\bar Y}_s \backslash \Sigma
\]
and 
\[
 U= \{ s \in \Pb: {\bar Y}_s \  \mbox{smooth} \ \mbox{and} \ {\bar Y}_s \cap \Sigma \  \mbox{ is a normal crossing} \}. 
 \]
 For $q \in M$, consider  the subspace
\[
V(q) =\{s \in H^0(\bar M, L): s(q)=0\}, \ \ \ \Pb(q) := \Pb(V(q)^*)
\]
 and
 \[
 U(q) = U \cap \Pb(q).
 \]
 The sets $U$, $U(q)$  are Zariski open subsets of $\Pb$, $\Pb(q)$ respectively.

\begin{lemma}
\label{unique!}
The pluriharmonic section $u: M \rightarrow  \tilde M \times_\rho \tilde X$ of  Theorem~\ref{pluriharmonicity} (i.e.~$\tilde M$ is a complex surface) is unique amongst all sections $M \rightarrow \tilde M \times_\rho \tilde X$ with sub-logarithmic growth.
\end{lemma}

\begin{proof}
Let $v: M \rightarrow  \tilde M \times_\rho \tilde X$ also be a  pluriharmonic section with sub-logarithmic growth.  For $q \in M$ and $s \in  U(q)$, 
the restrictions $u_{Y_s}: Y_s \rightarrow \tilde Y_s \times_{\rho_{Y_s}} \tilde X$ and $v_{Y_s}: Y_s \rightarrow \tilde Y_s \times_{\rho_{Y_s}} \tilde X$ (as in  (\ref{restrictionsection})) have sub-logarithmic growth.    
By  Theorem~\ref{uniqueness}, $u_{Y_s}=v_{Y_s}$.   Since  $q$ is an arbitrary point in $M$, we conclude $u=v$.
\end{proof}

We observe that {\sc statement} $(2)$ follows from Theorem~\ref{theorem:pluriharmonicDim2}, Theorem~\ref{pluriharmonicity},  Lemma~\ref{uniqueH} and Lemma~\ref{unique!}.

 \subsection{Inductive step} \label{MT}

The goal of this subsection is to 
assume {\sc statement}  $(2)$, $\dots$, {\sc statement}  $(n-1)$,    then  to prove {\sc statement} ($n$).

To this end, we start with the following:
For each   $s \in \Pb$, let 
\[
u^s: Y_s \rightarrow \tilde Y_s \times_{\rho_{Y_s}} \tilde X \simeq \iota^*(\tilde M \times_\rho \tilde X) \hookrightarrow \tilde M \times_\rho \tilde X
\]
be the pluriharmonic section of logarithmic growth whose existence is guaranteed by the inductive hypothesis.

\begin{lemma} \label{welldefined}   
For $q \in M$ and $s_1, s_2 \in U(q)$, we have
$
u^{s_1}(q) = u^{s_1}(q).
$
\end{lemma}
\begin{proof}
For  $i=1,2$ and $q \in M$, let 
 \[ 
 U(s_i, q)= \{s \in U(q) :  {{\bar Y}_s}\ \mbox{transversal to} \ {\bar Y}_{s_i} \ \mbox{and} \ \Sigma \cap {\bar Y}_s\cap {\bar Y}_{s_i} \ \mbox{normal crossing}\}.
 \]
 The set $U(s_i, q)$ is Zariski open in $U(q)$ which implies $U(s_1, q) \cap U(s_2, q) \neq \emptyset$.  Fix $s\in U(s_1, q) \cap U(s_2, q)$.
Let $\iota:Y_{s_i} \cap Y_s \rightarrow M$ be the inclusion map.  Consider the pluriharmonic sections of sub-logarithmic growth
 \[
u^s:  Y_s \rightarrow \tilde Y_s \times_{\rho_{Y_s}}  \tilde X \ \mbox{ and } \ 
 \tilde u^{s_i}:  Y_{s_i} \rightarrow \tilde Y_{s_i} \times_{Y_{s_i}} \tilde X
 \] 
and their restriction 
 \[
u^s_{Y_{s_i} \cap Y_s}: Y_{s_i} \cap Y_s \rightarrow {\widetilde{Y_{s_i} \cap Y_s}} \times_{\rho_{Y_{s_i} \cap Y_s}}  \tilde X
  \]
  and 
   \[
u^{s_i}_{Y_{s_i} \cap Y_s}: Y_{s_i} \cap Y_s \rightarrow {\widetilde{Y_{s_i} \cap Y_s}} \times_{\rho_{Y_{s_i} \cap Y_s}}  \tilde X.
   \]
Since these restrictions are both pluriharmonic sections of sub-logarithmic growth by {\sc statement} $(n-2)$, they  are in fact the same section; i.e.~the sections $u^s$ and $u^{s_i}$ restricted to $Y_{s_i} \cap Y_s$ is equal equal.
Since $q \in Y_{s_i} \cap Y_s$, we conclude $u^{s_i}(q) =u^s(q)$.    \end{proof}

By Lemma~\ref{welldefined}, we can define
\[
 u:M \rightarrow \tilde M \times_\rho \tilde X, \ \ u(q):=u_s(q) \text{ for } s\in U(q).
\]
To complete the inductive step, we are left to  show that  $u$  is  a pluriharmonic section with sub-logarithmic growth and moreover unique amongst pluriharmonic sections with sub-logarithmic growth. 
Given $q \in M$ and any complex 1-dimensional subspace $\PP \subset T_q^{1,0}(M)$, we consider the non-empty algebraic set
\[
A_\PP(q) = \{s \in \Pb(q): \PP \subset T_q^{1,0}(Y_s) \}
\]
and its Zariski open subset 
\[
U_\PP(q)= A_\PP(q) \cap U(q).
\]
Let $s\in U_\PP(q)$.  By the construction of $u$, its restriction $u_{Y_s}$ is the  unique pluriharmonic section of sub-logarithmic growth. Thus,  all derivatives of $u|_{Y_s}$ exist in the direction of $\PP$ and 
\[
\partial_{E'} \bar \partial|_\PP u=\partial_{E'} \bar \partial|_\PP u_s=0,
\]
thereby proving that $u$ is  pluriharmonic with sub-logarithmic growth.  

To prove uniqueness of $u$, let $v:  M \rightarrow \tilde M \times_\rho \tilde X$ be a pluriharmonic section of sub-logarithmic growth.  For $q \in M$, let  $s \in U(q)$.  
The restriction $v_{Y_s}$ (as in (\ref{restrictionsection})) is a pluriharmonic map of sub-logarithmic growth.     By the uniqueness assertion in  {\sc statement ($n-1$)}, we conclude $u_{Y_s}=v_{Y_s}$.   Since  $q$ is an arbitrary point in $M$, we conclude $u=v$.  This completes the proof of Theorem~\ref{theorem:pluriharmonic}.

\subsection{Other cases}
We include the proofs of Theorem~\ref{theorem:holomorphic} and Theorem~\ref{theorem:sameconclusion} in this subsection.\\

\begin{proofTheorem2}
Follows from an argument exactly as in \cite{siu1}.  
\end{proofTheorem2}

\begin{proofTheorem3}
The only difference between the proof of Theorem~\ref{theorem:pluriharmonic}
and this case is that the uniqueness of $\rho$-equivariant harmonic maps here follows from Theorem~\ref{uniqueness} (i) instead of Theorem~\ref{uniqueness} (ii).
\end{proofTheorem3}

\chapter{Euclidean buildings and Teichm\"uller spaces} \label{chap:Euclidean buildings}

In this chapter, we consider harmonic maps into  a locally finite Euclidean building $\tilde X$ and the Weil-Petersson completion of Teichm\"uller space.  
Both these spaces are singular spaces, and 
 harmonic maps into $\tilde X$ may have singular points.

\section{Singlar sets of harmonic maps}

\label{sec:singularsets}
\begin{definition}
Let $u:\Omega \rightarrow \tilde X$  be a harmonic map from a Riemannian domain  into an NPC space $\tilde X$.  A point $x \in \Omega$ is said to be a {\it regular point} if there exists a neighborhood ${\mathcal N}$ of $x$ and a totally geodesic subspace  $M \subset \tilde X$ isometric to a smooth Riemannian manifold such that $u(\NN) \subset M$.
A {\it singular point} of $u$ is a point of $\Omega$ that is not a regular point.   
A regular (resp.~singular) set of $u$ is the set of all regular (resp.~singular) points of $u$.
\end{definition}

We adopt the following definition from \cite[Section 7]{gromov-schoen}.

\begin{definition}
A harmonic map $u: \Omega \rightarrow  \tilde X$   is said to be {\it pluriharmonic} provided it is pluriharmonic in the usual sense (i.e. $\partial_{E'} \bar \partial u=0$ as in Section~\ref{KtoK}) on the regular set.
\end{definition}

In Section~\ref{sec:pluriharmonicity}, we proved that a harmonic map  into a smooth manifold is pluriharmonic.  The idea was to   
multiply the Bochner formula by a smooth function $\chi_N$ with support away from the divisor.  We then apply integration by parts  to obtain equalities (\ref{multiply}) and (\ref{multiply'}).  For a harmonic map into a singular space considered here,  we need to 
 justify integration by parts in the presence of  its singular set.  To do so, we need the following cut-off functions.

Recall that a $k$-dimensional Euclidean building $\tilde X$ is the union of $k$-flats.   The following theorem due to Gromov and Schoen (cf.~\cite[Theorem 6.4]{gromov-schoen} asserts that away from a small closed set,  a harmonic map  into a locally finite Euclidean building maps is locally a harmonic map into a $k$-flat.

\begin{theorem} \label{gs}
If $u: \Omega \rightarrow \tilde X$   is a harmonic map into a locally finite Euclidean building, then  the the singular set ${\mathcal S}(u)$ is a closed set of Hausdorff codimension  greater or equal to 2.
  For any compact subdomain $\Omega_1$ of $\Omega$, there is a sequence 
of Lipschitz functions $\{\psi_i\}$ with $\psi_i \equiv 0$ in a neighborhood of ${\mathcal S}(u) \cap \bar \Omega_1$, $0 \leq \psi_i \leq 1 $ and $\psi_i(x) \rightarrow 1$ for all $x \in \Omega_1 \backslash {\mathcal S}(u)$ such that
\[
\lim_{i \rightarrow \infty} \int_\Omega|\nabla \psi_i| d\mu =0
\]
and 
\[
\lim_{i \rightarrow \infty}\int |\nabla \nabla u| \, | \nabla \psi_i| d\mu =0.
\]
\end{theorem}

 \begin{proof}
The assertions of this theorem except for the first equality are contained in \cite[Theorem 6.4]{gromov-schoen}.    
The first equality  can  proved by following the proof of \cite[Theorem 6.4]{gromov-schoen}. We omit the details.
 \end{proof}
 
 \begin{lemma} \label{Bintegrable}
For $u$ as in Theorem~\ref{gs},
 \[
\int \partial \bar \partial \{ \bar \partial u, \bar \partial u\}  \chi_N 
= \int   \{ \bar \partial u, \bar \partial u\} \wedge \bar \partial \partial \chi_N.
\]
\end{lemma}

\begin{proof}
With $\psi_i$ as in Theorem~\ref{gs}, we have
 \begin{eqnarray*}
\lefteqn{ \int \partial \bar \partial \{ \bar \partial u, \bar \partial u\}  \chi_N \psi_i }
\\
& = & \int \bar \partial \{ \bar \partial u, \bar \partial u\} \wedge \partial (\chi_N \psi_i)
\\
& = & \int (\bar \partial \{ \bar \partial u, \bar \partial u\} \wedge \partial \chi_N)  \psi_i + \int (\bar \partial \{ \bar \partial u, \bar \partial u\} \wedge \partial \psi_i)  \chi_N
\\
& = & -\int  \left( \{ \bar \partial u, \bar \partial u\} \wedge \bar \partial \partial \chi_N \right)  \psi_i + \int \{ \bar \partial u, \bar \partial u\} \wedge \partial \chi_N \wedge \bar \partial \psi_i + \int \left( \bar \partial \{ \bar \partial u, \bar \partial u\} \wedge \partial \psi_i  \right) \chi_N.
 \end{eqnarray*}
Furthermore, there exists a constant $C>0$ depending only on the Lipschitz constant of $u$ and $\chi_N$   such that 
 \begin{eqnarray*}
 \left|  \int \{ \bar \partial u, \bar \partial u\} \partial \chi_N \wedge \bar \partial \psi_i \right| & \leq & C\int |\nabla \psi_i| \omega^2,
 \end{eqnarray*}
 
 \begin{eqnarray*}
 \left| \int \left( \bar \partial \{ \bar \partial u, \bar \partial u\} \wedge \partial \psi_i  \right) \chi_N \right| \leq  C\int |\nabla \nabla u| |\nabla \psi_i| \omega^2.
 \end{eqnarray*}
Thus, the assertion follows from letting $i \rightarrow \infty$ and applying Theorem~\ref{gs}.
 \end{proof}
 
 \begin{lemma} \label{Bpluri}
For $u$ as in Theorem~\ref{gs},

 \[
 - \int_M \chi_N d  \{ \bar\partial_E \partial u, \partial u -\bar{\partial} u\} = 
 \int_M d \chi_N \wedge  \{ \bar\partial_E \partial u, \partial u -\bar{\partial} u\}.
 \]
\end{lemma}

\begin{proof}
With $\psi_i$ as in Theorem~\ref{gs}, we have
\begin{eqnarray*}
\lefteqn{ - \int_M \chi_N \psi_i d  \{ \bar\partial_E \partial u, \partial u -\bar{\partial} u\} }\\
& = &  
 \int_M \psi_i d \chi_N \wedge  \{ \bar\partial_E \partial u, \partial u -\bar{\partial} u\} +  \int_M \chi_N d\psi_i  \wedge  \{ \bar\partial_E \partial u, \partial u -\bar{\partial} u\}.
\end{eqnarray*}
Thus,  there exists a constant $C>0$ depending only on the Lipschitz constant of $u$  in the support of $\chi_N$ such that
\[
\left| \int_M \chi_N d\psi_i  \wedge  \{ \bar\partial_E \partial u, \partial u -\bar{\partial} u\} \right| \leq  C \int |\nabla \nabla u| \, |\nabla \psi|,
 \]
 the assertion follows from letting $i \rightarrow \infty$ and applying Theorem~\ref{gs}.
\end{proof}

\section{Proof of  Theorem~\ref{theorem:buildings} and Theorem~\ref{theorem:teichmuller}}
\label{proofofthm4}
\begin{proofTheorem4}
First assume $\dim_C M =2$.  To
 apply Theorem~\ref{theorem:pluriharmonicDim2}, we use the fact that the isometry group of a Euclidean building without a Euclidean factor consists of only elliptic and hyperbolic elements (cf.~\cite[Theorem 4.1]{parreau}). 
Therefore, any element (resp.~commuting pair) of the isometry group has exponential decay (cf.~Definition~\ref{jz} and Definition~\ref{jz2}).
We can thus invoke Theorem~\ref{theorem:pluriharmonicDim2} to conclude that, with the Poincar\'e-type K\"ahler metric $g$  on $M$ (cf.~Section~\ref{poincarekahler}), there exists    a   harmonic section 
$
u: M \rightarrow \tilde M \times_\rho \tilde{X}
$
 of logarithmic energy growth.

Next, we follow the proof of inequality (\ref{ibp}) to prove
\[
\lim_{N \rightarrow \infty}
\int \partial \bar \partial \chi_N \wedge  \{ \bar \partial u, \bar \partial u\} < \infty.
\]  
Consequently, Lemma~\ref{Bintegrable} implies
\[
\int    \bar \partial \partial \{ \bar \partial u, \bar \partial u\}=
\lim_{N \rightarrow \infty} \int    \bar \partial \partial \{ \bar \partial u, \bar \partial u\} \chi_N<\infty.
\]
Since Siu's Bochner formula (cf.~Theorem~\ref{siubochner}) holds for $u$ in the regular set which is of full measure,  the conclusion of  Proposition~\ref{integrable} holds.  Consequently, we can now  apply the proof of (\ref{limitis0}) to show
 that 
 \[
 \lim_{N \rightarrow \infty}  \int_V d \chi_N \wedge  \{\bar\partial_{E'}\partial u, \partial u -\bar{\partial} u\}   = 
0.
\]
By Lemma~\ref{Bpluri}, 
\[
\int_M d  \{ \bar\partial_E \partial u, \partial u -\bar{\partial} u\} =\lim_{N \rightarrow \infty} \int_M \chi_N d  \{ \bar\partial_E \partial u, \partial u -\bar{\partial} u\}=0.
\]
 Since the variation of the Siu's Bochner formula of Theorem~\ref{mochizukibochnerformula} holds for $u$ in the regular set, we conclude  that $\left|{\partial}_{E'} \bar \partial u\right|=0$ on the regular set of $u$.  In other words, $u$ is pluriharmonic.

 The rest of the proof follows exactly as in the case when $\tilde X$ is a symmetric space.  Namely, the uniqueness follows from the same argument as Lemma~\ref{uniqueH} and Lemma~\ref{unique!}.  Finally,   the  inductive argument of Chapter~\ref{chap:higher dimensions} proves the conclusion of Theorem~\ref{theorem:pluriharmonic} for $\dim_C M >2$.
\end{proofTheorem4}

We next prove Theorem~\ref{theorem:teichmuller}.
Set theoretically, recall that the Weil-Petersson completion $\overline{\mathcal T}$ of the Teichm\"uller space $\mathcal T$  is the augmented Teichmüller space, i.e.~$\overline{\mathcal T}$ is obtained by adding points corresponding to nodal Riemann surfaces to $\mathcal T$. Its boundary $\partial \mathcal T$ can be stratified by smooth open strata of lower dimensional Teichm\"uller spaces.  In other words, $\overline{\mathcal T}$ is a stratified space (with the original Teichmüller space $\mathcal T$ being the top dimensional open stratum).\\

\begin{proofTheorem5}
We  repeat the proof of Theorem~\ref{theorem:buildings} but with the following three differences:
\begin{itemize}
\item The isometry group $\mathsf{Isom}(\overline{\mathcal T})$  consists of only elliptic and hyperbolic elements by Thurston's classification of mapping classes and \cite{daskal-wentworth}.   (Also see \cite[Theorem A]{bridson}.)   
\item   We  use our regularity theorem for harmonic maps into $\overline{\mathcal T}$ \cite[Theorem 1.5 and Theorem 1.6]{daskal-meseHR} instead of the Gromov-Schoen regularity theorem for Euclidean buildings (i.e.~\cite[Theorem 6.4]{gromov-schoen} or Theorem~\ref{gs}).  Consequently, Lemma~\ref{Bintegrable} and Lemma~\ref{Bpluri} hold for harmonic map $u:\Omega \rightarrow \overline{\mathcal T}$.
\item  The assumption that $\rho$ does not fix an unbounded closed convex strict subset is not needed. Indeed, in the proof of Theorem~\ref{unique!}, we  apply part (i)  instead of part (ii) of Theorem~\ref{uniqueness} to prove the uniqueness  of harmonic maps with sub-logarithmic growth. This is possible because the Weil-Petersson metric has negative sectional curvature.
\end{itemize}
\end{proofTheorem5}


\begin{thebibliography}{ABC}
\label{references}

\bibitem[B]{bridson} M.~R.~Bridson, {\it Semisimple actions of mapping class groups on CAT(0)-spaces}.  ``The geometry of Riemann surfaces" London Math.~Soc.~Lecture Notes 368, Cambridge University Press (2010) 1-14.

\bibitem[BH]{bridson-haefliger}  M.~R.~Bridson and A.~Haefliger.  {\it Metric Spaces of Non-Positive Curvature}.  Springer-Verlag, Berlin 1999.



\bibitem[Co1]{corlette} K.~Corlette. {\it Archimedian superrigidity and hyperbolic geometry.} Ann. Math. 135 (1990) 165-182.

\bibitem[Co2]{corlette2}  K.~Corlette. {\it Flat $G$-bundles with anonical metrics}.   J.~Differential Geom. 28 (1988) 361-382.

\bibitem[CoGr]{cornalba-griffiths} M. Cornalba and P. Griffiths.  {\it Analytic cycles and vector bundles on noncompact algebraic varieties.} Invent.
Math. 28 (1975) 1-106.








\bibitem[DM1]{daskal-meseDM} G. Daskalopoulos and C. Mese.  {\it On the Singular Set of Harmonic Maps into DM-Complexes.} Memoirs of the AMS (2016) vol. 239.  

\bibitem[DM2]{daskal-meseER}  G. Daskalopoulos and C. Mese.  {\it Essential Regularity of the Model Space for the Weil-Petersson Metric}.  Journal f\"ur die reine und angewandte Mathematik 750 (2019) 53-96.

\bibitem[DM3]{daskal-meseC1} G. Daskalopoulos and C. Mese.
 {\it $C^1$-estimates for the Weil-Petersson metric.} Trans.~AMS   369 (2017) 2917-2950.

\bibitem[DM4]{daskal-meseHR}
G. Daskalopoulos, C. Mese.  {\it Rigidity of Teichm\"uller Space.}  Invent.~Math. 224 (2021) 791- 916.

\bibitem[DM5]{daskal-meseUnique}
G. Daskalopoulos, C. Mese.  {\it Uniqueness of equivariant harmonic maps to symmetric spaces and buildings.}  Submitted for publication.  arXiv:2111.11422


 
\bibitem[DMV]{daskal-meseGAFA} G.~Daskalopoulos, C.~Mese and A.~Vdovina. {\it Superrigidity of Hyperbolic Buildings.} GAFA 21 (2011) 1-15.

\bibitem[DW]{daskal-wentworth} G. Daskalopoulos and R. Wentworth. {\it Classification of Weil-Petersson isometries} Amer. J. Math. 125 (2003) 941-975. 



\bibitem[Do]{donaldson} S.~Donaldson. {\it Twisted harmonic maps and the self-duality equations.} Proc.~London Math.~Soc.~55 (1987) 127-131.



\bibitem[Eb]{eberlein2} P.~B.~Eberlein.  {\it Structure of manifolds of  nonpositive curvature}. Global geometry and global analysis, Lecture Notes in Math., vol.~1156.  Springer-Verlag, Berlin, 1984, 86-153.

\bibitem[ES]{eells-sampson}
J.~Eells, Jr. and J.~H.~Sampson.  {\it Harmonic Mappings of Riemannian Manifolds.}
American Journal of Mathematics
86 (1964) 109-160.

\bibitem[FSY]{fsy}
K.~Fujiwara,
T.~Shioya, and 
S.~Yamagata.  {\it Parabolic isometries of CAT(0) spaces and CAT(0) dimensions}.  Algebraic \& Geometric Topology 4 (2004) 861-892.



\bibitem[GS]{gromov-schoen} M. Gromov and R. Schoen. {\it Harmonic maps into singular
spaces and $p$-adic superrigidity for lattices in groups of rank
one.} Publ. Math. IHES 76 (1992)  165-246.


\bibitem[Ha]{hartman} P.~Hartman.  {\it On homotopic harmonic maps.} Canad.~J.~Math.19 (1967) 673–687.

\bibitem[HK]{hayman-kennedy} W.~K.~Hayman and P.~B.~Kennedy, Subharmonic functions Vol. 1, London Mathematical Society Monographs No 9, Academic Press, London, 1976.


\bibitem[Iv]{ivanov} N. Ivanov. {\it Mapping Class Groups.} In Handbook of geometric topology.  North-Holland, Amsterdam,  (2002) 523-633. 

\bibitem[J]{jost}  J. Jost.  {\it Nonpositive curvature:
geometric and analytic aspects.}  Lectures in Mathematics.  ETH
Z\"{u}rich, Birkh\"{a}user Verlag 1997.




\bibitem[JY1]{jost-yau} J.~Jost and S.~T.~Yau.  {\it  Harmonic maps and superrigidity}.   Differential Geometry: Partial Differential Equations on Manifolds, Proc. Sympos. Pure Math. 54, part 1 (1993), 245-280.

\bibitem[JY2]{jost-yau2} J.~Jost and S.~T.~Yau. {\it On the rigidity of certain discrete groups and algebraic varieties.} Math.~Ann 278 (1987) 481-496.

\bibitem[JZ]{jost-zuo} J. Jost and K. Zuo.  {\it Harmonic maps of infinite energy and rigidity results for representations of fundamental groups of quasi-projective varieties.} J. of Diff. Geom. 47 (1997) 469-503.



\bibitem[Ko]{koba}  S. Kobayashi.  {\it
Differential geometry of complex vector bundles. }  Princeton University Press (1987).

\bibitem[KS1]{korevaar-schoen1} N.~Korevaar and R.~Schoen.  {\it Global existence theorems for harmonic maps to non-locally compact spaces.} Comm.~Anal.~Geom. 5 (1997) 213-266.

\bibitem[KS2]{korevaar-schoen2}  N. Korevaar and R. Schoen.  {\it
Global existence theorem for harmonic maps to non-locally compact spaces.}  Comm.  Anal. Geom. 5 (1997), 333-387.

\bibitem[KS3]{korevaar-schoen3}  N. Korevaar and R. Schoen.  {\it
Global existence theorem for harmonic maps:  finite rank spaces and an approach to rigidity for smooth actions.}  Unpublished manuscript.


\bibitem[L]{labourie} F.~Labourie.  {\it  Existence d’applications harmoniques tordues a valeurs dans lesvarietes a courbure negative.} Proc.~Amer.~Math.~Soc. 111 (1991) 877-882.

\bibitem[Lo]{loustau} B.~Loustau.  {\it 
Harmonic maps from Kähler manifolds.}
Preprint: arXiv:2010.03545

\bibitem[LY]{liu-yang} K.~Liu and X.~Yang.  {\it Hermitian harmonic maps and non-degenerate curvatures}.  Mathematical Research Letters 21 (2014) 831-862.

\bibitem[Loh]{lohkamp} J. Lohkamp. {\it An existence theorem for harmonic maps}. Manuscripta Math. 67 (1990), 21-23. 

\bibitem[Me]{mese} C.~Mese.  {\it Uniqueness theorems for harmonic maps into metric spaces.} Communications in Contemporary Mathematics 4 (2002) 725-750.



\bibitem[M]{mochizuki-memoirs}  T.~Mochizuki.  {\it Asymptotic behaviour of tame harmonic bundles and an application to pure twistor D-modules}. Memoirs of the AMS 185 (2007). 




\bibitem[MSY]{mok-siu-yeung} N. Mok, Y.-T. Siu, and S.~K. Yeung, {\it Geometric superrigidity} Invent.~Math. 113 (1993) 57–83.


\bibitem[Mo]{mostow} G.~D.~Mostow.  {\it Strong Rigidity of Locally Symmetric Spaces}.   Annals of Mathematical Studies.
Princeton University Press,  1973.




\bibitem[McP]{papa} J. McCarthy and A. Papadopoulos. {\it  Dynamics on Thurston's sphere of projective measured foliations}. Comment. Math. Helv.~64 (1989) 133-166.

\bibitem[Pa]{parreau} A.~Parreau.  {\it Immeubles affines: construction par les normes et \'etude des
isom\'etries, Crystallographic groups and their generalizations} (Kortrijk, 1999), Contemp. Math., vol. 262, Amer. Math. Soc., Providence, RI, 2000, pp. 263–302. 


\bibitem[RC]{ramos-cuevas} C.~Ramos-Cuevas.  {\it On the displacement function of isometries of Euclidean buildings}.  Indagationes Mathematicae 26 (2015) 355-362.

\bibitem[Re]{reshetnyak} Y.~G.~Reshetnyak.  {\it Nonexpanding maps in a space of curvature no greater than K}.  Siberian Math.~Journ.~9 (1968) 918-927.

\bibitem[Sa]{sampson} J.~H.~Sampson. {\it Harmonic maps in K\"ahler geometry. Harmonic mappings and minimal
immersions,} 193–205, Lecture Notes in Math., 1161, Springer, Berlin, 1985.




\bibitem[Sch]{schumacher}  G. Schumacher. {\it The curvature of the Petersson-Weil metric on the moduli space of K\"{a}hler-Einstein manifolds.} In Complex analysis and geometry, Univ. Ser. Math. (1993) 339-354, Plenum, New York.


\bibitem[Siu1]{siu1} Y.-T. Siu. {\it The complex analyticity of harmonic maps and the strong rigidity of
compact K\"{a}hler manifolds}. Ann.~of~Math. 112 (1980) 73-111.

\bibitem[Siu2]{siu2} Y.-T. Siu. {\it Strong rigidity for K\"{a}hler manifolds and the construction of bounded holomorphic functions}. Discrete groups in  Geometry and Analysis: Papers in honor of G.D. Mostow on his sixtieth birthday. Howe, R (ed) 124-151, Birkh\"auser (1987).



\bibitem[Su]{sun} X.~Sun. {\it Regularity of Harmonic Maps to Trees.}
Amer.~J.~Math 
125 (2003) 737-771. 


\bibitem[W1]{wei1} A. Weil.  {\it On discrete subgroups of Lie groups}. Annals of Mathematics. Second Series 72 (1960),  369-384.  

\bibitem[W2]{wei2} A. Weil.  {\it On discrete subgroups of Lie groups. II}.  Annals of Mathematics. Second Series 75 (1962), 578-602.

\bibitem[Wo]{wolf} M. Wolf. {\it Infinite energy harmonic maps and degeneration of hyperbolic surfaces in moduli spaces}. J. of Diff. Geom. 33 (1991) 487-539.

\bibitem[Z]{zuo}  K. Zuo.  {\it
Representations of fundamental groups of algebraic varieties.}  Lecture Notes in Mathematics 1708, Springer 1999.
\end{thebibliography}
\end{document}